\documentclass[12pt]{amsart}
\usepackage{pgf,tikz,pgfplots} 
\pgfplotsset{compat=1.14}
\usepackage{mathrsfs}
\usetikzlibrary{arrows}
\usetikzlibrary{patterns}
\usepackage{pgfplots}
\usetikzlibrary{intersections, pgfplots.fillbetween}
\usepackage[colorlinks=true,urlcolor=blue, citecolor=red,linkcolor=blue,linktocpage,pdfpagelabels, bookmarksnumbered,bookmarksopen]{hyperref}
\usepackage[hyperpageref]{backref}
\usepackage{cleveref}
\usepackage{xcolor}
\usepackage{amsthm} 
\usepackage{latexsym,amsmath,amssymb}
\usepackage{accents}
\usepackage[colorinlistoftodos,prependcaption,textsize=tiny]{todonotes}
\usepackage{a4wide}
\usepackage{soul}
\usepackage{mathtools} 
\usepackage{xparse} 

\pgfdeclarelayer{ft}
\pgfdeclarelayer{bg}
\pgfsetlayers{bg,main,ft}

\title[Nowhere smooth critical points of polyconvex functionals]{
Nowhere smooth critical points of polyconvex functionals in arbitrary dimension}

\author{Katarzyna Mazowiecka}
\address[Katarzyna Mazowiecka]{
Institute of Mathematics,%
University of Warsaw,
Banacha 2,
02-097 Warszawa, Poland}
\email{k.mazowiecka@mimuw.edu.pl}
\author{Armin Schikorra}
\address[Armin Schikorra]{Department of Mathematics,
University of Pittsburgh,
301 Thackeray Hall,
Pittsburgh, PA 15260, USA}
\email{armin@pitt.edu}
\definecolor{indigo}{rgb}{0.29, 0.0, 0.51}
\definecolor{p1}{gray}{0.4}
\definecolor{p2}{gray}{0.6}
\definecolor{p3}{gray}{0.98}
\definecolor{p4}{gray}{0.8}
\definecolor{p5}{gray}{0.9}

plus 6pt minus 12pt
\newcommand{\Aff}{{\rm Aff}}

\def\eps{\varepsilon}

\def\im{{\rm im\, }}
\def\span{{\rm span\, }}
\renewcommand{\ul}[1]{\underline{#1}}

\def\B{{B}}

\def\N{{\mathbb N}}

\def\H{{\mathcal H}}

\def\S{{\mathbb S}}

\renewcommand{\div}{{\rm div}}

\newtheorem{theorem}{Theorem}
\newtheorem{lemma}[theorem]{Lemma}

\newtheorem{proposition}[theorem]{Proposition}

\newtheorem{remark}[theorem]{Remark}
\newtheorem{definition}[theorem]{Definition}

\def\diam{{\rm diam\,}}
\def\dist{{\rm dist\,}}

\def\curl{{\rm curl\,}}
\def\lip{{\rm Lip\,}}
\def\rank{{\rm rank\,}}

\newcommand{\R}{\mathbb{R}}

\newcommand{\brac}[1]{\left (#1 \right )}
\newcommand{\abs}[1]{\left |#1 \right |}
\newcommand{\Ep}{\bigwedge\nolimits}

\newcommand{\barint}{
\rule[.036in]{.12in}{.009in}\kern-.16in \displaystyle\int }

\newcommand{\barcal}{\mbox{$ \rule[.036in]{.11in}{.007in}\kern-.128in\int $}}

\newcommand{\hdg}{\star}

\def\mvint_#1{\mathchoice
          {\mathop{\vrule width 6pt height 3 pt depth -2.5pt
                  \kern -8pt \intop}\nolimits_{\kern -3pt #1}}%
          {\mathop{\vrule width 5pt height 3 pt depth -2.6pt
                  \kern -6pt \intop}\nolimits_{#1}}%
          {\mathop{\vrule width 5pt height 3 pt depth -2.6pt
                  \kern -6pt \intop}\nolimits_{#1}}%
          {\mathop{\vrule width 5pt height 3 pt depth -2.6pt
                  \kern -6pt \intop}\nolimits_{#1}}}

\numberwithin{theorem}{section} \numberwithin{equation}{section}

\newcommand{\aleq}{\precsim}

\newcommand{\aeq}{\asymp}

\let\latexchi\chi
\makeatletter
\renewcommand\chi{\@ifnextchar_\sub@chi\latexchi}
\newcommand{\sub@chi}[2]{
  \@ifnextchar^{\subsup@chi{#2}}{\latexchi_{#2}}%
}
\newcommand{\subsup@chi}[3]{
  \latexchi_{#1}^{#3}%
}
\makeatother

\makeatletter
\def\tikz@arc@opt[#1]{
  {%
    \tikzset{every arc/.try,#1}%
    \pgfkeysgetvalue{/tikz/start angle}\tikz@s
    \pgfkeysgetvalue{/tikz/end angle}\tikz@e
    \pgfkeysgetvalue{/tikz/delta angle}\tikz@d
    \ifx\tikz@s\pgfutil@empty%
      \pgfmathsetmacro\tikz@s{\tikz@e-\tikz@d}
    \else
      \ifx\tikz@e\pgfutil@empty%
        \pgfmathsetmacro\tikz@e{\tikz@s+\tikz@d}
      \fi%
    \fi
    \tikz@arc@moveto
    \xdef\pgf@marshal{\noexpand%
    \tikz@do@arc{\tikz@s}{\tikz@e}
      {\pgfkeysvalueof{/tikz/x radius}}
      {\pgfkeysvalueof{/tikz/y radius}}}%
  }%
  \pgf@marshal%
  \tikz@arcfinal%
}
\let\tikz@arc@moveto\relax
\def\tikz@arc@movetolineto#1{%
  \def\tikz@arc@moveto{\tikz@@@parse@polar{\tikz@arc@@movetolineto#1}(\tikz@s:\pgfkeysvalueof{/tikz/x radius} and \pgfkeysvalueof{/tikz/y radius})}}
\def\tikz@arc@@movetolineto#1#2{#1{\pgfpointadd{#2}{\tikz@last@position@saved}}}
\tikzset{%
  move to start/.code=\tikz@arc@movetolineto\pgfpathmoveto,%
  line to start/.code=\tikz@arc@movetolineto\pgfpathlineto}
\makeatother

\newcommand{\tnn}{\aleph} 

\newcommand{\tnell}{\mathfrak{l}}
\newcommand{\tni}{\mathfrak{i}}
\newcommand{\tnj}{\mathfrak{j}}
\newcommand{\tnk}{\mathfrak{k}}
\newcommand{\tnl}{\mathfrak{l}}

\begin{document}
\begin{abstract}
For any $M, n \geq 2$ and any open set $\Omega \subset \mathbb{R}^n$ we find a smooth, strongly polyconvex function $F\colon \mathbb{R}^{M\times n}\to \mathbb{R}$ and a Lipschitz map $u\colon \mathbb{R}^n \to \mathbb{R}^M$ that is a weak local minimizer of the energy
\[
 \int_{\Omega} F(Du).
\]
but with nowhere continuous partial derivatives. This extends celebrated results by M\"uller--\u{S}ver\'ak \cite{MS} and Sz\'ekelyhidi \cite{Sz04} to higher dimensions.
\end{abstract}

\sloppy

\maketitle
\tableofcontents
\sloppy

\section{Introduction}
Hilbert's 19th problem, resolved by De Giorgi \cite{DG57} and Nash \cite{Nash58}, asserts that maps $u \in W^{1,2}(\R^n,\R)$ that locally minimize a functional $I(u) \coloneqq \int_{\Omega} F(Du)$ with uniformly convex integrand $F$ are actually smooth. This sort of result does not hold for systems, i.e. when considering vector-valued $u\colon \R^n \to \R^M$, \cite{DG68,GiustiMiranda68,Mooney20}. Evans in \cite{Eva86} and Kristensen--Taheri in \cite{KT03}, showed that \emph{partial} regularity results for \emph{strong local minimizers} are still possible for systems $u\colon \R^n \to \R^M$ if $F$ is strongly quasiconvex.

In striking contrast, M\"uller-\u{S}ver\'ak \cite{MS} developed convex integration techniques to construct a strongly quasiconvex map $F\colon \R^{2 \times 2} \to \R$ and a Lipschitz but nowhere $C^1$-solution to the Euler--Lagrange equation of the corresponding energy $I(u)$. Shortly after, Sz\'ekelyhidi \cite{Sz04} constructed strongly \emph{polyconvex} examples $F\colon \R^{2\times2} \to \R$ and Lipschitz but nowhere $C^1$ maps $u\colon \R^2 \to \R^2$ which are, by \cite{KT03}, weak local minimizers of the corresponding energy.

In this work we extend Sz\'ekelyhidi's construction from the two-dimensional setting $F\colon \R^{2\times 2} \to \R$ and $u\colon \Omega \subset \R^2 \to \R^2$ into strongly polyconvex $F\colon \R^{M \times n} \to \R$ and $u\colon \Omega \subset \R^n \to \R^M$, where the dimensions $M,n \geq 2$ are arbitrary. A function $F\colon \R^{M \times n} \to \R$ is called polyconvex if the map
\[
\R^{M \times n} \ni X \mapsto F(X)
\]
can be written as a convex function of the $\ell \times \ell$ subdeterminants of $X$, $\ell \in \{1,\ldots, \min\{n,M\}\}$. And $F$ is strongly polyconvex if $F(X) - \gamma |X|^2$ is polyconvex for some $\gamma > 0$; cf. \Cref{s:polyconvLH}.

Precisely, our main result reads as follows.

\begin{theorem}\label{th:main}
Take any $n,M \geq 2$. There exists a $C^2$-smooth, strongly polyconvex function $F\colon \R^{M \times n} \to \R$ and $u \in \lip(\R^n,\R^M)$ such that $u$ is a weak local minimizer of $u \mapsto \int_{\B^n} F(Du)$, but $u$ is nowhere continuously differentiable. More precisely, $\partial_\alpha u$ is discontinuous in any open set $\Omega \subset \B^n$ and any $\alpha \in \{1,\ldots,n\}$.
\end{theorem}
The notion of weak local minimizer is taken from \cite{KT03} and in our setup the weak local minimizing property comes from the fact that $u$ solves the Euler--Lagrange equations and $D^2 F$ is non-degenerate at a.e. point of $\nabla u(x)$.

Let us remark that a simple extension by zero of Sz\'ekelyhidi's \cite{Sz04} example does not lead to a \emph{strongly} polyconvex function. Instead we need to adapt each step in his construction.  The main subtlety that is when we work in $\R^n$ the operation $\div$ and $\curl$ are not ``simply a rotation'' of each other --- leading to necessary adaptations of the notion of ``rank $1$''-connected sets, which in turn requires larger $T_\tnn$-configurations, $\tnn \gg 1$.

We also record that as a corollary of the proof \Cref{th:main} we have
\begin{theorem}\label{th:uniqueness}
Take any $n,M \geq 2$. Let $\Omega \subset \R^n$ be any bounded, open, convex set. There exists a $C^2$-smooth, strongly polyconvex function $F\colon \R^{M \times n} \to \R$ and two different $u_1,u_2 \in \lip(\Omega,\R^M)$ such that
\begin{itemize}
 \item $u_1$ and $u_2$ both satisfy the Euler--Lagrange equation of the functional $I(u) = \int_{\B^n} F(Du)$;
 \item $u_1 = u_2$ on $\partial \Omega$.
\end{itemize}
\end{theorem}

{\bf Outline:} Since we work in arbitrary dimension $n \geq 2$, we will work with differential forms, actually tensor products of differential forms. The basic definitions and the reformulation of the Euler--Lagrange equation of $I$ to a differential inclusion $du \in K_F$ is performed in \Cref{s:differentialinclusion}.

Since we work with differential forms, the notion of ``Rank one''-matrices needs to be replaced by what we call $\mathscr{R}$-connections, this leads to an adapted notion of $T_{\tnn}$-configurations. Definitions and dimensional analysis of those objects are presented in \Cref{s:Rconnections}.

A given $F$ which admits a $T_{\tnn}$-configuration, can be transformed into an $\tilde{F}$ which allows for an abundance of transversal $T_{\tnn}$-configurations --- at least when $\tnn$ is large enough. This is discussed in \Cref{s:conditionc} and is one of the more subtle arguments in this work, although it still follows the basic philosophy of \cite{Sz04}.

Having a suitable replacement for ``rank one'' matrices, the notion of laminates, laminates of finite order change only slightly from \cite{MS}, we discuss this in \Cref{s:laminates}. The way we add wiggles is then almost verbatim to \cite{MS}, see \Cref{s:wigglelemmata}; so is the construction of the solution $u$ which is described in \Cref{s:tnnthm}. The only (small) subtlety is that $df \in L^\infty$ does not imply Lipschitz bounds, but this is dealt with standard arguments from the Sobolev theory of differential forms \cite{ISS99}.

The construction of an $F$ whose associated partial relation condition admits \emph{some} $T_{\tnn}$-configuration, \Cref{s:polyexample}, is more subtle: For $M=n=2$ and $\tnn = 5$ this was proven with the help of a computer-algebra system by Sz\'ekelyhidi in \cite{Sz04}. In principle this strategy seems to work for any $M,n \geq 2$ as well -- we checked this for $M=n=3$ in \Cref{s:mathematica}. But in order to prove our result in any dimension $M,n \geq 2$ we of course not rely on a computer-algebra system, especially since we need $T_{\tnn}$-configurations with $\tnn$ very large. For this reason we present a method how to extend Sz\'ekelyhidi's $T_5$-configuration for $M=2$, $n=2$ into a general, non-degenerate $T_{\tnn}$-configuration in any dimension $M$ and $n$, for any $\tnn$ large.

With all these techniques in place, the proof of \Cref{th:main} is then analogous to \cite{Sz04}, see \Cref{th:proofmain}.

{\bf Notation:}
\begin{itemize}
\item We will use $\tnn$ for the $T_\tnn$-configuration and $\mathfrak{i}$, $\mathfrak{k}$, $\mathfrak{j}$, $\tnell$ for finite sequences, such as the components of a $T_{\tnn}$-configuration
\item Greek letters are reserved for domain indices $\alpha =1,\ldots n$
\item latin letters are used for everything else.
\end{itemize}

{\bf Acknowledgement:}
Discussions with Jacek Ga\l{}eski, James Scott, Swarnendu Sil are gratefully acknowledged. Funding is acknowledged as follows

\begin{itemize}
\item A.S. is an Alexander-von-Humboldt Fellow.
\item A.S. is funded by NSF Career DMS-2044898.
\item The project is co-financed by the Polish National Agency for Academic Exchange within Polish Returns Programme - BPN/PPO/2021/1/00019/U/00001 (K.M.).
\item The work presented in this paper was conducted as part of the Thematic Research Programme \emph{Geometric Analysis and PDEs}, which received funding
from the University of Warsaw via IDUB (Excellence Initiative Research University).
\end{itemize}
Mutual visits of the authors at their respective institutions are gratefully acknowledged.

\section{Notation and setup of the differential inclusion}\label{s:differentialinclusion}
As usual we denote the gradient of a map $u\colon \R^n \to \R^M$ as a matrix
\[
 Du = \left ( \begin{array}{ccc}
              \partial_{1} u^1 & \ldots & \partial_n u^1\\
              \partial_{1} u^2 & \ldots & \partial_n u^2\\
              \vdots \\
              \partial_{1} u^M & \ldots & \partial_n u^M\\
             \end{array}
\right ) \in \R^{M \times n}.
\]
The Euler--Lagrange system of the energy on some open set $\Omega \subset \R^n$
\[
I_F(u) \coloneqq \int_{\Omega} F(Du)\quad u\colon B^n \subset \R^n \to \R^M
\]
for a Lagrangian $F\colon \R^{M \times n} \to \R$ is given by
\begin{equation}\label{eq:ELsystem1}
\sum_{\alpha =1}^n \partial_\alpha \brac{(\partial_{X_{i\alpha}}F) (Du)} = 0, \quad i=1,\ldots,M.
\end{equation}
Since we are interested in particular in dimensions $n \geq 3$ and the above is a divergence-type equation, the notion of differential forms simplifies notation, cf., e.g., \cite{doCarmo94}. Let us also remark that some results regarding differential inclusions for differential forms were obtained in \cite{BDK15}.

For $k \in \{0,1,\ldots,n\}$ we denote by $\Ep^k \R^n$ the vector space of $k$-forms
\[
\Ep^k \R^n = \left \{\sum_{1 \leq \alpha_1 < \alpha_2 < \ldots < \alpha_k \leq n} \lambda_{\alpha_1,\alpha_2,\ldots,\alpha_n}\, dx^{\alpha_1} \wedge dx^{\alpha_2} \wedge \ldots \wedge dx^{\alpha_k}: \quad  \lambda_{\alpha_1,\alpha_2,\ldots,\alpha_k} \in \R\right \}
\]
spanned by the basis $(dx^{\alpha_1} \wedge dx^{\alpha_2} \wedge \ldots dx^{\alpha_k})_{1 \leq \alpha_1 < \alpha_2 < \ldots < \alpha_k \leq n}$. It is completely equivalent to consider
\[
\Ep^k \R^n = \left \{\sum_{1 \leq \alpha_1, \alpha_2 ,\ldots ,\alpha_k \leq n} \mu_{\alpha_1,\alpha_2,\ldots,\alpha_n}\, dx^{\alpha_1} \wedge dx^{\alpha_2} \wedge \ldots \wedge dx^{\alpha_k}: \quad  \begin{array}{c}\mu_{\alpha_1,\alpha_2,\ldots,\alpha_n} \in \R\\
 \mu_{\alpha_1,\ldots,\beta,\gamma,\ldots, \alpha_k}=-\mu_{\alpha_1,\ldots,\gamma,\beta,\ldots, \alpha_k}                                                                                                                                                                                                                                                                                                                                                                                                                       \end{array}
\right \},
\]
since we henceforth agree on the convention that
\[
 dx^\beta \wedge dx^\gamma = - dx^\gamma \wedge dx^\beta, \quad \text{in particular } dx^\beta \wedge dx^\beta = 0.
\]
For $k \in \{0,1,\ldots,n\}$ and $M \geq 1$ we denote by $\R^M \otimes \Ep^k \R^n$ the vector space
\[
\R^M \otimes \Ep^k \R^n = \left \{\sum_{1 \leq \alpha_1 < \alpha_2 < \ldots < \alpha_k \leq n} \lambda_{\alpha_1,\alpha_2,\ldots,\alpha_n}\, dx^{\alpha_1} \wedge dx^{\alpha_2} \wedge \ldots \wedge dx^{\alpha_k} \colon  \quad \lambda_{\alpha_1,\alpha_2,\ldots,\alpha_n} \in \R^M \right \}.
\]
In particular we will have ``differential form operations'' act component-wise with respect to the $\R^M$-component.

One of these operators is the Hodge star operator
\[
\hdg \colon \R^M \otimes \Ep^k \R^n \to \R^M \otimes \Ep^{n-k} \R^n.
\] It is defined as the unique linear operator with the following property. Fix $1 \leq \alpha_1 < \alpha_2 < \ldots < \alpha_k \leq n$. Take $1 \leq \beta_1 < \beta_2 < \ldots < \beta_{n-k} \leq n$ the complement of $\alpha_1,\ldots,\alpha_k$, i.e., $\{\alpha_1,\ldots,\alpha_k\}\cup \{\beta_1,\ldots,\beta_{n-k}\} = \{1,\ldots,n\}$. Then we set
\[
\hdg \brac{dx^{\alpha_1} \wedge dx^{\alpha_2} \wedge \ldots \wedge dx^{\alpha_k}} \coloneqq c dx^{\beta_1} \wedge dx^{\beta_2} \wedge \ldots \wedge dx^{\beta_{n-k}},
\]
where $c \in \{-1,1\}$ is chosen so that
\begin{equation}\label{eq:dxalphawedgestardxalpha}
dx^{\alpha_1} \wedge dx^{\alpha_2} \wedge \ldots \wedge dx^{\alpha_k} \wedge \hdg \brac{dx^{\alpha_1} \wedge dx^{\alpha_2} \wedge \ldots \wedge dx^{\alpha_k}} = + dx^1 \wedge dx^2 \wedge \ldots \wedge dx^n.
\end{equation}
For a vector $v \in \R^M \otimes \Ep^k \R^n$ represented as a vector of $\Ep^k\R^n$-vectors,
\[
 v = \left ( \begin{array}{c}
              v^1\\
              \vdots\\
              v^M\\
             \end{array}\right ), \quad v^i \in \Ep^k \R^n
\]
we define
\[
\hdg v = \left ( \begin{array}{c}
              \hdg v^1\\
              \vdots\\
              \hdg v^M\\
             \end{array}\right ) \in \R^M \otimes \Ep^{n-k} \R^n.
\]
An elementary, yet tedious, computation establishes that
\[
 \hdg \hdg v = (-1)^{k(n-k)} v \qquad \forall v \in \R^M \otimes \Ep^{n-k} \R^n.
\]
Another useful observation is
\begin{equation}\label{eq:hdgdalpha}
\hdg  dx^\alpha = (-1)^{\alpha +1} dx^1 \wedge \ldots \wedge dx^{\alpha-1} \wedge dx^{\alpha+1} \wedge \ldots \wedge dx^n.
\end{equation}

The other operation that is important to us is the differential
\[
d\colon C^1(\R^n,\R^M \otimes \Ep^k \R^n) \to C^0(\R^n,\R^M \otimes \Ep^{k+1} \R^n).
\]

For $\ell \in \N$ we say that a map $u \in C^\ell(\R^n,\R^M \otimes \Ep^k \R^n)$  if
\[
 u(x) = \sum_{1 \leq \alpha_1 < \alpha_2 < \ldots < \alpha_k \leq n} \lambda_{\alpha_1,\alpha_2,\ldots,\alpha_n}(x)\, dx^{\alpha_1} \wedge dx^{\alpha_2} \wedge \ldots \wedge dx^{\alpha_k}
\]
and its coefficients $\lambda_{\alpha_1,\alpha_2,\ldots,\alpha_n}(x) \in C^\ell(\R^n,\R^M)$.


If $u \in C^1(\R^n,\R^M \otimes \Ep^k \R^n)$ as above
we define the differential $du \in C^0(\R^n,\R^M \otimes \Ep^{k+1} \R^n)$ via
\[
du(x) = \sum_{j=1}^n \sum_{1 \leq \alpha_1 < \alpha_2 < \ldots < \alpha_k \leq n} \frac{\partial}{\partial x^j}\lambda_{\alpha_1,\alpha_2,\ldots,\alpha_n}(x)\, dx^j \wedge dx^{\alpha_1} \wedge dx^{\alpha_2} \wedge \ldots \wedge dx^{\alpha_k}.
\]

In particular any $u\colon \R^n \to \R^M$ can be identified as $u\colon \R^n \to \R^M \otimes \Ep^0\R^n$ and $Du\colon \R^n \to \R^{M\times n}$ is then equivalent to considering $du\colon \R^n \to \R^M \otimes \Ep^1 \R^n$.

Here is the point of all these algebraic considerations:

Using \eqref{eq:dxalphawedgestardxalpha}, the Euler--Lagrange system \eqref{eq:ELsystem1} written in the language of differential forms is
\begin{equation}\label{eq:deltaFzero}
\begin{split}
&\hdg d\hdg \brac{\sum_{\alpha=1}^n(\partial_{X_{i\alpha}}F)(du) dx^\alpha}
=0, \qquad i=1,\ldots,M.
\end{split}
\end{equation}
For \eqref{eq:deltaFzero} it is sufficient (on star-shaped domains: equivalent) to obtain the existence of $w^i \colon \R^n \to \Ep^{n-2} \R^n$ such that
\begin{equation}\label{eq:dwieqhdgblabla}
dw^i =\hdg \sum_{\alpha=1}^n(\partial_{X_{i\alpha}}F)(du) dx^\alpha.
\end{equation}

We thus see that $u\colon \Omega \subset \R^n \to \R^{M}$ satisfies \eqref{eq:ELsystem1} if there exists a $w\colon \Omega \subset \R^n \to \R^M \otimes \Ep^{n-2} \R^n$ such that
\[
 \left ( \begin{array}{c}
          du\\
          dw
         \end{array}
\right ) \in
 K_F \subset \left (\begin{array}{c}
  \R^M \otimes \Ep^1 \R^n\\
  \R^M \otimes \Ep^{n-1} \R^n\\
 \end{array}\right ),\]
 where \begin{equation}\label{eq:Kfdef}K_F \coloneqq \left \{\left (\begin{array}{c}
  X\\
  Y
 \end{array}\right ) \in \left (\begin{array}{c}
  \R^M \otimes \Ep^1 \R^n\\
  \R^M \otimes \Ep^{n-1} \R^n\\
 \end{array}\right )\colon \quad   \left (\begin{array}{c}
  X\\
  Y
 \end{array}\right ) = \left (\begin{array}{c}
  X\\
\brac{\sum_{\alpha=1}^n\partial_{X_{i\alpha}}F(X) \hdg dx^\alpha}_{i=1}^M
 \end{array}\right )\right \}.
\end{equation}

Observe that we can naturally identify $\Ep^1 \R^n$ with $\R^n$ and $\Ep^{n-1} \R^n$ with $\R^n$, so we will, equivalently, refer to
\[
 K_F \subset  \left (\begin{array}{c}
  \R^{M  \times n}\\
  \R^{M  \times n}\\
 \end{array}\right ).\]
However, this identification should be taken with care, since the Hodge star operator changes signs.

It is worth to compare this in the classical case $n=2$ and $M=2$ to the formulations in \cite{MS,Sz04,Sv94}: in that case we have $\hdg dx^1 = dx^2$ and $\hdg dx^2 = -dx^1$. Moreover, we can identify $\R^2 \otimes \Ep^1 \R^2$ with $\R^{2 \times 2}$, so that we have
\[\begin{split}
K_F \coloneqq& \left \{\left (\begin{array}{c}
  X\\
  Y
 \end{array}\right ) \in \left (\begin{array}{c}
  \R^{2\times 2}\\
  \R^{2\times 2}\\
 \end{array}\right )\colon \quad   Y_{i1}=-\partial_{X_{i2}}F(X),\
Y_{i2} = \partial_{X_{i1}}F(X)
\right \}\\
=&\left \{\left (\begin{array}{c}
  X\\
  Y
 \end{array}\right ) \in \left (\begin{array}{c}
  \R^{2\times 2}\\
  \R^{2\times 2}\\
 \end{array}\right )\colon \quad   Y=-DF(X)J
\right \}\\
\end{split}
\]
where we denote
\[
 J = \left ( \begin{array}{cc}
              0 & -1\\
              1 & 0
             \end{array}
\right ).
\]
Thus, for $M=n=2$ we see that we consider the same situation as in \cite{MS,Sz04,Sv94}.


%
%

\section{\texorpdfstring{$\mathscr{R}$-}{R}-connections and \texorpdfstring{$T_{\tnn}$}{TN}-configurations}\label{s:Rconnections}
\subsection{\texorpdfstring{$\mathscr{R}$-}{R}-connections}\label{ss:rank1replacement}
In our convex integration scheme we need to replace the notion of rank-1 connectedness to incorporate the linear algebra that the higher-dimensional case necessitates. For this, we replace the tensor product $q \otimes a$ for $a \in \R^4$ and $q \in \R^2$ \cite{MS,Sz04} by an alternative tensor product. It turns out that the right notion is the following (the main reason for this particular choice is \Cref{la:wiggle}).

\begin{definition}\label{def:R}
Let $A,B \in \left (\begin{array}{c}
  \R^M \otimes \Ep^1 \R^n\\
  \R^M \otimes \Ep^{n-1} \R^n\\
 \end{array}\right )$. We say that $A,B$ are $\mathscr{R}$-connected if
 \[
  A-B \in \mathscr{R},
 \]
where
\begin{equation}\label{eq:setR}
\begin{split}
\mathscr{R} &= \left \{\sum_{\alpha=1}^nb_\alpha\, dx^\alpha \wedge a\colon \quad \text{for some $b \in \R^n$ and $a \in \left ( \begin{array}{c} \R^M\\ \R^M \otimes \Ep^{n-2} \R^n \end{array}\right )$} \right \}\\
&\subset  \left (\begin{array}{c}
  \R^M \otimes \Ep^1 \R^n\\
  \R^M \otimes \Ep^{n-1} \R^n\\
 \end{array}\right ).
 \end{split}
\end{equation}
Again, in dimension $n=M=2$ the set $\mathscr{R}$ simply consists $\R^{4 \times 2}$-matrices with rank $\leq 1$.

We want to rule out trivial movements in $\mathscr{R}$ so we will denote
\begin{equation}\label{eq:setRo}
 \mathscr{R}^o =  \left \{\sum_{\alpha=1}^nb_\alpha\, dx^\alpha \wedge a\colon \quad \text{for some $b \in \R^n\setminus \{0\}$ and $a \in \left ( \begin{array}{c} \R^M \setminus \{0\}\\ \R^M \otimes \Ep^{n-2} \R^n \end{array}\right )$} \right \}.
\end{equation}
\end{definition}

In colloquial terms one may refer to $\mathscr{R}$-matrices as ``rank one'', and more generally one could say a matrix is ``Rank $k$'' if it can be represented as a sum of $k$ elements of $\mathscr{R}$ --- but caution is needed: if $n > 2$ the set $\left ( \begin{array}{c} \R^M \setminus \{0\}\\ \R^M \otimes \Ep^{n-2} \R^n \end{array}\right )$ is \emph{not} $2M$-dimensional, but larger. On the other hand, if $n>2$ the ``tensor product'' $b_\alpha dx^\alpha \wedge a$ has more cancellations than in the case $n=2$. We will discuss the proper representation and do a careful dimensional analysis below.

But first we record the some simple properties of $\mathscr{R}$:
\begin{lemma}\label{le:propertiesofA}
  The set $\mathscr{R}$ and $\mathscr{R}^o$ have the following properties:
  \begin{enumerate}
 \item $\mathscr{R}$ and $\mathscr{R}^o$ are symmetric. If $A-B \in \mathscr{R}$ then $B-A \in \mathscr{R}$.
 \item $\mathscr{R}$ and $\mathscr{R}^o$ are a cone, i.e., if $A-B \in \mathscr{R}$ then $\lambda (A-B) \in \mathscr{R}$ for all $\lambda \in \R \setminus \{0\}$.
 \end{enumerate}
 \end{lemma}

The above formulation of $\mathscr{R}$ is helpful for differential computations, for dimensional analysis it is less so. This is why we record the following identification:

\begin{remark}\label{la:rewriteR}
The following is a linear bijection:
\[
 \mathfrak{T}:  \left ( \begin{array}{c} \R^M\otimes \Ep^1 \R^n\\ \R^M \otimes \Ep^{n-1} \R^n \end{array}\right ) \to \left ( \begin{array}{c} \R^{M \times n}\\ \R^{M\times n}\end{array}\right )
\]

\[
 \brac{\mathfrak{T} \left ( \begin{array}{c}
            \omega'\\
            \omega''
           \end{array} \right )}_{i,\alpha} \coloneqq  \brac{\hdg \brac{\omega' \wedge \hdg dx^\alpha}}^i \quad i=1,\ldots,M, \ \alpha =1,\ldots,n
\]
and
\[
 \brac{\mathfrak{T} \left ( \begin{array}{c}
            \omega'\\
            \omega''
           \end{array} \right )}_{(M+i),\alpha} \coloneqq  \brac{\hdg \brac{dx^\alpha \wedge \omega''}}^i \quad i=1,\ldots,M, \ \alpha =1,\ldots,n.
\]

Consider the following set $ \tilde{\mathscr{R}}  \subset \left ( \begin{array}{c} \R^{M \times n}\\ \R^{M\times n}\end{array}\right )$
\[
 \tilde{\mathscr{R}} \coloneqq \left \{\left ( \begin{array}{cccc} \vec{u}\, b_1 & \vec{u}\, b_2 & \ldots & \vec{u}\, b_n \\
\sum_{\alpha=1}^n\vec{v}_{1\alpha}  b_\alpha & \sum_{i=1}^n\vec{v}_{2\alpha} \, b_\alpha & \ldots & \sum_{\alpha=1}^n\vec{v}_{n\alpha} \, b_\alpha \end{array}\right ): \text{ where }
\begin{array}{l}
 b \in \R^n,\ \vec{u} \in \R^M,\\
 \vec{v}_{\alpha \beta}= -\vec{v}_{\beta\alpha} \in \R^{M}, \quad \alpha,\beta=1,\ldots,n
 \end{array}
\right \}
\]

Then $\mathscr{R}$ is equivalent to $\tilde{\mathscr{R}}$ in the sense that we can pass from one set to the other by identifying
\begin{equation}\label{eq:Rrelation}
 a = \left ( \begin{array}{c} \vec{u}\\ \sum\limits_{1 \leq \alpha_1 < \alpha_2 \leq n}\vec{v}_{\alpha_1 \alpha_2} \hdg \brac{dx^{\alpha_1}\wedge dx^{\alpha_2}} \end{array}\right ) \in \left ( \begin{array}{c} \R^M\\ \R^M \otimes \Ep^{n-2} \R^n \end{array}\right ).
\end{equation}
\end{remark}
\begin{proof}
%
Observe that for $1 \leq \alpha_1 < \alpha_2 \leq n$ we have
\[
\hdg\brac{ dx^{\alpha_1} \wedge  dx^{\alpha_2}} = (-1)^{\alpha_1+\alpha_2 + 1} dx^1 \wedge \ldots \wedge dx^{\alpha_1-1} \wedge dx^{\alpha_1+1} \wedge \ldots \wedge dx^{\alpha_2-1} \wedge dx^{\alpha_2+1} \wedge \ldots \wedge dx^n.
\]
Thus, in view of \eqref{eq:hdgdalpha},
\[
dx^{\alpha_1} \wedge \hdg\brac{ dx^{\alpha_1} \wedge  dx^{\alpha_2}} = (-1)^{\alpha_2} dx^1 \wedge \ldots \wedge dx^{\alpha_2-1} \wedge dx^{\alpha_2+1} \wedge \ldots \wedge dx^n = -\hdg dx^{\alpha_2}
\]
and
\[
 dx^{\alpha_2} \wedge \hdg\brac{ dx^{\alpha_1} \wedge  dx^{\alpha_2}} = (-1)^{\alpha_1+ 1} dx^1 \wedge \ldots \wedge dx^{\alpha_1-1} \wedge dx^{\alpha_1+1} \wedge \ldots \wedge dx^n=\hdg dx^{\alpha_1}.
\]
Then, if we have \eqref{eq:Rrelation},
\[
\sum_{\alpha=1}^n b_\alpha dx^\alpha \wedge a = \left ( \begin{array}{c} \vec{u}\ \sum_{\alpha=1}^n b_\alpha dx^\alpha \\ \sum\limits_{1 \leq \alpha_1 < \alpha_2 \leq n} \vec{v}_{\alpha_1,\alpha_2}  \brac{b_{\alpha_1}\, \underbrace{dx^{\alpha_1} \wedge \hdg \brac{dx^{\alpha_1}\wedge dx^{\alpha_2}}}_{=- \hdg dx^{\alpha_2}}+b_{\alpha_2}\, \underbrace{dx^{\alpha_2} \wedge \hdg \brac{dx^{\alpha_1}\wedge dx^{\alpha_2}}}_{{=\hdg dx^{\alpha_1}}} } \end{array}\right ).
\]
Using antisymmetry we rewrite
\begin{equation}\label{eq:toomanythpdge}
\begin{split}
 \sum\limits_{1 \leq \alpha_1 < \alpha_2 \leq n}
 &\vec{v}_{\alpha_1,\alpha_2}  \brac{b_{\alpha_1}\, (- \hdg dx^{\alpha_2}) +b_{\alpha_2}\, \hdg dx^{\alpha_1}}\\
&=\frac{1}{2} \sum\limits_{\alpha_1,\alpha_2 =1}^n \vec{v}_{\alpha_1,\alpha_2}  \brac{b_{\alpha_2}\, \hdg dx^{\alpha_1}-b_{\alpha_1}\, \hdg dx^{\alpha_2}}\\
&= \sum\limits_{\alpha=1}^n \brac{\sum_{\beta=1}^n \vec{v}_{\alpha,\beta}  b_{\beta}}\, \hdg dx^{\alpha}.
\end{split}
 \end{equation}

 Thus,
 \[
  \mathfrak{T} \brac{\sum_{\alpha=1}^n b_\alpha dx^\alpha \wedge a } = \left ( \begin{array}{cccc} \vec{u}\, b_1 & \vec{u}\, b_2 & \ldots & \vec{u}\, b_n \\
\sum_{\alpha=1}^n\vec{v}_{1\alpha}  b_\alpha & \sum_{i=1}^n\vec{v}_{2\alpha} \, b_\alpha & \ldots & \sum_{\alpha=1}^n\vec{v}_{n\alpha} \, b_\alpha \end{array}\right ).
 \]

%
\end{proof}

We also observe the following ``quantitative connectivity'' of $\left (\begin{array}{c}
  \R^M \otimes \Ep^1 \R^n\\
  \R^M \otimes \Ep^{n-1} \R^n\\
 \end{array}\right )$ under movements in $\mathscr{R}$: We can connect any two ``matrices'' in $A,B\in \left (\begin{array}{c}
  \R^M \otimes \Ep^1 \R^n\\
  \R^M \otimes \Ep^{n-1} \R^n\\
 \end{array}\right )$ by $\mathscr{R}$-matrices of length roughly comparable to $|A-B|$.
\begin{lemma}\label{la:distinR}
Fix $M,n \geq 2$. There exists a constant $\Lambda > 1$ and $N \geq 2$ such that the following holds. For any
\[
A,B \in  \left (\begin{array}{c}
  \R^M \otimes \Ep^1 \R^n\\
  \R^M \otimes \Ep^{n-1} \R^n\\
 \end{array}\right )
\]
there are
\[
( A_\tni)_{\tni=1}^{N} \subset \left (\begin{array}{c}
  \R^M \otimes \Ep^1 \R^n\\
  \R^M \otimes \Ep^{n-1} \R^n\\
 \end{array}\right )
\]
such that
\begin{itemize}
 \item $A_1 = A$ and $A_N = B$
 \item $A_{\tni+1}-A_{\tni} \in \mathscr{R}$
 \item $|A_{\tni+1}-A_{\tni}| \leq \Lambda |A-B|$.
\end{itemize}
\end{lemma}
\begin{proof}
In view of \Cref{la:rewriteR} it suffices to show that for any $A,B \in \R^{2M\times n}$ there exists $A_{\tni}$ as above with $A_{\tni+1}-A_{\tni}  \in \tilde{\mathscr{R}}$.

Let $(e_1,\ldots,e_n) \subset \R^n$ be the unit vectors of $\R^n$. Then we can write
\[
A -B= \sum_{\alpha=1}^{n} a_\alpha \otimes e_\alpha
\]
with
\[
 |a_\alpha| \aleq |A-B|.
\]
So in order to prove the claim, we may assume w.l.o.g.
\[
 A-B= a \otimes e_1 = \left ( \begin{array}{cccc}
                       \vec{a}'&0&\ldots&0\\
                       \vec{a}''&0&\ldots&0
                      \end{array}
\right )=\left ( \begin{array}{cccc}
                       \vec{a}'&0&\ldots&0\\
                       0&0&\ldots&0
                      \end{array}
\right )+\left ( \begin{array}{cccc}
                       0&0&\ldots&0\\
                       \vec{a}''&0&\ldots&0
                      \end{array}
\right )
\]
for some $a = \left ( \begin{array}{c}
                       \vec{a}'\\
                       \vec{a}''
                      \end{array}
\right )\in \R^{2M}$ with $|\vec{a}'|,|\vec{a}''| \aleq |A-B|$.

Set $\vec{u}^1\coloneqq \vec{a'}^1$, $\vec{v}^1 = 0$, $\vec{b}^1=e_1$. Then the corresponding element $C^1 \in \tilde{\mathscr{R}}$ is
\[
C^1 = \left ( \begin{array}{cccc} \vec{u}^1\,  & 0 & \ldots & 0 \\
0&0&\ldots&0\end{array}\right ) = \left ( \begin{array}{cccc}
                       \vec{a}'&0&\ldots&0\\
                       0&0&\ldots&0
                      \end{array}
\right ).
\]
Set $\vec{u}^2\coloneqq 0$, $\vec{v}^2_{12} = \vec{a''}$, and $\vec{v}^2_{\alpha \beta} = 0$ for $(\alpha,\beta) \neq (1,2), (2,1)$, and lastly $\vec{b}^2=e_2$. Then the corresponding element $C^2 \in \tilde{\mathscr{R}}$ is
\[
C^2 = \left ( \begin{array}{ccccc} 0&  0 &0& \ldots & 0 \\
\vec{v}_{12}^2   & \vec{v}_{22}^2 & \vec{v}_{32}^2 &\ldots & \vec{v}_{n2}^2 \end{array}\right ) = \left ( \begin{array}{cccc}
                       0&0&\ldots&0\\
                       \vec{a}''&0&\ldots&0
                      \end{array}
\right ).
\]
Thus,
\[
 C^1+C^2= B-A
\]
and we can set
\[
 A_1 = A,\ A_2 = A+C_1,\ A_3=B=A+C_1+C_2.
\]
which implies the claim.
\end{proof}

The identification from \Cref{la:rewriteR} is also useful for the dimensional analysis:
\begin{lemma}\label{la:Rrep}
Let $b,\tilde{b} \in \R^n$ with $|b| = |\tilde{b}|=1$, and $\vec{u},\vec{\tilde{u}} \in \R^M$ not both zero, and $\vec{v}_{\alpha \beta}= -\vec{v}_{\beta\alpha} \in \R^{M}, \quad \alpha,\beta=1,\ldots,n$ as well as $\vec{\tilde{v}}_{\alpha \beta}= -\vec{\tilde{v}}_{\beta\alpha} \in \R^{M}, \quad \alpha,\beta=1,\ldots,n$.

Then the following are equivalent
\begin{enumerate}
 \item\label{it:identification1} $(b,\vec{u},\vec{v})$ and $(\tilde{b},\vec{\tilde{u}},\vec{\tilde{v}})$ represent the same object in $\mathscr{R}^o$, that is we have \[\left ( \begin{array}{cccc} \vec{u}\, b_1 & \vec{u}\, b_2 & \ldots & \vec{u}\, b_n \\
\sum_{\alpha=1}^n\vec{v}_{1\alpha}  b_\alpha & \sum_{\alpha=1}^n\vec{v}_{2\alpha} \, b_\alpha & \ldots & \sum_{\alpha=1}^n\vec{v}_{n\alpha} \, b_\alpha \end{array}\right )
=
\left ( \begin{array}{cccc} \vec{\tilde{u}}\, \tilde{b}_1 & \vec{\tilde{u}}\, \tilde{b}_2 & \ldots & \vec{\tilde{u}}\, \tilde{b}_n \\
\sum_{\alpha=1}^n\vec{\tilde{v}}_{1\alpha}  \tilde{b}_\alpha & \sum_{\alpha=1}^n\vec{\tilde{v}}_{2\alpha} \, \tilde{b}_\alpha & \ldots & \sum_{\alpha=1}^n\vec{\tilde{v}}_{n\alpha} \, \tilde{b}_\alpha \end{array}\right )
\]
\item\label{it:identification2} for $\lambda \in \{-1,1\}$ we have
\[
 b = \lambda \tilde{b}, \quad \vec{u} = \lambda \vec{\tilde{u}}
\]
and if we take any orthonormal basis $(o_1,\ldots,o_n)$ of $\R^n$ with $o_1 = b$ then we can write
\begin{equation}\label{eq:asdlkasdjvecvvectv}
 \vec{v} - \lambda \vec{\tilde{v}} = \sum_{2 \leq \alpha \leq \beta \leq n} \vec{w}_{\alpha \beta} \brac{o_{\alpha} \otimes o_{\beta} - o_{\beta} \otimes o_{\alpha}}
\end{equation}
with $\vec{w}_{\alpha \beta} \in \R^M$ for $2 \leq \alpha \leq \beta \leq n$.
\end{enumerate}

In particular, $\mathscr{R}^o$ is locally a manifold of dimension
\begin{equation}\label{eq:dimA}
 \dim \mathscr{R}^o
=n(M+1)-1.
\end{equation}
\end{lemma}
\begin{proof}
Once the equivalence is proven the dimension \eqref{eq:dimA} is clear: the degrees of freedom for $b \in \R^n$, $|b| = 1$ are $n-1$. The degrees of freedom for $\vec{u}$ are $M$. The degrees of freedom for $w_{\alpha \beta} \in \R^M$ for $2 \leq \alpha \leq \beta \leq n$ is $M$ times the dimension of $\mathfrak{so}(n-1)$ which is $\frac{(n-1)(n-2)}{2}$. In view of \eqref{eq:asdlkasdjvecvvectv} the latter is the dimension of the kernel of the map $(b,\vec{u},\vec{v}) \mapsto \mathscr{R}^o$, so the degrees of freedom for $\vec{v}$ are
\[
M \brac{\underbrace{\frac{n(n-1)}{2}}_{\dim \mathfrak{so(n)}} - \frac{(n-1)(n-2)}{2}}=M (n-1).
\]
Adding up the degrees of freedom gives \eqref{eq:dimA}.

It remains to establish the equivalence.

\underline{\eqref{it:identification2} $\Rightarrow$ \eqref{it:identification1}} is relatively obvious: Assuming (2) we clearly have
\[
 \vec{u}\, b_\beta = \vec{\tilde{u}}\, \tilde{b}_\beta \quad \beta \in \{1,\ldots,n\}.
\]
Moreover since $o_\alpha \perp b$ for $\alpha \geq 2$, we have
\[
 \sum_{2 \leq \alpha \leq \beta \leq n} \vec{w}_{\alpha \beta} \brac{o_{\alpha} \otimes o_{\beta} - o_{\beta} \otimes o_{\alpha}} b = 0.
\]
Thus, from \eqref{eq:asdlkasdjvecvvectv} we obtain
\[
 \sum_{\alpha=1}^n\vec{\tilde{v}}_{\gamma \alpha}  \tilde{b}_\alpha = \lambda^2 \sum_{\alpha=1}^n\vec{v}_{\gamma \alpha}  b_\alpha
 + 0.
 \]
This establishes \eqref{it:identification1}.

\underline{\eqref{it:identification1} $\Rightarrow$ \eqref{it:identification2}} We assume
\begin{equation}\label{eq:twomatrxRsame}
\left ( \begin{array}{cccc} \vec{u}\, b_1 & \vec{u}\, b_2 & \ldots & \vec{u}\, b_n \\
\sum_{\alpha=1}^n\vec{v}_{1\alpha}  b_\alpha & \sum_{\alpha=1}^n\vec{v}_{2\alpha} \, b_\alpha & \ldots & \sum_{\alpha=1}^n\vec{v}_{n\alpha} \, b_\alpha \end{array}\right ) =
\left ( \begin{array}{cccc} \vec{\tilde{u}}\, \tilde{b}_1 & \vec{\tilde{u}}\, \tilde{b}_2 & \ldots & \vec{\tilde{u}}\, \tilde{b}_n \\
\sum_{\alpha=1}^n\vec{\tilde{v}}_{1\alpha}  \tilde{b}_\alpha & \sum_{i=1}^n\vec{\tilde{v}}_{2\alpha} \, \tilde{b}_\alpha & \ldots & \sum_{\alpha=1}^n\vec{\tilde{v}}_{n\alpha} \, \tilde{b}_\alpha \end{array}\right )
\end{equation}
The first entry of the above equation is equivalent to
\[
 \vec{u} \otimes b = \vec{\tilde{u}} \otimes \tilde{b}.
\]
Since by assumption $|b| = |\tilde{b}| = 1$, this in turn is equivalent to $b = \pm \tilde{b}$ and $\vec{u} = \pm \vec{\tilde{u}}$.

So in order to prove the equivalence we may w.l.o.g. assume $b = \tilde{b}$, $\vec{u} = \pm \vec{\tilde{u}}$ and that
\[\left ( \begin{array}{cccc} \sum_{\alpha=1}^n\vec{v}_{1\alpha}  b_\alpha & \sum_{\alpha=1}^n\vec{v}_{2\alpha} \, b_\alpha & \ldots & \sum_{\alpha=1}^n\vec{v}_{n\alpha} \, b_\alpha
\end{array}\right )=
\left ( \begin{array}{cccc} \sum_{\alpha=1}^n\vec{\tilde{v}}_{1\alpha}  b_\alpha & \sum_{\alpha=1}^n\vec{\tilde{v}}_{2\alpha} \, b_\alpha & \ldots & \sum_{\alpha=1}^n\vec{\tilde{v}}_{n\alpha} \, b_\alpha\end{array}\right )
\]That is
\begin{equation}\label{eq:aaarghasldkjasd}
 \sum_{\alpha =1}^n \brac{\vec{v}_{\beta \alpha}  -\vec{\tilde{v}}_{\beta \alpha}}b_\alpha =0 \quad \text{$\beta \in \{1,\ldots,n\}$}.
\end{equation}
Let $(o_1,\ldots,o_n)$ be an orthonormal basis of $\R^n$, with $o_1 = b$. We can then write
\[
 \vec{v} -\vec{\tilde{v}} =  \sum_{1 \leq \alpha <\beta \leq n} \vec{w}_{\alpha \beta} \brac{o_{\alpha} \otimes o_{\beta} - o_{\beta} \otimes o_{\alpha}}
\]
for some ``coefficients'' $\vec{w}_{\alpha \beta} \in \R^M$.
All we need to show is that \eqref{eq:aaarghasldkjasd} is equivalent to to $w_{1,\beta} = 0$ for any $\beta \in \{2,\ldots,n\}$. But \eqref{eq:aaarghasldkjasd} is simply the matrix product of $\vec{v} -\vec{\tilde{v}}$ (componentwise in $\R^M$-vector considered as an $\R^{n \times n}$-matrix). That is \eqref{eq:aaarghasldkjasd} implies
\[
 \sum_{1 \leq \alpha <\beta \leq n} \underbrace{\vec{w}_{\alpha \beta}}_{\in \R^N} \underbrace{\brac{o_{\alpha} \otimes o_{\beta} - o_{\beta} \otimes o_{\alpha}}}_{\in \R^{n\times n}} \underbrace{b}_{=o_1 \in \R^n} = 0
\]
Since $(o_1,\ldots,o_n)$ form an orthonormal basis of $\R^n$ (and $\beta \geq 2$)
\[0=
 \sum_{1 \leq \alpha <\beta \leq n} \vec{w}_{\alpha \beta} \brac{ o_{\alpha} \delta_{\beta 1} - o_{\beta} \delta_{\alpha 1}} = \sum_{\beta =2}^n  \vec{w}_{1\beta} o_{\beta}.
\]
Since $o_2,\ldots,o_n$ are linear independent vectors in $\R^n$, considering the above equation componentwise in $\R^M$, we see that $\vec{w}_{1\beta} = 0$ for $\beta \in \{2,\ldots,n\}$. We can conclude.

\end{proof}

Next we state a continuity result of the parametrization of elements in $\mathscr{R}^o$.

\begin{lemma}\label{la:lipschdepR}
Let $C \in \mathscr{R}$. Then there exists $b \in \R^n$ with $|b|=1$ and $a= \left ( \begin{array}{c} a'\\a''\end{array} \right ) \in \left ( \begin{array}{c} \R^M\\ \R^M \otimes \Ep^{n-2} \R^n \end{array}\right )$ such that
\[
 C = \sum_{\alpha=1}^n b_\alpha\, dx^\alpha \wedge a
\]
and the following holds.

For any $\tilde{C} \in \mathscr{R}$ there exists $\tilde{b} \in \R^n$ with $|\tilde{b}|=1$ and $\tilde{a}= \left ( \begin{array}{c} \tilde{a}'\\\tilde{a}''\end{array} \right ) \in \left ( \begin{array}{c} \R^M\\ \R^M \otimes \Ep^{n-2} \R^n \end{array}\right )$ such that
\[
 \tilde{C} = \sum_{\alpha=1}^n \tilde{b}_\alpha\, dx^\alpha \wedge \tilde{a},
\]
and denoting by $|C-\tilde{C}|$ any norm of $\left ( \begin{array}{c} \Ep^1 \R^M\\ \R^M \otimes \Ep^{n-1} \R^n \end{array}\right )$ we have
\begin{equation}\label{eq:lipschdepRclaim1}
|a''-\tilde{a}''|_{\R^M \otimes \Ep^{n-2} \R^n} \aleq |C-\tilde{C}| + |b-\tilde{b}|_{\R^n},
\end{equation}
and
\begin{equation}\label{eq:lipschdepRclaim2}
 |a'-\tilde{a}'| \aleq |C-\tilde{C}|
\end{equation}
and there exists constant $\sigma > 0$ (depending only on the choice of norms, otherwise uniform)
\begin{equation}\label{eq:lipschdepRclaim3}
\text{if $|C-\tilde{C}| \leq \sigma |a'|$} \quad \text{then } |b-\tilde{b}|_{\R^n} \aleq \frac{1}{|a'|} |C-\tilde{C}|.
\end{equation}

In particular, if $C \in \mathscr{R}^o$ there exists a small $\eps = \eps(C) > 0$ for any $\tilde{C} \in \mathscr{R}$ with $|C-\tilde{C}| < \eps$ then $\tilde{C} \in \mathscr{R}^o$ and we can choose a parametrization $\tilde{a},\tilde{b}$ that depends on $\tilde{C}$ in a locally Lipschitz way.
\end{lemma}
\begin{proof}
By \Cref{la:Rrep} any $C$ can be identified with some $|b|=1$, $\vec{u} \in \R^M$, $\vec{v}_{\alpha\beta}=-\vec{v}_{\beta\alpha} \in \R^M$ such that
\[
 \vec{v}=\sum_{\beta =2}^n \vec{w}_{\beta} \brac{b \otimes o_{\beta} - o_{\beta} \otimes b},
\]
for some $\vec{w}_{\beta} \in \R^M$, $\beta \in \{2,\ldots,n\}$. The values of $\vec{w}_\beta$ depend on the choice of the orthonormal frame $(o_{2},\ldots,o_n)$ that span $b^\perp \subset \R^n$: if we adapt the orthonormal frame, i.e. for some $O \in O(n-1)$ we set
\[
 \overline{o}_\beta \coloneqq \sum_{\gamma=2}^n O_{\beta-1,\gamma-1} o_\gamma
\]
we simply have
\begin{equation}\label{eq:cont:defv}
 \vec{v}=\sum_{\beta =2}^n \vec{\overline{w}}_{\beta} \brac{b \otimes \overline{o}_{\beta} - \overline{o}_{\beta} \otimes b},
\end{equation}
where
\[
 \vec{\overline{w}}_\beta  = \sum_{\gamma=2}^nO_{\gamma-1,\beta-1} \vec{w}_\gamma.
\]

We compute, for each $\gamma \in \{1,\ldots,n\}$
\[
\begin{split}
 &\sum_{\alpha=1}^n\vec{v}_{\gamma \alpha} \, b_\alpha\\
 =&\sum_{\alpha=1}^n \brac{\sum_{\beta =2}^n \vec{w}_{\beta} \brac{b \otimes o_{\beta} - o_{\beta} \otimes b}}_{\gamma \alpha} \, b_\alpha\\
 =&\sum_{\beta =2}^n \vec{w}_{\beta}  \sum_{\alpha=1}^n \brac{\brac{b \otimes o_{\beta} - o_{\beta} \otimes b}}_{\gamma \alpha} \, b_\alpha\\
 =&\sum_{\beta =2}^n \vec{w}_{\beta}  \brac{\brac{b \otimes o_{\beta} - o_{\beta} \otimes b}b}_{\gamma}\\
 =&-\sum_{\beta =2}^n \vec{w}_{\beta} \brac{o_{\beta} }_{\gamma}\\
 \end{split}
\]
That is, $C$ is represented by
\[\left ( \begin{array}{cccc} &b \otimes \vec{u}\\
-\sum_{\beta =2}^n \vec{w}_{\beta} \brac{o_{\beta} }_{1} & -\sum_{\beta =2}^n \vec{w}_{\beta} \brac{o_{\beta} }_{2} & \ldots & -\sum_{\beta =2}^n \vec{w}_{\beta} \brac{o_{\beta} }_{n} \end{array}\right )
\]
Similarly, any $\tilde{C} \in \mathscr{R}$ can be represented by
\[\left ( \begin{array}{cccc} &\tilde{b} \otimes \tilde{\vec{u}}\\
-\sum_{\beta =2}^n \tilde{\vec{w}}_{\beta} \brac{\tilde{o}_{\beta} }_{1} & -\sum_{\beta =2}^n \tilde{\vec{w}}_{\beta} \brac{\tilde{o}_{\beta} }_{2} & \ldots & -\sum_{\beta =2}^n \tilde{\vec{w}}_{\beta} \brac{\tilde{o}_{\beta} }_{n} \end{array}\right )
\]
Adapting $(o_2,\ldots,o_n)$ and $(\tilde{o}_2,\ldots,\tilde{o}_n)$, but without changing $b,\tilde{b}, \vec{u}, \vec{v}_{\alpha\beta}$ we can ensure that
\begin{equation}\label{eq:contR:choiceofbo2}
 \span \{b,o_2\}= \span \{\tilde{b},\tilde{o}_2\} \quad \text{ and } o_\beta \in \{\pm \tilde{o}_\beta\} \subset \span \{b,o_2\}^\perp \text{ for $\beta = 3,\ldots,n$}
\end{equation}
Indeed, if $\tilde{b}$ is parallel to $b$ we simply choose $\tilde{o}_2=o_2$. If $\tilde{b}$ is not parallel to $b$ we use the Gram-Schmidt-algorithm on $(\tilde{b},b)$ to find $\tilde{o}_2$ such that \eqref{eq:contR:choiceofbo2} holds.

The Hilbert-Schmidt scalar product of $\tilde{C}$ and $C$ becomes
\[
 \langle C,\tilde{C}\rangle = \langle b,\tilde{b}\rangle \langle \vec{u},\tilde{\vec{u}}\rangle +
 \sum_{\beta =2}^n \sum_{\gamma =2}^n \langle \vec{w}_{\beta},\tilde{\vec{w}}_{\gamma}\rangle_{\R^M} \langle o_{\beta}, \tilde{o}_{\gamma} \rangle_{\R^n}
\]
Thus
\[
\begin{split}
 |C-\tilde{C}|^2 =& |\vec{u}|_{\R^M}^2 +|\tilde{\vec{u}}|_{\R^M}^2 - 2\langle b,\tilde{b}\rangle \langle \vec{u},\tilde{\vec{u}}\rangle\\
 &+|\tilde{\vec{w}}|_{\R^{M \times n-1}}^2  + |\vec{w}|_{\R^{M \times n-1}}^2-2\sum_{\beta =2}^n \sum_{\gamma =2}^n \langle \vec{w}_{\beta},\tilde{\vec{w}}_{\gamma}\rangle_{\R^M}\, \langle o_{\beta}, \tilde{o}_{\gamma} \rangle_{\R^n}\\
  =& |\vec{u}-\tilde{\vec{u}}|_{\R^M}^2 +|b-\tilde{b}|^2 \langle \vec{u},\tilde{\vec{u}}\rangle\\
  &+\sum_{\beta =2}^n \sum_{\gamma =2}^n \brac{|\vec{w}|_{\R^{M \times n-1}}^2   \delta_{\beta\gamma}+ |\vec{w}|_{\R^{M \times n-1}}^2 \delta_{\beta\gamma}  -2 \langle \vec{w}_{\beta},\tilde{\vec{w}}_{\gamma}\rangle_{\R^M} \langle o_{\beta}, \tilde{o}_{\gamma} \rangle_{\R^n}}\\
  =& |\vec{u}-\tilde{\vec{u}}|_{\R^M}^2 +|b-\tilde{b}|^2 \langle \vec{u},\tilde{\vec{u}}\rangle\\
  &+|\vec{w}-\tilde{\vec{w}}|_{\R^{M \times n-1}}^2  +\sum_{\beta =2}^n \sum_{\gamma =2}^n\langle \vec{w}_{\beta},\tilde{\vec{w}}_{\gamma}\rangle_{\R^M} \brac{2\delta_{\beta \gamma}-2\langle o_{\beta}, \tilde{o}_{\gamma} \rangle_{\R^n}}\\
  =& |\vec{u}-\tilde{\vec{u}}|_{\R^M}^2 +|b-\tilde{b}|^2 \langle \vec{u},\tilde{\vec{u}}\rangle\\
  &+|\vec{w}-\tilde{\vec{w}}|_{\R^{M \times n-1}}^2  \\
  &+\sum_{\beta =2}^n \langle \vec{w}_{\beta},\tilde{\vec{w}}_{\beta}\rangle_{\R^M} \brac{2-2\langle o_{\beta}, \tilde{o}_{\beta} \rangle_{\R^n}}\\
  &+\sum_{\gamma =2}^n \sum_{\beta=2, \beta \neq \gamma}^n  \langle \vec{w}_{\beta},\tilde{\vec{w}}_{\gamma}\rangle_{\R^M} \brac{-2\langle o_{\beta}, \tilde{o}_{\gamma} \rangle_{\R^n}}.
%
 \end{split}
\]
In view of \eqref{eq:contR:choiceofbo2} this simplifies to
\[
\begin{split}
 |C-\tilde{C}|^2  =& |\vec{u}-\tilde{\vec{u}}|_{\R^M}^2 +|b-\tilde{b}|^2 \langle \vec{u},\tilde{\vec{u}}\rangle\\
  &+|\vec{w}-\tilde{\vec{w}}|_{\R^{M \times n-1}}^2  \\
  &+ \langle \vec{w}_{2},\tilde{\vec{w}}_{2}\rangle_{\R^M} \abs{o_{2}- \tilde{o}_{2}}^2.
%
 \end{split}
\]
Observe that for $\lambda,\mu \in \{-1,1\}$ we have
\[
\begin{split}
&\left ( \begin{array}{cccc} &\lambda \tilde{b} \otimes \lambda \tilde{\vec{u}}\\
-\sum_{\beta =2}^n \tilde{\mu\vec{w}}_{\beta} \brac{\mu \tilde{o}_{\beta} }_{1} & -\sum_{\beta =2}^n \tilde{\mu \vec{w}}_{\beta} \brac{\mu \tilde{o}_{\beta} }_{2} & \ldots & -\sum_{\beta =2}^n \tilde{\mu\vec{w}}_{\beta} \brac{\mu \tilde{o}_{\beta} }_{n} \end{array}\right )\\
 &=\left ( \begin{array}{cccc} &\tilde{b} \otimes \tilde{\vec{u}}\\
 -\sum_{\beta =2}^n \tilde{\vec{w}}_{\beta} \brac{\tilde{o}_{\beta} }_{1} & -\sum_{\beta =2}^n \tilde{\vec{w}}_{\beta} \brac{\tilde{o}_{\beta} }_{2} & \ldots & -\sum_{\beta =2}^n \tilde{\vec{w}}_{\beta} \brac{\tilde{o}_{\beta} }_{n} \end{array}\right )
\end{split}
\]
Since \eqref{eq:contR:choiceofbo2} does not change if we replace $\tilde{b}$ by $\lambda  \tilde{b}$ and $\tilde{o}_\beta$ by $\lambda \tilde{o}_\beta$, we repeat the above computation and actually have for any $\lambda,\mu \in \{-1,1\}$
\[
\begin{split}
 |C-\tilde{C}|^2  =& |\vec{u}-\lambda \tilde{\vec{u}}|_{\R^M}^2 +|b-\lambda \tilde{b}|^2 \lambda \langle \vec{u},\tilde{\vec{u}}\rangle\\
  &+|\vec{w}-\mu\tilde{\vec{w}}|_{\R^{M \times n-1}}^2  \\
  &+ \mu \langle \vec{w}_{2},\tilde{\vec{w}}_{2}\rangle_{\R^M} \abs{o_{2}- \mu\tilde{o}_{2}}^2\\
%
 \end{split}
\]
We choose $\lambda \in \{-1,1\}$ so that $\lambda \langle \vec{u},\tilde{\vec{u}}\rangle \geq 0$ and $\mu \in \{-1,1\}$ so that $\mu \langle \vec{w}_{2},\tilde{\vec{w}}_{2}\rangle_{\R^M} \geq 0$. Consequently,
\begin{equation}\label{eq:Rcont:asdadfg1}
\begin{split}
 |C-\tilde{C}|^2  =& |\vec{u}-\lambda \tilde{\vec{u}}|_{\R^M}^2 +|b-\lambda \tilde{b}|^2 \abs{\langle \vec{u},\tilde{\vec{u}}\rangle}\\
  &+|\vec{w}-\mu\tilde{\vec{w}}|_{\R^{M \times n-1}}^2  \\
  &+ \abs{\langle \vec{w}_{2},\tilde{\vec{w}}_{2}\rangle_{\R^M}} \abs{o_{2}- \tilde{o}_{2}}^2\\
%
 \end{split}
\end{equation}
A first consequence worth recording
\begin{equation}\label{eq:Rcont:asdadfg2}
\begin{split}
 |C-\tilde{C}|^2  \geq& |\vec{u}-\lambda \tilde{\vec{u}}|_{\R^M}^2 + |\vec{w}-\mu\tilde{\vec{w}}|_{\R^{M \times n-1}}^2  \\
 \end{split}
\end{equation}
This implies \eqref{eq:lipschdepRclaim2} using the the representation \eqref{eq:Rrelation}.

From \eqref{eq:Rcont:asdadfg1} we also find that if
\[
 |C-\tilde{C}| \leq \frac{1}{2} |\vec{u}| \quad \text{then }|\vec{u}-\lambda \tilde{\vec{u}}| \leq \frac{1}{2} |\vec{u}|
\]
and thus
\[
 \abs{\langle \vec{u},\tilde{\vec{u}}\rangle } = \lambda \langle \vec{u},\tilde{\vec{u}}\rangle = \frac{1}{2} \brac{|\vec{u}|^2+|\tilde{\vec{u}}|^2-|\vec{u}-\lambda \tilde{\vec{u}}|^2} \geq \frac{1}{4} |\vec{u}|^2.
\]
And thus
\[
\begin{split}
 |C-\tilde{C}| \leq \frac{1}{2} |\vec{u}| \quad \Rightarrow |C-\tilde{C}|^2  \geq & |\vec{u}-\lambda \tilde{\vec{u}}|_{\R^M}^2 +\frac{1}{4}|b-\lambda \tilde{b}|^2 |\vec{u}|^2
 \end{split}
\]
This implies \eqref{eq:lipschdepRclaim3} using the the representation \eqref{eq:Rrelation}.

Similarly, if
\[
 |C-\tilde{C}| \leq \frac{1}{2} |\vec{w_2}| \quad \text{then } \abs{\langle \vec{w_2},\tilde{\vec{w_2}}\rangle } \geq \frac{1}{4} |\vec{w}_2|^2.
\]
and thus
\begin{equation}\label{eq:Rcont:asuoivycx}
 |C-\tilde{C}| \leq \frac{1}{2} |\vec{w_2}| \quad \text{then } |C-\tilde{C}|^2  \geq |\vec{w}-\mu\tilde{\vec{w}}|_{\R^{M \times n-1}}^2  + \frac{1}{4} |\vec{w}_2|^2 \abs{o_{2}- \tilde{o}_{2}}^2\\
\end{equation}
From the definitions of $\vec{v}$ and $\vec{\tilde{v}}$, \eqref{eq:cont:defv}, also taking into account \eqref{eq:contR:choiceofbo2},
\[
\begin{split}
 |\vec{v}-\vec{\tilde{v}}|\aleq& |\vec{w}-\mu \tilde{\vec{w}}|_{\R^N}+|b-\lambda \tilde{b}|_{\R^n} + |w_2|_{\R^M}\, |o_{2}-\mu\tilde{o}_2|_{\R^n}\\
 \overset{\eqref{eq:Rcont:asdadfg2}}{\aleq}& |C-\tilde{C}| + |b-\lambda \tilde{b}|_{\R^n} + |\vec{w_2}|_{\R^M}\, |o_{2}-\mu\tilde{o}_2|_{\R^n}\\
 \end{split}
\]
This implies
\[
|\vec{v}-\vec{\tilde{v}}|\aleq|C-\tilde{C}| + |b-\lambda \tilde{b}|_{\R^n},
\]
because either $|w_2| \leq 2|C-\tilde{C}|$ or otherwise we can apply \eqref{eq:Rcont:asuoivycx}. This implies \eqref{eq:lipschdepRclaim1}, using the representation \eqref{eq:Rrelation}.

For the last part of the claim, the continuous dependence of the parametrization of $\tilde{C} \approx C \in \mathscr{R}^o$ we simply observe that up to the sign $a' \neq 0$ and $b$, $|b|=1$, are uniquely determined by $C$, by \Cref{la:Rrep} and \eqref{eq:Rrelation}, and thus \eqref{eq:lipschdepRclaim3} is applicable uniformly.
\end{proof}

From the continuous representation we obtain in particular

\begin{lemma}\label{la:Risclosed}
$\mathscr{R}$ is a closed set in $\left (\begin{array}{c}
  \R^M \otimes \Ep^1 \R^n\\
  \R^M \otimes \Ep^{n-1} \R^n\\
 \end{array}\right ) \hat{=} \left (\begin{array}{c}
  \R^{M \times n}\\
  \R^{M\times n}\\
 \end{array}\right ) $
\end{lemma}
\begin{proof}
Take any sequence $(C^\tni)_{\tni \in \N}$ in $\mathscr{R}$ that converges in $\left (\begin{array}{c}
  \R^M \otimes \Ep^1 \R^n\\
  \R^M \otimes \Ep^{n-1} \R^n\\
 \end{array}\right )$.
Apply \Cref{la:lipschdepR} to $C^1$ and $C^{\tni}$ then we find a representation
We write
\[
 C^{\tni} = \sum_{\alpha=1}^n b^{\tni}_\alpha\, dx^\alpha \wedge a^{\tni}
\]
with $b^{\tni} \in \R^n$, $|b^{\tni}| = 1$ and
\[
 a^{\tni} = \left ( \begin{array}{c} a^{\tni'}\\a^{\tni''} \end{array} \right ) \in \left ( \begin{array}{c} \R^M\\ \R^M \otimes \Ep^{n-2} \R^n \end{array}\right ),
\]
such that
\[
 \sup_{\tni} |a^{\tni} - a^{1}| \aleq \sup_{\tni} |C^1-C^{\tni}| + 2 < \infty.
\]
Up to passing to a subsequence we thus can assume that $a^{\tni}$ and $b^{\tni}$ converge to some $a$ and $b$, respectively, and we have
\[
 C \coloneqq  \sum_{\alpha=1}^n b_\alpha\, dx^\alpha \wedge a = \lim_{\tni \to \infty} \sum_{\alpha=1}^n b^{\tni}_\alpha\, dx^\alpha \wedge a^{\tni} = \lim_{\tni \to \infty} C^{\tni}.
\]
\end{proof}

\begin{lemma}\label{la:paramaterization}
Let $\overline{C} \in \mathscr{R}^o$.

Define the map
\begin{equation}\label{eq:deffrakP}
 \mathfrak{P}: \R^M \times \brac{\mathfrak{so}(n)\otimes \R^M} \times \brac{\R^{n}} \to \mathscr{R}
\end{equation}
by
\[
 \mathfrak{P}\brac{\vec{u},\vec{v}, b}
 \coloneqq \left ( \begin{array}{c} \sum_{\alpha=1}^n \vec{u}\, b_\alpha dx^\alpha \\ \sum\limits_{\alpha=1}^n \brac{\sum_{\beta=1}^n \vec{v}_{\alpha,\beta}  b_{\beta}}\, \hdg dx^{\alpha} \end{array}\right )
\]
for $\vec{u} \in \R^M$, $(\vec{v}_{\alpha \beta})_{\alpha, \beta =1}^n \in \mathfrak{so}(n)\otimes \R^M$ and $b \in \R^n$.

Then there exists an $\eps > 0$ and a smooth manifold  $\mathfrak{C}$
\[
 \mathfrak{C} \subset \R^M \times \brac{\mathfrak{so}(n)\otimes \R^M} \times \S^{n-1} \to \mathscr{R}
\]
which has
\[
 \dim\mathfrak{C} = Mn+n-1
\]
and $\mathfrak{P}: \mathfrak{C} \to B(\overline{C},\eps) \cap \mathscr{R}$ is a diffeomorphism.
\end{lemma}
\begin{proof}
Fix $\overline{C} \in \mathscr{R}^o$ and take \Cref{la:lipschdepR} combined with \Cref{la:Rrep} some $\vec{\overline{u}} \in \R^M \setminus \{0\}$, $(\vec{\overline{v}}_{\alpha \beta})_{\alpha, \beta =1}^n \in \mathfrak{so}(n)\otimes \R^M$ and $\overline{b} \in \S^{n-1}$ such that
\[
 \mathfrak{P}\brac{\vec{\overline{u}},\vec{\overline{v}}, \overline{b}} = C.
\]
We define
\[
 \tilde{\mathfrak{C}} \subset \R^M \times \brac{\mathfrak{so}(n)\otimes \R^M} \times \S^{n-1}
\]
the set consisting of $\vec{u} \in \R^M \setminus \{0\}$, $(\vec{v}_{\alpha \beta})_{\alpha, \beta =1}^n \in \mathfrak{so}(n)\otimes \R^M$ and $b \in \S^{n-1}$ with the properties
\begin{itemize}
 \item $|b-\overline{b}| < \frac{1}{2}$ (this fixes the orientation of $b$)
 \item $\vec{v}$ can be represented by for some $\vec{w}_{\beta} \in \R^M$, $\beta \in \{2,\ldots,n\}$
 \[
  \vec{v} = \sum_{\beta =2}^n \vec{w}_{\beta} \brac{b \otimes o_{\beta} - o_{\beta} \otimes b},
 \]
where $(b,o_2,\ldots,o_n)$ denotes an arbitrary choice of orthonormal basis of $\R^n$ which extends $b$.
\end{itemize}
Clearly $\tilde{\mathfrak{C}}$ is a smooth manifold of dimension $M+(n-1)+M(n-1)=Mn+n-1$ ($M$-degrees of freedom for $\vec{u}$, $n-1$ degrees of freedom of $b$, and then by the representation of $\vec{v}$ we have $M(n-1)$-degrees of freedom for $\vec{w}$.

and $\mathfrak{P}: \tilde{\mathfrak{C}} \to \mathscr{R}^o$ is a smooth operator. Also, if $\eps$ is chosen small enough, then by \Cref{la:Rrep} and \Cref{la:lipschdepR} any $C \in B(\overline{C},\eps)$ corresponds to exactly one tuple $(\vec{u}, (\vec{v}_{\alpha \beta})_{\alpha, \beta =1}^n,b) \in \tilde{\mathfrak{C}}$, indeed \Cref{la:lipschdepR} implies that $\mathfrak{P}$ is bi-Lipschitz (Here we used $a' \neq 0$).

Any smooth bi-Lipschitz map is locally a diffeomorphism onto its image, \Cref{la:IFTbiLipschitz}, so we can conclude.
\end{proof}

\begin{lemma}\label{la:mathfrakCi}
Fix \[
(\overline{b},\vec{\overline{u}}, \vec{\overline{v}}) \in    \S^{n-1} \times \R^M \times \mathfrak{so}(n)\otimes \R^M,
    \]
so that the corresponding
\[
 \overline{C} = \left ( \begin{array}{cccc} \vec{\overline{u}}\, \overline{b}_1 & \vec{\overline{u}}\, \overline{b}_2 & \ldots & \vec{\overline{u}}\, {\overline{b}_n} \\
\sum_{\alpha=1}^n\vec{\overline{v}}_{1\alpha}  \overline{b}_\alpha & \sum_{\alpha=1}^n\vec{\overline{v}}_{2\alpha} \, \overline{b}_\alpha & \ldots & \sum_{\alpha=1}^n\vec{\overline{v}}_{n\alpha} \, \overline{b}_\alpha \end{array} \right ) \in \mathscr{R}^o.
\]
We also assume for some fixed orthonormal basis of $\R^n$ $(\overline{b},o_2,\ldots,o_n)$ that
\[
\vec{\overline{v}} = \sum_{\beta =2}^n \vec{\overline{w}}_{\beta} \brac{\overline{b} \otimes \overline{o}_{\beta} - \overline{o}_{\beta} \otimes \overline{b}},
\]
Define the linear space
\[\begin{split}
 \mathfrak{C}\coloneqq  \Big \{C \in \left (\begin{array}{c}\R^{M \times n}\\
\R^{M \times n}
\end{array}\right ): C =& \left ( \begin{array}{cccc} \vec{\overline{u}}\, b_1 & \vec{\overline{u}}\, b_2 & \ldots & \vec{\overline{u}}\, b_n \\
\sum_{\alpha=1}^n\vec{\overline{v}}_{1\alpha}  b_\alpha & \sum_{\alpha=1}^n\vec{\overline{v}}_{2\alpha} \, b_\alpha & \ldots & \sum_{\alpha=1}^n\vec{\overline{v}}_{n\alpha} \, b_\alpha \end{array} \right )\\
&+ \left ( \begin{array}{cccc} \vec{u}\, \overline{b}_1 & \vec{u}\, \overline{b}_2 & \ldots & \vec{u}\, \overline{b}_n \\
\sum_{\alpha=1}^n\vec{v}_{1\alpha}  \overline{b}_\alpha & \sum_{\alpha=1}^n\vec{v}_{2\alpha} \, \overline{b}_\alpha & \ldots & \sum_{\alpha=1}^n\vec{v}_{n\alpha} \, \overline{b}_\alpha \end{array} \right ):\\
&(b,u, \vec{v}) \in    \R^{n} \times \R^M \times \mathfrak{so}(n)\otimes \R^M
\Big \}\\
\end{split}
\]
Then
\[
 \dim \mathfrak{C} = n(M+1)-1
\]
\end{lemma}
\begin{proof}
This follows in the spirit of
\Cref{la:paramaterization}.
Fix an orthonormal basis $(\overline{b},o_2,\ldots,o_n) \subset \R^n$. Then we may assume that
\begin{equation}\label{eq:akjlchvxcvi}
\vec{v} = \sum_{\beta =2}^n \vec{w}_{\beta} \brac{b \otimes o_{\beta} - o_{\beta} \otimes b},
\end{equation}
because all terms of the type $\brac{o_\alpha \otimes o_{\beta} - o_{\beta} \otimes o_\alpha}$ for $\alpha,\beta \neq 1$ get cancelled.

If we write
\[
 b = \underbrace{b - \overline{b}\langle \overline{b},b\rangle}_{b^T} + \underbrace{\overline{b} \langle\overline{b},b\rangle}_{b^\perp}
\]
We see that for $\vec{\tilde{u}} \coloneqq  \vec{u} + \langle \overline{b},b\rangle$
\[
\begin{split}
& \left ( \begin{array}{cccc} \vec{\overline{u}}\, b_1 & \vec{\overline{u}}\, b_2 & \ldots & \vec{\overline{u}}\, b_n \\
\sum_{\alpha=1}^n\vec{\overline{v}}_{1\alpha}  b_\alpha & \sum_{\alpha=1}^n\vec{\overline{v}}_{2\alpha} \, b_\alpha & \ldots & \sum_{\alpha=1}^n\vec{\overline{v}}_{n\alpha} \, b_\alpha \end{array} \right )\\
&+ \left ( \begin{array}{cccc} \vec{u}\, \overline{b}_1 & \vec{u}\, \overline{b}_2 & \ldots & \vec{u}\, \overline{b}_n \\
\sum_{\alpha=1}^n\vec{v}_{1\alpha}  \overline{b}_\alpha & \sum_{\alpha=1}^n\vec{v}_{2\alpha} \, \overline{b}_\alpha & \ldots & \sum_{\alpha=1}^n\vec{v}_{n\alpha} \, \overline{b}_\alpha \end{array} \right )\\
=& \left ( \begin{array}{cccc} \vec{\overline{u}}\, b^{T}_1 & \vec{\overline{u}}\, b^{T}_2 & \ldots & \vec{\overline{u}}\, b^{T}_n \\
\sum_{\alpha=1}^n\vec{\overline{v}}_{1\alpha}  b_\alpha & \sum_{\alpha=1}^n\vec{\overline{v}}_{2\alpha} \, b_\alpha & \ldots & \sum_{\alpha=1}^n\vec{\overline{v}}_{n\alpha} \, b_\alpha \end{array} \right )\\
&+ \left ( \begin{array}{cccc} \vec{\tilde{u}}\, \overline{b}_1 & \vec{\tilde{u}}\, \overline{b}_2 & \ldots & \vec{\tilde{u}}\, \overline{b}_n \\
\sum_{\alpha=1}^n\vec{v}_{1\alpha}  \overline{b}_\alpha & \sum_{\alpha=1}^n\vec{v}_{2\alpha} \, \overline{b}_\alpha & \ldots & \sum_{\alpha=1}^n\vec{v}_{n\alpha} \, \overline{b}_\alpha \end{array} \right )\\
\end{split}
\]
Thus
\[
 \begin{split}
 \mathfrak{C}\coloneqq  \Big \{C \in \left (\begin{array}{c}\R^{M \times n}\\
\R^{M \times n}
\end{array}\right ): C =& \left ( \begin{array}{cccc} \vec{\overline{u}}\, b_1 & \vec{\overline{u}}\, b_2 & \ldots & \vec{\overline{u}}\, b_n \\
\sum_{\alpha=1}^n\vec{\overline{v}}_{1\alpha}  b_\alpha & \sum_{\alpha=1}^n\vec{\overline{v}}_{2\alpha} \, b_\alpha & \ldots & \sum_{\alpha=1}^n\vec{\overline{v}}_{n\alpha} \, b_\alpha \end{array} \right )\\
&+ \left ( \begin{array}{cccc} \vec{u}\, \overline{b}_1 & \vec{u}\, \overline{b}_2 & \ldots & \vec{u}\, \overline{b}_n \\
\sum_{\alpha=1}^n\vec{v}_{1\alpha}  \overline{b}_\alpha & \sum_{\alpha=1}^n\vec{v}_{2\alpha} \, \overline{b}_\alpha & \ldots & \sum_{\alpha=1}^n\vec{v}_{n\alpha} \, \overline{b}_\alpha \end{array} \right ):\\
&(b,u, \vec{v}) \in    T_{\overline{b}}\S^{n-1} \times \R^M \times \mathfrak{so}(n)\otimes \R^M, \quad \text{\eqref{eq:akjlchvxcvi} holds}
\Big \}\\
\end{split}
\]
This is an at most $(n-1)+M+(n-1)M$-dimensional space. And indeed,
from \Cref{la:paramaterization}, $(\overline{b}+\eps b,\overline{u}+\eps u,\overline{v}+\eps v) \in \tilde{\mathfrak{C}}+o(\eps)$, and we have that
\[
(b,u,v) \mapsto \mathfrak{P}(\overline{b}+\eps b,\overline{u}+\eps u,\overline{v}+\eps v) \in \mathscr{R}^o
\]
is a local diffeomorphism, so the $\mathfrak{C}$ is simply the tangent space, and thus has the same dimension as the manifold $\mathscr{R}^o$, which was computed in \Cref{la:paramaterization}. We can conclude.
\end{proof}

\begin{lemma}\label{la:fuckyouallofyoufuckfuckfuck1}
Fix \[
(\overline{b}^{\tni},\vec{\overline{u}^{\tni}}, \vec{\overline{v}^{\tni}}) \in    \R^n \times \R^M \times \mathfrak{so}(n)\otimes \R^M, \quad \tni=1,2.
    \]
and set
\[\begin{split}
 \mathfrak{C}^\tnell\coloneqq  \Big \{C \in \left (\begin{array}{c}\R^{M \times n}\\
\R^{M \times n}
\end{array}\right ): C =& \left ( \begin{array}{cccc} \vec{\overline{u}^{\tni}}\, b^{\tni}_1 & \vec{\overline{u}}\, b^{\tni}_2 & \ldots & \vec{\overline{u}}\, b^{\tni}_n \\
\sum_{\alpha=1}^n\vec{\overline{v}}_{1\alpha}  b^{\tni}_\alpha & \sum_{\alpha=1}^n\vec{\overline{v}}_{2\alpha} \, b^{\tni}_\alpha & \ldots & \sum_{\alpha=1}^n\vec{\overline{v}}_{n\alpha} \, b^{\tni}_\alpha \end{array} \right )\\
&+ \left ( \begin{array}{cccc} \vec{u^{\tni}}\, \overline{b^{\tni}}_1 & \vec{u^{\tni}}\, \overline{b^{\tni}}_2 & \ldots & \vec{u^{\tni}}\, \overline{b^{\tni}}_n \\
\sum_{\alpha=1}^n\vec{v^{\tni}}_{1\alpha}  \overline{b^{\tni}}_\alpha & \sum_{\alpha=1}^n\vec{v^{\tni}}_{2\alpha} \, \overline{b^{\tni}}_\alpha & \ldots & \sum_{\alpha=1}^n\vec{v^{\tni}}_{n\alpha} \, \overline{b^{\tni}}_\alpha \end{array} \right )  \Big \}\\
\end{split}
\]
Also set
\[
 \overline{C}^\tnell \coloneqq  \left ( \begin{array}{cccc} \vec{\overline{u}^{\tni}}\, \overline{b}^{\tni}_1 & \vec{\overline{u}}\, \overline{b}^{\tni}_2 & \ldots & \vec{\overline{u}}\, \overline{b}^{\tni}_n \\
\sum_{\alpha=1}^n\vec{\overline{v}}_{1\alpha}  \overline{b}^{\tni}_\alpha & \sum_{\alpha=1}^n\vec{\overline{v}}_{2\alpha} \, \overline{b}^{\tni}_\alpha & \ldots & \sum_{\alpha=1}^n\vec{\overline{v}}_{n\alpha} \, \overline{b}^{\tni}_\alpha \end{array} \right )
\]

Assume that $(\overline{b}^1,\overline{b}^2)$ are linearly independent in $\R^n$ and $(\overline{u}^1,\overline{u}^2)$ are linearly independent in $\R^M$.

Then
\[
\overline{C}^2 \not \in \mathfrak{C}^1
\]
\end{lemma}
\begin{proof}
Assume to the contrary that there exists $C^1 \in \mathfrak{C}^1$ with $C^1 =\overline{C}^2$.
We take the first component of the equality which is
\[
 \overline{u}^1 \otimes b^1 + u^1 \otimes \overline{b}^1 = \overline{u}^2 \otimes \overline{b}^2
\]
There exists $P \in O(M)$ and $O \in O(n)$ such that $e_1 = P\overline{u}^1 \in \R^M$, $e_2 = P\overline{u}^2 \in \R^M$ and $o_1 = O\overline{b}^1 \in \R^n$, $o_2 = O\overline{b}^2 \in \R^n$, where $e_1,\ldots,e_M$ and $o_1,\ldots,o_n$ are the standard basis of $\R^M$ and $\R^n$, respectively.  Multiplying the equation from the left with $P$ and from the right with $O$, setting $\tilde{b}^\tni = O b^{\tni}$ and $\tilde{u}^{\tni}  = P u$, we find
\begin{equation}\label{eq:xcvoijxlkcjv}
 e_1 \otimes \tilde{b}^1 + \tilde{u}^1 \otimes o_1 = e_2 \otimes o_2.
\end{equation}
Multiplying with $e_\ell$ and $o_\alpha$ respectively we find
\[
 \delta_{1\ell} \langle \tilde{b}^1,o_\alpha\rangle_{\R^n} + \langle \tilde{u}^1,e_\ell\rangle_{\R^M}\delta_{1\alpha} = \delta_{2\ell} \delta_{2\alpha} \quad \forall \ell \in \{1,\ldots,M\},\ \alpha \in \{1,\ldots,n\}.
\]
Taking $\ell=1$ and $\alpha =2,\ldots,n$ we find that $\langle  \tilde{b}^1,o_\alpha\rangle_{\R^n} = 0$, thus $\tilde{b}^1 = \lambda o_1$ for some $\lambda \in \R$. Taking $\ell = 2,\ldots,M$ and $\alpha=1$ we find that $\langle \tilde{u}^1, e_\ell\rangle_{\R^M} = 0$, thus $\tilde{u}^1 = \mu e_1$ for some $\mu \in \R$. Taking $\ell = 1$ and $\alpha = 1$ we find
\[
 \lambda+\mu = 0.
\]
Thus, \eqref{eq:xcvoijxlkcjv} becomes
\[
 0=e_2 \otimes o_2
\]
a contradiction. We can conclude.
\end{proof}

\begin{lemma}\label{la:dimofsumguy}
For $M,n \geq 2$ there exists $\lambda = \lambda(n,M) \in (1,2)$ such that the following holds.
Fix \[
(\overline{b}^{\tni},\vec{\overline{u}^{\tni}}, \vec{\overline{v}^{\tni}}) \in    \R^n \times \R^M \times \mathfrak{so}(n)\otimes \R^M, \quad \tni=1,2.
    \]
corresponding to some $\overline{C}_{\tni} \in \mathscr{R}^o$, as in \Cref{la:mathfrakCi}. As in that lemma we set
\[\begin{split}
 \mathfrak{C}^\tnell\coloneqq  \Big \{C \in \left (\begin{array}{c}\R^{M \times n}\\
\R^{M \times n}
\end{array}\right ): C =& \left ( \begin{array}{cccc} \vec{\overline{u}^{\tni}}\, b^{\tni}_1 & \vec{\overline{u}}\, b^{\tni}_2 & \ldots & \vec{\overline{u}}\, b^{\tni}_n \\
\sum_{\alpha=1}^n\vec{\overline{v}}_{1\alpha}  b^{\tni}_\alpha & \sum_{\alpha=1}^n\vec{\overline{v}}_{2\alpha} \, b^{\tni}_\alpha & \ldots & \sum_{\alpha=1}^n\vec{\overline{v}}_{n\alpha} \, b^{\tni}_\alpha \end{array} \right )\\
&+ \left ( \begin{array}{cccc} \vec{u^{\tni}}\, \overline{b^{\tni}}_1 & \vec{u^{\tni}}\, \overline{b^{\tni}}_2 & \ldots & \vec{u^{\tni}}\, \overline{b^{\tni}}_n \\
\sum_{\alpha=1}^n\vec{v^{\tni}}_{1\alpha}  \overline{b^{\tni}}_\alpha & \sum_{\alpha=1}^n\vec{v^{\tni}}_{2\alpha} \, \overline{b^{\tni}}_\alpha & \ldots & \sum_{\alpha=1}^n\vec{v^{\tni}}_{n\alpha} \, \overline{b^{\tni}}_\alpha \end{array} \right )\\
&(b,u, \vec{v}) \in    \R^{n} \times \R^M \times \mathfrak{so}(n)\otimes \R^M,
\Big \}\\
\end{split}
\]
Assume that $(\overline{b}^1,\overline{b}^2)$ are linearly independent in $\R^n$ and $(\overline{u}^1,\overline{u}^2)$ are linearly independent in $\R^M$.

Then the set
\[
 K_0\coloneqq  \left \{(C_1,C_2) \in \mathfrak{C}^1 \times \mathfrak{C}^2: \quad C_1-C_2 = 0 \right \} \subset \left (\begin{array}{c}\R^{M \times n}\\
\R^{M \times n}
\end{array}\right )^2
\]
has dimension
\[
 \dim K_0 \leq \lambda \brac{n(M+1)-1}
\]
\end{lemma}
\begin{proof}
As in the proof of \Cref{la:mathfrakCi} we can rewrite
\[\begin{split}
 \mathfrak{C}^\tnell\coloneqq  \Big \{C \in \left (\begin{array}{c}\R^{M \times n}\\
\R^{M \times n}
\end{array}\right ): C =& \left ( \begin{array}{cccc} \vec{\overline{u}^{\tnell}}\, b^{\tnell}_1 & \vec{\overline{u}}\, b^{\tnell}_2 & \ldots & \vec{\overline{u}}\, b^{\tnell}_n \\
\sum_{\alpha=1}^n\vec{\overline{v}}_{1\alpha}  b^{\tnell}_\alpha & \sum_{\alpha=1}^n\vec{\overline{v}}_{2\alpha} \, b^{\tnell}_\alpha & \ldots & \sum_{\alpha=1}^n\vec{\overline{v}}_{n\alpha} \, b^{\tnell}_\alpha \end{array} \right )\\
&+ \left ( \begin{array}{cccc} \vec{u^{\tnell}}\, \overline{b^{\tnell}}_1 & \vec{u^{\tnell}}\, \overline{b^{\tnell}}_2 & \ldots & \vec{u^{\tnell}}\, \overline{b^{\tnell}}_n \\
\sum_{\alpha=1}^n\vec{v^{\tnell}}_{1\alpha}  \overline{b^{\tnell}}_\alpha & \sum_{\alpha=1}^n\vec{v^{\tnell}}_{2\alpha} \, \overline{b^{\tnell}}_\alpha & \ldots & \sum_{\alpha=1}^n\vec{v^{\tnell}}_{n\alpha} \, \overline{b^{\tnell}}_\alpha \end{array} \right )\\
&(b,u, \vec{v}) \in    T_{\overline{b}^\tnell} \S^{n-1} \times \R^M \times \mathfrak{so}(n)\otimes \R^M, \ \text{with \eqref{eq:akjlchvxcvi}}
\Big \}\\
\end{split}
\]
and we may assume $\vec{v}^{\tni}$ satisfies \eqref{eq:akjlchvxcvi}, i.e. is $(n-1)M$-dimensional, $b^{\tni} \in T_{\overline{b}^{\tni} }\S^{n-1}$ i.e. is $n-1$-dimensional.

Consider the equation
\begin{equation}\label{eq:C1eqC2}
 C^1 - C^2=0
\end{equation}
We consider the first components,
\[\begin{split}
 \vec{\overline{u}^{1}} \otimes b^{1} + \vec{u^{1}}\, \otimes \overline{b^{1}}
 = \vec{\overline{u}^{2}} \otimes b^{2} + \vec{u^{2}}\, \otimes \overline{b^{2}}
\end{split}
\]
There exists $P \in O(M)$ and $O \in O(n)$ such that $\vec{e}_1 = P\vec{\overline{u}}^1 \in \R^M$, $\vec{e}_2 = P\vec{\overline{u}}^2 \in \R^M$ and $o_1 = O\overline{b}^1 \in \R^n$, $o_2 = O\overline{b}^2 \in \R^n$, then we have
\[\begin{split}
 \vec{e}_1 \otimes (Ob^{1}) + P\vec{u^{1}}\, \otimes o_1
 = \vec{e}_2 \otimes Ob^{2} + P\vec{u^{2}}\, \otimes o_2
\end{split}
\]
Multiply this with $e_\ell$ and $o_\alpha$ to obtain
\begin{equation}\label{eq:deusexmachina}
\begin{split}
 \delta_{\ell1} \langle (Ob^{1}),o_\alpha\rangle_{\R^n} + \langle P\vec{u^{1}},\vec{e}_\ell \rangle_{\R^M} \delta_{\alpha 1}
 = \delta_{\ell 2} \langle Ob^{2},o_\alpha \rangle_{\R^n} + \langle P\vec{u^{2}}, e_\ell\rangle_{\R^M} \delta_{\alpha 2}
\end{split}
\end{equation}

Take $\alpha=1$ and $\ell = 1$, then we find
\begin{equation}\label{eq:xclkxcvijx}
\langle P\vec{u^{1}},\vec{e}_1 \rangle_{\R^M} = -\langle (Ob^{1}),o_1\rangle_{\R^n}.
\end{equation}
This equation restricts $u^1$ in terms of $b^1$.
Set
\[
\begin{split}
 \tilde{\mathfrak{C}}^1\coloneqq  \Big \{C \in \left (\begin{array}{c}\R^{M \times n}\\
\R^{M \times n}
\end{array}\right ): C =& \left ( \begin{array}{cccc} \vec{\overline{u}^{1}}\, b^{1}_1 & \vec{\overline{u}}\, b^{1}_2 & \ldots & \vec{\overline{u}}\, b^{1}_n \\
\sum_{\alpha=1}^n\vec{\overline{v}}_{1\alpha}  b^{1}_\alpha & \sum_{\alpha=1}^n\vec{\overline{v}}_{2\alpha} \, b^{1}_\alpha & \ldots & \sum_{\alpha=1}^n\vec{\overline{v}}_{n\alpha} \, b^{1}_\alpha \end{array} \right )\\
&+ \left ( \begin{array}{cccc} \vec{u^{1}}\, \overline{b^{1}}_1 & \vec{u^{1}}\, \overline{b^{1}}_2 & \ldots & \vec{u^{1}}\, \overline{b^{1}}_n \\
\sum_{\alpha=1}^n\vec{v^{1}}_{1\alpha}  \overline{b^{1}}_\alpha & \sum_{\alpha=1}^n\vec{v^{1}}_{2\alpha} \, \overline{b^{1}}_\alpha & \ldots & \sum_{\alpha=1}^n\vec{v^{1}}_{n\alpha} \, \overline{b^{1}}_\alpha \end{array} \right )\\
&(b,u, \vec{v}) \in    T_{\overline{b}^1} \S^{n-1} \times \R^M \times \mathfrak{so}(n)\otimes \R^M, \ \text{with \eqref{eq:akjlchvxcvi} and \eqref{eq:xclkxcvijx}}
\Big \}\\
\end{split}
\]
Then we see that
\[
 \dim  \tilde{\mathfrak{C}}^1 \leq (n-1)+(M-1)+(n-1)M,
\]
and
\[
 \{(C^1,C^2): C^1-C^2=0\} \subset \tilde{\mathfrak{C}}^1 \times \mathfrak{C}^2.
\]
By Cartesian coordinates, the dimension of the right-hand side space is at most
\[
 (n-1)+(M-1)+(n-1)M + (n-1)+M+(n-1)M = \underbrace{\frac{2(Mn+n-1)-1}{(Mn+n-1)}}_{=: \lambda <2} (Mn+n-1).
\]
We can conclude.

\end{proof}

\subsection{\texorpdfstring{$T_{\tnn}$}{TN}-configurations}
We are now ready to define what we mean by a $T_{\tnn}$-configuration -- it is essentially the same definition as in \cite{MS} -- but we replace ``rank-one connections'' with $\mathscr{R}$-connections. Note, however, that the notion of non-degeneracy is relatively strong.

\begin{definition}\label{def:Tnconfig}
Let $\tnn \in \N$. A $T_{\tnn}$-configuration consist of
\begin{itemize}
\item a base point $P \in \left (\begin{array}{c}
  \R^M \otimes \Ep^1 \R^n\\
  \R^M \otimes \Ep^{n-1} \R^n\\
 \end{array}\right )\hat{=} \left (\begin{array}{c}
  \R^{M \times n}\\
  \R^{M\times n}\\
 \end{array}\right )$
\item $(C_\tni)_{\tni=1}^{\tnn} \subset \mathscr{R}$ such that $\sum_{\tni=1}^{\tnn} C_\tni = 0$
\item $(\kappa_\tni)_{\tni=1}^{\tnn} \in (1,\infty)$
\end{itemize}

Given such a $T_{\tnn}$-configuration for $\tnk \in \{1,\ldots,\tnn\}$ we set
\[
 \pi_\tnk(P,(C_\tni)_{\tni=1}^{\tnn},(\kappa_{\tni})_{\tni =1}^{\tnn} ) \coloneqq P + C_1 + C_2 + \ldots  C_{\tnk -1}
\]
and
\begin{equation}\label{eq:phikdef}
 Z_{\tnk} \coloneqq \phi_\tnk(P,(C_\tni)_{\tni=1}^{\tnn},(\kappa_{\tni})_{\tni =1}^{\tnn} ) \coloneqq P + C_1 + C_2 + \ldots  C_{\tnk -1} + \kappa_{\tnk} C_{\tnk}.
\end{equation}
We call a $T_{\tnn}$-configuration \emph{non-degenerate} if all of the following hold in addition to the above assumptions.
\begin{enumerate}
\item $C_\tni \in \mathscr{R}^o$ for all $\tni$,
\item if we write $C^{\tni}$ as
\[
C^{\tni} = \sum_{\alpha=1}^n b^\tni_\alpha\, dx^\alpha \wedge a^\tni, \quad \tni=1,\ldots \tnn,
\]
for $b^{\tni} \in \R^n$, $|b^{\tni}|=1$, and
\[
a^{\tni} \in \left ( \begin{array}{c} \R^M \setminus \{0\}\\ \R^M \otimes \Ep^{n-2} \R^n \end{array}\right )
\]
then we have \begin{equation} \label{eq:nondegenerate}
{\rm span} \{b_1,\ldots,b_{\tnn}\}=\R^n
          \end{equation}
\item and
\begin{equation}\label{eq:nondegbli}
 \{b_\tni,b_{\tni+1}\} \text{ linear independent} \quad \text{whenever } \tni \in \{1,\ldots,\tnn\}
\end{equation}
(where we identify $\tnn+1$ with $1$)
\item and
\[
 \{a_\tni',a_{\tni+1}'\} \text{ linear independent} \quad \text{whenever }\tni \in \{1,\ldots,\tnn\}
\]
\item and
\begin{equation}\label{eq:TNnorankone}
 Z_{\tnk} - Z_{\tnl} \not \in \mathscr{R} \quad \text{for all $\tnk \neq \tnl \in \{1,\ldots,\tnn\}$}.
\end{equation}

\end{enumerate}
\medskip
We call a $T_{\tnn}$-configuration \emph{wild} for some $\beta$ if, when we write
\[
 Z_\tnk = \left ( \begin{array}{c}X_{\tnk}\\
 Y_{\tnk} \end{array}\right )\in \left (\begin{array}{c}
  \R^M \otimes \Ep^1 \R^n\\
  \R^M \otimes \Ep^{n-1} \R^n\\
 \end{array}\right )
\]
and identify $X_{\tnk}$ with the canonical matrix in $\R^{M \times n}$, then
\begin{equation}\label{eq:tnniswild}
\left \{X_{\tnk} e_\beta: \quad \tnk \in \{1,\ldots,\tnn\}\right \} \subset \R^M
\end{equation}
consists of at least two different vectors, where $e_\beta = (0,\ldots,0,1,0,\ldots,0)^T \in \R^n$ is the usual $\alpha$-th unit vector.

We call a $T_{\tnn}$-configuration \emph{wild} if it is wild for each $\beta \in \{1,\ldots,n\}$.
\end{definition}
We observe that necessarily $\tnn \geq n$ for any non-degenerate $T_{\tnn}$-configuration. Also our definition \eqref{eq:nondegenerate} is a somewhat strong generalization of the non-degeneracy for $T_{\tnn}$-configurations in $\R^{4 \times 2}$ which assumes that there are no rank-$1$ connections between the points $\phi_\tnk(P,(C_\tni)_{\tni=1}^{\tnn},(\kappa_{\tni})_{\tni =1}^{\tnn} )$, $\tnk=1,\ldots,n$. Our assumption \eqref{eq:nondegenerate} is needed to ensure that the collection of $T_{\tnn}$-configurations around a specific non-degenerate one form a smooth manifold.

The wildness \eqref{eq:tnniswild} will be used to ensure that for our map $u$ from \Cref{th:main} the partial derivative $\partial_\alpha u$ has high oscillation for each $\alpha \in \{1,\ldots,n\}$.

\begin{lemma}
Assume that $\brac{P,(C_\tni)_{\tni=1}^{\tnn},(\kappa_{\tni})_{\tni =1}^{\tnn}}$ is a non-degenerate $T_{\tnn}$-configuration.

Then there exists an $\eps > 0$ such that if $\tilde{P},(\tilde{C}_\tni)_{\tni=1}^{\tnn},(\tilde{\kappa}_{\tni})_{\tni =1}^{\tnn}$ is an (a priori possibly degenerate) $T_{\tnn}$-configuration which is $\eps$-close to $P,(C_\tni)_{\tni=1}^{\tnn},(\kappa_{\tni})_{\tni =1}^{\tnn}$, i.e. if
\begin{itemize}
 \item $|P-\tilde{P}| \leq \eps$
 \item $|C^{\tni}-\tilde{C}_{\tni}| \leq \eps$ for all $\tni \in \{1,\ldots,\tnn\}$
 \item $|\kappa_{\tni}-\tilde{\kappa}_{\tni}| \leq \eps$ for all $\tni \in \{1,\ldots,\tnn\}$
\end{itemize}
Then $\brac{\tilde{P},(\tilde{C}_\tni)_{\tni=1}^{\tnn},(\tilde{\kappa}_{\tni})_{\tni =1}^{\tnn}}$ is also non-degenerate.
\end{lemma}
\begin{proof}
\begin{enumerate}
 \item By the continuity result of \Cref{la:lipschdepR}, if $C \in \mathscr{R}^o$ then for suitably close-by $\tilde{C} \in \mathscr{R}$ we have $\tilde{C} \in \mathscr{R}^o$.
 \item Follows from the same continuity continuity result \Cref{la:lipschdepR}
 \item Follows from the same continuity continuity result \Cref{la:lipschdepR}
 \item Follows from the same continuity continuity result \Cref{la:lipschdepR}
 \item According to \Cref{la:Risclosed} $\mathscr{R}$ is a closed set. Thus, if \eqref{eq:TNnorankone} holds for $\brac{P,(C_\tni)_{\tni=1}^{\tnn},(\kappa_{\tni})_{\tni =1}^{\tnn}}$ then it holds also for all $T_{\tnn}$-configuration which are suitably close.
\end{enumerate}
\end{proof}

The following elementary, yet crucial observation about $T_{\tnn}$-configurations use only the definition of $T_{\tnn}$-configurations, and \Cref{le:propertiesofA}.
\begin{remark}\label{T4:v1}
Assume that $(P,(C_\tni)_{\tni=1}^{\tnn},(\kappa_{\tni})_{\tni =1}^{\tnn})$ is a $T_{\tnn}$-configuration. As in \eqref{eq:phikdef} we set the base points
\[
 P_\tnk \coloneqq \pi_\tnk(P,(C_\tni)_{\tni=1}^{\tnn},(\kappa_{\tni})_{\tni =1}^{\tnn} ), \quad \tnk\in \{1,\ldots,\tnn\}
\]
and the end-points
\[
 Z_\tnk \coloneqq \phi_k(P,(C_\tni)_{\tni=1}^{\tnn},(\kappa_{\tni})_{\tni =1}^{\tnn} ), \quad \tnk\in \{1,\ldots,\tnn\}.
\]
For any $\mu \in (0,1]$ the set $(P,(\mu C^{\tni})_{i=1}^{\tnn},(\kappa_{\tni})_{\tni =1}^{\tnn})$ is still a $T_{\tnn}$-configuration, with

\[
 \pi_\tnk(P,(\mu C^{\tni})_{i=1}^{\tnn},(\kappa_{\tni})_{\tni =1}^{\tnn} ) =P+\mu C_1 + \ldots + \mu C_{\tnk -1},
\]
and
\[
\begin{split}
 \phi_\tnk(P,(\mu C^{\tni})_{i=1}^{\tnn},(\kappa_{\tni})_{\tni =1}^{\tnn} )=&P+\mu C_1 + \ldots + \mu C^{\tnk -1}+\kappa_{\tnk} \mu C^{\tnk}\\
 =&(1-\mu)P + \mu Z_\tnk\\
 \end{split}
\]
The latter implies in particular that if $(P,(C_\tni)_{\tni=1}^{\tnn},(\kappa_{\tni})_{\tni =1}^{\tnn})$ is a nondegenerate $T_{\tnn}$-configuration, so is $(P,(\mu C^{\tni})_{i=1}^{\tnn},(\kappa_{\tni})_{\tni =1}^{\tnn})$.
\end{remark}

\subsection{Dimensional Analysis}\label{s:dimensions}

\begin{lemma}\label{la:dimK}\label{la:tangentspaceKF}
Let $F: \R^{M \times n} \to \R$ be a smooth, strongly polyconvex function. Then $K_F$ as in \eqref{eq:Kfdef}
%
is an $Mn$-dimensional manifold.

Moreover, at the point $Z = \left ( \begin{array}{c} X\\Y \end{array} \right ) \in \left (\begin{array}{c}
  \R^{M \times n}\\
  \R^{M\times n}\\
 \end{array}\right )$ we have the tangent space
                                           \[
                                            T_Z K_F = \left \{ \left ( \begin{array}{c} \tilde{X}\\ D^2 F(X)[\tilde{X}] \end{array}\right ): \quad \tilde{X} \in \R^{M \times n}  \right \}
                                           \]

%
 \end{lemma}
 \begin{proof}
$K_F$ consists of
\[
 \left(\begin{array}{c} X\\ Y \end{array} \right ) \in \left (\begin{array}{c}
  \R^M \otimes \Ep^1 \R^n\\
  \R^M \otimes \Ep^{n-1} \R^n\\
 \end{array}\right ) \hat{=} \left (\begin{array}{c}
  \R^{M \times n}\\
  \R^{M\times n}\\
 \end{array}\right )
\]
such that $Y=\brac{\sum_{\alpha=1}^n\partial_{X_{i\alpha}}F(X) \hdg dx^\alpha}_{i=1}^M$, which means that $K_F$ can be identified with
\[
 \tilde{K}_F = \left \{\left(\begin{array}{c} X\\ Y \end{array} \right ) \in \left(\begin{array}{c} \R^{M \times n}\\ \R^{M \times n} \end{array} \right ): \quad  \begin{array}{l}Y_{i\alpha} = \partial_{X_{i\alpha}} F(X)\\i = 1,\ldots,M \text{ and } \alpha=1,\ldots,n \end{array}\right \}\subset \left (\begin{array}{c}
  \R^{M \times n}\\
    \R^{M \times n}\\
 \end{array}\right )
\]
Since $F$ is strongly polyconvex $D^2 F(X)$ satisfies the strong Legendre-Hadamard condition and thus $\R^{M \times n}\ni Z \mapsto D^2 F(X) Z$ is invertible, \Cref{la:stronglypolyconvexinvertible}. By the implicit function theorem we find that then $\tilde{K}_F$ must be a $Mn$-dimensional manifold.
%
\end{proof}

%
%
%

For $\tnn \in \N$ we denote the set
\begin{equation}\label{eq:mathcalKf}
 \mathcal{K}_F \coloneqq K_F \times \ldots K_F \subset \left (\begin{array}{c}
  \R^M \otimes \Ep^1 \R^n\\
  \R^M \otimes \Ep^{n-1} \R^n\\
 \end{array}\right )^{\tnn} \hat{=} \left (\begin{array}{c}
  \R^{M \times n}\\
  \R^{M\times n}\\
 \end{array}\right )^{\tnn}.
\end{equation}
We also want to define a manifold within the ``endpoints of $T_{\tnn}$-configurations''
\[
 \mathcal{M}_{\tnn} \subset \left (\begin{array}{c}
  \R^M \otimes \Ep^1 \R^n\\
  \R^M \otimes \Ep^{n-1} \R^n\\
 \end{array}\right )^{\tnn}  \hat{=} \left (\begin{array}{c}
  \R^{M \times n}\\
  \R^{M\times n}\\
 \end{array}\right )^{\tnn}
\]
i.e.
\[
 \mathcal{M}_{\tnn}  \subset \left \{ (Z_1,\ldots,Z_\tnn) \in \left (\begin{array}{c}
  \R^{M\times n} \\
  \R^{M\times n}\\
 \end{array}\right )^{\tnn}: \quad
 \begin{array}{l}
 Z_\tnk = \phi_\tnk(P,(C_\tni)_{\tni=1}^{\tnn},(\kappa_{\tni})_{\tni =1}^{\tnn})\\
  \text{for some $T_{\tnn}$-configuration $P,(C_\tni)_{\tni=1}^{\tnn},(\kappa_{\tni})_{\tni=1}^{\tnn}$}
 \end{array}
 \right \}
\]
where we recall the definition of $T_{\tnn}$-configuration from \Cref{def:Tnconfig} and in particular \eqref{eq:phikdef}. A word of warning: we will -- in a common abuse of notation -- sometimes call $\mathcal{M}_{\tnn}$ as the manifold of $T_{\tnn}$-configurations (instead of the more precise ``set of endpoints of some specific $T_{\tnn}$-configurations'').

The precise choice of $\mathcal{M}_{\tnn}$, and the word ``manifold'' are from following important observation: around any fixed $(\overline{Z}_1,\ldots,\overline{Z}_{\tnn}) \in \mathcal{M}_{\tnn}$ derived from a \emph{non-degenerate} $T_{\tnn}$-configuration, we can find such a manifold $\mathcal{M}_{\tnn}$ whose dimension we can compute. For the case $M=n=2$ this has been done in \cite[Section 4.2]{MS}, \cite[Proposition 4.26]{Kir03} and we extend the approach in \cite[Lemma 2]{Sz04}.

\begin{proposition}\label{pr:dimensionM}
Let $\brac{\overline{P},(\overline{C}_\tni)_{\tni=1}^{\tnn},(\overline{\kappa}_\tni)_{\tni=1}^{\tnn}}$ be a non-degenerate $T_{\tnn}$-configuration.

Denote the $\tnk$-th endpoint of the above $T_{\tnn}$-configuration as
\[
\overline{Z}_\tnk \coloneqq \phi_\tnk (\overline{P},(\overline{C}_\tni)_{\tni=1}^{\tnn},(\overline{\kappa}_\tni)_{\tni=1}^{\tnn}).
\]

Then there exists $\eps > 0$ and a smooth manifold $\mathcal{M}_{\tnn} \subset \left (\begin{array}{c}
  \R^M \otimes \Ep^1 \R^n\\
  \R^M \otimes \Ep^{n-1} \R^n\\
 \end{array}\right )^{\tnn} \hat{=} \left (\begin{array}{c}
  \R^{M \times n}\\
  \R^{M \times n}\\
 \end{array}\right )^{\tnn}$ of $T_{\tnn}$-endpoints containing ($\overline{Z}_1,\ldots,\overline{Z}_{\tnn}$).
This manifold is of dimension
\[
 \dim \brac{ \mathcal{M}_\tnn \cap B_{\eps}((\overline{Z}_1,\ldots,\overline{Z}_{\tnn}))} = {\tnn}\, n(M+1)
\]
Moreover, for any tuple $(Z_1,\ldots,Z_{\tnn}) \subset \mathcal{M}_{\tnn}$ there exists a non-degenerate $T_\tnn$-configuration $({P},({C}_\tni)_{\tni=1}^{\tnn},({\kappa}_\tni)_{\tni=1}^{\tnn})$ such that
\[
 Z_\tni = \phi_\tni ({P},({C}_\tni)_{\tni=1}^{\tnn},({\kappa}_\tni)_{\tni=1}^{\tnn}),
\]
and the dependency of $P,C_\tni,\kappa$ on $Z_{\tni}$ is smooth.

In particular there exists a smooth map (with a slight but justifiable abuse of notation called again $\pi_{\tnk}$)
\[
 \pi_{\tnk}: \mathcal{M}_{\tnn} \to \left (\begin{array}{c}
  \R^M \otimes \Ep^1 \R^n\\
  \R^M \otimes \Ep^{n-1} \R^n\\
 \end{array}\right )\hat{=} \left (\begin{array}{c}
  \R^{M \times n}\\
  \R^{M\times n}\\
 \end{array}\right )
\]
such that for any tuple $(Z_1,\ldots,Z_{\tnn}) \subset \mathcal{M}_{\tnn}$ if we take $({P},({C}_\tni)_{\tni=1}^{\tnn},({\kappa}_\tni)_{\tni=1}^{\tnn})$ from above
then
\[
 \pi_{\tnk}(Z_{\tni}) = P + C_1 + \ldots+C_{\tnk}.
\]
Moreover, for each $\tnk \in \{1,\ldots,\tnn\}$,
\begin{equation}\label{eq:Dpiissurjective}
D\pi_{\tnk}(\overline{Z})\Big|_{T_{\overline{Z}}  \mathcal{M}} : T_{\overline{Z}}  \mathcal{M} \to \left (\begin{array}{c}
  \R^M \otimes \Ep^1 \R^n\\
  \R^M \otimes \Ep^{n-1} \R^n\\
 \end{array}\right ) \hat{=} \left (\begin{array}{c}
  \R^{M \times n}\\
  \R^{M\times n}\\
 \end{array}\right )
\end{equation}
is surjective.
\end{proposition}

To prove \Cref{pr:dimensionM}, the main point is to properly parametrize $(C_\tni)_{\tni=1}^{\tnn}$ with the property $\sum_{\tni=1}^{\tnn} C_{\tni} = 0$.

\begin{lemma}\label{la:mathfrakN}
Set
\[
 \mathscr{C}_{\tnn} \coloneqq \left \{(C_\tni)_{\tni=1}^{\tnn} \subset \mathscr{R}: \quad \sum_{\tni=1}^{\tnn} C_{\tni} = 0 \right \}
\]

Assume that $(\overline{C}_\tni)_{\tni=1}^{\tnn} \in \mathscr{C}_{\tnn}$ are nondegenerate in the sense that
\begin{enumerate}
 \item $\overline{C}_{\tni} \in \mathscr{R}^o$ for $\tni \in \{1,\ldots,\tnn\}$,
 \item if we write
 \[
\overline{C}_\tni = \sum_{\alpha=1}^n \overline{b}^\tni_\alpha\, dx^\alpha \wedge \overline{a}^\tni, \quad \tni=1,\ldots \tnn,
\]
for some
\[\overline{a}^\tni \in \left ( \begin{array}{c} \R^M \setminus \{0\}\\ \R^M \otimes \Ep^{n-2} \R^n \end{array}\right )
\]
and $\overline{b}^\tni \in \R^n$, $|\overline{b}^{\tni}| = 1$, then we have
\begin{equation} \label{eq:binondegenerate}
{\rm span} \{\overline{b}_1,\ldots,\overline{b}_{\tnn}\}=\R^n
\end{equation}
\end{enumerate}

Then there exist $\eps > 0$ such that the set
\[
 B((\overline{C}_1,\ldots,\overline{C}_{\tnn}),\eps) \cap \mathscr{C}_{\tnn}
\]
\begin{itemize}
\item consists only of tuples $(C_1,\ldots,C_{\tnn})$ which are non-degenerate in the above sense
\item can be smoothly parametrized via a smooth submanifold \[\mathfrak{C}_{\tnn} \subset \brac{\R^M}^{\tnn} \times \brac{\mathfrak{so}(n)\otimes \R^M}^{\tnn} \times \brac{\R^{n}}^{\tnn}\]
of dimension
\begin{equation}\label{eq:mathfrakCNdim}
 \dim \mathfrak{C}_N = {\tnn}\, (Mn+n-1)-2Mn
\end{equation}
More precisely consider if we extend the map defined in \eqref{eq:deffrakP} to tuples
\[
 \mathfrak{P}: \brac{\R^M}^{\tnn} \times \brac{\mathfrak{so}(n)\otimes \R^M}^{\tnn} \times \brac{\R^{n}}^{\tnn} \to \mathscr{R}^{\tnn}
\]
via
\[
\begin{split}
 &\mathfrak{P}\brac{(\vec{u}^1,\ldots,\vec{u}^{\tnn}),(\vec{v}^1,\ldots,\vec{v}^{\tnn}), (b^1,\ldots,b^{\tnn})}\\
 \coloneqq&
\brac{ \mathfrak{P}\brac{(\vec{u}^1,\vec{v}^1,b^1)}, \mathfrak{P}\brac{(\vec{u}^2,\vec{v}^2,b^2)}, \ldots, \mathfrak{P}\brac{(\vec{u}^\tnn,\vec{v}^1,b^\tnn)}}
\end{split}
%
\]

Then $\mathfrak{P} \Big |_{\mathfrak{C}_{\tnn} }: \mathfrak{C}_\tnn \to B((\overline{C}_1,\ldots,\overline{C}_{\tnn}),\eps) \cap \mathscr{C}_{\tnn}$ is a diffeomorphic parametrization.

\end{itemize}

\end{lemma}
\begin{proof}
Fix
\[
\overline{C}_\tni = \sum_{\alpha=1}^n \overline{b}^\tni_\alpha\, dx^\alpha \wedge \overline{a}^\tni, \quad \tni=1,\ldots \tnn,
\]

By continuity of parametrization, \Cref{la:lipschdepR}, any tuple $(C_1,\ldots,C_{\tnn}) \in B((\overline{C}_1,\ldots,\overline{C}_{\tnn}),\eps) \cap \mathscr{C}_{\tnn}$ must be non-degenerate, if $\eps$ is small enough.
%

For the given nondegenerate $(\overline{C}_1,\ldots,\overline{C}_{\tnn}) \in \mathscr{R}^{\tnn}$ we choose
\[
(\vec{\overline{u}}^1,\ldots,\vec{\overline{u}}^{\tnn}),(\vec{\overline{v}}^1,\ldots,\vec{\overline{v}}^{\tnn}), (\overline{b}^1,\ldots,\overline{b}^{\tnn}) \in \mathfrak{C}^{\tnn} \subset \brac{\R^M}^{\tnn} \times \brac{\mathfrak{so}(n)\otimes \R^M}^{\tnn} \times \brac{\R^{n}}^{\tnn} \to \mathscr{R}^{\tnn},
\]
where $\mathfrak{C}$ is from \Cref{la:paramaterization} (mind that this is a slight abuse of notation $\mathfrak{C}^{\tnn} = \mathfrak{C}_1 \times \ldots \times \mathfrak{C}_\tnn$
 where each $\mathfrak{C}_\tni$ is from \Cref{la:paramaterization} for the specific $C_\tni$), and we have
\[
\mathfrak{P}(\overline{u}^{\tni}, \vec{\overline{v}}^{\tni}, \overline{b}^\tni) = \overline{C}_\tni,
\]
and we recall the non-degeneracy condition ${\rm span} \{\overline{b}^1,\ldots,\overline{b}^{\tnn}\} = \R^n$ (again, for $\eps$ small enough the $\overline{b}^\tni$ are unique up to sign, as we evince from \Cref{la:lipschdepR}).

Next we define $\Phi: \mathscr{R}^{\tnn} \to \left (\begin{array}{c}
  \R^M \otimes \Ep^1 \R^n\\
  \R^M \otimes \Ep^{n-1} \R^n\\
 \end{array}\right )$
via
\[
 \Phi(C_1,\ldots,C_{\tnn}) \coloneqq \sum_{\tni=1}^{\tnn} C_{\tni}.
\]
We are going to show below that $\Phi$
has full rank at $(\overline{C}_1,\ldots,\overline{C}_{\tnn})$, i.e. \begin{equation}\label{eq:sumeqzeroparam:asda}\rank D\Phi (\overline{C}_1,\ldots,\overline{C}_{\tnn}) = 2Mn.\end{equation} Once we have \eqref{eq:sumeqzeroparam:asda}, the implicit function theorem implies the existence of a small $\eps$ such that $ B((\overline{C}_1,\ldots,\overline{C}_{\tnn}),\eps) \cap \mathscr{C}_{\tnn}$ is a smooth manifold with dimension
\[
\dim  B((\overline{C}_1,\ldots,\overline{C}_{\tnn}),\eps) \cap \mathscr{C}_{\tnn} = {\tnn}\, \dim \mathscr{R}^o-2Mn,
\]
where $\dim \mathscr{R}^o = Mn+n-1$ was computed in \Cref{la:paramaterization}. This establishes \eqref{eq:mathfrakCNdim}, inverting $\mathfrak{P}$ for each tuple, by \Cref{la:paramaterization}, we have found $\mathfrak{C}_\tnn$.

It remains to establish \eqref{eq:sumeqzeroparam:asda}. To show that $\Phi$ has full rank, it suffices to show that $\Phi \circ \mathfrak{P}$ has full rank at the point $(\vec{\overline{u}}^1,\ldots,\vec{\overline{u}}^{\tnn}),(\vec{\overline{v}}^1,\ldots,\vec{\overline{v}}^{\tnn}), (\overline{b}^1,\ldots,\overline{b}^{\tnn})$.

For this we compute $D\brac{\Phi \circ \mathfrak{P}}$ at this point, and show it is surjective.
That is we need to show the following map is surjective
\[
\begin{split}
& (\vec{{u}}^1,\ldots,\vec{{u}}^{\tnn}),(\vec{{v}}^1,\ldots,\vec{{v}}^{\tnn}), ({b}^1,\ldots,{b}^{\tnn}) \mapsto\\
&  \brac{D(\Phi \circ \mathfrak{P})\brac{(\vec{\overline{u}}^1,\ldots,\vec{\overline{u}}^{\tnn}),(\vec{\overline{v}}^1,\ldots,\vec{\overline{v}}^{\tnn}), (\overline{b}^1,\ldots,\overline{b}^{\tnn})}} [(\vec{{u}}^1,\ldots,\vec{{u}}^{\tnn}),(\vec{{v}}^1,\ldots,\vec{{v}}^{\tnn}), ({b}^1,\ldots,{b}^{\tnn})]
\end{split}
\]
To simplify the notation slightly, we use the representation in \Cref{la:rewriteR}, and use as replacement for $\Phi \circ \mathfrak{P}$ the map
\[
\begin{split}
 &\tilde{\Phi}((\underbrace{b^1}_{\in \R^n},\underbrace{u^1}_{\in \R^M},\underbrace{v^1}_{\in {\mathfrak so}(n)\otimes \R^M}), (b^2,u^2,v^2),\ldots,(b^{\tnn},u^{\tnn},v^{\tnn})) \\
  \coloneqq& \sum_{j=1}^{\tnn} \left ( \begin{array}{cccc} \vec{u^{\tnj}}\, b^{\tnj}_1 & \vec{u^{\tnj}}\, b^{\tnj}_2 & \ldots & \vec{u^{\tnj}}\, b^{\tnj}_n \\
\sum_{\alpha=1}^n\vec{v^{\tnj}}_{1\alpha}  b^{\tnj}_\alpha & \sum_{\alpha=1}^n\vec{v^{\tnj}}_{2\alpha} \, b^{\tnj}_\alpha & \ldots & \sum_{\alpha=1}^n\vec{v^{\tnj}}_{n\alpha} \, b^{\tnj}_\alpha \end{array}\right )
 \in
 \left (\begin{array}{c}
 \R^{M \times n}\\
 \R^{M\times n}
\end{array}
 \right )
\end{split}
\]
Fix a matrix $W \in \R^{2M\times n}$ written as
\[
 W = \left ( \begin{array}{c}U\\
 V\\
 \end{array}
 \right )
\]
where $U \in \R^{M\times n}$ and $V \in \R^{M \times n}$.

For surjectivity we need to find $\vec{u}^{\tnj}$, $\vec{v}^{\tnj}$ and $b^{\tnj}$ such that
\[
\begin{split}
\brac{D\tilde{\phi}\brac{(\vec{\overline{u}}^1,\ldots,\vec{\overline{u}}^{\tnn}),(\vec{\overline{v}}^1,\ldots,\vec{\overline{v}}^{\tnn}), (\overline{b}^1,\ldots,\overline{b}^{\tnn})}} [(\vec{{u}}^1,\ldots,\vec{{u}}^{\tnn}),(\vec{{v}}^1,\ldots,\vec{{v}}^{\tnn}), ({b}^1,\ldots,{b}^{\tnn})] =\left ( \begin{array}{c}U\\
 V\\
 \end{array}
 \right ).
\end{split}
\]
Our first choice is to set $b^{\tnj}=0$ for all $\tnj$. Then we need to find $\vec{u}^{\tnj}$ and $\vec{v}^{\tnj}$ such that
\[
\begin{split}
\left ( \begin{array}{c}U\\
 V\\
 \end{array}
 \right )=\sum_{j=1}^{\tnn} \left ( \begin{array}{cccc} \vec{u^{\tnj}}\, \overline{b}^{\tnj}_1 & \vec{u^{\tnj}}\, \overline{b}^{\tnj}_2 & \ldots & \vec{u^{\tnj}}\, \overline{b}^{\tnj}_n \\
\sum_{\alpha=1}^n\vec{v^{\tnj}}_{1\alpha}  \overline{b}^{\tnj}_\alpha & \sum_{\alpha=1}^n\vec{v^{\tnj}}_{2\alpha} \, \overline{b}^{\tnj}_\alpha & \ldots & \sum_{\alpha=1}^n\vec{v^{\tnj}}_{n\alpha} \, \overline{b}^{\tnj}_\alpha \end{array}\right )\\
\end{split}
\]
It is easy to guess $u^{\tnj}$: Since ${\rm span}\{\overline{b}_1,\ldots,\overline{b}_{\tnn}\} = \R^n$ we can write
\[
 U = u_1 \otimes \overline{b}_1 + u_2 \otimes \overline{b}_2 + \ldots + u_{\tnn} \otimes \overline{b}_{\tnn},
\]
that is we have
\[
 U = \left (\begin{array}{cccc} \vec{u^{\tnj}}\, \overline{b}^{\tnj}_1 & \vec{u^{\tnj}}\, \overline{b}^{\tnj}_2 & \ldots & \vec{u^{\tnj}}\, \overline{b}^{\tnj}_n \\ \end{array} \right )
\]

It remains to find $v^1,\ldots,v^{\tnn}$ such that
\begin{equation}\label{eq:aslkdjsdfgvVeq}
 V = \sum_{j=1}^{\tnn}\left ( \begin{array}{cccc} \sum_{\alpha=1}^n\vec{v^{\tnj}}_{1\alpha}  \overline{b}^{\tnj}_\alpha & \sum_{\alpha=1}^n\vec{v^{\tnj}}_{2\alpha} \, \overline{b}^{\tnj}_\alpha & \ldots & \sum_{\alpha=1}^n\vec{v^{\tnj}}_{n\alpha} \, \overline{b}^{\tnj}_\alpha \end{array}\right )
\end{equation}
which we can do essentially in the same way as the construction for $U$, but with a little bit more linear algebra.

We define the pairing $L: \brac{\mathfrak{so}(n) \otimes \R^M} \times \R^n \to \R^{M \times n}$ as follows: for $\vec{v}_{\alpha \beta} = -\vec{v}_{\beta\alpha} \in \R^M$, $\alpha, \beta \in \{1,\ldots,n\}$ and $b \in \R^n$ we set
\[
 L_{(\vec{v}_{\alpha\beta})_{\alpha,\beta=1}^n}[b] \coloneqq \left ( \begin{array}{cccc} \sum_{\alpha=1}^n\vec{v}_{1\alpha}  b_\alpha & \sum_{\alpha=1}^n\vec{v}_{2\alpha} \, b_\alpha & \ldots & \sum_{\alpha=1}^n\vec{v}_{n\alpha} \, b_\alpha \end{array}\right )  \in \R^{M \times n}
\]
Our goal is to find $(\vec{v}^\tni)_{\tni =1}^{\tnn} \in \brac{\mathfrak{so}(n) \otimes \R^M}^{\tnn}$ so that \eqref{eq:aslkdjsdfgvVeq} holds, i.e.
\begin{equation}\label{eq:aslkdjsdfgvVeq2}
V = \sum_{\tnj=1}^{\tnn} L_{v^\tnj} (\overline{b}^\tnj).
\end{equation}

Now we pick $\vec{\nu}^\gamma_{\alpha \beta} \in \R^M$, $1 \leq \alpha,\beta,\gamma \leq n$ with the following properties
\begin{itemize}
\item $V =  \left ( \begin{array}{cccc} \vec{\nu^2}_{12}   & -\vec{\nu^1}_{12}& \ldots & -\vec{\nu^1}_{1n} \end{array}\right )$,
\item $v^\gamma_{\alpha \beta} = - v^\gamma_{\beta \alpha}$
\item all other $v^\gamma_{\alpha \beta}$ are set to zero.
\end{itemize}

Then
\[
\begin{split}
 V=& \sum_{\gamma=1}^n
 \left ( \begin{array}{cccc} \vec{\nu^{\gamma}}_{1\gamma}  & \vec{\nu^{\gamma}}_{2\gamma}& \ldots & \vec{\nu^{\gamma}}_{n\gamma} \end{array}\right )  \\
 =& \left ( \begin{array}{cccc} 0  & -\vec{\nu^1}_{12}& \ldots & -\vec{\nu^1}_{1n} \end{array}\right )\\
 &+\left ( \begin{array}{cccc} \vec{\nu^2}_{12}  & 0& \ldots & -\vec{\nu^2}_{2n} \end{array}\right )\\
 &+\ldots\\
 &+\left ( \begin{array}{cccc} \vec{\nu^n}_{1n}  & \vec{\nu^n}_{2n}& \ldots & 0 \end{array}\right )
 \end{split}
\]

Denote by $e_\gamma = (0,0,\ldots,0,1,0,\ldots)^T \in \R^n$ the $\gamma$-th unit vector. Then
\[
 L_{(\vec{\nu}_{\alpha \beta}^\gamma)}[e_\gamma] = \left ( \begin{array}{cccc} \vec{\nu}^\gamma_{1\gamma}  & \vec{\nu}^\gamma_{2\gamma} & \ldots & \vec{v}^\gamma_{n\gamma} \end{array}\right )
\]
So, from our previous choice of $\vec{v}_{\alpha \beta}^i$ we have
\[
 V=\sum_{\gamma=1}^n L_{(\vec{\nu}_{\alpha \beta}^\gamma)}[e_\gamma]
\]

Since $(\overline{b}^\tnell)_{\tnell=1}^{\tnn}$ spans all of $\R^n$ we have can find $\lambda_{\tnell \gamma} \in \R$, $\gamma \in \{1,\ldots,n\}$ and $\tnell \in \{1,\ldots,\tnn\}$,
\[
 e_\gamma = \sum_{\tnell=1}^{\tnn} \lambda_{\tnell \gamma} \overline{b}^\tnell.
\]
Then, by multilinearity,
we have
\[
\begin{split}
 V=&\sum_{\tnell=1}^{\tnn}\sum_{\gamma=1}^n  \lambda_{\tnell \gamma}L_{\brac{(\vec{\nu}_{\alpha \beta}^\gamma)_{\alpha \beta}}}[ \overline{b}^\tnell]\\
 =&\sum_{\tnell=1}^{\tnn}  L_{ \brac{(\sum_{\gamma=1}^n\lambda_{\tnell \gamma}\vec{\nu}_{\alpha \beta}^\gamma)_{\alpha \beta}}}[ \overline{b}^\tnell]
\end{split}
\]
Observe we still have $\sum_{\gamma=1}^n\lambda_{\tnell \gamma}\vec{\nu}_{\alpha \beta}^\gamma = -\sum_{\gamma=1}^n\lambda_{\tnell \gamma}\vec{\nu}_{\beta \alpha}^\gamma$, so indeed we have \eqref{eq:aslkdjsdfgvVeq2}.

So this implies if $(\overline{b}_1,\ldots,\overline{b}_\tnn)$ spans all of $\R^n$, then for any
\[
 W \in \R^{2M \times n}
\]
there exists some choice $(\vec{u}^1,\ldots,\vec{u}^{\tnn})$, $(\vec{v}^1,\ldots,\vec{v}^{\tnn})$ such that
\[
\brac{D(\Phi \circ \mathfrak{P})\brac{(\vec{\overline{u}}^1,\ldots,\vec{\overline{u}}^{\tnn}),(\vec{\overline{v}}^1,\ldots,\vec{\overline{v}}^{\tnn}), (\overline{b}^1,\ldots,\overline{b}^{\tnn})}} [(\vec{{u}}^1,\ldots,\vec{{u}}^{\tnn}),(\vec{{v}}^1,\ldots,\vec{{v}}^{\tnn}), (0,\ldots,0)]=W.
\]
And thus, $D\Phi(\overline{b},\overline{u},\overline{v})$ is a surjective linear map, thus it has full rank, that is \eqref{eq:sumeqzeroparam:asda} is established.

We can conclude.
%
%
%
%
%
%
%
%
%
%
%
%
\end{proof}

\begin{proof}[Proof of \Cref{pr:dimensionM}]

Fix the $T_\tnn$-configuration $\brac{\overline{P},(\overline{C}_\tni)_{\tni=1}^{\tnn},(\overline{\kappa}_\tni)_{\tni=1}^{\tnn}}$.

Consider $\mathscr{C}_\tnn$ from \Cref{la:mathfrakN}
\[
 \Phi : \left (\begin{array}{c}
  \R^M \otimes \Ep^1 \R^n\\
  \R^M \otimes \Ep^{n-1} \R^n\\
 \end{array}\right ) \times \mathscr{C}_\tnn \times (1,\infty)^\tnn \to \left (\begin{array}{c}
  \R^M \otimes \Ep^1 \R^n\\
  \R^M \otimes \Ep^{n-1} \R^n\\
 \end{array}\right )^{\tnn},
\]
given by
\begin{equation}\label{eq:mathcalMparam}
 \Phi\brac{P,(C_\tni)_{\tni=1}^{\tnn},(\kappa_\tni)_{\tni=1}^{\tnn}} \coloneqq \brac{\phi_\tnk (P,(C_\tni)_{\tni=1}^{\tnn},(\kappa_\tni)_{\tni=1}^{\tnn}) }_{\tnk=1}^{\tnn}
\end{equation}
We see that
\[
 D_P \Phi\brac{\overline{P},(\overline{C}_\tni)_{\tni=1}^{\tnn},(\overline{\kappa}_\tni)_{\tni=1}^{\tnn}}[Q] = (Q,Q,\ldots,Q)
\]
\[
 D_{\kappa_{\tnk}} \Phi\brac{\overline{P},(\overline{C}_\tni)_{\tni=1}^{\tnn},(\overline{\kappa}_\tni)_{\tni=1}^{\tnn}}[\kappa] = (0,\ldots,0,\kappa C_{\tnk},0,\ldots,0)
\]
and for $C \in T_{\overline{C}_{\tni}} \mathscr{C}$
\[
 D_{C_\tni}\Phi\brac{\overline{P},(\overline{C}_\tni)_{\tni=1}^{\tnn},(\overline{\kappa}_\tni)_{\tni=1}^{\tnn}}[C]=(0,\ldots,0, \kappa_\tni C, C,\ldots, C)
\]
These seem to be all linear independent, so $D\Phi$ has full rank. Observe that the domain has dimension

\begin{itemize}
\item Choice of the base point $P \in \left (\begin{array}{c}
  \R^M \otimes \Ep^1 \R^n\\
  \R^M \otimes \Ep^{n-1} \R^n\\
 \end{array}\right )$. Degrees of freedom: $2Mn$.
\item Dimension of $\mathcal{C}_{\tnn}$ is ${\tnn}\, (Mn+n-1)-2Mn$, by \Cref{la:mathfrakN}.
 \item choice of $\kappa_1,\ldots,\kappa_\tnn$. Degrees of freedom: ${\tnn}$.
\end{itemize}
So the dimension of the domain of $\Phi$ is $\tnn n(M+1)$ and its map into an $2Mn\tnn$-dimensional space. By the implicit function space, for a suitably small $\eps$ we can set
\[
 \mathcal{M}_{\tnn} \coloneqq\left \{ \Phi (P,(C_\tni)_{\tni=1}^{\tnn},(\kappa_\tni)_{\tni=1}^{\tnn}):\
 \begin{array}{l}
|P-\overline{P}| < \eps\\
\max_{\tni}|\kappa_\tni-\overline{\kappa}_\tni| < \eps\\
|C_{\tni}-\overline{C}_{\tni}| < \eps
\end{array}
 \right   \}
\]
is a $\tnn n(M+1)$-manifold and $\Phi$ is a diffeomorphism of a neighborhood of $(\overline{P},\overline{C}_1,\ldots,\overline{C}_n,\overline{\kappa}_1,\ldots,\overline{\kappa}_{\tnn})$ onto $\mathcal{M}_{\tnn}$. In particular it is one-to-one onto its target, so for each $(Z_1,\ldots,Z_\tnn) \in \mathcal{M}_{\tnn}$ there exists exactly one $({P},{C}_1,\ldots,{C}_n,{\kappa}_1,\ldots,{\kappa}_{\tnn})$ in the neighborhood of $(\overline{P},\overline{C}_1,\ldots,\overline{C}_n,\overline{\kappa}_1,\ldots,\overline{\kappa}_{\tnn})$ that $\Phi$ maps into $(Z_1,\ldots,Z_\tnn)$. So if we set
\[
 \pi_{\tnk}(Z_1,\ldots,Z_\tnn) \coloneqq P+C_1+\ldots+C_{\tnk}
\]
we have a well-defined smooth map on $\mathcal{M}_{\tnn}$.

Fix some $\tnk$ then for $|P-\overline{P}| < \eps$, $\max_{\tni}|\kappa_\tni-\overline{\kappa}_\tni| < \eps$, $|C_{\tni}-\overline{C}_{\tni}|$,
\[
 \pi_{\tnk} \circ \Phi (P,(C_\tni)_{\tni=1}^{\tnn},(\kappa_\tni)_{\tni=1}^{\tnn}) = P + C_1 + \ldots + C_{\tnk}
\]
Already by the variations in $P$ it is clear that \eqref{eq:Dpiissurjective} is satisfied.
%
%
%
%
\end{proof}

\begin{lemma}\label{la:goddammofo}
For $\tnn \geq 1$ let \[
(\overline{b}^{\tni},\vec{\overline{u}^{\tni}}, \vec{\overline{v}^{\tni}}) \in    \R^n \times \R^M \times \mathfrak{so}(n)\otimes \R^M, \quad \tni=1,\ldots,\tnn.
    \]
Set
\[\begin{split}
 \mathfrak{C}^\tnell\coloneqq  \Big \{C \in \left (\begin{array}{c}\R^{M \times n}\\
\R^{M \times n}
\end{array}\right ): C =& \left ( \begin{array}{cccc} \vec{\overline{u}^{\tni}}\, b^{\tni}_1 & \vec{\overline{u}}\, b^{\tni}_2 & \ldots & \vec{\overline{u}}\, b^{\tni}_n \\
\sum_{\alpha=1}^n\vec{\overline{v}}_{1\alpha}  b^{\tni}_\alpha & \sum_{\alpha=1}^n\vec{\overline{v}}_{2\alpha} \, b^{\tni}_\alpha & \ldots & \sum_{\alpha=1}^n\vec{\overline{v}}_{n\alpha} \, b^{\tni}_\alpha \end{array} \right )\\
&+ \left ( \begin{array}{cccc} \vec{u^{\tni}}\, \overline{b^{\tni}}_1 & \vec{u^{\tni}}\, \overline{b^{\tni}}_2 & \ldots & \vec{u^{\tni}}\, \overline{b^{\tni}}_n \\
\sum_{\alpha=1}^n\vec{v^{\tni}}_{1\alpha}  \overline{b^{\tni}}_\alpha & \sum_{\alpha=1}^n\vec{v^{\tni}}_{2\alpha} \, \overline{b^{\tni}}_\alpha & \ldots & \sum_{\alpha=1}^n\vec{v^{\tni}}_{n\alpha} \, \overline{b^{\tni}}_\alpha \end{array} \right )  \Big \}\\
\end{split}
\]
Also set
\[
 \overline{C}^\tnell \coloneqq  \left ( \begin{array}{cccc} \vec{\overline{u}^{\tni}}\, \overline{b}^{\tni}_1 & \vec{\overline{u}}\, \overline{b}^{\tni}_2 & \ldots & \vec{\overline{u}}\, \overline{b}^{\tni}_n \\
\sum_{\alpha=1}^n\vec{\overline{v}}_{1\alpha}  \overline{b}^{\tni}_\alpha & \sum_{\alpha=1}^n\vec{\overline{v}}_{2\alpha} \, \overline{b}^{\tni}_\alpha & \ldots & \sum_{\alpha=1}^n\vec{\overline{v}}_{n\alpha} \, \overline{b}^{\tni}_\alpha \end{array} \right )
\]
Lastly fix $\overline{\kappa}_\tni > 0$ for $\tni \in (1,\ldots,\tnn)$.

Assume that $\kappa_\tni \in \R$ and $C_\tni \in \mathfrak{C}^\tni$ for $\tni \in \{1,\ldots,\tnn\}$ such that
\begin{equation}\label{eq:askldmuniquev1}
  \kappa_{\tnell} \overline{C}_\tnell  =\sum_{\tnk=1}^{\tnell-1} C_\tnk +\overline{\kappa_\tnell} C_{\tnell} \quad \tnell = 1,\ldots,\tnn
\end{equation}
and
\begin{equation}\label{eq:askldmuniquev2}
 \sum_{\tnell=1}^{\tnn} C_{\tnell} = 0.
\end{equation}

Then $\kappa_{\tni} = 0$ and $C_\tni \in \mathfrak{C}^\tni$ for $\tni \in \{1,\ldots,\tnn\}$.
\end{lemma}
\begin{proof}
We argue by induction. If $\tnn = 1$ \eqref{eq:askldmuniquev2} implies $C_1 = 0$, and thus \eqref{eq:askldmuniquev1} implies $\kappa_1 = 0$.

Assume now the claim is proven for $\tnn-1$, and $\tnn \geq 2$.

We write down the first two equations \eqref{eq:askldmuniquev1}
\begin{equation}\label{eq:askldmuniquev21fuck}
  \kappa_{1} \overline{C}_1  = \overline{\kappa_1} C_{1}
\end{equation}
\[
   \kappa_{2} \overline{C}_2  = C_1 +\overline{\kappa_2} C_{2}
\]
Plugging the first equation into the second one we have
\[
   \kappa_{2} \overline{C}_2  -\overline{\kappa_2} C_{2}= \frac{\kappa_1}{\overline{\kappa_1}}\overline{C}_1.
\]
Observe that $\overline{C}_2 \in \mathfrak{C}^2$, which is a linear space, so we have
\[
\frac{\kappa_1}{\overline{\kappa_1}}\overline{C}_1 \in \mathfrak{C}^2,
\]
by \Cref{la:fuckyouallofyoufuckfuckfuck1} this implies $\kappa_1 = 0$. By \eqref{eq:askldmuniquev21fuck} we also have $C_1=0$. Now we can use the induction hypothesis.

\end{proof}

\begin{lemma}\label{la:kerDpi}
Let $\brac{\overline{P},(\overline{C}_\tni)_{\tni=1}^{\tnn},(\overline{\kappa}_{\tni})_{\tni =1}^{\tnn}}$ be a non-degenerate $T_{\tnn}$-configuration with endpoints $\overline{Z} = (\overline{Z}_1,\ldots,\overline{Z}_\tnn) \in
\left (\begin{array}{c}\R^{M \times n}\\
\R^{M \times n}
\end{array}\right )^{\tnn}
$.
Take from from \Cref{pr:dimensionM} $\mathcal{M}_{\tnn}$ the manifold of $T_{\tnn}$-endpoints around $\overline{Z}$ and $\pi_{1}: \mathcal{M}_{\tnn} \to \left (\begin{array}{c}\R^{M \times n}\\
\R^{M \times n}
\end{array}\right )$.

Then
\[
 \ker d\pi_{1}(\overline{Z}) \subset \left (\begin{array}{c}\R^{M \times n}\\
\R^{M \times n}
\end{array}\right )^{\tnn}
\]
consists of tuples
\[
 (Z_1,\ldots,Z_{\tnn}) \in \left (\begin{array}{c}\R^{M \times n}\\
\R^{M \times n}
\end{array}\right )^{\tnn}
\]
such that each $Z_{\tnell}$, $\tnell \in \{1,\ldots,\tnn\}$ has the representation,
\begin{equation}\label{eq:Sz04:l5:Tangentspacev1}
 Z_{\tnell} =\sum_{\tnk=1}^{\tnell-1} C_\tnk +\overline{\kappa_\tnell} C_{\tnell} + \kappa_{\tnell} \overline{C}_\tnell
 \end{equation}
where $(C_1,\ldots,C_{\tnn}) \in T_{(\overline{C}_1,\ldots,\overline{C}_{\tnn})} \mathcal{C}_{\tnn}$ and $(\kappa_1,\ldots,\kappa_{\tnn}) \in \R^\tnn$ -- Here we take the manifold $\mathcal{C}_{\tnn}$ around $(\overline{C}_1,\ldots,\overline{C}_{\tnn})$ from \Cref{la:mathfrakN}.

The representation in \eqref{eq:Sz04:l5:Tangentspacev1} is unique (in $C_1,\ldots,C_\tnn$ and $\kappa_1,\ldots,\kappa_{\tnn}$) and we have
\begin{equation}\label{eq:kerdpi1dim}
 \dim \ker D\pi_1(\overline{Z}) = {\tnn}\, (Mn+n)-2Mn
\end{equation}
More generally, if we set
\[
 \mathfrak{Z} \coloneqq  \left \{(Z_1,\ldots,Z_\tnn) \in T_{\overline{Z}} \mathcal{M}_\tnn: \quad Z_{\tnell} =\sum_{\tnk=1}^{\tnell-1} C_\tnk +\overline{\kappa_\tnell} C_{\tnell} \text{ for some  $({C}_1,\ldots,{C}_{\tnn}) \in T_{(\overline{C}_1,\ldots,\overline{C}_{\tnn})} \mathcal{C}_{\tnn}$} \right\}
\]
\[
 \mathfrak{K} \coloneqq  \left \{ (Z_1,\ldots,Z_\tnn) \in T_{\overline{Z}} \mathcal{M}_\tnn : \quad Z_{\tnell} = \kappa_{\tnell} \overline{C}_{\tnell} \text{ for some $(\kappa_1,\ldots,\kappa_{\tnn}) \in \R^{\tnn}$}\right \}
\]
Then $\mathfrak{Z}  \cap \mathfrak{K} = \{0\}$ and
\[
 \ker D\pi_1(\overline{Z}) = \mathfrak{K} + \mathfrak{Z}.
\]

\end{lemma}
\begin{proof}
 We begin by computing $\ker D\pi_1(\overline{Z})$, that is we need to compute all
\[
 Z \coloneqq (Z_1,\ldots,Z_\tnn) \in T_{\mathcal{M}_\tnn} \brac{\R^{4 \times 2}}^{\tnn}
\]
such that
\[
 D\pi_1(\overline{Z})[Z] =0.
\]

From \Cref{la:mathfrakN} we take the manifold $\mathcal{C}_{\tnn}$ around $(\overline{C}_1,\ldots,\overline{C}_{\tnn})$ that consists of nondegenerate $\mathscr{R}$-tuples, parametrized from a submanifold
\[
 \mathfrak{C}_{\tnn} \subset  \brac{\R^M}^{\tnn} \times \brac{\mathfrak{so}(n)\otimes \R^M}^{\tnn} \times \brac{\R^{n}}^{\tnn} \to \mathscr{R}^{\tnn}
\]
via the map $\mathfrak{P} \Big |_{\mathfrak{C}_{\tnn}}$,
\[
 \mathfrak{P}: \underbrace{\brac{\R^M}^{\tnn}}_{\vec{u}} \times \underbrace{\brac{\mathfrak{so}(n)\otimes \R^M}^{\tnn} \times \brac{\R^{n}}^{\tnn}}_{\vec{v}} \to \underbrace{\mathscr{R}^{\tnn}}_{b}
\]
Then from \Cref{pr:dimensionM} we infer that $\mathcal{M}_{\tnn}$ is parametrized by \eqref{eq:mathcalMparam}
\[
 \Phi : \left (\begin{array}{c}
  \R^{M \times n}\\
  \R^{M \times n}\\
 \end{array}\right ) \times \mathscr{C}_\tnn \times (1,\infty)^\tnn \to \left (\begin{array}{c}
  \R^{M \times n}\\
  \R^{M \times n}\\
 \end{array}\right )^{\tnn},
\]
\[
 \Phi\brac{P,(C_\tni)_{\tni=1}^{\tnn},(\kappa_\tni)_{\tni=1}^{\tnn}} \coloneqq \brac{\phi_\tnk (P,(C_\tni)_{\tni=1}^{\tnn},(\kappa_\tni)_{\tni=1}^{\tnn}) }_{\tnk=1}^{\tnn}
\]
Clearly
\[
 \pi_1 \circ \Phi\brac{P,(C_\tni)_{\tni=1}^{\tnn},(\kappa_\tni)_{\tni=1}^{\tnn}} = P
\]
So for any $P \in \left (\begin{array}{c}
  \R^{M \times n}\\
  \R^{M \times n}\\
 \end{array}\right )$, $(C_1,\ldots,C_\tnn) \in T_{\overline{C}_1,\ldots,\overline{C}_{\tnn}} \mathcal{C}_{\tnn}$, $\kappa_1,\ldots,\kappa_n \in \R$ we have
\[
 \frac{d}{d\eps}\big |_{\eps =0}\pi_1 \circ \Phi\brac{\brac{\overline{P},(\overline{C}_\tni)_{\tni=1}^{\tnn},(\overline{\kappa}_\tni)_{\tni=1}^{\tnn}} + \eps \brac{P,(C_\tni)_{\tni=1}^{\tnn},(\kappa_\tni)_{\tni=1}^{\tnn}}} = P
\]
That is
\[
 \frac{d}{d\eps}\big |_{\eps =0}\Phi\brac{\brac{\overline{P},(\overline{C}_\tni)_{\tni=1}^{\tnn},(\overline{\kappa}_\tni)_{\tni=1}^{\tnn}} + \eps \brac{P,(C_\tni)_{\tni=1}^{\tnn},(\kappa_\tni)_{\tni=1}^{\tnn}}} \in \ker D\pi(\overline{Z})
\]
if and only if $P = 0$.

From the definition of $\phi_1,\ldots,\phi_{\tnn}$ we find
\[
\begin{split}
  \frac{d}{d\eps}\big |_{\eps =0}\Phi\brac{\brac{\overline{P},(\overline{C}_\tni)_{\tni=1}^{\tnn},(\overline{\kappa}_\tni)_{\tni=1}^{\tnn}} + \eps \brac{P,(C_\tni)_{\tni=1}^{\tnn},(\kappa_\tni)_{\tni=1}^{\tnn}}} =: (Z_1,Z_2,\ldots,Z_{\tnn})
\end{split}
  \]
where
\begin{equation}\label{eq:kerDpiparam}
 Z_{\tnell} = \sum_{\tnk=1}^{\tnell-1} C_\tnk +\overline{\kappa_\tnell} C_{\tnell} + \kappa_{\tnell} \overline{C}_\tnell.
\end{equation}
Since $\Phi$ is a parametrization of $\mathcal{M}_{\tnn}$ we have shown
\[
 (Z_1,\ldots,Z_{\tnn}) \in \ker D\pi_1 (\overline{Z})
\]
if and only if \eqref{eq:kerDpiparam} holds for some $(C_1,\ldots,C_\tnn) \in T_{\overline{C}_1,\ldots,\overline{C}_{\tnn}} \mathcal{C}_{\tnn}$, $\kappa_1,\ldots,\kappa_n \in \R$.

So consider the linear map
\[
 A: T_{\overline{C}_1,\ldots,\overline{C}_{\tnn}} \mathcal{C}_{\tnn} \times \R^n \to \ker D\pi_1(\overline{Z}) \subset \left (\begin{array}{c}
  \R^{M \times n}\\
  \R^{M \times n}\\
 \end{array}\right )^{\tnn}
\]
given by
\[
 A(\brac{C_1,\ldots,C_\tnn},\kappa_1,\ldots,\kappa_\tnn) = (Z_1,\ldots,Z_\tnn)
\]
where $Z_{\tni}$ are given by \eqref{eq:kerDpiparam}. By \Cref{la:goddammofo} $A$ is an invertible matrix. This establishes \eqref{eq:kerdpi1dim} since by \Cref{la:mathfrakN}
\[
 \brac{T_{\overline{C}_1,\ldots,\overline{C}_{\tnn}} \mathcal{C}_{\tnn}} \times \R^n= {\tnn}\, (Mn+n-1)-2Mn+n
\]
The decomposition into $\mathfrak{K}$ and $\mathfrak{Z}$ follows from \Cref{la:goddammofo}.
\end{proof}

The following is our version of \cite[Lemma 5]{Sz04}
\begin{lemma}\label{la:sz:la5}
Fix $M,n \in \N$, $n,M \geq 2$. There exist $\tnn_0=\tnn_0(n,M)$ such that the following holds.
Assume that $\tnn \geq \tnn_0$ and let $\brac{\overline{P},(\overline{C}_\tni)_{\tni=1}^{\tnn},(\overline{\kappa}_{\tni})_{\tni =1}^{\tnn}}$ be a non-degenerate $T_{\tnn}$-configuration with endpoints $\overline{Z} = (\overline{Z}_1,\ldots,\overline{Z}_\tnn) \in
\left (\begin{array}{c}\R^{M \times n}\\
\R^{M \times n}
\end{array}\right )^{\tnn}
$.

Take from from \Cref{pr:dimensionM} $\mathcal{M}_{\tnn}$ the manifold of $T_{\tnn}$-endpoints around $\overline{Z}$ and $\pi_{1}: \mathcal{M}_{\tnn} \to \left (\begin{array}{c}\R^{M \times n}\\
\R^{M \times n}
\end{array}\right )$.

Denote for $1\le i_1 < i_2 < \ldots < i_k \leq \tnn$ the map
\[
 p_{i_1,\ldots,i_k}: (\R^{Mn})^{\tnn} \to (\R^{Mn})^{\tnn}
\]
which sets the $i_{j}$-components to be zero, i.e.
\[
 p_{i_1,\ldots,i_k} (v_1,\ldots,v_{\tnn}) = (v_1,\ldots,v_{i_1-1},0,v_{i_1+1}, \ldots, v_{i_2-1},0,v_{i_2+1},\ldots,\nu_{{i_k}-1},0,\nu_{{i_k}+1},\ldots,\nu_{\tnn}).
\]

Then $\ker D\pi_1 \subset T_{\overline{Z}} \mathcal{M}_{\tnn} \subset \left (\begin{array}{c}\R^{M \times n}\\
\R^{M \times n}
\end{array}\right )^{\tnn}$ contains an $Mn\tnn$-dimensional subspace $L$ with the property that
\[
 \dim (L \cap \im p_{i_1\ldots,i_{k}}) \leq Mn(\tnn-k)
\]
for any $k \in \{1,\ldots,\tnn\}$ and $1 \leq i_1 < \ldots i_k \leq \tnn$.
\end{lemma}

\begin{proof}
Fix $\overline{Z} = (\overline{Z}_1,\ldots,\overline{Z}_\tnn) \in \mathcal{M}_\tnn \subset (\R^{M \times n})^{\tnn}$.

We will show below that
\begin{equation}\label{eq:sz04:l5:24}
\dim(\ker D\pi_1(\overline{Z}) \cap \im p_{i_1\ldots,i_{k}}) \leq  {\tnn}\, (Mn+n)-2Mn-Mnk
\end{equation}
for all $1\leq i_1 < \ldots < i_k \leq \tnn$ and all $k \in \{1,\ldots,\tnn\}$.

Assume for the moment that \eqref{eq:sz04:l5:24} is established, take an orthonormal basis
\[
 \span \{o_1,\ldots,o_K\} = \ker D\pi_1(\overline{Z}) \cap \im p_{i_1,\ldots,i_{k}},
\]
where $K \leq \tnn (Mn+n)-2Mn-Mnk$. Recall from \Cref{la:kerDpi} that
\[
 \dim \ker D\pi_1(\overline{Z}) = {\tnn}\, (Mn+n)-2Mn.
\]
So we can extend the orthonormal system $\{o_1,\ldots,o_K\}$ above to an orthonormal system $\{o_1,\ldots,o_{\tnn (Mn+n)-2Mn}\}$ such that
\[
 \span \{o_1,\ldots,o_K, o_{K+1}, \ldots, o_{\tnn (Mn+n)-2Mn}\} = \ker D\pi_1(\overline{Z}).
\]
Take
\[
 L \coloneqq \span \{o_1,\ldots,o_{Mn\tnn-Mnk},o_{\tnn (Mn+n)-2Mn-Mnk+1}, \ldots,o_{\tnn (Mn+n)-2Mn}\}.
\]
Observe that this defines an $Mn\tnn$-dimensional space, if
\[
 Mn\tnn-Mnk < \tnn (Mn+n)-2Mn-Mnk+1,
\]
i.e., if $\tnn \geq 2M$.

By definition of $L$,
\[
 \dim (L \cap \im p_{i_1\ldots,i_{k}}) = Mn(\tnn - k).
\]
That is, once we have established \eqref{eq:sz04:l5:24} we can conclude.

We need \underline{to prove \eqref{eq:sz04:l5:24}}.

For \underline{$k=\tnn$} this is easy, since $\im p_{1,\ldots,\tnn} = \{0\}$ and $\tnn \geq 2M$.

So from now on we assume \underline{$k \leq \tnn-1$}.

By \Cref{la:kerDpi} any element $(Z_1,\ldots,Z_{\tnn}) \in \ker D\pi_1(\overline{Z}) \subset T_{\overline{Z}} \mathcal{M}$ is then of the form
\begin{equation}\label{eq:Sz04:l5:Tangentspace}
\begin{split}
 Z_{\tnell} =&\sum_{\tnk=1}^{\tnell-1} C_\tnk +\overline{\kappa_\tnell} C_{\tnell} + \kappa_{\tnell} \overline{C}_\tnell
\end{split}
 \end{equation}
where $(C_1,\ldots,C_{\tnn}) \in T_{(\overline{C}_1,\ldots,\overline{C}_{\tnn})} \mathcal{C}_{\tnn}$. Such a $(Z_1,\ldots,Z_{\tnn}) \in \ker D\pi_1(\overline{Z}) \cap \im p_{i_1\ldots i_k}$ if and only if
\begin{equation}\label{eq:la5:equation}
 \sum_{\tnk=1}^{{i_\ell}-1} C_\tnk +\overline{\kappa}_{i_\ell} C_{{i_\ell}} =- \kappa_{{i_\ell}} \overline{C}_{i_\ell}, \quad \ell = 1,\ldots,k.
\end{equation}
These are $k$ conditions, but recall that we also have a $k+1$-st condition
\begin{equation}\label{eq:la5:equationzero}
 \sum_{\tnk=1}^{\tnn} C_{\tnk} =0.
\end{equation}

\eqref{eq:la5:equation} and \eqref{eq:la5:equationzero} are linear equations on $(Z_1,\ldots,Z_\tnn)$. So we have by \Cref{la:kerDpi} and \Cref{la:RCsets},
\begin{equation}\label{eq:solutionspaceasdA}
\begin{split}
 &\dim \{(Z_1,\ldots,Z_\tnn) \in \left (\begin{array}{c}\R^{M \times n}\\
\R^{M \times n}
\end{array}\right )^{\tnn}: \quad \text{ \eqref{eq:Sz04:l5:Tangentspace}, \eqref{eq:la5:equation}, \eqref{eq:la5:equationzero} hold}\}\\
\leq&\underbrace{\dim \mathfrak{K}}_{=\tnn} + \dim \{(Z_1,\ldots,Z_\tnn) \in \mathfrak{Z}: \quad \text{ \eqref{eq:Sz04:l5:Tangentspacev2zero}, \eqref{eq:la5:equationv2zero}, \eqref{eq:la5:equationzero} hold}\}
\end{split}
\end{equation}
where
\begin{equation}\label{eq:la5:equationv2zero}
 \sum_{\tnk=1}^{{i_\ell}-1} C_\tnk +\overline{\kappa}_{i_\ell} C_{{i_\ell}} =0, \quad \ell = 1,\ldots,k.
\end{equation}
and
\begin{equation}\label{eq:Sz04:l5:Tangentspacev2zero}
\begin{split}
 Z_{\tnell} =&\sum_{\tnk=1}^{\tnell-1} C_\tnk +\overline{\kappa_\tnell} C_{\tnell}
\end{split}
 \end{equation}

Next we observe that the map
\[
 \psi (Z_1,\ldots,Z_{\tnn}) \coloneqq  (\frac{1}{\overline{\kappa_1}}Z_1, \frac{Z_2-Z_1}{\overline{\kappa_2}}, \ldots,\frac{Z_{\tnn}-Z_{\tnn-1}}{\overline{\kappa_{\tnn}}})
\]
Then
\[
 \psi : \mathfrak{Z} \to \{(C_1,\ldots,C_\tnn): C_\tni \in \mathfrak{C}_\tni\}
\]
is bijective (recall $\mathfrak{C}_{\tni}$ from \Cref{la:goddammofo}). Thus,
\[
\begin{split}
 &\dim \{(Z_1,\ldots,Z_\tnn) \in \mathfrak{Z}: \quad \text{ \eqref{eq:Sz04:l5:Tangentspacev2zero}, \eqref{eq:la5:equationv2zero}, \eqref{eq:la5:equationzero} hold}\}\\
= &\dim \{(C_1,\ldots,C_\tnn): C_\tni \in \mathfrak{C}^\tni, \quad \text{\eqref{eq:la5:equationv2zero}, \eqref{eq:la5:equationzero} hold}\}.
\end{split}
\]
In view of this and \eqref{eq:solutionspaceasdA}, our goal \eqref{eq:sz04:l5:24} becomes
\begin{equation}\label{eq:sz04:l5:24v2}
\dim \{(C_1,\ldots,C_\tnn): C_\tni \in \mathfrak{C}^\tni, \quad \text{\eqref{eq:la5:equationv2zero}, \eqref{eq:la5:equationzero} hold}\}\leq  {\tnn}\, (Mn+n-1)-2Mn-Mnk
\end{equation}
Observe
\[
 {\tnn}\, (Mn+n-1)-2Mn-Mnk \geq 0 \Leftarrow \tnn \geq 2M\frac{n}{n-1},
\]
which is always ensured since by assumption $\tnn \geq 2M$.

We can diagonalize \eqref{eq:la5:equationv2zero} and \eqref{eq:la5:equationzero}. Subtract the $(\ell-1)$st line from the $\ell$-th line we find
\begin{equation}\label{eq:sz5:lines}
0=\begin{cases}
 \sum_{\tnk=1}^{{i_1}-1} C_\tnk +\overline{\kappa}_{i_1} C_{{i_1}}  \\
 (1-\overline{\kappa}_{i_{\ell-1}})C_{i_{\ell-1}}+ \sum_{\tnk=i_{\ell-1}+1}^{{i_\ell}-1} C_\tnk +\overline{\kappa}_{i_\ell} C_{{i_\ell}}
 \quad &\ell =2,\ldots,k\\
 (1-\overline{\kappa}_{i_{k}})C_{i_{k}}+ \sum_{\tnk=i_{k}+1}^{\tnn} C_\tnk
 \end{cases}
\end{equation}
The $C_{i_\ell}$ and $C_{\tnn}$ are uniquely determined in the $\ell$-th equation.
So we can rewrite this further as (modification if $i_k=\tnn$)
\begin{equation}\label{eq:sz5:linesv2}
\begin{cases}
C_{{i_1}} =- \frac{1}{\overline{\kappa}_{i_1}}\sum_{\tnk=1}^{{i_1}-1} C_\tnk  \\
  C_{{i_\ell}}+\frac{1}{\overline{\kappa}_{i_\ell}} \brac{(1-\overline{\kappa}_{i_{\ell-1}})C_{i_{\ell-1}}} =-\frac{1}{\overline{\kappa}_{i_\ell}}  \sum_{\tnk=i_{\ell-1}+1}^{{i_\ell}-1} C_\tnk
 \quad &\ell =2,\ldots,k\\
 C_{\tnn}+\brac{(1-\overline{\kappa}_{i_{k}})C_{i_{k}}}=-\brac{\sum_{\tnk=i_{k}+1}^{\tnn-1}C_\tnk }
 \end{cases}
\end{equation}
Let $1 \leq j_1 < j_2 < \ldots < j_{\tnn-k-1} \leq \tnn$ be the coefficients complementary to $(i_1,\ldots,i_k,\tnn)$. We can write \eqref{eq:sz5:linesv2} as
\begin{equation}\label{eq:sz5:linesv2matrix}
 A \left (\begin{array}{c}C_{i_1}\\
           \vdots\\
           C_{i_k}\\
           C_{\tnn}
          \end{array}
 \right ) = B \left (\begin{array}{c}C_{j_1}\\
           \vdots\\
           C_{i_k}\\
           C_{j_{\tnn-k-1}}
          \end{array}
 \right ),
\end{equation}
where the coefficient matrix on the left-hand side acting is
\[
 A = \left ( \begin{array}{ccccccc}
              1 & 0 & 0 & \ldots & 0 & 0 & 0 \\
\frac{1}{\overline{\kappa}_{i_2}} (1-\overline{\kappa}_{i_{1}}) &                1&0&\ldots&0&0&0\\
0&\frac{1}{\overline{\kappa}_{i_3}} (1-\overline{\kappa}_{i_{2}})&1 &\ldots&0&0&0\\
\vdots & \vdots & \ddots& & \vdots& \vdots&\vdots\\
                0&0&0&\ldots & \frac{1}{\overline{\kappa}_{i_k}} (1-\overline{\kappa}_{i_{k-1}})  & 1&0\\
                0&0&0&\ldots&0&(1-\overline{\kappa}_{i_k}) &1
             \end{array}
\right ) \in \R^{(k+1) \times (k+1)}
\]
and $B \in \R^{(k+1)\times (\tnn-k-1)}$. Then \eqref{eq:sz5:linesv2matrix} is equivalent to
\[
 \left (\begin{array}{c}C_{i_1}\\
           \vdots\\
           C_{i_k}\\
           C_{\tnn}
          \end{array}
 \right ) = A^{-1}B \left (\begin{array}{c}C_{j_1}\\
           \vdots\\
           C_{i_k}\\
           C_{j_{\tnn-k-1}}
          \end{array}
 \right )
\]
and consequently,
\[
\begin{split}
&\dim \{(C_1,\ldots,C_\tnn): C_\tni \in \mathfrak{C}^\tni, \quad \text{\eqref{eq:la5:equationv2zero}, \eqref{eq:la5:equationzero} hold}\}\\
=&\dim \{(C_{j_1},\ldots,C_{j_{\tnn-k-1}}): C_\tni \in \mathfrak{C}^\tni, \quad \text{\eqref{eq:sz5:linesvaskdj} holds}\},\\
\end{split}
\]
where
\begin{equation}\label{eq:sz5:linesvaskdj}
\begin{cases}
0 = \sum_{\tnk=1}^{{i_1}-1} C_\tnk  \\
  0=\sum_{\tnk=i_{\ell-1}+1}^{{i_\ell}-1} C_\tnk
 \quad &\ell =2,\ldots,k\\
 0=\sum_{\tnk=i_{k}+1}^{\tnn-1}C_\tnk
 \end{cases}
\end{equation}
Each equation is now completely independent of the other, and since the dimension of $\mathfrak{C}^{\tni}$ is $(n(M+1)-1)$, \Cref{la:mathfrakCi}, a rough estimate is
\[
\dim \{(C_1,\ldots,C_\tnn): C_\tni \in \mathfrak{C}^\tni, \quad \text{\eqref{eq:sz5:linesvaskdj} holds}\} \leq (\tnn-(k+1))  \brac{n(M+1)-1}
\]
Compare this with our desired equation \eqref{eq:sz04:l5:24v2},
\[
\begin{split}
 &(\tnn-(k+1))  \brac{n(M+1)-1} \leq {\tnn}\, (Mn+n-1)-2Mn-Mnk\\
  \Leftrightarrow& k \geq M\frac{n}{n-1}-1.
\end{split}
 \]
That is our rough estimate establishes \underline{\eqref{eq:sz04:l5:24v2}, but only for any $k \geq M\frac{n}{n-1}-1$}.

To get the result for all $k$ we need to be more careful.

Define $A_\ell$ to be the number of independent terms in the $\ell$-th line of \eqref{eq:sz5:linesvaskdj}, i.e.
\[
 A_\ell = \begin{cases}
           i_1-1 \quad &\ell=1\\
           i_\ell-i_{\ell-1}-1 \quad& \ell = 2,\ldots,k\\
           \tnn-i_k-1
          \end{cases}
\]
observe that $\sum_{\ell=1}^{k+1} A_\ell = \tnn - (k+1)$.
Set $B_\ell$ the number of many consecutive pairs $C_\tni,C_{\tni+1}$ in the $\ell$-line of \eqref{eq:sz5:linesvaskdj}, i.e., set
\[
 B_\ell = \left \lfloor \frac{A_\ell}{2} \right \rfloor.
\]

From \Cref{la:dimofsumguy} we have for some $\lambda \in (1,2)$ the space
\[
\{(C_{\tni},C_{\tni+1}) \in \mathfrak{C}^\tni \times \mathfrak{C}^{\tni+1}\colon C_{\tni} + C_{\tni+1} = Y\}
\]
has dimension at most $\lambda (n(M+1)-1)$. Observe that we can assume $\lambda \approx 2$.

Then the $\ell$-th line has a solution space of dimension
\[
\begin{split}
 &\dim \{(C_{i_{\ell-1}+1},\ldots,C_{i_{\ell}-1})\in \mathfrak{C}_{i_{\ell-1}+1} \times \ldots \times \mathfrak{C}_{i_{\ell}-1}: \quad \text{$\ell$-th line of \eqref{eq:sz5:linesvaskdj} holds}\}\\
 \leq &B_\ell \lambda (n(M+1)-1) + (A_\ell - 2B_\ell) (n(M+1)-1).
\end{split}
 \]
Observe that $(A_\ell - 2B_\ell) \in \{0,1\}$. Let $L\leq k+1$ be the number which counts how many times $A_\ell-2B_\ell \neq 0$,
\[
 L = \#\{\ell \in \{1,\ldots, k+1\}: \quad A_\ell-2B_\ell =1\} \leq k+1.
\]
Clearly \[L+2\sum_{\ell=1}^{k+1}B_\ell = \tnn-k-1.\]
We also observe that $\tnn - (k+1)-L$ must necessarily be even. In total we get a better estimate for the dimension of the solution space (better since $\lambda < 2$)
\[
\begin{split}
&\dim \{(C_1,\ldots,C_\tnn): C_\tni \in \mathfrak{C}_\tni, \quad \text{\eqref{eq:sz5:linesvaskdj} holds}\}\\
\leq  &\sum_{\ell=1}^{k+1} B_\ell \lambda (n(M+1)-1) + (A_\ell - 2B_\ell) (n(M+1)-1)\\
=&(\tnn-k-1-L) \frac{\lambda}{2} (n(M+1)-1) + L (n(M+1)-1)\\
=&(\tnn-k-1) \frac{\lambda}{2} (n(M+1)-1) + L (\frac{2-\lambda}{2})(n(M+1)-1)\\
\end{split}
\]
Since $\lambda < 2$ the worst case is when $L$ is maximal, where the restrictions on $L$ are
\[
 \begin{cases}
 L \leq k+1\\
 L+k+1 \leq \tnn\\
 \tnn - (k+1)-L \quad \text{is even}.
 \end{cases}
\]
So we have the estimate
\[
\begin{split}
&\dim \{(C_1,\ldots,C_\tnn): C_\tni \in \mathfrak{C}_\tni, \quad \text{\eqref{eq:sz5:linesvaskdj} holds}\}\\
\leq&(\tnn-k-1) \frac{\lambda}{2} (n(M+1)-1) + (k+1) (\frac{2-\lambda}{2})(n(M+1)-1)\\
\end{split}
\]
We compare this to condition \eqref{eq:sz04:l5:24v2},
\[
\begin{split}
& (\tnn-k-1) \frac{\lambda}{2} (n(M+1)-1) + (k+1) (\frac{2-\lambda}{2})(n(M+1)-1) \leq \tnn (Mn+n-1)-2Mn-Mnk\\
\Leftrightarrow &
  (k+1) \brac{ (\frac{2-2\lambda}{2})(n(M+1)-1) +Mn} +Mn\leq \tnn \frac{2-\lambda}{2} (n(M+1)-1)\\
\Leftrightarrow &
  (k+1) \brac{ (2-\lambda)nM - (\lambda-1)(n-1)} +Mn\leq \tnn \frac{2-\lambda}{2} (n(M+1)-1)\\
 \end{split}
\]
We can ensure that the first term on the left-hand side is negative. Namely we may assume w.l.o.g. $\lambda<2$ sufficiently close to $2$, namely
\[
(2-\lambda)nM - (\lambda-1)(n-1) \leq 0
\Leftrightarrow \frac{2nM +n-1}{Mn + n-1}\leq \lambda,
\]
Then we have
\[
\begin{split}
& (\tnn-k-1) \frac{\lambda}{2} (n(M+1)-1) + (k+1) (\frac{2-\lambda}{2})(n(M+1)-1) \leq \tnn (Mn+n-1)-2Mn-Mnk\\
\Leftarrow &  Mn\leq \tnn \frac{2-\lambda}{2} (n(M+1)-1)\\
 \end{split}
\]
Thus for $\tnn$ suitably large (depending on $\lambda < 2$, $n$, $M$) the last inequality is true.
That is, \eqref{eq:sz04:l5:24v2} is established in all remaining cases and we can conclude.

\end{proof}

\section{The condition (C) and its genericity}\label{s:conditionc}
Assume that $\tnn \geq \tnn_0$ and let $\brac{\overline{P},(\overline{C}_\tni)_{\tni=1}^{\tnn},(\overline{\kappa}_{\tni})_{\tni =1}^{\tnn}}$ be a non-degenerate $T_{\tnn}$-configuration with endpoints $\overline{Z} = (\overline{Z}_1,\ldots,\overline{Z}_\tnn) \in
\left (\begin{array}{c}\R^{M \times n}\\
\R^{M \times n}
\end{array}\right )^{\tnn}
$.

Take from from \Cref{pr:dimensionM} $\mathcal{M}_{\tnn}$ the manifold of $T_{\tnn}$-endpoints around $\overline{Z}$ and $\pi_{\tnk}: \mathcal{M}_{\tnn} \to \left (\begin{array}{c}\R^{M \times n}\\
\R^{M \times n}
\end{array}\right )$.
For some given $F: \R^{M \times n} \to \R$ take $\mathcal{K}_F$ from \eqref{eq:mathcalKf}. Here is our version of condition (C) from \cite{MS}.

\begin{definition}[Condition (C)]\label{def:condC}

We say that that a given $T_{\tnn}$- configuration satisfies the condition (C) if for some small $\delta > 0$
\begin{itemize}
 \item $\mathcal{M}_{\tnn} $ and $\mathcal{K}_{F}$ meet transversally at $\overline{Z}$, that is
 \begin{equation}\label{eq:conditionCtransversal}
  T_{\overline{Z}} \mathcal{M}_{\tnn}  + T_{\overline{Z}} \mathcal{K}_{F} = \left (\begin{array}{c}
  \R^{M \times n}\\
  \R^{M\times n}\\
 \end{array}\right )^{\tnn}
 \end{equation}
 \item If for any $\tnk$ we restrict the map $\pi_{\tnk}$ to the intersection of $\mathcal{M}_{\tnn}$ and $\mathcal{K}_{F}$ in a neighborhood of $\overline{Z}$, i.e. if for some small enough $\delta >0$ we consider
 \[
 \pi_{\tnk} \Big |_{\mathcal{M}_{\tnn} \cap \mathcal{K}_{F} \cap B_\delta(\overline{Z}_1,\ldots,\overline{Z}_\tnn)} ,
\]
then
\begin{equation}\label{eq:conditionCpisurjective}
 \rank D\pi_{\tnk} \Big |_{\mathcal{M}_{\tnn} \cap \mathcal{K}_{F} \cap B_\delta(\overline{Z}_1,\ldots,\overline{Z}_\tnn)} (Z) = 2Mn
\end{equation}
for all $Z \in \mathcal{M}_{\tnn} \cap \mathcal{K}_{F} \cap B_\delta(\overline{Z}_1,\ldots,\overline{Z}_\tnn)$, i.e. it is a local submersion.
\end{itemize}
\end{definition}

\begin{remark}
By \Cref{pr:dimensionM} the dimension of $\mathcal{M}_{\tnn}$ is $\tnn n (M+1)$. By \Cref{la:tangentspaceKF} the dimension of $\mathcal{K}_{F}=(K_F)^{\tnn}$ is $Mn\tnn$. If these two manifolds meet transversely, the dimension of the intersection
\[
 \dim \brac{\mathcal{M}_{\tnn} \cap \mathcal{K}_{F} \cap B_\delta(\overline{Z}_1,\ldots,\overline{Z}_\tnn) }= \tnn n (M+1) + Mn\tnn - 2Mn\tnn = \tnn n
\]
The rank of $D\pi_{\tnk} \Big |_{\mathcal{M}_{\tnn} \cap \mathcal{K}_{F} \cap B_\delta(\overline{Z}_1,\ldots,\overline{Z}_\tnn)}$ can thus not be larger than $\tnn n$, so a necessary condition to even dream of having condition $(C)$ is
\[
 \tnn n \geq 2Mn,
\]
that is
\[
 \tnn \geq 2M.
\]
By the assumption of non-degeneracy of the $T_{\tnn}$ we also need $\tnn \geq n$.

So it makes only sense to work with $T_{\tnn}$-configurations where
\[
 \tnn \geq \max\{2M,n\}.
\]
Observe that in the case $M=n=2$ from \cite{MS,Sz04} this is exactly the assumption that we look for (at least) $T_4$-configurations -- and as \cite{Sz04} showed, it is easier (and in some sense: only possible) to find a $T_\tnn$-configuration for $\tnn=5$. Our argument below will work with $\tnn$ even larger (although for specific $n$, $M$, it might be simple to compute an ``optimal'' $\tnn$).
\end{remark}

A crucial result is that condition $(C)$ is generic. Recall that we can always identify $F: \R^{M \times n} \to \R$ with $F: \R^M \otimes \Ep^1 \R^n \to \R$.
\begin{theorem}\label{th:conditionCgen}
Fix $M,n \geq 2$. There exists $\tnn_0 = \tnn_0(M,n)$ such that for any $\tnn \geq \tnn_0$ the following holds.

Assume $F_0 \in C^2(\R^{M \times n},\R)$ satisfies the strong Legendre-Hadamard condition, cf. \Cref{def:stronglegendrehad}, and that
$\mathcal{K}_{F_0}$ contains a non-degenerate $T_\tnn$ configuration $(\overline{Z}_1,\ldots,\overline{Z}_\tnn)$ with its manifold $\mathcal{M}_{\tnn}$.

Then for any $\delta > 0$ there exists $F \in C^2(\R^{M \times n},\R)$ with
\begin{equation}\label{eq:FmF0C2lessdelta}
 \|F-F_0\|_{C^2} < \delta
\end{equation}
such that $F$ still satisfies the strong Legendre-Hadamard condition, $(\overline{Z}_1,\ldots,\overline{Z}_\tnn)$ still belongs to $\mathcal{K}_F$, but the condition $(C)$ is satisfied for $F$.

If $F_0$ is \emph{strongly} polyconvex then we can ensure that $F$ is \emph{strongly} polyconvex.
\end{theorem}

First we record the following basic observations
\begin{lemma}\label{la:Htnkf0pdeltaprop}
Let $H_\tnk \in C_c^2(B(0,R),\R)$  and fix some points $(X_{\tnk})_{\tnk=1}^{\tnn} \subset \R^{M \times n}$.
Set
\[
 F(X) \coloneqq F_0(X) + \delta_1 \sum_{\tnk=1}^{\tnn} H_\tnk(X-X_\tnk),
\]
for some $\delta_1 > 0$.

For any $\delta > 0$ there exists $\delta_1 > 0$ such that
\begin{itemize}
\item \eqref{eq:FmF0C2lessdelta} holds,
 \item if $F_0$ satisfies the strong Legendre-Hadamard condition, so does $F$,
 \item if $F_0$ is strongly polyconvex, so is $F$,
 \item Recall the definition of $K_F$ from \eqref{eq:Kfdef}. There exist $R_0>0$ depending on $(X_\tnk)_{\tnk=1}^{\tnn}$ such that if $R < R_0$ for any $\tni \in \{1,\ldots,\tnn\}$ we have the following: If $DH_\tni(0) = 0$ and if $(X_\tni,Y_{\tni}) \in K_{F_0}$ for some $Y_{\tni} \in \R^M \otimes \Ep^{n-1} \R^n$ then $(X_\tni,Y_{\tni}) \in K_{F_0}$.
 \end{itemize}
\end{lemma}
For the last bullet point, let us stress however that in general $K_{F_0} \neq K_{F}$
\begin{proof}
\eqref{eq:FmF0C2lessdelta} is clearly satisfied if we choose $\delta$ suitably small (depending on $\H_{\tnk}$).

Since $F_0$ satisfies the the strong Legendre-Hadamard condition \eqref{eq:stronglegendre}, taking $\delta > 0$ suitably small (depending on $H_{\tnk}$) we can ensure that
\[
\begin{split}
 \partial_{X_{i\alpha}}  \partial_{X_{j\beta}} F(X) \xi_i \xi_j\, \eta_\alpha \eta_\beta
 =&\partial_{X_{i\alpha}}  \partial_{X_{j\beta}} F_0(X) \xi_i \xi_j\, \eta_\alpha \eta_\beta
 + \delta \sum_{\tnk=1}^{\tnn} \partial_{X_{i\alpha}}  \partial_{X_{j\beta}}H_\tnk(X-X_\tnk) \xi_i \xi_j\, \eta_\alpha \eta_\beta \\
 \geq& \brac{\lambda -\delta \sum_{\tnk=1}^{\tnn} \|D^2 H_{\tnk}\|} |\xi|_{\R^M}^2 |\eta|_{\R^n}^2\\
 \geq& \frac{\lambda}{2} |\xi|_{\R^M}^2 |\eta|_{\R^n}^2.
 \end{split}
%
\]
That is $F$ satisfies the strong Legendre-Hadamard condition.

If $F_0$ is polyconvex, i.e. for some $\gamma > 0$ we know that $F_0(X)-\gamma|X|^2$ is polyconvex, then we observe that
\[
 F(X)-\frac{\gamma}{2} |X|^2 = \underbrace{F_0(X)-\gamma |X|^2}_{\text{polyconvex}} + \frac{\gamma}{2} |X|^2 +\delta \sum_{\tnk=1}^{\tnn} H_\tnk(X-X_\tnk),
\]
which is polyconvex, since for suitably small $\delta$
\[
 X \mapsto \frac{\gamma}{2} |X|^2 +\delta \sum_{\tnk=1}^{\tnn} H_\tnk(X-X_\tnk)
\]
is convex.

If $DH_\tni(0) = 0$, then
\[
 D F(X_{\tni}) = DF_0(X_{\tni}) + \delta_1 \sum_{\tnk\neq \tni} DH_\tnk(X_{\tni}-X_\tnk),
\]
If $R < \inf_{\tnell \neq \tnk}|X_{\tnell}-X_\tnk|$ then by the support of of $H$
\[
 DF(X_{\tni}) = DF_0(X_{\tni}).
\]
Thus the claim follows from the definition \eqref{eq:Kfdef}.
\end{proof}

\begin{lemma}\label{la:twoconditionssame}
Let $F: \R^{M \times n} \to \R$ be strongly polyconvex, $\overline{Z} = (\overline{Z}_1,\ldots,\overline{Z}_\tnn)$ be a nondegenerate $T_{\tnn}$-configuration with associated manifold $\mathcal{M}_{\tnn}$ and map map $\pi_{\tnk}: \mathcal{M}_{\tnn} \to \left (\begin{array}{c}
  \R^{M \times n}\\
  \R^{M\times n}\\
 \end{array}\right )$ from \Cref{pr:dimensionM}. Assume moreover that $\overline{Z} \in \mathcal{K}_F$.

 If for some $\tnk \in \{1,\ldots,\tnn\}$ we have
\begin{equation}\label{eq:laslzlo5:goalv2}
 T_{\overline{Z}} \mathcal{K}_F + \ker D\pi_{\tnk}(\overline{Z})\Big|_{T_{\overline{Z}}  \mathcal{M}} = \left (\begin{array}{c}
  \R^{M \times n}\\
  \R^{M\times n}\\
 \end{array}\right )^{\tnn}
\end{equation}
then necessarily
\[
 \tnn \geq 2M.
\]
More importantly, if \eqref{eq:laslzlo5:goalv2} holds for all $\tnk$ then condition (C) holds at $\overline{Z}$.
\end{lemma}

\begin{proof}
We have by the rank nullity theorem
\[
\dim\brac{ \ker D\pi_{\tnk}(\overline{Z})\Big|_{T_{\overline{Z}}  \mathcal{M}}}=\dim\brac{ T_{\overline{Z}}  \mathcal{M}}-\dim\brac{ {\rm im} D\pi_{\tnk}(\overline{Z})\Big|_{T_{\overline{Z}}  \mathcal{M}}}
\]
From \Cref{pr:dimensionM} we have
\[
 \dim {\rm im} D\pi_{\tnk}(\overline{Z})\Big|_{T_{\overline{Z}}  \mathcal{M}} = 2Mn
\]
and
\[
 \dim (T_{\overline{Z}}  \mathcal{M}) = {\tnn}n(M+1),
\]
and thus (recall that $\tnn \geq n \geq 2$)
\[
\dim\brac{ \ker D\pi_{\tnk}(\overline{Z})\Big|_{T_{\overline{Z}}  \mathcal{M}}}={\tnn}n(M+1)-2Mn \geq 0
\]
On the other hand by \Cref{la:dimK}
\[
 \dim T_{\overline{Z}}  \mathcal{K}_F = {\tnn}Mn
\]
Also
\[
 \dim
 \left (\begin{array}{c}
  \R^{M \times n}\\
  \R^{M \times n}\\
 \end{array}\right )^{\tnn} = 2M n{\tnn}.
\]

We recall that if $X$ and $Y$ are linear subspaces then $\dim (X \cap Y) = \dim X + \dim Y - {\rm span} \{X,Y\}$. Thus, \eqref{eq:laslzlo5:goalv2} implies
\begin{equation}\label{eq:DpiKFcapTz}
 T_{\overline{Z}} \mathcal{K}_F \cap \ker D\pi_{\tnk}(\overline{Z})\Big|_{T_{\overline{Z}}  \mathcal{M}} = {\tnn}Mn+{\tnn}n(M+1)-2Mn-2Mn\tnn =(\tnn -2M)n
\end{equation}
In particular the right-hand side number must be nonnegative, so $\tnn \geq 2M$.

From the rank-nullity equation again we have
\[
\begin{split}
 &\dim {\rm im} D\pi_{\tnk}(\overline{Z}) \Big|_{T_{\overline{Z}} \mathcal{M}_\tnn\cap T_{\overline{Z}}  \mathcal{K}_F }\\
  =& \dim\{T_{\overline{Z}} \mathcal{M}_\tnn\cap T_{\overline{Z}}  \mathcal{K}_F \} - \dim \ker D\pi_{\tnk}(\overline{Z}) \Big|_{T_{\overline{Z}} \mathcal{M}_\tnn\cap T_{\overline{Z}}  \mathcal{K}_F }\\
 =& \dim T_{\overline{Z}} \mathcal{M}_\tnn + \dim T_{\overline{Z}}  \mathcal{K}_F - \dim \brac{{\rm span} \{T_{\overline{Z}} \mathcal{M}_\tnn,T_{\overline{Z}}  \mathcal{K}_F\}} - (\tnn -2M)n \\
 \overset{\eqref{eq:laslzlo5:goalv2}}{=}&{\tnn}n(M+1)+{\tnn}Mn-2Mn\tnn- (\tnn -2M)n \\
 =&2Mn \\
 \end{split}
\]
That is we have
\begin{equation}\label{eq:laslzlo5:goal2}
 \dim \brac{{\rm im} D \pi_{\tnk}(\overline{Z}) \Big |_{T_Z \mathcal{K}_F \cap T_Z \mathcal{M}_\tnn} }= 2Mn \quad \forall \tnk \in \{1,\ldots,\tnn\},
\end{equation}
\eqref{eq:laslzlo5:goal2} readily implies the openness condition of Condition (C), \eqref{eq:conditionCpisurjective}. It remains to establish the transversality condition, \eqref{eq:conditionCtransversal}
\[
 T_{\overline{Z}} \mathcal{K}_F + T_{\overline{Z}} \mathcal{M}_{\tnn} = \left (\begin{array}{c}
  \R^{M \times n}\\
  \R^{M\times n}\\
 \end{array}\right )^{\tnn}.
\]

Observe that by \Cref{la:tangentspaceKF} and \Cref{pr:dimensionM}
\[
 \dim(T_{\overline{Z}} \mathcal{K}_F) = \tnn M n, \quad \dim(T_{\overline{Z}} \mathcal{M}_\tnn) = \tnn n (M+1)
\]
Since
\[
 \dim(\span \{T_{\overline{Z}} \mathcal{K}_F,T_{\overline{Z}} \mathcal{M}_\tnn) = \dim(T_{\overline{Z}} \mathcal{K}_F)+\dim(T_{\overline{Z}} \mathcal{M}_\tnn)-
 \dim(T_{\overline{Z}} \mathcal{K}_F \cap T_{\overline{Z}} \mathcal{M}_\tnn)
\]
So we see that we always have
\begin{equation}\label{eq:tranversality:have}
 \dim(T_{\overline{Z}} \mathcal{K}_F \cap T_{\overline{Z}} \mathcal{M}_\tnn)\geq \tnn n,
\end{equation}
and transversality is equivalent to
\begin{equation}\label{eq:tranversality:want}
 \dim(T_{\overline{Z}} \mathcal{K}_F \cap T_{\overline{Z}} \mathcal{M}_\tnn)=\tnn n
\end{equation}

By the rank-nullity theorem
\[
 \dim \brac{{\rm ker} D \pi_\tnk(\overline{Z}) \Big |_{T_{\overline{Z}} \mathcal{K}_F \cap T_{\overline{Z}} \mathcal{M}_\tnn} }+\dim \brac{{\rm im} D \pi_\tnk(\overline{Z}) \Big |_{T_{\overline{Z}} \mathcal{K}_F \cap T_{\overline{Z}} \mathcal{M}_\tnn} } = \dim (T_{\overline{Z}} \mathcal{K}_F \cap T_{\overline{Z}} \mathcal{M}_\tnn)
\]
Then \eqref{eq:DpiKFcapTz} and \eqref{eq:laslzlo5:goal2}  imply \eqref{eq:tranversality:want}, i.e. transversality \eqref{eq:conditionCtransversal} is established.
\end{proof}

We also have the following, see \cite[Lemma 4]{Sz04} where this is proven for $N=4$. For general $N$ the proof is completely analogous.

\begin{lemma}\label{la:sz:la4}
Let $N \in \N$.
Suppose $L \subset (\R^{2N})^{\tnn}$ is a subspace with $\dim L \leq N\tnn$

Denote for $1=i_1 < i_2 < \ldots < i_k \leq \tnn$ the map
\[
 p_{i_1,\ldots,i_k}: (\R^{2N})^{\tnn} \to (\R^{2N})^{\tnn}
\]
which sets the $i_{j}$-components to be zero, i.e.
\[
 p_{i_1,\ldots,i_k} (v_1,\ldots,v_{\tnn}) = (v_1,\ldots,v_{i_1-1},0,v_{i_1+1}, \ldots, v_{i_2-1},0,v_{i_2+1},\ldots)
\]

Assume that
\begin{equation}\label{eq:szla4:piLdim}
 \dim \brac{L \cap {\rm im} (p_{i_1,\ldots,i_k}
 )} \leq N(\tnn -k)
\end{equation}
Then for any $(\overline{A}_{\tni})_{\tni=1}^{\tnn} \subset \R^{N \times N}_{sym}$ and any $\delta > 0$ there exists $(A_{\tni})_{\tni=1}^{\tnn} \subset \R^{N \times N}_{sym}$ with $|A_{\tni}-\overline{A}_{\tni}| < \delta$ such that if we set
\[
 V_\tni \coloneqq  \left (\begin{array}{c}
         I_{N \times N}\\
         A_{\tni}
        \end{array}\right ) \R^N \subset \R^{2N}
\]
we have
\[
 V_{1} \times \ldots \times V_{\tnn} \cap L = \{0\}.
\]
\end{lemma}

\begin{proof}[Proof of \Cref{th:conditionCgen}]

For $\tnk \in \{1,\ldots,\tnn\}$ pick
\[
 A_{\tnk} : \R^{M \times n} \to \R^{M \times n}
\]
a symmetric linear map to be determined later. We set
 \[
 F(X) \coloneqq F_0(X) + \delta \sum_{\tnk=1}^{\tnn} H_\tnk(X-X_\tnk),
\]
for suitably small $\delta$, $H_{\tnk}$ with compact support around $0$ and
\[
 DH_\tnk(0) = 0,
\]
and
\[
 D^2H_{\tnk}(0)=:A_{\tnk}
\]
In view of \Cref{la:Htnkf0pdeltaprop} we $F$ satisfies all conditions besides condition (C).

Denote by $\mathcal{M}_{\tnn}$ the manifold of $T_{\tnn}$-configurations around $\overline{Z}$, cf. \Cref{pr:dimensionM}, with associated map $\pi_{\tnk}: \mathcal{M}_{\tnn} \to \left (\begin{array}{c}
  \R^{M \times n}\\
  \R^{M\times n}\\
 \end{array}\right )$.

In view of  \Cref{la:tangentspaceKF} we can write for $\overline{Z}=(\overline{Z}_1,\ldots,\overline{Z}_{\tnn})$
\[
 T_{\overline{Z}} \mathcal{K}_F = V_1 \times \ldots \times V_\tnn \subset \left ( \begin{array}{c}\R^{M \times n}\\\R^{M \times n}\end{array} \right )^{\tnn}
\]
where
\[
 V_\tni = \left \{ \left ( \begin{array}{c} \tilde{X}\\ A_{\tni} [\tilde{X}] \end{array}\right ): \quad \tilde{X} \in \R^{M \times n}\right \}
\]

Our goal will be to choose the symmetric linear map $A_{\tni}: \R^{M \times n} \to \R^{M \times n}$, $\tni \in \{1,\ldots,\tnn\}$ such that for all $\tnk \in \{1,\ldots,\tnn\}$
\begin{equation}\label{eq:goallaszlolemma4}
 V_1 \times \ldots \times V_{\tnn}+ \ker D\pi_{\tnk}(\overline{Z})\Big|_{T_{\overline{Z}}  \mathcal{M}} = \left (\begin{array}{c}
  \R^{M \times n}\\
  \R^{M\times n}\\
 \end{array}\right )^{\tnn}
\end{equation}
holds, since then, by \Cref{la:twoconditionssame}, condition (C) is satisfied.

Now assume we have chosen $(A_{\tni})_{\tni = 1}^{\tnn}$ such that \eqref{eq:goallaszlolemma4} holds for some $\tnk \in \{1,\ldots,\tnn\}$. Then for all $(\tilde{A}_{\tni})_{\tni = 1}^{\tnn}$ with $|A_{\tni}-\tilde{A}_{\tni}| \ll 1$ condition \eqref{eq:goallaszlolemma4} still holds for the same $\tnk$.

We are now going to show that given any tuple $(\overline{A}_1,\ldots,\overline{A}_\tnn)$ with corresponding spaces $\overline{V}_1,\ldots,\overline{V}_{\tnn}$, any $\delta > 0$ and any $\tnk \in \{1,\ldots,\tnn\}$ we can find $(A_1,\ldots,A_\tnn)$ and corresponding spaces $(V_1,\ldots,V_\tnn)$ such that $|A_{\tni}-\overline{A}_{\tni}| < \delta$ and \eqref{eq:goallaszlolemma4} holds for $\tnk$. If we are able to do this, we can conclude by induction that we can find $V_1,\ldots,V_{\tnn}$ so that \eqref{eq:goallaszlolemma4} holds.

It suffices to do this for $\tnk = 1$.
Fix $(\overline{A}_1,\ldots,\overline{A}_\tnn)$ and $\delta > 0$

We need to show that there exists $L \subset \ker D\pi_1(\overline{Z}) \subset \left (\begin{array}{c}
  \R^{M \times n}\\
  \R^{M\times n}\\
 \end{array}\right )^{\tnn} $ such that
\[
 \dim L \geq \tnn Mn
\]
and there are $(A_1,\ldots,A_\tnn)$ and corresponding subspaces $(V_1,\ldots,V_{\tnn})$ with $|A_\tni - \overline{A}_{\tni}| < \delta$ such that
\[
 L \cap (V_1 \times \ldots \times V_\tnn) = \{0\}
\]
If we have this then
\[
\dim (L+(V_1\times \ldots V_{\tnn}) = \dim(L) + \dim(V_1\times \ldots \times V_{\tnn}) -  \dim(L \cap (V_1 \times \ldots \times V_\tnn))
\]
implies (since $\dim V_\tni = Mn$)
\[
 \dim (L+(V_1\times \ldots V_{\tnn})  = 2\tnn Mn
\]
and thus \eqref{eq:goallaszlolemma4} holds for $\kappa =1$.

We now apply \Cref{la:sz:la5}, which implies that the conditions of \Cref{la:sz:la4} are met for $N=Mn$. The latter lemma then implies the existence of $L$ as desired. We can conclude.
\end{proof}

\section{On \texorpdfstring{$\mathscr{R}$}{R}-convex hulls, laminates, and laminates of finite order}\label{s:laminates}
Since we replaced the notion of ``rank-one'' in \Cref{ss:rank1replacement} (for the benefit of later being able to use \Cref{la:wiggle}), we need to replace the notion of rank-one convex functions, and re-establish the important properties that were used in \cite{MS}.

\begin{definition}\label{def:Rconvex}
A function $f:  \left (\begin{array}{c}
  \R^M \otimes \Ep^1 \R^n\\
  \R^M \otimes \Ep^{n-1} \R^n\\
 \end{array}\right ) \to \R$ is $\mathscr{R}$-convex if
 \[
  f(tA + (1-t)B) \leq tf(A) + (1-t) f(B) \quad \forall t \in[0,1].
 \]
holds whenever $A,B \in \left (\begin{array}{c}
  \R^M \otimes \Ep^1 \R^n\\
  \R^M \otimes \Ep^{n-1} \R^n\\
 \end{array}\right )$ with $A-B \in \mathscr{R}$.
\end{definition}

\begin{lemma}\label{la:Rconvexiscontinuous}
Any $\mathscr{R}$-convex function is locally Lipschitz.
\end{lemma}
\begin{proof}
Since $f$ is $\mathscr{R}$-convex,
\[
\R \ni t\mapsto f(A+tC)
\]
is locally Lipschitz for any $A \in \left (\begin{array}{c}
  \R^M \otimes \Ep^1 \R^n\\
  \R^M \otimes \Ep^{n-1} \R^n\\
 \end{array}\right )$ and $C \in \mathscr{R}$.

General local Lipschitz continuity is a consequence of \Cref{la:distinR}.
\end{proof}

\begin{definition}[$\mathscr{R}$-convex hull]
Let $K \subset \left (\begin{array}{c}
  \R^M \otimes \Ep^1 \R^n\\
  \R^M \otimes \Ep^{n-1} \R^n\\
 \end{array}\right )$ be compact.
 Then we define its $\mathscr{R}$-convex hull $K^r$ as follows:
 \[
  K^{\mathscr{R}c} \coloneqq \left \{X \in \left (\begin{array}{c}
  \R^M \otimes \Ep^1 \R^n\\
  \R^M \otimes \Ep^{n-1} \R^n\\
 \end{array}\right ): \quad f(X) \leq \sup_{K} f \quad \forall f: \left (\begin{array}{c}
  \R^M \otimes \Ep^1 \R^n\\
  \R^M \otimes \Ep^{n-1} \R^n\\
 \end{array}\right ) \to \R \quad \text{rank-one convex}\right \}.
 \]
For general sets $U \subset \left (\begin{array}{c}
  \R^M \otimes \Ep^1 \R^n\\
  \R^M \otimes \Ep^{n-1} \R^n\\
 \end{array}\right )$ the $\mathscr{R}$-convex hull is defined as
\[
 U^{\mathscr{R}c} \coloneqq \bigcup_{K \subset U \text{ compact}} K^{\mathscr{R}c}.
\]
\end{definition}

\begin{definition}[Laminates and laminates of finite order]\label{def:laminate}
In analogy to \cite{MS} we define laminates and laminates of finite order, simply replacing rank-one convexity by $\mathscr{R}$-convexity:

\begin{itemize}
\item Denote by $\mathcal{P}$ the set of all compactly supported probability measures in $\left (\begin{array}{c}
  \R^M \otimes \Ep^1 \R^n\\
  \R^M \otimes \Ep^{n-1} \R^n\\
 \end{array}\right )$.
 \item For a compact set $K \subset \left (\begin{array}{c}
  \R^M \otimes \Ep^1 \R^n\\
  \R^M \otimes \Ep^{n-1} \R^n\\
 \end{array}\right )$, $\mathcal{P}(K)$ denotes the set of all measures in $\mathcal{P}$ that are supported in $K$.
 \item For a radon measure $\nu$ we denote by $\overline{\nu}$ its baricenter \[\overline{\nu} = \int X d\nu(X).\]
 \item A measure $\nu \in \mathcal{P}$ is a \emph{laminate} if
 \[
 f(\overline{\nu}) \leq \int f d\nu\quad \text{ for all $\mathscr{R}$-convex $f$}
 \]
\item For $K \subset \left (\begin{array}{c}
  \R^M \otimes \Ep^1 \R^n\\
  \R^M \otimes \Ep^{n-1} \R^n\\
 \end{array}\right )$ we denote by
\[
 \mathcal{P}^\ell(K) = \left \{\nu \in \mathcal{P}(K): \text{$\nu$ is laminate}\right \}
\]

\item  Let $\mathcal{O} \subset \left (\begin{array}{c}
  \R^M \otimes \Ep^1 \R^n\\
  \R^M \otimes \Ep^{n-1} \R^n\\
 \end{array}\right )$ be an open set.
 Assume that $\nu \in \mathcal{P}$ is of the form
 \[
  \nu=\sum_{j=1}^r \lambda_j \delta_{A_j}
 \]
  where $A_j \in \mathcal{O}$, and $A_j \neq A_k$ for $j\neq k$, $\sum_{j=1}^r \lambda_j = 1$.

  Another measure $\nu' \in \mathcal{P}$ can be obtained from $\nu$ by an \emph{elementary splitting in $\mathcal{O}$} if for some $i \in \{1,\ldots,r\}$ and some $\lambda \in [0,1]$ there exists $B_2,B_1 \in \mathcal{O}$, $[B_1,B_2] \in \mathcal{O}$, with $B_2-B_1 \in \mathscr{R}$, and
  \[
   A_i = (1-s) B_1 + s B_2
  \]
and
\[
\begin{split}
 \nu' =& \nu + \lambda \lambda_i \brac{(1-s) \delta_{B_1} + s \delta_{B_2} - \delta_{A_i}}\\
 =&\sum_{j=1, j \neq i}^r \lambda_j \delta_{A_j} + \lambda_i (1-\lambda) \delta_{A_i}+ \lambda \lambda_i (1-s) \delta_{B_1} + \lambda \lambda_i s \delta_{B_2}
 \end{split}
\]
\item A \emph{laminate of finite order} in $\mathcal{O}$ is a measure $\nu$ for which there exists a finite sequence of measures $\nu_1,\ldots,\nu_m=\nu$ such that $\nu_1 = \delta_A$, and $\nu_{i}$ can be obtained from $\nu_{i-1}$ by an elementary splitting in $\mathcal{O}$ as above. The set of laminates of finite order in $\mathcal{O}$ is denoted by $\mathcal{L}(\mathcal{O})$.
\end{itemize}
\end{definition}

A few simple observations:
\begin{lemma}\label{la:easystufflaminates}
\begin{enumerate}
 \item If $\nu'$ can be obtained from $\nu$ by an elementary splitting, then
\[
 \overline{\nu} = \overline{\nu'}.
\]
\item If $\nu$ is a laminate and $\nu'$ can be obtained from $\nu$ by an elementary splitting, then $\nu'$ is also a laminate.
\item Any laminate of finite order is also a laminate.
\item Assume that $\nu_1,\nu_2 \in \mathcal{L}(\mathcal{O})$ and $[\overline{\nu}_1,\overline{\nu}_2] \subset \mathcal{O}$ and $\overline{\nu}_1-\overline{\nu}_2 \in \mathscr{R}$.

Fix $\lambda \in (0,1)$ and set
\[
\nu_3 \coloneqq  \lambda \nu_1 + (1-\lambda) \nu_2
\]
Then $\nu_3 \in \mathcal{L}(\mathcal{O})$.
In particular, \cite[Lemma~2.2]{MS} still holds for our notion of $\mathscr{R}$-convex.
\item If $B_1,B_2 \in K^{\mathscr{R}c}$ and $B_2-B_1 \in \mathscr{R}$ then
\[
 [B_1,B_2] \subset K^{\mathscr{R}c}.
\]
\item for any set $U$ and any $X \in U$ we have $X \in U^{\mathscr{R}c}$.
\end{enumerate}
\end{lemma}
\begin{proof}
\begin{enumerate}
 \item obvious
 \item Let $f$ be $\mathscr{R}$-convex, and assume
 \[
  f(\overline{\nu}) \leq \int f d\nu=\sum_{j=1}^r \lambda_j f(A_j)
 \]
Assume the splitting happens at $A_i = (1-s)B_1 + sB_2$ with $B_2-B_1 \in \mathscr{R}$ then we have by $\mathscr{R}$-convexity
\[
\lambda_i\lambda f(A_i) \leq (1-s)\lambda_i\lambda f(B_1) + s\lambda_i\lambda f(B_2),
\]
so that (recall: $\overline{\nu'}=\overline{\nu}$) we have
\[
\begin{split}
  f(\overline{\nu'}) \leq& \sum_{j=1, j \neq i}^r \lambda_j f(A_j) + \lambda_i (1-\lambda) f(A_i) + (1-s)\lambda_i\lambda f(B_1) + s\lambda_i\lambda f(B_2)\\
  =&\int f d\nu'.
  \end{split}
  \]
  \item obvious since $\delta_{A}$ is a trivially laminate.
  \item Set
  \[
   P_1 \coloneqq  \overline{\nu_1}, \quad P_2 \coloneqq  \overline{\nu_2},
  \]
and
  \[
  P_3 \coloneqq  \overline{\nu_3} = \lambda P_1 + (1-\lambda) P_2.
  \]
  Set
  \[
   \mu_0 = \delta_{\overline{\nu_3}}
  \]
  Then we set
  \[
   \mu_1 = \lambda \delta_{P_1} + (1-\lambda) \delta_{P_2}.
  \]
Clearly $\mu_1$ can be obtained by an elementary splitting from $\mu_0$. Now we simply do the construction steps for $\delta_{P_1}$ as in $\nu_1$ and $\delta_{P_2}$ as in $\nu_2$.
\item Assume that $B_1,B_2 \in K^{\mathscr{R}c}$ and $B_1-B_2 \in \mathscr{R}$.

Take any $B = sB_1 + (1-s) B_2$ for some $s \in [0,1]$. We need to show that $B \in K^{\mathscr{R}c}$.

Let $f: \left (\begin{array}{c}
  \R^M \otimes \Ep^1 \R^n\\
  \R^M \otimes \Ep^{n-1} \R^n\\
 \end{array}\right ) \to \R$ be $\mathscr{R}$-convex. Since $B_1-B_2 \in \mathscr{R}$
 \[
  f(B) \leq sf(B_1) + (1-s) f(B_2)
 \]
Since $B_1$ and $B_2$ are in $K^{\mathscr{R}c}$
\[
 f(B) \leq s\sup_{K} f + (1-s) \sup_{K} f = \sup_{K} f
\]
This holds for any $\mathscr{R}$-convex $f$, so $B \in K^{\mathscr{R}c}$.
\item obvious
\end{enumerate}

\end{proof}

We can characterize the $\mathscr{R}$-convex hull of a compact set $K$ by baricenters of measures in $\mathcal{P}^\ell(K)$.
\begin{proposition}\label{th:jimmystheorem}
Let $K \subset  \left (\begin{array}{c}
  \R^M \otimes \Ep^1 \R^n\\
  \R^M \otimes \Ep^{n-1} \R^n\\
 \end{array}\right )$ be compact. Then
\[
 K^{\mathscr{R}c} = \left \{\overline{\nu}: \quad \nu \in \mathcal{P}^\ell(K) \right \}
\]
\end{proposition}
\begin{proof}
\underline{$\supseteq$:} Assume $\nu \in \mathcal{P}^\ell(K)$, then $f(\overline{\nu}) \leq \int f d\nu$ for all $\mathscr{R}$-convex $f$. In particular
\[
 f(\overline{\nu}) \leq \sup_{K} f
\]
and thus $\overline{\nu} \in K^r$.

\underline{$\subseteq$} Let $P \in K^r$. We need to find $\nu \in \mathcal{P}^\ell(K)$, with $\overline{\nu} = P$.

Let
\[
 d_K(X) \coloneqq \dist(X,K), \quad X \in \left (\begin{array}{c}
  \R^M \otimes \Ep^1 \R^n\\
  \R^M \otimes \Ep^{n-1} \R^n\\
 \end{array}\right ).
\]
Let $\mathcal{O} \supset K$ be an open, bounded set that contains $P$.

We define the $\mathscr{R}$-convex envelope for some open set $\mathcal{O} \subset \left (\begin{array}{c}
  \R^M \otimes \Ep^1 \R^n\\
  \R^M \otimes \Ep^{n-1} \R^n\\
 \end{array}\right )$ as
\begin{equation}\label{eq:Kr1}
 R_{\mathcal{O}} d_K(X) \coloneqq \sup \{g(X): \quad g: \mathcal{O} \to \R \quad \text{$\mathscr{R}$-convex in $\mathcal{O}$ and $g \leq d_K$ in $\mathcal{O}$}\}.
\end{equation}
$R_{\mathcal{O}}d_K(X)$ is $\mathscr{R}$-convex,
and by \cite[Lemma 2.3]{MS} (whose proof can be copied verbatim in our case) we can assume $R_{\mathcal{O}} d$ is defined in all of $\left (\begin{array}{c}
  \R^M \otimes \Ep^1 \R^n\\
  \R^M \otimes \Ep^{n-1} \R^n\\
 \end{array}\right )$ -- and globally $\mathscr{R}$-convex.

Also since $g(X) \equiv 0 \leq d_K(X)$ is $\mathscr{R}$-convex, we have
\[
 R_{\mathcal{O}}d_K(X) \geq 0 \quad \text{in $\mathscr{O}$}.
\]
Since $P \in K^r$ by assumption, we find
\[R_{\mathcal{O}} d_K(P) \leq 0.\]
Consequently,
\[
 R_{\mathcal{O}}d_K(X)=0.
\]

On the other hand, by \cite[Lemma 2.2]{MS} (which holds since we have \Cref{la:easystufflaminates})
\begin{equation}\label{eq:Kr2}
 R_{\mathcal{O}} d_K(X)  = \inf \left \{ \int d_K\ d\nu: \quad \nu \in \mathcal{L}(\mathcal{O}): \quad \overline{\nu}  = X \right \}.
\end{equation}
where $\mathcal{L}(\mathcal{O})$ denotes laminates of finite order, \Cref{def:laminate}.

Thus, from \eqref{eq:Kr2} for any $k \in \N$ we find  some $\nu_k \in \mathcal{L}(\mathcal{O})$ such that $\overline{\nu}_k = P$ and
\[
 0 = R_{\mathcal{O}} d_K(P) \geq \int d_K\, d\nu_k - \frac{1}{k}.
\]
Since $\nu_k$ are Radon measures, up to not relabelled subsequence, they converge weak* to some Radon measure $\nu$ and we have $\overline{\nu} = p$, and
\[
 0 = \int d_K\, d\nu.
\]
Since $d_K$ is zero exactly in $K$, we conclude that $\nu$ must be supported in $K$, that is $\nu \in \mathcal{P}(K)$. The last property we need to show is that $\nu \in \mathcal{P}^\ell(K)$, i.e. that $\nu$ is a laminate in the sense of \Cref{def:laminate}:

And indeed, each $\nu_k$ is a laminate, so we have for any $\mathscr{R}$-convex $f$ (which is necessarily continuous by \Cref{la:Rconvexiscontinuous})
\[
 f(P)= f(\overline{\nu_k}) \leq \int f\, d\nu_k.
\]
The right-hand side converges and we have
\[
 f(P) \leq \int f\, d\nu.
\]
This holds for any $\mathscr{R}$-convex $f$, so $\nu$ is a laminate in the definition of \Cref{def:laminate}, and thus we have found $\nu \in \mathcal{P}^\ell(K)$, with $\overline{\nu} = P$. We can conclude.
\end{proof}

%
%

Recall the notion of  $T_{\tnn}$-configurations in \Cref{def:Tnconfig}. We begin with the following observation which takes the role of \cite[Theorem 2.1.]{MS} for us.
\begin{lemma}\label{la:TNhasfinitelaminateapprox}
Assume that $\brac{\tilde{P},(\tilde{C}_\tni)_{\tni=1}^{\tnn},(\tilde{\kappa}_{\tni})_{\tni =1}^{\tnn}}$ is a $T_{\tnn}$-configuration in the sense of \Cref{def:Tnconfig}.

Denote the endpoints by
\[
 Z_{\tnk} \coloneqq \phi_\tnk(P,(C_\tni)_{\tni=1}^{\tnn},(\kappa_{\tni})_{\tni =1}^{\tnn} )
\]
and the base points by
\[
P_{\tnk} \coloneqq  \pi_\tnk(P,(C_\tni)_{\tni=1}^{\tnn},(\kappa_{\tni})_{\tni =1}^{\tnn} ) \coloneqq P + C_1 + C_2 + \ldots  C_{\tnk -1}
\]
Fix $\tnk$, then there exist $(\lambda_\tni)_{\tni=1}^\tnn \subset (0,1)$, $\sum_{\tni=1}^{\tnn} \lambda_{\tni} = 1$, such that
\[
 \nu_{\tnk} \coloneqq  \sum_{\tni=1}^{\tnn} \lambda_{\tni} \delta_{Z_{\tni}}
\]
is a laminate and
\[
 \overline{\nu_{\tnk}} = P_{\tnk}.
\]
$\nu$ is in general \emph{not} a laminate of finite order, however for any open set $\mathcal{O}$ containing $\{Z_1,\ldots,Z_{\tnn}\}^{\mathscr{R}c}$ there exists a sequence of laminates of finite order $\nu_i \in \mathcal{L}(\mathcal{O})$ such that
\[
 \nu_i \text{ weak* converges to } \nu
\]
and
\[
 \overline{\nu_i} = P_{\tnk},
\]
and they can be written as
\[
 \nu_i = \sum_{\tnell=1}^{\tnn} \mu_{\tnell;i} \delta_{P_\tnell} + \sum_{\tnell=1}^{\tnn} \lambda_{\tnell;i} \delta_{Z_\tnell}
\]
where \begin{itemize} \item $\mu_{\tnell;i},\lambda_{\tnell;i} \in [0,1]$
\item $i \mapsto \lambda_{\tnell;i}$ are non-decreasing for each $\tnell \in \{1,\ldots,\tnn\}$
\item $i \mapsto \max_{\tnell} \mu_{\tnell;i}$ is non-increasing and converges to zero as $i \to \infty$.
\item for each $i$, the coefficient $\mu_{\tnell;i}\neq 0$ for at most one $\tnell \in \{1,\ldots,\tnn\}$
\item we have
\[
 \sum_{\tnell=1}^\tnn \mu_{\tnell;i} + \sum_{\tnell=1}^{\tnn} \lambda_{\tnell;i} =1.
\]
\item We have \[
\lim_{i \to \infty} \sum_{\tnell=1}^{\tnn} \lambda_{\tnell;i} =1
              \]

\end{itemize}
\end{lemma}
\begin{proof}
We observe that for each $\tni \in \{1,\ldots,\tnn\}$ there is $s_{\tni} \in (0,1)$ (since $\kappa_{\tni}>1$)
\[
 P_{\tni} = s_\tni P_{\tni-1} + (1-s_{\tni})Z_{\tni-1}
\]
And
\[
 P_{\tni-1}-Z_{\tni-1} \in \mathscr{R}.
\]

We start by setting
\[
 \mu_0 \coloneqq  \delta_{P_{\tnk}}.
\]
$\mu_0$ is clearly a laminate of finite order. And by the above argument we can do an elementary splitting to get
\[
 \mu_1 = s_{\tnk} \delta_{P_{\tnk-1}} + (1-s_{\tnk}) \delta_{Z_{\tnk-1}}
\]
Let us stress that $s_{\tnk} \neq 0, 1$.

We now iteratively perform an elementary splitting the $P_{\tni}$-term (there is exactly one in each step), and find a sequence $\mu_k$ of laminates of finite order of the form
\[
 \mu_k = t_k\delta_{P_{\tni_k}} + \sum_{\tni=1}^\tnn \lambda_{\tni;k} \delta_{Z_{\tni}}
\]
where $t_k+\sum_{\tni=1}^\tnn \lambda_{\tni;k}=1$, $t_k \in (0,1)$ and $\lambda_{\tni;k} \in [0,1]$ -- but $k \mapsto \lambda_{\tni;k}$ is non-decreasing. After $\tnn$ iterations of this arguments all $\lambda_{\tni;k} > 0$.
We also have
\[
 t_{k+1} \leq t_{k} \max_{\tni \in \{1,\ldots,\tnn\}} s_{\tni}
\]
and thus $t_{k} \xrightarrow{k \to \infty} 0$, and $\lambda_{\tni;k} \xrightarrow{k \to \infty} \lambda_{\tni} \in (0,1)$.

Since $\mu_k$ is a Radon measure (up to subsequence) it weak* converges to some $\mu$.

Since each $\mu_k$ is a laminate of finite order, each $\mu_k$ is in particular a laminate, and thus for any (necessarily continuous) $\mathscr{R}$-convex $f$
\[
 f(\overline{P_{\tnk}}) = f(\overline{\mu_k}) \leq \int f d\mu_k \xrightarrow{k \to \infty} \int f d\mu.
\]

That is,
\[
 \mu = \sum_{\tni=1}^\tnn \lambda_{\tni} \delta_{Z_{\tni}} \in \mathcal{P}^{\ell}(\{Z_1,\ldots,Z_{\tnn}\})
\]
is a laminate (not necessarily of finite order) with $\overline{\mu} = P_{\tnk}$.

In particular, in view of \Cref{th:jimmystheorem}
\[
 P_{\tnk} \in \{Z_1,\ldots,Z_{\tnn}\}^{\mathscr{R}c}.
\]
This holds for any $\tnk$ so we actually have
\[
\{P_{1},\ldots,P_{\tnn}\} \subset \{Z_1,\ldots,Z_{\tnn}\}^{\mathscr{R}c}.
\]
Take now any open set $\mathcal{O}$ containing $\{Z_1,\ldots,Z_{\tnn}\}^{\mathscr{R}c}$. In view of \Cref{la:easystufflaminates} any of the elementary splittings above are within $\mathcal{O}$. Thus, $\mu_k \in \mathcal{L}(\mathcal{O})$.

We can conclude.

\end{proof}

We also have
\begin{lemma}\label{la:T4iteration}
Assume that $\brac{{P},({C}_\tni)_{\tni=1}^{\tnn},({\kappa}_{\tni})_{\tni =1}^{\tnn}}$ is a $T_{\tnn}$-configuration in the sense of \Cref{def:Tnconfig}.

Denote the endpoints by
\[
 Z_{\tnk} \coloneqq \phi_\tnk(P,(C_\tni)_{\tni=1}^{\tnn},(\kappa_{\tni})_{\tni =1}^{\tnn} )
\]
and the base points by
\[
P_{\tnk} \coloneqq  \pi_\tnk(P,(C_\tni)_{\tni=1}^{\tnn},(\kappa_{\tni})_{\tni =1}^{\tnn} )
\]

Let $\mu \in (0,1)$ such that $\mu\kappa_\tni > 1$

Then
\begin{itemize}
\item we have \[ \tilde{Z}_\tnk \coloneqq  (1-\mu) P_\tnk + \mu Z_\tnk = \phi_{\tnk}(P,(C_i)_{\tni=1}^{\tnn},(\lambda \kappa_i)_{\tni=1}^{\tnn})\]
\item If we set for $\lambda \in [0,\mu]$.
\[
  \overline{Z}_\tnk \coloneqq (1-\lambda) P_\tnk + \lambda Z_\tnk
\]
then \[\overline{Z}_\tnk \in \{\tilde{Z}_1,\tilde{Z}_2, \ldots , \tilde{Z}_{\tnn}\}^{\mathscr{R}c}\] for any $\tnk \in \{1,\ldots,\tnn\}$.
\item For any $\tnk \in \{1,\ldots,\tnn\}$, any $\mathcal{O}$ open such that $\mathcal{O} \supset \{\tilde{Z}_1,\tilde{Z}_2, \ldots , \tilde{Z}_{\tnn}\}^{\mathscr{R}c}$ there exists a laminate of finite order $\sigma_i \in \mathcal{L}(\mathcal{O})$ with $\overline{\sigma_i} = \overline{Z}_{\tnk}$ such that $\sigma_i$ weak* converges to a laminate
\[
 (1-\frac{\lambda}{\mu}) \sum_{\tnell=1}^{\tnn} \lambda_{\tnell} \delta_{\tilde{Z}_{\tnell}} + \frac{\lambda}{\mu} \delta_{\tilde{Z}_{\tnk}},
\]
where $\lambda_{\tnell} \in (0,1)$ and $\sum_{\tnell=1}^{\tnn} \lambda_{\tnell} = 1$.
$\sigma_i$ has the form
\[
 \sigma_i \coloneqq  (1-\frac{\lambda}{\mu}) \brac{\sum_{\tnell=1}^{\tnn} \mu_{\tnell;i} \delta_{P_\tnell}  + \sum_{\tnell=1}^{\tnn} \lambda_{\tnell;i} \delta_{\tilde{Z}_\tnell}} + \frac{\lambda}{\mu} \delta_{\tilde{Z}_{\tnk}}
\]
\end{itemize}
\end{lemma}
\begin{proof}
We have
\[
 P_{\tnk} = P + \sum_{\tni=1}^{\tnk-1} C_{\tni}
\]
and
\[
\begin{split}
 Z_{\tnk} =& P + \sum_{\tni=1}^{\tnk-1} C_{\tni} + \kappa C_{\tnk}.
\end{split}
 \]
So,
\[
\tilde{Z}_{\tnk} \coloneqq  ( 1-\mu)P_{\tnk} + \mu Z_{\tnk}=P+\sum_{\tni=1}^{\tnk-1} C_\tni + \mu \kappa_{\tnk} C_{\tnk} =\phi_{\tnk}(P,(C_i)_{\tni=1}^{\tnn},(\lambda \kappa_i)_{\tni=1}^{\tnn})
\]

In particular we have that $(\tilde{Z}_{1},\ldots,\tilde{Z}_{\tnn})$ are endpoints of a $\tnn$-configuration, and we notice that
\[
 \pi_{\tnk}(P,(C_i)_{\tni=1}^{\tnn},(\lambda \kappa_i)_{\tni=1}^{\tnn}) = P_{\tnk}
\]
By \Cref{la:TNhasfinitelaminateapprox} we have $P_{\tnk} \in \{\tilde{Z}_{1},\ldots,\tilde{Z}_{\tnn}\}^{\mathscr{R}c}$, \Cref{la:easystufflaminates} we have $\tilde{Z}_{\tnk} \in \{\tilde{Z}_{1},\ldots,\tilde{Z}_{\tnn}\}^{\mathscr{R}c}$ and -- since $\tilde{Z}_\tnk-P_{\tnk} \in \mathscr{R}$
\[
 [P_{\tnk},\tilde{Z}_{\tnk}] \in \{\tilde{Z}_{1},\ldots,\tilde{Z}_{\tnn}\}^{\mathscr{R}c}.
\]
We observe that $\overline{Z}_\tnk \in  [P_{\tnk},\tilde{Z}_{\tnk}]$, because $\lambda \leq \mu$ and
\[
\begin{split}
 \overline{Z}_\tnk =&  (1-\lambda) P_\tnk + \lambda Z_\tnk\\
 =&  (1-\frac{\lambda}{\mu}) P_{\tnk} + \frac{\lambda}{\mu} \tilde{Z}_\tnk.
 \end{split}
\]
For a fixed $\tnk \in \{1,\ldots,\tnn\}$ we take from \Cref{la:TNhasfinitelaminateapprox} the laminate of finite order
\[
 \nu_i = \sum_{\tnell=1}^{\tnn} \mu_{\tnell;i} \delta_{P_\tnell} + \sum_{\tnell=1}^{\tnn} \lambda_{\tnell;i} \delta_{\tilde{Z}_\tnell}
\]
such that
\[
 \overline{\nu_i} = P_{\tnk}
\]
Then
\[
 \sigma_i \coloneqq  (1-\frac{\lambda}{\mu}) \brac{\sum_{\tnell=1}^{\tnn} \mu_{\tnell;i} \delta_{P_\tnell}  + \sum_{\tnell=1}^{\tnn} \lambda_{\tnell;i} \delta_{\tilde{Z}_\tnell}} + \frac{\lambda}{\mu} \delta_{\tilde{Z}_{\tnk}}
\]
satisfies
\[
 \overline{\sigma_i} = \overline{Z}_\tnk.
\]
Moreover, $\sigma_i$ is a laminate of finite order: We can set
\[
 \sigma_{i;0} \coloneqq  \delta_{\overline{Z}_\tnk}
\]
and
\[
 \sigma_{i;1} \coloneqq  (1-\frac{\lambda}{\mu}) \delta_{P_{\tnk}} + \frac{\lambda}{\mu} \delta_{\tilde{Z}_{\tnk}}
\]
which we can obtain by elementary splitting since $P_{\tnk}-\tilde{Z}_{\tnk} \in \mathscr{R}$. Now we continue to define $\sigma_{i;1}$ by going from $\delta_{P_{\tnk}}$ to $\nu_i$ via elementary splittings (which is possible since $\nu_i$ is a laminate of finite order).

Thus $\sigma_i \in \mathcal{L}(\mathcal{O})$ for any open set $\mathcal{O} \supset \{\tilde{Z}_1,\ldots,\tilde{Z}_{\tnn}\}$.

\end{proof}

%
%
%
%

\section{Basic construction: the wiggle lemmata}\label{s:wigglelemmata}
We are going to consider maps $f: \R^n \to \left (\begin{array}{c}
  \R^M \\
  \R^M \otimes \Ep^{n-2} \R^n\\
 \end{array}\right )$.
As usual we call such a map \emph{affine linear} if it is differentiable everywhere and
\[
\left (\begin{array}{c}
  \R^M \\
  \R^M \otimes \Ep^{n-2} \R^n\\
 \end{array}\right ) \ni  \partial_\alpha F \text{ is constant}
\]
We will often call a given affine linear maps
\[
\Aff(x).
\]
A map $f: \Omega \subset \R^n \to \left (\begin{array}{c}
  \R^M \\
  \R^M \otimes \Ep^{n-2} \R^n\\
 \end{array}\right )$ is called \emph{piecewise affine linear} if it is locally Lipschitz and there are countably many $(\Omega_j)_{j=1}^\infty$ of open and pairwise disjoint subsets of $\Omega$ such that
 \[
  \mathcal{L}^n(\Omega \setminus \bigcup_{j=1}^\infty \Omega_j) = 0
 \]
 and $f \Big |_{\Omega_j}$ is affine linear.

The main issue, that is responsible for the linear algebra torture we have to deal with, is that if $n> 2$ then $df$ and $\nabla f$ are two different operations -- which is the reason we chose $\mathscr{R}$ in \Cref{def:R} in the way we did.

Namely, we have the following version of the ``wiggle'' lemma (\cite[Lemma 3.1]{MS}, \cite[Lemma 2.1]{Astala-Faraco-Sz}).

\begin{lemma}\label{la:wiggle}
Let \[
  A_1,A_2 \in \left (\begin{array}{c}
  \R^M \otimes \Ep^1 \R^n\\
  \R^M \otimes \Ep^{n-1} \R^n\\
 \end{array}\right )
\]
such that $A_1-A_2 \in \mathscr{R}$. For $\lambda \in (0,1)$ set $C \coloneqq (1-\lambda) A_1 + \lambda A_2$.

Also assume that we have any affine linear map on $\R^n$
\[
 \Aff(x): \R^n \to \left (\begin{array}{c}
  \R^M \\
  \R^M \otimes \Ep^{n-2} \R^n\\
 \end{array}\right )
\]
such that
\[
 d (\Aff(x)) \equiv C.
\]

Then for any $\eps,\delta > 0$ there exists a piecewise affine mapping
\[
 f: \Omega \subset \R^n \to \left (\begin{array}{c}
  \R^M \\
  \R^M \otimes \Ep^{n-2} \R^n\\
 \end{array}\right )
\]
such that
\begin{itemize}
\item \begin{equation}\label{eq:wiggle:goal1}\dist(df,\{A_1,A_2\}) < \delta \quad \text{a.e. in $\Omega$} \end{equation}
\item \begin{equation}\label{eq:wiggle:goal2}\|\nabla f\|_{L^\infty(\Omega)} \aleq_{|A_1|,|A_2|,|\nabla \Aff|} 1.\end{equation}
\item If $A_1-A_2 \neq 0$ then we have \begin{equation}\label{eq:wiggle:goal3a}|\{x \in \Omega: |d f(x)-A_1| < \delta\}| = (1-\lambda) |\Omega|\end{equation} and
\begin{equation}\label{eq:wiggle:goal3b}
 |\{x \in \Omega: |d f(x)-A_2| < \delta\}| = \lambda |\Omega|
\end{equation}
\item
\begin{equation}\label{eq:wiggle:goal4}
 \abs{\Aff(x)-f} < \eps \quad \text{in $\Omega$}
\end{equation}
and
\[
 f = \Aff(x) \quad \text{on $\partial \Omega$}
\]
\end{itemize}
\end{lemma}
\begin{proof}
If $A_1 = A_2$ (and thus $A_1=A_2=C$) then we set $f \coloneqq  \Aff(x)$ which satisfies all the assumptions. So assume from now on $A_1 \neq A_2$.

For $\lambda \in (0,1)$ let $s: \R \to \R$ be the Lipschitz map defined as
\[
 s(t) \coloneqq \brac{\lambda(1-\lambda) + t[(1-\lambda)\chi_{(-\lambda,0)}(t) - \lambda \chi_{(0,1-\lambda)}(t)} \chi_{(-\lambda,1-\lambda)}(t).
\]
Then
\[
s'(t) = \begin{cases}
  1-\lambda \quad &t \in (-\lambda,0)\\
  -\lambda \quad &t \in (0,1-\lambda)\\
  0 \quad &\text{a.e. otherwise}.
 \end{cases}
\]
Since $A_2-A_1 \in \mathscr{R}$, cf. \eqref{eq:setR}, we find some $b \in \R^n$ and $a \in \left ( \begin{array}{c} \R^M\\ \R^M \otimes \Ep^{n-2} \R^n \end{array}\right )$ such that
\[
 A_2-A_1 = b_\alpha\, dx^\alpha \wedge a.
\]
If $b = 0$ then $A_1=A_2$, so by assumption $b \neq 0$. Simply by otherwise redefining $a$ as $|b| a$, we may assume w.l.o.g. that $|b| = 1$. Pick $o^2,\ldots,o^n \in \R^n$ so that $(b,o^2,\ldots,o^n)$ is an orthonormal basis.

For $\gamma \in (0,1]$, $\sigma \in (0,1]$ and $p \in \R^n$ we set
\[
 w_{\gamma,\sigma,p}(x) \coloneqq \sigma \gamma \brac{s\brac{\frac{q \cdot (x-p)}{\sigma \gamma}} - \sum_{i=2}^n |o^i \cdot (x-p)/\gamma|}
\]
We observe that $w$ is piecewise affine linear.

Moreover,
\[
 dw_{\gamma,\sigma,p} = s'\brac{\frac{q \cdot (x-p)}{\sigma \gamma}} q_\alpha dx^\alpha + O_{\Ep^1}(\sigma).
\]
and
\begin{equation}\label{eq:wiggle:aspjxvc}
 \partial_\beta w_{\gamma,\sigma,p} = s'\brac{\frac{q \cdot (x-p)}{\sigma \gamma}} q_\beta + O(\sigma).
\end{equation}

So if we set
\[
 f_{\gamma,\sigma,p}(x) = \Aff(x) + w_{\gamma,\sigma,p}(x)a
\]
then we have
\[
\begin{split}
 df_{\gamma,\sigma,p}(x) = &C +s'\brac{\frac{q \cdot (x-p)}{\sigma \gamma}}\, q_\alpha\, dx^\alpha \wedge a + O(\sigma)\\
 =& C+s'\brac{\frac{q \cdot (x-p)}{\sigma \gamma}} \brac{A_2-A_1} + O(\sigma)\\
 =&\begin{cases}
                                         A_2  + O(\sigma)\quad &\text{if }q \cdot (x-p)\in (-\sigma\gamma\lambda,0)\\
                                        A_1 +O(\sigma)\quad &\text{if } q \cdot (x-p) \in (0,(1-\lambda)\sigma\gamma )\\
                                        C+O(\sigma) \quad &\text{otherwise}.
                                       \end{cases}
 \end{split}
\]

We also need to control $\nabla f$, but a rough estimate suffices, so with \eqref{eq:wiggle:aspjxvc} we obtain
%
for any $\gamma,\sigma \in (0,1]$ (recall that $|q| = 1$)
\[
 |\nabla \brac{f_{\gamma,\sigma,p}(x)-\Aff(x)}| \leq_{a} 1.
\]

The set
\[
 \Omega_{\gamma,\sigma,p} \coloneqq \{x \in \R^n: w_{\gamma,\sigma,p} > 0\} = b+\sigma  \Omega_{\gamma,1,0}
\]
is certainly an open set (it is even a polytope, i.e. the convex hull of finitely many points). It is also bounded.
Trivially we have
\begin{equation}\label{eq:wiggleasdasd}
 f_{\gamma,\sigma,p}(x) = \Aff(x) \quad \text{on $\partial \Omega_{\gamma,\sigma,p}$}
\end{equation}

In conclusion, we have $f_{\gamma,\sigma,p}$ is Lipschitz in $\overline{\Omega_{\gamma,\sigma,p}}$, and
\[
\begin{split}
 df_{\gamma,\sigma,p} = &\begin{cases}
                                         A_2  + O(\sigma)\quad &\text{if }q \cdot (x-p)\in (-\sigma\gamma\lambda,0)\\
                                        A_1 +O(\sigma)\quad &\text{if } q \cdot (x-p) \in (0,(1-\lambda)\sigma\gamma )\\
                                        C+O(\sigma) \quad &\text{a.e. otherwise}.
                                       \end{cases}
 \end{split}
\]
Since the last case is a zero-set in $\Omega_{\gamma,\sigma,p}$ and $\partial \Omega_{\gamma,\sigma,p}$ is a zero-set we find that
\begin{equation}\label{eq:wiggle:aslkdjhvxcoi}
 \dist(df_{\gamma,\sigma,p},\{A_1,A_2\}) \leq O(\sigma) \quad \text{a.e. in $\overline{\Omega_{\gamma,\sigma,p}}$}
\end{equation}

We also observe that (for $\sigma$ so small that $\dist(A,B) \gg O(\sigma)$) -- since by the definition of $s$ we have all $x$ with $q \cdot (x-p)\in [-\sigma\gamma\lambda,0]$ or $q \cdot (x-p) \in [0,(1-\lambda)\sigma\gamma ]$ belong to $\Omega_{\gamma,\sigma,p}$,
\begin{equation}\label{eq:wiggle:towardsgoal3a}
 \abs{\{x \in \overline{\Omega_{\gamma,\sigma,p}}: \dist(df_{\gamma,\sigma,p},A_1)\}} = (1-\lambda) |\Omega_{\gamma,\sigma,p}|
\end{equation}
\begin{equation}\label{eq:wiggle:towardsgoal3b}
 \abs{\{x \in \overline{\Omega_{\gamma,\sigma,p}}: \dist(df_{\gamma,\sigma,p},A_2)\}} = \lambda |\Omega_{\gamma,\sigma,p}|
\end{equation}

We collect the main properties we care about:
\begin{itemize}
 \item $f_{\gamma,\sigma,p}(x) =  \Aff(x)$ on $\partial \Omega_{\gamma,\sigma,p}$
 \item \begin{equation}\label{eq:wiggle:fLinfty}
 |f_{\gamma,\sigma,p}- \Aff(x)| \aleq \sigma\gamma + \sigma |x-p| \quad \text{in $\Omega_{\gamma,\sigma,p}$}
 \end{equation}
 \item We also have (independent of $\gamma, \sigma \in (0,1]$)
 \[
  \|\nabla f_{\gamma,\sigma,p}\|_{L^\infty(\Omega_{\gamma,\sigma,p})} \aleq_{A_1,A_2,C,\Aff} 1.
 \]

\end{itemize}

The set $\overline{\Omega_{\gamma,\sigma,p}}$ is a closed set with nonempty interior, and if we set
\[
 \mathcal{V} \coloneqq \left \{\overline{\Omega_{r,\sigma,b}} \equiv p+r \overline{\Omega_{1,\sigma,0}}: \quad p \in \R^n: r \in (0,1] \right \}
\]
then $\mathcal{V}$ is a regular family and a Vitali covering as in Vitali's theorem, \Cref{th:vitali}. So there exists an (at most) countable sequence $(p_j)_{j=1}^\infty \subset \R^n$, $(r_j)_{j=0}^\infty \subset (0,1]$ such that
\[
 \overline{\Omega_{r_j,\sigma,p_j}} \cap \overline{\Omega_{r_i,\sigma,p_i}} = \emptyset \quad \forall i \neq j
\]
and
\begin{equation}\label{eq:wiggle:vitalicoverresult}
 \mathcal{L}^{n}\brac{\Omega \setminus \bigcup_{j=1}^\infty \overline{\Omega_{r_j,\sigma,p_j}}} = 0.
\end{equation}
We set
\[
 f(x) \coloneqq \begin{cases}
          \Aff(x) \quad &\text{$x \not \in \bigcup_{j=1}^\infty \overline{\Omega_{r_j,\sigma,p_j}}$}\\
          f_{r_j,\sigma,p_j}(x) \quad &\text{if $x \in p_j + r_j \overline{\Omega_{r_j,\sigma,p_j}}$ for some $j$}
         \end{cases}
\]
We then have that $f: \Omega \to \R$ is continuous, \eqref{eq:wiggleasdasd}, since the Lipschitz constant $\|\nabla f\|$ is uniformly bounded in any $\Omega_{r_j,\sigma,p_j}$ we find that $f: \Omega \to \R$ is indeed Lipschitz, and the Lipschitz constant
\[
 \|\nabla f\|_{L^\infty(\Omega)} \aleq_{a,C} 1.
\]
This establishes \eqref{eq:wiggle:goal2}.

In view of \eqref{eq:wiggle:fLinfty} we have
\[
 |f- (\Aff(x))| \aleq \sigma + \sigma \diam(\Omega) \quad \text{in $\Omega$}
\]
For $\sigma$ suitably small this implies \eqref{eq:wiggle:goal4}.

In view of \eqref{eq:wiggle:aslkdjhvxcoi} combined with \eqref{eq:wiggle:vitalicoverresult} we find that
\[
  \dist(df,\{A_1,A_2\}) \leq O(\sigma) \quad \text{a.e. in $\Omega$}
\]
so again taking $\sigma \ll 1$ we conclude \eqref{eq:wiggle:goal1}.

As for \eqref{eq:wiggle:goal3a} it follows from \eqref{eq:wiggle:towardsgoal3a}, and \eqref{eq:wiggle:goal3b} follows from \eqref{eq:wiggle:towardsgoal3b} with $\sigma \ll 1$.
%
%
%
%
%
%
%
\end{proof}

Iterating \Cref{la:wiggle} we obtain
\begin{lemma}[Wiggle lemma for laminates of finite order]\label{la:wigglelamfinite}
Let $\nu  \in \mathcal{L}(\Omega)$ a laminate of finite order in some open set $\mathcal{O}$, as defined in \Cref{def:laminate}, i.e.
\[
 \nu = \sum_{\tni=1}^{\tnn} \lambda_\tni \delta_{A_\tni}
\]
such that $\lambda_{\tni} \in [0,1]$ and $\sum_{\tni} \lambda_{\tni}=1$ with $A_{\tni} \neq A_{\tnj}$ for $\tni \neq \tnj$.

Also assume that we have any affine linear map on $\R^n$
\[
 \Aff(x): \R^n \to \left (\begin{array}{c}
  \R^M \\
  \R^M \otimes \Ep^{n-2} \R^n\\
 \end{array}\right )
\]
such that
\[
 d (\Aff(x)) \equiv \overline{\nu}.
\]

Then for any $\eps,\delta > 0$ there exists a piecewise affine mapping
\[
 f: \Omega \subset \R^n \to \left (\begin{array}{c}
  \R^M \\
  \R^M \otimes \Ep^{n-2} \R^n\\
 \end{array}\right )
\]
such that
\begin{itemize}
\item \[\dist(df,\{A_1,\ldots,A_\tnn\}) < \delta \quad \text{a.e. in $\Omega$} \]
\item \[\|\nabla f\|_{L^\infty(\Omega)} \aleq_{|A_1|,\ldots,|A_\tnn|,|\nabla \Aff|} 1\]
\item We have \[|\{x \in \Omega: |d f(x)-A_\tni| < \delta\}| = \lambda_\tni |\Omega|\] for any $\tni \in\{1,\ldots,\tnn\}$
\item
\[
 \abs{\Aff(x)-f} < \eps \quad \text{in $\Omega$}
\]
and
\[
 f = \Aff(x) \quad \text{on $\partial \Omega$}
\]
\end{itemize}
\end{lemma}
\begin{proof}
This can be proven by induction. The case $\tnn = 1$ is trivial, $\tnn =2$ is \Cref{la:wiggle}. Assume now that we have $\tilde{f}$ for some laminate of finite order
\[
 \tilde{\nu} = \sum_{\tni=1}^{\tnn} \lambda_{\tni} \delta_{A_\tni},
\]
and $\nu$ is obtained from $\tilde{\nu}$ by an elementary splitting in $\mathcal{O}$, say
\[
A_{\tnn} = (1-s)A_{\tnn+1} +s A_{\tnn+2}
\]
\[
 \nu = \sum_{\tni=1}^{\tnn-1} \lambda_{\tni} \delta_{A_\tni} + \lambda_{\tnn} (1-\lambda) \delta_{A_{\tnn}} + \lambda \lambda_{\tnn} (1-s) \delta_{A_{\tnn+1}} + \lambda \lambda_\tnn s \delta_{A_{\tnn+2}}
\]
Since $\tilde{f}$ is piecewise affine linear we can restrict our attention to (at most countably many) open subsets $\tilde{\Omega}_i$ where $d\tilde{f} \equiv A_{\tnn}$. In these each of these sets we take an open subset $\tilde{\tilde{\Omega}}_i \subset \tilde{\Omega}_i$ such that
\[
|\tilde{\Omega}_i \setminus \overline{\tilde{\tilde{\Omega}}_i}|=(1-\lambda)|\tilde{\Omega}_i|
\]
We do not change $\tilde{f}$ in $\tilde{\Omega}_i \setminus \overline{\tilde{\tilde{\Omega}}_i}$. Inside $\tilde{\tilde{\Omega}}$ (observe that $\tilde{f}$ is affine linear in this set) we apply use \Cref{la:wiggle} to replace $\tilde{f}$ by some $g_i$ such that
\[
|\{x \in \tilde{\tilde{\Omega}}_i: |d g_i(x)-A_{\tnn+1}| < \delta\}| = (1-s) |\tilde{\tilde{\Omega}}_i|
\]
and
\[
|\{x \in \tilde{\tilde{\Omega}}_i: |d g_i(x)-A_{\tnn+2}| < \delta\}| = s |\tilde{\tilde{\Omega}}_i|
\]
We set
\[
 f \coloneqq  \begin{cases}
 g_i \quad &\text{in $\tilde{\tilde{\Omega}}_i$} \\
 \tilde{f} \quad &\text{otherwise}\\
 \end{cases}
\]
Since $g=\tilde{f}$ on $\partial\tilde{\tilde{\Omega}}_i$ the resulting $f$ is still Lipschitz, and since $|\tilde{f}-\Aff| \ll 1$ and $|g_i-\tilde{f}| \ll 1$ we find that $|f-\Aff| \ll1$.

The remaining properties are now easy to check.
\end{proof}

Combining \Cref{la:wigglelamfinite} with \Cref{la:T4iteration}, we obtain the following (cf. \cite[Theorem 3.1]{MS}, \cite[Proposition 1]{Sz04}).

\begin{lemma}[wiggle lemma for $T_{\tnn}$-configurations]\label{la:MS:3.2:v2}
Assume that $\brac{{P},({C}_\tni)_{\tni=1}^{\tnn},({\kappa}_{\tni})_{\tni =1}^{\tnn}}$ is a $T_{\tnn}$-configuration in the sense of \Cref{def:Tnconfig}.

Denote the endpoints by
\[
 Z_{\tnk} \coloneqq \phi_\tnk(P,(C_\tni)_{\tni=1}^{\tnn},(\kappa_{\tni})_{\tni =1}^{\tnn} )
\]
and the base points by
\[
P_{\tnk} \coloneqq  \pi_\tnk(P,(C_\tni)_{\tni=1}^{\tnn},(\kappa_{\tni})_{\tni =1}^{\tnn} )
\]
Let $\mu \in (0,1)$ such that $\mu\kappa_\tni > 1$ for any $\tni \in \{1,\ldots,\tnn\}$ and $0 < \lambda < \mu$.

Set
\[
\tilde{Z}_\tnk \coloneqq  (1-\mu) P_\tnk + \mu Z_\tnk
\]
and
\[
 \overline{Z}_\tnk \coloneqq (1-\lambda) P_\tnk + \lambda Z_\tnk
\]

Fix any $\tnk \in \{1,\ldots,\tnn\}$ and take any affine function $\Aff: \R^n \to \left (\begin{array}{c}
  \R^M \\
  \R^M \otimes \Ep^{n-2} \R^n\\
 \end{array}\right )$ such that
 \[d\Aff = \overline{Z}_{\tnk}.\]

Then for any open set $\Omega$ and any $\delta > 0$ there exist a piecewise affine $f: \Omega \to \left (\begin{array}{c}
  \R^M \\
  \R^M \otimes \Ep^{n-2} \R^n\\
 \end{array}\right )$ such that

 \begin{enumerate}
  \item $\dist(df,\{Z_1,\ldots,Z_{\tnn},P_1,\ldots,P_\tnn\}) <\delta $ a.e. in $\Omega$
  \item $\|\nabla f\|_{L^\infty(\Omega)} \aleq_{\abs{Z_1},\ldots,\abs{Z_{\tnn}}, \abs{\nabla \Aff}} 1$.
  \item Assuming that $\tilde{Z}_{\tni} \neq \tilde{Z}_{\tnj}$ for $\tni \neq \tnj$ for each $\eps > 0$ we can choose $f$ so that \[\abs{\left \{x \in \Omega:\ |df(x)-\tilde{Z}_{\tnk}|<\eps\right \}} \geq \frac{\lambda}{\mu} |\Omega| \]
  and
  \[\abs{\left \{x \in \Omega:\ |df(x)-\tilde{Z}_{\tni}|<\eps \right \}} > 0 \quad \forall \tni \]
  \item and
  \[
   \abs{\left \{x \in \Omega: \dist(df(x),\{Z_1,\ldots,Z_{\tnn}\})>\delta \right \}} < \eps |\Omega|
  \]

  \item $\sup_{\Omega} |f-\Aff | < \delta$
  \item $f = \Aff$ on $\partial \Omega$
 \end{enumerate}
\end{lemma}
\begin{proof}
Fix $\tnk$, set $f_0 \coloneqq  \Aff$.

Take from \Cref{la:T4iteration}
\[
 \sigma_i \coloneqq  (1-\frac{\lambda}{\mu}) \brac{\sum_{\tnell=1}^{\tnn} \mu_{\tnell;i} \delta_{P_\tnell}  + \sum_{\tnell=1}^{\tnn} \lambda_{\tnell;i} \delta_{\tilde{Z}_\tnell}} + \frac{\lambda}{\mu} \delta_{\tilde{Z}_{\tnk}}
\]
which is a laminate of finite order. In view of \Cref{la:wigglelamfinite} we find an piecewise affine Lipschitz function $f$ arbitrarily close to $\Aff$ and
  \[
  \abs{\left \{x \in \Omega:\ |df(x)-\tilde{Z}_{\tnk}|<\eps \right \}} \geq \frac{\lambda}{\mu}
  \]
and
  \[
  \abs{\left \{x \in \Omega:\ |df(x)-\tilde{Z}_{\tni}|<\eps \right \}} > 0
  \]
However we do not necessarily have
\[
\dist(df,\{\tilde{Z}_{1},\ldots,\tilde{Z}_{\tnn}\}) \ll 1
\]
but instead
\[
\dist(df,\{\tilde{Z}_{1},\ldots,\tilde{Z}_{\tnn},P_1,\ldots,P_{\tnn}\}) \ll 1
\]
Since $P_{\tnl} \neq \tilde{Z}_{\tnk}$ we can however ensure that
\[
|\{x: \dist(df(x),\{\tilde{Z}_{1},\ldots,\tilde{Z}_{\tnn}\}) \not \ll 1\}| = (1-\frac{\lambda}{\mu}) \sum_{\tnell=1}^{\tnn} \mu_{\tnl} |\Omega|
\]

\end{proof}

\section{The \texorpdfstring{$T_\tnn$}{TN}-theorem}\label{s:tnnthm}

The following is the essence of \cite{MS}, see also \cite[Proposition 2]{Sz04} whose proof we can essentially follow.
\begin{theorem}\label{th:TnnwithCgivesexample}
Let $\tnn \in \N$ and $\Omega \subset \R^n$ open convex set. Assume that $F \in C^2(\R^{M \times n},\R)$ and $K_F$ as in \eqref{eq:Kfdef}.

Assume that $\brac{\overline{P},(\overline{C}_\tni)_{\tni=1}^{\tnn},(\overline{\kappa}_\tni)_{\tni=1}^{\tnn}}$ is a non-degenerate $T_{\tnn}$-configuration in the sense of \Cref{def:Tnconfig} with
\[
\overline{Z}_\tnk \coloneqq  \phi_\tnk(P,(C_\tni)_{\tni=1}^{\tnn},(\kappa_{\tni})_{\tni =1}^{\tnn} ) \in K_F
\]
Denote by $\mathcal{M}_\tnn$ be the manifold of $T_{\tnn}$-tuples close-by from \Cref{pr:dimensionM}, and assume $\mathcal{M}_\tnn \cap \mathcal{K}_F$ satisfies the condition (C) from \Cref{def:condC}.

There exists $P_0 \in \{\overline{Z}^1,\ldots,\overline{Z}^n\}^{\mathscr{R}c}$ such\footnote{Indeed, one can adapt the proof to show this is actually true for any $P_0 \in \{\overline{Z}^1,\ldots,\overline{Z}^n\}^{\mathscr{R}c}$, but we do not need this so we don't bother} that the following holds

$\Aff: \Omega \to \left (\begin{array}{c}
  \R^M \\
  \R^M \otimes \Ep^{n-2} \R^n\\
 \end{array}\right )$ is a given affine function with $d\Aff \equiv P_0$.

 Then for any $\delta > 0$ there exists $f = \left ( \begin{array}{c} f'\\ f''\end{array}\right ) \Omega \to \left (\begin{array}{c}
  \R^M \\
  \R^M \otimes \Ep^{n-2} \R^n\\
 \end{array}\right )$ with the following properties
 \begin{enumerate}
  \item $df \in K_F \cap \bigcup_{\tnk=1}^{\tnn} B_\delta(\overline{Z}_\tnk)$ a.e. in $\Omega$
  \item In particular $u\coloneqq f': \Omega \to \R^M$ is Lipschitz and solves \eqref{eq:ELsystem1} in distributional sense
  \item $f' = \Aff'$ on $\partial \Omega$ and $\sup_{\Omega} |f'-\Aff'| < \delta$
  \item $f'$ is nowhere $C^1$ in $\Omega$. If the $T_{\tnn}$-configuration is moreover wild in the sense of \Cref{def:Tnconfig} then actually $\partial_\alpha f'$ is nowhere continuous in $\Omega$ for each $\alpha \in \{1,\ldots,n\}$.
 \end{enumerate}

\end{theorem}

\begin{proof}
Recall, for $\overline{\delta}$ suitably small, we have the map $\pi_{\tnk}: \mathcal{M}_\tnn \cap \mathcal{K}_F \cap \brac{B_{\overline{\delta}}(\overline{Z}_1)\times \ldots \times B_\delta(\overline{Z}_\tnn)}$ from \Cref{def:condC}.

Take a strictly increasing sequence $(\lambda_i)_{i=1}^\infty$, $\lim_{i \to \infty} \lambda_i =1$.
\[
 \Phi_i^{\tnk} \coloneqq (1-\lambda_i) \pi_\tnk + \lambda_i \phi_\tnk: \mathcal{M}_\tnn \cap \mathcal{K}_F \to \left (\begin{array}{c}
  \R^M \otimes \Ep^1 \R^n\\
  \R^M \otimes \Ep^{n-1} \R^n\\
 \end{array}\right ).
\]
Of course we would prefer to $\lambda_i = 1$, but then $\phi_{i}^k$ is not a locally open mapping, so we need to assume $\lambda_i<1$ and let $\lambda_i \to 1$ (we assume that $\lambda_i$ so that $\Phi_i^k$ is indeed this locally open mapping by condition (C)). We shall also assume that $\lambda_1 \approx 1$, so that we can apply \Cref{la:MS:3.2:v2} at will. Lastly, observe that we have
\[
\lim_{k \to \infty} \sum_{i=k}^\infty \abs{\lambda_{i+1} - \lambda_i} = 1-\lambda_k \xrightarrow{k \to \infty} 0.
\]

Let $(\mathcal{O}_i)_{i=0}^\infty$ be an increasing sequence of relative open subsets of $\mathcal{M} \cap \mathcal{K} \cap \brac{B_{\overline{\delta}}(\overline{Z}_1)\times \ldots \times B_\delta(\overline{Z}_\tnn)}$,
\[
 \mathcal{O}_{i+1} \supset \overline{\mathcal{O}_i} \supset \mathcal{O}_i.
\]
Since $(\overline{Z}_1,\ldots,\overline{Z}_\tnn)$ are pairwise different we may assume that for each $i$
\[
 U_{i,\tnk}  \coloneqq  \Phi_i^\tnk(\mathcal{O}_i)
\]
are pairwise disjoint, $U_{i,\tnk} \cap U_{i,\tnj} = \emptyset$ (take possibly $\lambda $ even closer to $1$ and $\overline{\delta}$ small enough).

We set
\[
 U_{i} \coloneqq \bigcup_{\tnk=1}^{\tnn} \Phi_i^\tnk(\mathcal{O}_i).
\]
By condition $(C)$ each $U_i \subset  \left (\begin{array}{c}
  \R^M \otimes \Ep^1 \R^n\\
  \R^M \otimes \Ep^{n-1} \R^n\\
 \end{array}\right )$ is open.

From now on we fix any
\[
 P_0 \in U_0
\]

We are now inductively going to define piecewise affine maps $f_i: \Omega \to  \left ( \begin{array}{c} \R^M\\ \R^M \otimes \Ep^{n-2} \R^n \end{array}\right )$ with the properties

\begin{enumerate}
\item $f_i$ is a piecewise affine Lipschitz map
 \item $f_i = \Aff$ on $\partial \Omega$
 \item\label{it:Tnconst:dfiUi} $df_i(x) \in U_i$ for a.e. $x \in \Omega$
 \item $\|f_i-f_{i-1}\|_{L^\infty} \leq 2^{-i}$
\end{enumerate}

Assuming $f_i$ is already constructed we set
\[
 \Omega_{i,\tni} \coloneqq  \text{``}\{x \in \Omega: df_i \in U_{i,\tni}\}\text{''}
\]
by which we mean (since $f_i$ is piecewise affine) the union of open sets where $f_i$ is affine and $df_i \in U_{i,\tni}$. Since for each $i$ the $U_{i,\tni}$ are disjoint we have from \Cref{it:Tnconst:dfiUi} that
\[
\abs{ \Omega \setminus \bigcup_{\tni=1}^{\tnn} \Omega_{i,\tni}}=0.
\]

\begin{enumerate}
  \setcounter{enumi}{4}
 \item \label{TNthm:dfiUiisubstantial} We have
 \[
  |\{x \in \Omega_{i-1,\tni}: df_i \in U_{i,\tni}\}| \geq \frac{\lambda_{i-1}}{\lambda_{i}} \abs{\Omega_{i-1,\tni}}
 \]
\item \label{TNthm:dfiUiiabit}and \[
  |\{x \in \Omega_{i-1,\tnk}: df_i \in U_{i,\tni}\}| >0
 \]
 \item\label{cond:dfconvergence} \[
 \int_{\Omega} |df_{i+1}-df_i| \aleq \frac{\lambda_{i+1}-\lambda_i}{\lambda_1} |\Omega|
\]

\end{enumerate}

We set $f_0 = \Aff$ to start the induction, clearly this satisfies all the conditions.

Now assume we have constructed $f_i$. We are going to discuss the construction of $f_{i+1}$.

We can restrict our attention to $\Omega_{i,\tni}$ for some fixed $\tni \in \{1,\ldots,\tnn\}$. And within $\Omega_{i,\tni}$ we restrict our attention to some $\tilde{\Omega}_{i,\tni}$ where $f_i$ is affine, that is
\[
 df_i \equiv Q \in U_{i,\tni}.
\]
Simply by subdividing $\tilde{\Omega}_{i,\tni}$ we can assume that
\begin{equation}\label{eq:tildeomegaismall}
 \diam \tilde{\Omega}_{i,\tni} < \frac{1}{i}
\end{equation}

By definition of $U_{i,\tni}$, there exists some $T \in \mathcal{O}_i$ such that
\[
Q = (1-\lambda_i) \pi_\tni(T) + \lambda_i  \phi_\tni(T)
\]
For $\tnk \in \{1,\ldots,\tnn\}$ we set
\[
Z^\tnk \coloneqq  (1-\lambda_{{i+1}}) \pi_\tnk(T) + \lambda_{{i+1}}  \phi_\tnk(T) \in U^{i+1}.
\]
By \Cref{la:T4iteration}
\[
 Q \in \{Z^1,\ldots,Z^{\tnn}\}^{\mathscr{R}c}
\]
and we have an explicit laminate of finite order $\sigma_\ell$
\[
 \sigma_\ell \coloneqq  (1-\frac{\lambda_{i}}{\lambda_{i+1}}) \brac{\sum_{\tnell=1}^{\tnn} \tilde{\mu}_{\tnell;\ell} \delta_{P_\tnell}  + \sum_{\tnell=1}^{\tnn} \tilde{\lambda}_{\tnell;i} \delta_{Z_\tnell}} + \frac{\lambda_{i}}{\lambda_{i+1}} \delta_{Z_{\tni}}
\]
with $\overline{\sigma_i} = Q$ and for $\ell \to \infty$ the measures $\sigma_\ell$ weak*-converges to a laminate
\[
 (1-\frac{\lambda_{i}}{\lambda_{i+1}}) \sum_{\tnell=1}^{\tnn} \tilde{\mu}_{\tnell} \delta_{Z_{\tnell}} + \frac{\lambda_{i}}{\lambda_{i+1}} \delta_{Z_{\tni}}.
\]
Here
\[
 P_{\tnell} = \pi_\tnk(T).
\]
It is important to observe that since $T \in \mathcal{O}_i$ these are the endpoints of a $T_{\tnn}$-configuration which satisfies condition $(C)$, thus $\pi_{\tnk}$ is locally an open mapping, so there is a tiny $\delta > 0$  such that
\begin{equation}\label{eq:T4sdfklxvcklxvclk}
 B_{\delta} (P_{\tnell}) \subset \pi_{\tnk}(\mathcal{O}_i) \quad \forall \tnell \in \{1,\ldots,\tnn\}.
\end{equation}

We now apply \Cref{la:MS:3.2:v2} and obtain a candidate for $df_{i+1}$ in $\tilde{\Omega}_{i,\tni}$ which satisfies properties (1), (2), (4), (5), (6).

We also have (for an $\eps$ we can choose)
\begin{equation}\label{eq:xcjlio2}\abs{\left \{x \in \tilde{\Omega}_{i,\tni}:\ |df_{i+1}(x)-Z_{\tni}|<\eps\right \}} \geq \frac{\lambda_i}{\lambda_{i+1}} |\tilde{\Omega}_{i,\tni}| \end{equation}
Observe that we have
\[
 |Z_\tni-Q|\aleq |\lambda_{i+1}-\lambda_i|
\]
so that (choosing $\eps$ in \eqref{eq:xcjlio2} suitably small) we have
\begin{equation}\label{eq:dfconv}
 \int_{\tilde{\Omega}_{i,\tni}} |df_{i+1}-df_i|= \int_{\tilde{\Omega}_{i,\tni}} |df_{i+1}-Q| \aleq |\lambda_{i+1}-\lambda_i| \frac{\lambda_i}{\lambda_{i+1}} |\tilde{\Omega}_{i,\tni}|+(1-\frac{\lambda_i}{\lambda_{i+1}}) |\tilde{\Omega}_{i,\tni}| \aleq \frac{\lambda_{i+1}-\lambda_i}{\lambda_1} |\tilde{\Omega}_{i,\tni}|
\end{equation}
Below, we are not going to change $f_i$ where $|df_i-Z_\tni|\ll 1$, so this statement will stay correct and (summing up over all $\tilde{\Omega}_{i,\tni}$) we have \Cref{cond:dfconvergence}.

But, instead of $(3)$ we have
\[
 \dist(df_{i+1},\{Z_1,\ldots,Z_\tnn,P_1,\ldots,P_{\tnn}\}) \ll 1 \quad \text{a.e. in $\tilde{\Omega}_{i,\tni}$}.
\]
However, we already know that the set where $df_{i+1}$ is close to any of the $P_{\tnell}$ is very small. So we restrict again to subsets $\tilde{\tilde{\Omega}}_{i,\tni}$ of $\tilde{\Omega}_{i,\tni}$ where our new $f_{i+1}$ is affine, and $df_{i+1} \aeq P_{\tnell}$.
That is $df_{i+1} = \tilde{P}_{\tnell} \in \pi_{\tnell}(\mathcal{O}_{i})$, by \eqref{eq:T4sdfklxvcklxvclk}. Thus there exists $\tilde{T} \in \mathcal{O}_i$ such that
\[
 \tilde{P}_{\tnell}=\pi_{\tnell} (\tilde{T}).
\]
Setting
\[
\tilde{Z}^\tnk \coloneqq  (1-\lambda_{i+1}) \pi_\tnk(\tilde{T}) + \lambda_{i+1}  \phi_\tnk(\tilde{T}) \in U^{i+1}.
\]
we have again from \Cref{la:T4iteration} (this time with $\lambda = 0$) that $P \in \{\tilde{Z}^1,\ldots,\tilde{Z}^{\tnn}\}$ with an explicit laminate, to which we can apply \Cref{la:MS:3.2:v2} and replace $f_{i+1}$ in $\tilde{\tilde{\Omega}}_{i,\tni}$. Observe this does not disturb any of the other (1)$\sim$(7). Iterating this argument we only replace $f_{i+1}$ where $df_{i}$ is not close to any $U^{i+1}$, and in the limit we find $df_{i+1}$ satisfying also property yes).

Observe that the previous argument is essentially the same as what is usually used in \cite{Sz04}, namely that
\begin{equation}\label{eq:UiUi+1}
 U_i \subset U_{i+1}^{\mathscr{R}c} \quad \forall i \geq 0
\end{equation}

We observe
\[
\begin{split}
& \int_{\Omega} |df_{i+1}-df_{i}| \\
=&\sum_{\tni=1}^{\tnn} \int_{\Omega_{i,\tni}} |df_{i+1}-df_{i}|\\
=&\sum_{\tni=1}^{\tnn} \int_{\Omega_{i,\tni}, df_i \in U_{i+1,\tni}} |df_{i+1}-df_{i}|\\
&+\sum_{\tni=1}^{\tnn} \int_{\Omega_{i,\tni}, df_i \not \in U_{i+1,\tni}} |df_{i+1}-df_{i}|\\
\end{split}
\]

We now want to discuss the limiting procedure.

We observe that
\[
\sup_{i} \|f_i-\Aff f_i\|_{L^\infty(\Omega)} < \infty
\]
and
\[
\sup_{i} \|df_i\|_{L^\infty(\Omega)} < \infty
\]
But observe that we have no global Lipschitz bound!

If we write
\[
 f_i = \left ( \begin{array}{c}
        f_i'\\
        f_i''
       \end{array} \right ) \in \left ( \begin{array}{c} \R^M\\ \R^M \otimes \Ep^{n-2} \R^n \end{array}\right )
\]
we observe that $f_i'$ is indeed Lipschitz,
\[
\sup_{i} \|f_i'\|_{L^\infty}+ [f_i']_{\lip} < \infty
\]
By \cite[Theorem 6.4]{ISS99} there exist $g_i''$ with $dg_i'' = df_i''$ such that for any $p < \infty$
\[
 \sup_{i} \|g_i''\|_{W^{1,p}(\Omega)}  < \infty
\]
Set
\[
 g_i \coloneqq  \left (\begin{array}{c}
         f_i'\\
         g_i''
        \end{array}\right ), \quad \text{with }dg_i = df_i
\]

Then (up to non-relabelled subsequence) there exist $g \in W^{1,p}(\Omega,\left ( \begin{array}{c} \R^M\\ \R^M \otimes \Ep^{n-2} \R^n \end{array}\right ))$ as a weak limit of $g_i$. By \Cref{cond:dfconvergence}, $dg_i$ is a Cauchy sequence in $L^1$ and thus a.e. convergent. This implies
\[
dg \in K_F \cap \bigcup_{\tnk=1}^{\tnn} B_\delta(\overline{Z}_\tnk) \quad \text{a.e. in $\Omega$}
\]
which implies that $g'$ solves \eqref{eq:ELsystem1} in distributional sense.
In particular, since $dg'$ is bounded, from \eqref{eq:dwieqhdgblabla}
\[
dg''^i =\hdg \sum_{\alpha=1}^n(\partial_{X_{i\alpha}}F)(dg') dx^\alpha
\]
we see that $dg''$ is bounded.

Moreover, if $g'$ was $C^1$ then $dg'$ would be continuous, and thus, since we know \eqref{eq:dwieqhdgblabla}
\[
dg''^i =\hdg \sum_{\alpha=1}^n(\partial_{X_{i\alpha}}F)(dg') dx^\alpha.
\]
we would conclude that $dg''$ is continuous -- that is we conclude that $dg$ is continuous.

However, \underline{$dg$ is not continuous}: Take any $\tilde{\Omega} \subset \Omega$, for some large $i_0$ and some $\tni_0$ some of the $\tilde{\Omega}_{i_0,\tni_0}$ from the construction of $f_i$ must be a subset of $\tilde{\Omega}$, by  \eqref{eq:tildeomegaismall}.

By \Cref{TNthm:dfiUiiabit} there is some $\gamma > 0$ (depending on $i_0$) such that
\[
\abs{\{x \in \tilde{\Omega}_{i_0,\tni_0}: df_{i_0+1}(x) \in U_{i_0+1,\tnk} \}} > \gamma \quad \forall \tnk \in \{1,\ldots,\tnn\}.
\]
But then by \Cref{TNthm:dfiUiisubstantial} for any $i \geq i_0+1$
\[
 \abs{\{x \in \tilde{\Omega}_{i_0,\tni_0}: df_{i}(x) \in U_{i,\tnk} \}} \geq \frac{\lambda_{i-1}}{\lambda_i} \frac{\lambda_{i-2}}{\lambda_{i-1}} \ldots \frac{\lambda_{i_0}}{\lambda_{i_0+1}} \gamma = \frac{\lambda_{i_0}}{\lambda_i} \gamma.
\]
Since $df_i$ converges to $dg$ strongly in $L^1$ , we conclude that
\[
 \abs{\{x \in \tilde{\Omega}_{i_0,\tni_0}: dg(x) \in B_\delta(\overline{Z}_\tnk) \}} \geq \lambda_{i_0} \gamma.
\]
This hold for any $\tnk$, i.e. we have
\[
 \abs{\{x \in \tilde{\Omega}: dg(x) \in B_\delta(\overline{Z}_\tnk) \}} >0 \quad \forall \tnk \in \{1,\ldots,\tnn\}.
\]
and thus $dg$ has non-vanishing essential oscillation, and in particular $dg$ is not continuous.

If we write
\[
\overline{Z}_\tnk = \left ( \begin{array}{c}\overline{X}_{\tnk}\\
 \overline{Y}_{\tnk} \end{array}\right )\in \left (\begin{array}{c}
  \R^M \otimes \Ep^1 \R^n\\
  \R^M \otimes \Ep^{n-1} \R^n\\
 \end{array}\right )
\]
and identify $X_{\tnk}$ with the canonical matrix in $\R^{M \times n}$, then we conclude that for any $\alpha \in (1,\ldots,n)$,
\[
 \abs{\{x \in \tilde{\Omega}: \partial_\alpha g'(x) \in B_\delta(\overline{X}_\tnk e_\alpha) \}} >0 \quad \forall \tnk \in \{1,\ldots,\tnn\}.
\]
Assuming that $T_{\tnn}$ is wild, and $\delta$ is suitably small we conclude that $\partial_\alpha$ has non-vanishing essential oscillation, and thus $\partial_\alpha$ is not continuous in $\tilde{\Omega}$.
 \end{proof}

\section{Polyconvex example: An extension method for \texorpdfstring{$T_{\tnn}$}{TN}}\label{s:polyexample}
The goal of this section is to find some $F: \R^{M \times n} \to \R$, which for algebra reasons we identify with
\[
 F: \R^{M} \otimes \Ep^1 \R^n \to \R
\]
such that there exists some $T_\tnn$-configuration, $(Z_1,\ldots,Z_\tnn) \subset K_F$. For the application of the previous lemmata, in particular to vary $F$ to obtain the condition (C), \Cref{th:conditionCgen}, we need $\tnn$ to be large, and of course the $T_{\tnn}$-configuration must be degenerate in the space we are working in. The idea is to prove an extension theorem that extends the $M=n=2$-$T_{5}$-configuration obtained in \cite{Sz04} to arbitrary dimension and arbitrary $\tnn$ (and ensuring non-degeneracy). This will be accomplished in \Cref{th:pleaspleaseplease}.

\begin{theorem}\label{th:existenceofFwithTNN}
For any $M,n \geq 2$ and $\tnn_0 \in \N$ there exists some $\tnn \geq \tnn_0$ and a non-degenerate, wild $T_\tnn$-configuration in $\left (\begin{array}{c}
  \R^M \otimes \Ep^1 \R^n\\
  \R^M \otimes \Ep^{n-1} \R^n\\
 \end{array}\right )$
with endpoints
\[
 Z^\tni = \left ( \begin{array}{c}
                   X^\tni\\
                   Y^\tni
                  \end{array}
\right )
\]
such
\[
 \{Z_1,\ldots,Z_{\tnn}\} \subset K_F
\]
Moreover, writing
\[
 Z_\tni = \left ( \begin{array}{c}
                   X_\tni\\
                   Y_\tni
                  \end{array}
\right )
\]
we can ensure that $D^2F(X_\tni)$ is positive definite for each $\tni \in \{1,\ldots,\tnn\}$.
\end{theorem}

Recall \eqref{eq:Kfdef}
\[
 K_F = \left \{\left (\begin{array}{c}
  X\\
  Y
 \end{array}\right ) \in \left (\begin{array}{c}
  \R^M \otimes \Ep^1 \R^n\\
  \R^M \otimes \Ep^{n-1} \R^n\\
 \end{array}\right ): \quad   \left (\begin{array}{c}
  X\\
  Y
 \end{array}\right ) = \left (\begin{array}{c}
  X\\
  \brac{\sum_{\alpha =1}^n \frac{\partial F}{\partial_{X_{i\alpha}}}(X){\hdg} dx^\alpha}_{i=1}^M
 \end{array}\right )\right \}
\]

That is a configuration $Z = \left ( \begin{array}{c}
                                      X\\
                                      Y
                                     \end{array}
\right ) \in \left (\begin{array}{c}
  \R^M \otimes \Ep^1 \R^n\\
  \R^M \otimes \Ep^{n-1} \R^n\\
 \end{array}\right )$ belongs to $K_F$ is equivalent to

 \begin{equation}\label{eq:TNconst:condY}
  Y^{{i}}= \frac{\partial F}{\partial_{X_{{i} {\alpha}}}}(X) \hdg dx^{{\alpha}} \in \Ep^{n-1} \R^n, \quad {i} =1,\ldots,M
 \end{equation}

For our purposes, we want to find $F$ which is strongly polyconvex, cf. \Cref{s:polyconvLH}. That is, our $F$ is supposed to have the form
\begin{equation}\label{eq:wewantF}
 F(X)=\frac{\eps}{2}|X|^2  + \delta G(\tilde{X}),
\end{equation}
where
\[
 G: \R^L \to \R
\]
for
\[
 L=\sum_{i=1}^{\min\{n,M\}}\binom{n}{i}\binom{M}{i}
\]
which is the sum of the number of square $i \times i$-submatrices of an $\R^{M \times n}$ matrix.
We will denote  variables in $\R^L$ by $\tilde{X}$. And (by a slight abuse of notation)
for
\[
 X=(X_{i\alpha })_{i=1,\ldots,M;\alpha=1,\ldots,n} \in \R^{M \times n}
\]
and two multi-index $A = \{\alpha_1<i_2<i_3<\ldots<\alpha_k\} \subset \{1,\ldots,n\}$ and $I = \{i_1<i_2<i_3<\ldots<i_k\} \subset \{1,\ldots,M\}$ we set
\[
 X_{I,A} \in \R^{k \times k} = \brac{X_{i_j,\alpha_\beta}}_{j=1,\beta}^k,
\]
and define
\begin{equation}\label{eq:tildeXX}
 \tilde{X} \coloneqq  \times_{k=1}^{\min\{M,n\}}\times_{|A|=|I|=k} \det_{k \times k}(X_{I,A}) \in \R^L
\end{equation}


In order to find $G$ we recall the following simple existence result, cf. \cite{Sz04}.
 \begin{lemma}\label{la:ex:convex}
Assume we have
\[
\brac{ \tilde{X}_{\tni}}_{\tni=1}^{\tnn} \subset \brac{\R^L}^\tnn
\]
with
\begin{equation}\label{eq:endpointsdifferent}
\tilde{ X}_{\tni} \neq \tilde{X}_{\tnj} \quad \tni \neq \tnj,
\end{equation}
and
\[
 (c_\tni)_{\tni = 1}^{\tnn} \subset \R^\tnn
\]
and
\[
( B_{\tni})_{\tni =1}^{\tnn} \subset (\R^L)^{\tnn}
\]
such that
\begin{equation}\label{eq:Sz04:cjci}
c^{\tnj} >c^{\tni} + \langle B_\tni, \tilde{X}_\tnj - \tilde{X}_\tni\rangle_{\R^L} \quad \forall \tnj \neq \tni \in \{1,\ldots,{\tnn}\}
\end{equation}
Then there exists a smooth convex function $G: \R^L \to \R$ with
\[
 G(\tilde{X}_\tni)=c^{\tni}, \quad \tni = 1,\ldots,\tnn
\]
\[
 DG(\tilde{X}_\tni)=B_{\tni}, \quad \tni = 1,\ldots,\tnn
\]
\end{lemma}
\begin{proof}
Set
\[
 G_{\tni}(\tilde{X}) \coloneqq  c^{\tni} + \langle B_{\tni},\tilde{X}-\tilde{X}_{\tni}\rangle \quad \tilde{X} \in \R^L
\]
Then from the assumptions we have \eqref{eq:Sz04:cjci}
\[
 G_{\tnj}(\tilde{X}_\tnj) = c_j > G_{\tni}(\tilde{X}_{\tnj}) \quad \forall \tni \neq \tnj.
\]
By continuity there exists $\delta>0$ such that
\[
 G_{\tnj}(\tilde{X}) > G_{\tni}(\tilde{X}) \quad \forall \tni \neq \tnj, \tilde{X} \in B(\tilde{X}_\tnj,\delta).
\]
We set
\[
 G(\tilde{X}) \coloneqq  \max_{\tni=1,\ldots,\tnn}  G_{\tni}(\tilde{X})
\]
This is Lipschitz and convex and satisfies otherwise the assumptions, since
\[\left ( \begin{array}{cc}
                Y_{12} & -Y_{11}\\
                Y_{22} & -Y_{21}\\
               \end{array} \right )
 G(\tilde{X}) =  G_{\tnj}(\tilde{X}) \quad \forall \tilde{X} \in B(\tilde{X}_\tnj,\delta)
\]
By mollifying  $G$ outside of $\bigcup B(\tilde{X}_\tnj,\delta/2)$ we conclude.
\end{proof}

The following Lemma reduces the job of finding $F$ combined with a $T_\tnn$-tuple into a system of linear inequalities in $Y$. This was successfully employed in $M=n=2$-case by \cite{Sz04}.

\begin{lemma}\label{la:inequalityneeded}
Assume that for some $\delta > 0$ we can find a tuple
\[
\left ( \begin{array}{c}
          \overline{X}^{\tni}\\
          \hat{Y}^{\tni}
         \end{array}
\right) \in \left (\begin{array}{c}
  \R^{M \times n} \\
  \R^{M \times n}\\
 \end{array}\right ) \quad \tni \in \{1,\ldots,\tnn\},
\]
and $(c^{\tnj})_{\tnj=1}^\tnn \subset \R$ and $ d_{k;A,I}^\tni \in \R$, $k \in \{1,\ldots,\min\{M,n\}\}$, $|A|=|I| = k$ such that for all $\tnj \neq \tni \in \{1,\ldots,{\tnn}\}$ we have for any $\tni \neq \tnj \in \{1,\ldots,\tnn\}$,

\begin{equation}\label{eq:linearsystemass}
\begin{split}
c^{\tnj} >&c^{\tni} + \langle
\hat{Y}^\tni, \overline{X}^\tnj - \overline{X}^\tni\rangle\\
& +\delta\sum_{k=2}^{\min\{M,n\}}\sum_{|A|=|I|=k} d^\tni_{k;A,I}  \brac{ \det_{k \times k}(\overline{X}^{\tnj}_{I,A})-\det(\overline{X}^{\tni}_{I,A}) - \langle D_X \det_{I,A}(\overline{X}_{I,A}^\tni), \overline{X}^\tnj_{I,A} - \overline{X}_{I,A}^\tni\rangle}
\end{split}
\end{equation}
Then for all suitably small $\eps > 0$ there exists a convex map $G: \R^L \to \R$ such that $F$ as in \eqref{eq:wewantF}
we have
\begin{equation}\label{eq:constr}
\left ( \begin{array}{c}
          \overline{X}^{\tni}\\
          \overline{Y}^{\tni}
         \end{array}
\right) \in K_F \quad \tni = 1,\ldots,\tnn.
\end{equation}
where we identify $X^{\tni}_{i\alpha} \hat{=} X^{\tni}_{i\alpha} dx^\alpha$, and
where $\hat{Y}^\tni \in \R^{M \times n}$ and $\overline{Y}^\tni \in \R^M \otimes \Ep^{n-1}\R^n$ are related by the identification
\begin{equation}\label{eq:hatvsbarY}
 \overline{Y}^\tni =  \brac{\sum_{\alpha=1}^n \hat{Y}^\tni_{i \alpha} \hdg dx^\alpha}_{i=1}^{M}.
\end{equation}

If
\[
 d^i_{k;A;I} = 0
\]
whenever $A\neq (1,2)$ or $I \neq (1,2)$, then $G= G(X,\det_{2\times 2}(X_{2x2}))$ where $X_{2x2}$ is the upper left square matrix of $X \in \R^{M \times n}$.

\end{lemma}
\begin{remark}
A word of warning: If we compare \eqref{eq:linearsystemass} to the corresponding system of inequalities for $M=n=2$ from \cite{Sz04}, the matrix $\hat{Y}^\tni$ above corresponds  to $YJ$, where $J= \left ( \begin{array}{cc}
              0 & -1\\
              1 & 0
             \end{array}
\right )$.
\end{remark}

\begin{proof}[Proof of \Cref{la:inequalityneeded}]
If we have \eqref{eq:linearsystemass}, for suitably small $\eps > 0$, and changing the $c^j$ we also have
\begin{equation}\label{eq:linearsystemass2}
\begin{split}
\delta c^{\tnj} >&\delta c^{\tni} + \langle
 \hat{Y}^\tni, \overline{X}^\tnj - \overline{X}^\tni\rangle\\
 & - \eps \langle \overline{X}^{\tni},\overline{X}^\tnj - \overline{X}^\tni\rangle\\
&- \langle DF_0(\overline{X}^{\tni}), \overline{X}^\tnj - \overline{X}^\tni\rangle\\
& +\delta\sum_{k=2}^{\min\{M,n\}}\sum_{|A|=|I|=k} d_{k;A,I}  \brac{ \det_{k \times k}(\overline{X}^{\tnj}_{I,A})-\det(\overline{X}^{\tni}_{I,A}) - \langle D_X \det_{I,A}(\overline{X}_{I,A}), \overline{X}^\tnj - \overline{X}^\tni\rangle}
\end{split}
\end{equation}
We define $d_0^\tni \in \R^{M \times n}$ by
\begin{equation}\label{eq:constr:d0sadsad}
\delta d^\tni_1 \coloneqq   \hat{Y}^{\tni} - \eps \overline{X}^{\tni} - DF_0(\overline{X}^{\tni}) -\delta \sum_{k=2}^{\min\{M,n\}}\sum_{|A|=|I|=k}d_{k;A,I}^\tni  \brac{D_X\det_{I,A}}(\overline{X}_{I,A})
\end{equation}
Then \eqref{eq:linearsystemass2} becomes
\begin{equation}\label{eq:linearsystemass3}
\begin{split}
\delta c^{\tnj} >&\delta c^{\tni} + \delta \langle d_1^\tni ,\overline{X}^\tnj - \overline{X}^\tni\rangle\\
& +\delta\sum_{k=2}^{\min\{M,n\}}\sum_{|A|=|I|=k} d_{k;A,I}^\tni  \brac{\det_{k \times k}(\overline{X}^{\tnj}_{I,A})-\det(\overline{X}^{\tni}_{I,A})}
\end{split}
\end{equation}

For $\widetilde{\overline{X}^\tni}$ as in \eqref{eq:tildeXX}, and
\[
B^\tni = \left ( \begin{array}{c}
             d_1^\tni\\
             (d_{k;A,I}^\tni)_{k=2,\ldots,\min\{M,n\},\ |A|=|I|=k}
            \end{array}
\right ) \in \R^L
\]
the system \eqref{eq:linearsystemass3} becomes
\[
c^{\tnj} >c^{\tni} + \langle B^\tni,\widetilde{\overline{X}^\tnj} - \widetilde{\overline{X}^\tni}\rangle
\]
which is \eqref{eq:Sz04:cjci}. Thus, by \Cref{la:ex:convex} there exists $G: \R^L \to \R$ convex with $DG(\widetilde{\overline{X}^\tni})=B^\tni$.

We now rewrite \eqref{eq:constr:d0sadsad} and see that  it is
\[
 \hat{Y}^{\tni} = \eps \overline{X}^{\tni} + DF_0(\overline{X}^{\tni}) + \delta D_X\Big|_{X = \overline{X}^{\tni}} \brac{G\brac{\widetilde{X}}},
\]
that is
\[
 \hat{Y}^{\tni}_{i\alpha}  = \frac{\partial}{\partial X_{i\alpha}}F(\overline{X}^\tni), \quad \tni=1,\ldots,\tnn;\ i=1,\ldots,M;\ \alpha=1,\ldots,n.
\]
Thus
\[
\sum_{\alpha=1}^n \hat{Y}^{\tni}_{i\alpha} \hdg dx^\alpha = \sum_{\alpha=1}^n \frac{\partial}{\partial X_{i\alpha}}F(\overline{X}^\tni) \hdg dx^\alpha, \quad \tni=1,\ldots,\tnn;\ i=1,\ldots,M.
\]
Since we have by \eqref{eq:hatvsbarY}
\[
 \overline{Y}^i =   \sum_{\alpha=1}^n \hat{Y}^{\tni}_{i\alpha} \hdg dx^\alpha \in \Ep^{n-1} \R^n \quad \tni=1,\ldots,\tnn;\ i=1,\ldots,M.
\]
taking into account \eqref{eq:TNconst:condY}, we have established \eqref{eq:constr}.
\end{proof}

\Cref{th:existenceofFwithTNN} is a consequence of \Cref{th:pleaspleaseplease} below, combined with \eqref{la:inequalityneeded} and \eqref{la:ex:convex}. The positive definiteness of $D^2 F(X_{\tni})$ can be obtained exactly as in \cite[p.145]{Sz04}.

\begin{proposition}\label{th:pleaspleaseplease}
For any $M,n \geq 2$ and $\tnn_0 \in \N$ there exists some $\tnn \geq \tnn_0$ and a non-degenerate and wild $T_\tnn$-configuration in $\left (\begin{array}{c}
  \R^M \otimes \Ep^1 \R^n\\
  \R^M \otimes \Ep^{n-1} \R^n\\
 \end{array}\right )$
with endpoints
\[
 Z^\tni = \left ( \begin{array}{c}
                   X^\tni\\
                   Y^\tni
                  \end{array}
\right )
\]
such that we identify with matrices as in \Cref{la:inequalityneeded}, then \eqref{eq:linearsystemass} holds for some choice of $c^\tni$, $d^\tni_{2,(1,2),(1,2)} \in \R$ and all other $d^\tni_{k,A,I}=0$.
\end{proposition}
Let us remark, that the idea in \cite{Sz04} of solving \eqref{eq:linearsystemass} via a computer algebra system by ``guessing good $X^{\tni}$'' seems to work in any dimension and for any $\tnn$. We give an example for $M=n=3$ in \Cref{s:mathematica}. Since we want to consider every dimension $M,n \geq 2$, and were unable to abstractly prove that \eqref{eq:linearsystemass} is solvable, we developed a method to extend the $M=n=2$-$T_{5}$-example to any dimension and any $\tnn$.

The main idea for \Cref{th:pleaspleaseplease} is to extend the $T_{5}$ in $\R^{4\times 2}$ obtained in \cite{Sz04} to an $T_{\tnn}$ with large $\tnn$ and in any dimension $\left ( \begin{array}{c}
                                                                                                      \R^{M \times n}\\
                                                                                                      \R^{M \times n}
                                                                                                     \end{array}
\right)$. The ``lifting'' into higher dimensions is relatively straight forward. The extension of a $T_{\tnn}$ to a (useful) $T_{\tnn+2}$ is more work.

In the following, since we work with both, the $\R^{4\times 2}$-situation and the $\R^{2M\times n}$-situation, we denote objects in the latter with $\ul{b}$ etc.

First we discuss the lifting to higher dimensions.
\begin{lemma}\label{la:TnnLift}
Let $M,n \geq 2$. $\tnn \geq 3$.

Assume we have a non-degenerate $T_\tnn$-configuration of $\left ( \begin{array}{c}
                                                                                                      \R^{2 \times 2}\\
                                                                                                      \R^{2 \times 2}
                                                                                                     \end{array}
\right)$

given by
\[
 (0,(C'^\tni)_{\tni=1}^{\tnn},(\kappa_\tni)_{\tni=1}^{\tnn}) \in \R^{2 \times 2} \times \brac{\R^{2 \times 2}_{\rank =1}}^\tnn\times \R^\tnn
\]
\[
 (0,(C''^\tni)_{\tni=1}^{\tnn},(\kappa_\tni)_{\tni=1}^{\tnn}) \in \R^{2 \times 2} \times \brac{\R^{2 \times 2}_{\rank =1}}^\tnn\times \R^\tnn
\]
with endpoints
\[
 Z^\tni = \left ( \begin{array}{c}
              X^\tni\\
              Y^\tni
             \end{array} \right ) = \sum_{\tnk=1}^\tni \left ( \begin{array}{c}
              C'^\tni\\
              C''^\tni
             \end{array} \right ) + (\kappa_\tni-1) \left ( \begin{array}{c}
              C'^\tni\\
              C''^\tni
             \end{array} \right ), \tni = 1,\ldots,\tnn,
\]
where $C'^\tni = a'^\tni\otimes b^\tni$, $C''^\tni = a''^\tni \otimes b^\tni$ for $a^\tni=(a'^\tni,a''^\tni)^T \in \R^4$ and $b^{\tni} \in \R^2$, $|b^\tni|=1$.

Set, cf. \eqref{eq:Rrelation},
\[
\underline{C}^\tni = \sum_{\alpha=1}^n \underline{b}^\tni_{;\alpha} dx^\alpha \wedge \underline{a}^\tni \in \left ( \begin{array}{c} \R^M \otimes \Ep^1 \R^n\\ \R^M \otimes \Ep^{n-1} \R^n \end{array}\right )
\]
where
\begin{equation}\label{eq:DQMn:extbyzerob}
 \underline{b}^\tni_{} = \left ( \begin{array}{c}
                       b\\
                       0_{n-2}
                      \end{array} \right)\in \R^n,
\end{equation}

\begin{equation}\label{eq:DQMn:extbyzeroa}
 \underline{a}^\tni =
\left (\begin{array}{c} \underline{a}'^\tni \\ \underline{a}''^\tni \end{array}\right)
= \left ( \begin{array}{c} \left (\begin{array}{c}a'^\tni\\0_{M-2}\end{array} \right )\\ \left (\begin{array}{c}a''^\tni\\0_{M-2}\end{array} \right ) \hdg \brac{dx^{1}\wedge dx^{2}} \end{array}\right ) \in \left ( \begin{array}{c} \R^M\\ \R^M \otimes \Ep^{n-2} \R^n \end{array}\right ).
\end{equation}
Set
\[
 \underline{Z}^\tni = \left ( \begin{array}{c}
              \underline{X}^\tni\\
              \underline{Y}^\tni
             \end{array} \right ) = \sum_{\tnk=1}^\tni \left ( \begin{array}{c}
              \underline{C}'^\tni\\
              \underline{C}''^\tni
             \end{array} \right ) + (\kappa_\tni-1) \left ( \begin{array}{c}
              \underline{C}'^\tni\\
              \underline{C}''^\tni
             \end{array} \right ) \in \left ( \begin{array}{c} \R^M \otimes \Ep^1 \R^n\\ \R^M \otimes \Ep^{n-1} \R^n \end{array}\right )
\]
Then
\begin{itemize}
 \item $(0,(\underline{C}^\tni)_{\tni=1}^\tnn,(\kappa_\tni)_{\tni=1}^\tnn)$ is a $T_{\tnn}$-configuration in $\left (\begin{array}{c}
  \R^M \otimes \Ep^1 \R^n\\
  \R^M \otimes \Ep^{n-1} \R^n\\
 \end{array}\right )$ the sense of \Cref{def:Tnconfig}
 \item If $n=2$ this $T_{\tnn}$-configuration is non-degenerate. If $n > 2$ the configuration satisfies all condition of nondegenerate $T_{\tnn}$ condition except \eqref{eq:nondegenerate}
 \item Assume that the $\R^{4\times 2}$-$T_\tnn$-configuration satisfies, for some $c^\tni \in \R$ and $d^\tni \in \R$,
\begin{equation}\label{eq:szineq2x4vDQ}
\begin{split}
0 >&c^{\tni} - c^{\tnj}+ \langle
 {Y}^\tni J, {X}^\tnj - {X}^\tni\rangle + d^\tni  \det(X^\tnj - {X}^\tni) \quad \text{for all $\tni\neq \tnj \in \{1,\ldots,\tnn\}$}
\end{split}
\end{equation}
Here $J= \left ( \begin{array}{cc}
              0 & -1\\
              1 & 0
             \end{array}
\right )$.

 If we set
 \begin{equation}\label{eq:hatvsbarYv2v3}
\underline{\hat{Y}}^\tni_{\cdot \alpha} = \hdg \brac{dx^\alpha \wedge \underline{Y}^\tni} \in \R^M, \quad \alpha \in \{1,\ldots,n\}
\end{equation}
and identify
 \begin{equation}\label{eq:hatvsbarYv2v3X}
\R^{M \times n} \ni \underline{X}^\tni_{i\alpha}\hat{=}\underline{X}^\tni_{i \alpha} dx^\alpha \in \R^M \otimes \Ep^{1} \R^n
\end{equation}
then these matrices $(\underline{X}^\tni, \underline{\hat{Y}}^\tni)$ satisfy \eqref{eq:linearsystemass} for $d^\tni_{2,(1,2),(1,2)} = d^\tni$ and all other $d^{\tni}_{k,A,I} = 0$.
\end{itemize}

\end{lemma}

\begin{proof}
The first two bullet points are obvious, so is the last one. We just need to check  \eqref{eq:linearsystemass}.

Assume that $b = (b_1,b_2,0,\ldots,0)^T \in \R^n$ and $\underline{a}'' = (a'',0)^T \in \R^M$, for $a'' \in \R^2$ and set
\[
\underline{C}'' \coloneqq  \sum_{\alpha=1}^n b_\alpha dx^\alpha \wedge \underline{a}''
\]
We compute $\widehat{\underline{C}''}$ as defined in \eqref{eq:hatvsbarYv2v3}:
\[
\begin{split}
&\hdg \brac{dx^\alpha \wedge \brac{\sum_{\beta=1}^n b_\beta dx^\beta \wedge \underline{a}''}}\\
=&\hdg \brac{dx^\alpha \wedge \brac{\sum_{\beta=1}^n b_\beta dx^\beta \wedge \brac{\left (\begin{array}{c}a''\\0_{M-2}\end{array} \right ) \hdg \brac{dx^{1}\wedge dx^{2}}}}}\\
=&\left (\begin{array}{c}a''\\0_{M-2}\end{array} \right )  \hdg \brac{dx^\alpha \wedge \brac{\sum_{\beta=1}^2 b_\beta dx^\beta \wedge \brac{\hdg \brac{dx^{1}\wedge dx^{2}}}}}\\
\end{split}
\]
This is zero for $\alpha \not \in \{1,2\}$. For $\alpha = 1$,
\[
\begin{split}
&\hdg \brac{dx^1 \wedge \brac{\sum_{\beta=1}^n b_\beta dx^\beta \wedge \underline{a}''}}\\
=&\left (\begin{array}{c}a''\\0_{M-2}\end{array} \right )b_2   \hdg \brac{dx^1 \wedge dx^2\wedge \brac{\hdg \brac{dx^{1}\wedge dx^{2}}}}= \left (\begin{array}{c}a''\\0_{M-2}\end{array} \right )b_2
\end{split}
\]
For $\alpha = 2$
\[
\begin{split}
&\hdg \brac{dx^2 \wedge \brac{\sum_{\beta=1}^n b_\beta dx^\beta \wedge \underline{a}''}}\\
=&\left (\begin{array}{c}a''\\0_{M-2}\end{array} \right ) b_1 \hdg \brac{dx^2 \wedge \brac{ dx^1 \wedge \brac{\hdg \brac{dx^{1}\wedge dx^{2}}}}}\\
=&-\left (\begin{array}{c}a''\\0_{M-2}\end{array} \right ) b_1
\end{split}
\]
So we have
\[
\widehat{\underline{C}''} = - \left (\begin{array}{c}a''\\0_{M-2}\end{array} \right ) \otimes \left (\begin{array}{c} b^\perp\\ 0_{n-2} \end{array}\right )
\]
Which by \Cref{la:aotimesbJ} is equivalent to
\[
\widehat{\underline{C}''} =  \left (\begin{array}{cccc}\brac{a'' \otimes b} J& 0& \ldots &0\\
0 & 0 & \ldots &0
\end{array} \right )
\]
Thus,
\[
 \underline{\hat{Y}}^\tni  =  \left (\begin{array}{cccc}Y^\tni  & 0& \ldots &0\\
0 & 0 & \ldots &0
\end{array} \right ) \in \R^{M \times n}
\]

Thus it is clear that the claim is true.

\end{proof}

The basic idea to extend a given $T_{\tnn}$ to a larger (and non-degenerate) $T_{\tnn+2}$ is to distort two of the ``rank-1'' connections $C^{\tni_1}, C^{\tni_1} \in \mathscr{R}^o$, and replace them by
\[
 C^{\tni_1} + \delta \overline{C}
\]
\[
C^{\tni_2} - \delta \overline{C}
\]
Of course, in general neither of those two matrices above belongs to $\mathscr{R}^o$, but if $\overline{C}$ is itself ``rank 1'' then each matrix above is ``rank-2'' i.e. each can be decomposed into two $\mathscr{R}^o$-connections. The point is that we can choose the two $\mathscr{R}^o$-connections in a way that the resulting $T_{\tnn+2}$-endpoints do not deviate too much from their original $T_{\tnn}$ (this is the content of \eqref{eq:DQMn:kappaCkappa1Ct} below), and using the precise asymptotics as $\delta \to 0$ we shall see that choosing $\overline{C}$ wisely, we can ensure the linear system \eqref{eq:linearsystemass3} remains valid.

We begin with the decomposition result. We state it first in $\R^{4 \times 2}$. We are actually not going to use it, but just the $\R^{2M\times n}$-dimensional version, but we believe the $\R^{4\times 2}$-version is easier to process for first-time readers.
\begin{lemma}\label{la:onerankdec4x2}
Fix any matrix $C \in \R^{4 \times 2}$, $\lambda \in (0,1)$, $a \in \R^{4} \setminus \{0\}$ and $b \in \R^2$, $|b|=1$.

For any $\delta > 0$ exists two vectors $\tilde{b}^\tni \in \R^2$, $|\tilde{b}^\tni|=1$, and $\tilde{a}^\tni \in \R^4$, $\tni=1,2$ such that

\begin{itemize}
\item We have
\[
  \delta C + a \otimes b  = \tilde{a}^1 \otimes \tilde{b}^1 + \tilde{a}^2 \otimes \tilde{b}^2
\]
\item Any two vectors of $\{b,\tilde{b}^1,\tilde{b}^2\}$ are linearly independent.
\item $\lim_{\delta \to 0}|\tilde{b}^\tni-b| = 0$ for $\tni=1,2$ as $\delta \to 0$
\item $\lim_{\delta \to 0} |\tilde{a}^1-\lambda a| + |\tilde{a}^2-(1-\lambda) a|= 0$.
\item If
\[
 C = \overline{a}\otimes \overline{b} \quad \overline{a} \in \R^4, \quad \overline{b} \in \R^2
\]
More precisely, as $\delta \to 0$, denoting by $b^\perp = (-b_2,b_1)^T \in \R^2$, and $\sigma$ either $\sqrt{\delta}$ or $-\sqrt{\delta}$ we can ensure that
\[
\tilde{a}^1\otimes \tilde{b}^1 - \lambda a \otimes b  =\sigma \brac{\lambda a\otimes b^\perp+ (1-\lambda) |a| \langle b^\perp,\overline{b}\rangle \overline{a} \otimes b} + o(\sqrt{\delta})
\]
and
\[
 \tilde{a}^2\otimes \tilde{b}^2 - (1-\lambda) a \otimes b =-\sigma \brac{\lambda a  \otimes b^\perp+(1-\lambda) |a| \langle b^\perp,\overline{b}\rangle \overline{a} \otimes b}
+ o(\sqrt{\delta})
\]
\end{itemize}
\end{lemma}

Here is the higher-dimensional version of \Cref{la:onerankdec4x2}. We leave it as an exercise to the reader to deduce \Cref{la:onerankdec4x2} from  \Cref{la:onerankdecMxn}, or to prove \Cref{la:onerankdec4x2} by hand translating the arguments of \Cref{la:onerankdecMxn}.

Observe that we add $\delta^2 \ul{\xi} \in \R^n$ which allows us to make our $T_{\tnn}$ less flat, which is important to eventually obtain a non-degenerate $T_{\tnn}$, and ensure the $T_{\tnn}$ is wild in the sense of \Cref{def:Tnconfig}.
\begin{lemma}\label{la:onerankdecMxn}
Fix $\lambda \in (0,1)$, $\underline{C} \in \mathscr{R}^o$ given by
\[
\underline{C} = \sum_{\alpha=1}^n\underline{b}_\alpha dx^\alpha \wedge \underline{a} \in \left ( \begin{array}{c} \Ep^1 \R^n\\ \R^M \otimes \Ep^{n-1} \R^n \end{array}\right ),
\]
for some $\underline{b} \in \R^n$, $|\underline{b}|=1$ and some $\underline{a} = \left ( \begin{array}{c} \underline{a}'\\ \underline{a}'' \end{array}\right ) \in  \left (\begin{array}{c} \R^M\\ \R^M \otimes \Ep^{n-2} \R^n \end{array}\right )$.

Fix $\underline{\overline{a}} \in \left ( \begin{array}{c} \R^M\\ \R^M \otimes \Ep^{n-2} \R^n \end{array}\right )$ and $\underline{\overline{b}} \in \R^n$.

We assume that
\begin{equation}\label{eq:barbbulbnotld}
\ul{ \overline{b}}, \ul{b} \quad \text{linearly independent in $\R^n$}
\end{equation}
in particular
\begin{equation}\label{eq:DQp}
 p \coloneqq  \underline{\overline{b}}-\langle \underline{\overline{b}} , \underline{b}\rangle \underline{b} \neq 0.
\end{equation}

Also let $\underline{\xi} \in \R^n$.

For $\delta > 0$ set
\[
 \underline{\overline{C}}_\delta = \brac{\underline{\overline{b}}_\alpha  + \delta^2 \underline{\xi}_\alpha }dx^\alpha \wedge \underline{\overline{a}} \in \mathscr{R}^o.
\]
Then  there exists two vectors $\underline{\tilde{b}}^\tni \in \R^n$, $|\underline{\tilde{b}}^\tni|=1$, and $\underline{\tilde{a}}^\tni \in \left ( \begin{array}{c} \R^M\\ \R^M \otimes \Ep^{n-2} \R^n \end{array}\right )$, $\tni=1,2$ such that for

\[
\underline{\tilde{C}}^\tni \coloneqq   \sum_{\alpha=1}^n \underline{\tilde{b}}^\tni_\alpha dx^\alpha \wedge \underline{\tilde{a}}^\tni \in \left ( \begin{array}{c} \Ep^1 \R^M\\ \R^M \otimes \Ep^{n-1} \R^n \end{array}\right )
\]

\begin{itemize}
\item we have the decomposition
\begin{equation}\label{eq:onerankdec:dec}
  \delta \underline{\overline{C}}_\delta + \underline{C} = \underline{\tilde{C}}^1 + \underline{\tilde{C}}^2
\end{equation}
\item For all small $\delta$, any two vectors of $\{\underline{b},\tilde{\underline{b}}^1,\tilde{\underline{b}}^2\}$ are linearly independent in $\R^n$,
\item For all small $\delta$, if we write $\underline{\tilde{a}}^\tni =\in \left (\begin{array}{c}
                                                                                      \underline{\tilde{a}}'^\tni\\
                                                                                   \underline{\tilde{a}}''^\tni\\
                                                                                  \end{array}\right ) \in \left (
\begin{array}{c} \R^M\\ \R^M \otimes \Ep^{n-2} \R^n \end{array}\right )$ then if $\{\ul{\overline{a}}',\ul{\overline{a}}''\}$ are linearly independent in $\R^M$
then also
$\{\underline{\tilde{a}}'^1,\underline{\tilde{a}}'^2\}$ are linearly independent in $\R^M$.
\item Recall $p$ from \eqref{eq:DQp} then
\begin{equation}\label{eq:DQMnspantb1tb2}
        \span \{\ul{\tilde{b}}^1,\ul{\tilde{b}}^2\} = \span \{\ul{b}\} + \span \{p+\delta^2 \brac{\xi-\langle \xi,\ul{b}\rangle \ul{b}}\}
      \end{equation}

\item $\lim_{\delta \to 0}|\tilde{\underline{b}}^\tni-\underline{b}| = 0$ for $\tni=1,2$
\item $\lim_{\delta \to 0} |\tilde{\underline{a}}^1-\lambda \underline{a}| + |\tilde{\underline{a}}^2-(1-\lambda) \underline{a}|= 0$.
\item Indeed, we have the following asymptotics as $\delta \to 0$
\begin{equation}\label{eq:DQMn:C1lambdaC}
\begin{split}
&\underline{\tilde{C}}^1 - \lambda \underline{C}\\
=& \sqrt{\delta}\brac{\frac{1}{2}\frac{1-\lambda}{ 1-\lambda+ \lambda^2} \langle\ul{\overline{b}}, \frac{p}{|p|^2}\rangle \sum_{\alpha=1}^n\ul{b}_\alpha dx^\alpha \wedge \underline{\overline{a}} + \lambda \sum_{\alpha=1}^n p_\alpha dx^\alpha \wedge \underline{a}} + o(\sqrt{\delta})\\
 \end{split}
\end{equation}
and
\begin{equation}\label{eq:DQMn:C21mlambdaC}
\begin{split}
&\underline{\tilde{C}}^2 - (1-\lambda) \underline{C}\\
=& -\sqrt{\delta}\brac{\frac{1}{2}\frac{1-\lambda}{ 1-\lambda+ \lambda^2} \langle\ul{\overline{b}}, \frac{p}{|p|^2}\rangle \sum_{\alpha=1}^n\ul{b}_\alpha dx^\alpha \wedge \underline{\overline{a}} + \lambda\sum_{\alpha=1}^n p_\alpha dx^\alpha \wedge \underline{a}} + o(\sqrt{\delta})\\
 \end{split}
\end{equation}
Here recall \eqref{eq:DQp} for the definition of $p$.
\item Fix $\kappa > 1$ and choose $\lambda \coloneqq  \kappa - \sqrt{\kappa-\kappa} \in (0,1)$, i.e. such that
\begin{equation}\label{eq:DQkappalpambda}
\tilde{\kappa}^1 =\tilde{\kappa}^2  \coloneqq  \frac{\kappa}{\lambda} = \frac{\kappa  - \lambda}{1-\lambda}>1
\end{equation}

Then \begin{equation}
\label{eq:DQMn:kappaCkappa1Ct}\lim_{\delta \to 0}\abs{\kappa \ul{C} -\tilde{\kappa}_1 \ul{\tilde{C}}^1}+\abs{\kappa \ul{C}  - \brac{\ul{\tilde{C}}^1+\tilde{\kappa}_2 \ul{\tilde{C}}^2}}=0
\end{equation}

%
%
%
\end{itemize}
\end{lemma}

\begin{proof}[Proof of \Cref{la:onerankdecMxn}]

Observe that $p \perp \ul{b}$ and in view of \eqref{eq:barbbulbnotld} $0 < |p| \leq 1$.

Define $\sigma = -\sqrt{\delta} \frac{\lambda}{1-\lambda}$
\[
\begin{split}
 \underline{\tilde{b}}^1 \coloneqq & \frac{\ul{b}+\sqrt{\delta} \brac{\underline{\overline{b}}+\delta^2 \ul{\xi}-\langle \underline{\overline{b}}+\delta^2 \ul{\xi} , \underline{b}\rangle \underline{b}} }{\abs{
 \ul{b}+\sqrt{\delta} \brac{\underline{\overline{b}}+\delta^2 \ul{\xi}-\langle \underline{\overline{b}}+\delta^2 \ul{\xi} , \underline{b}\rangle \underline{b}} }}\\
\overset{\eqref{eq:DQp}}{=}& \frac{\ul{b}+\sqrt{\delta} \brac{p+\delta^2 \ul{\xi}-\langle \delta^2 \ul{\xi} , \underline{b}\rangle \underline{b}} }{\abs{\ul{b}+\sqrt{\delta} \brac{p+\delta^2 \ul{\xi}-\langle \delta^2 \ul{\xi} , \underline{b}\rangle \underline{b}} }}\\
=& \frac{\ul{b}+\sqrt{\delta} p }{\abs{
 \ul{b}+\sqrt{\delta} p }} + O(\delta^2)\\
 =&\frac{\ul{b}+\sqrt{\delta} p}{\sqrt{1+\delta|p|^2}}\\
 =&\ul{b}+\sqrt{\delta} p+\brac{\ul{b}+\sqrt{\delta} p} \delta\brac{
\frac{|p|^2}{\sqrt{1+\delta|p|^2} \brac{\sqrt{1+\delta|p|^2}+1}}} + O(\delta^2)\\
 =&\ul{b}+\sqrt{\delta} p+\delta \ul{b} |p|^2  + \delta^{\frac{3}{2}} p|p|^2 + O(\delta^2)\\
 \end{split}
\]
Similarly,
\[
\begin{split}
 \underline{\tilde{b}}^2 \coloneqq & \frac{\ul{b}+\sigma \brac{\underline{\overline{b}}+\delta^2 \ul{\xi}-\langle \underline{\overline{b}}+\delta^2 \ul{\xi} , \underline{b}\rangle \underline{b}} }{\abs{
 \ul{b}+\sigma \brac{\underline{\overline{b}}+\delta^2 \ul{\xi}-\langle \underline{\overline{b}}+\delta^2 \ul{\xi} , \underline{b}\rangle \underline{b}} }}\\
=& \frac{\ul{b}+\sigma \brac{p+\delta^2 \ul{\xi}-\langle \delta^2 \ul{\xi} , \underline{b}\rangle \underline{b}} }{\abs{
 \ul{b}+\sigma \brac{p+\delta^2 \ul{\xi}-\langle \delta^2 \ul{\xi} , \underline{b}\rangle \underline{b}} }}\\
 =&\ul{b}+\sigma p+\sigma^2 \ul{b} |p| + \sigma^3 p|p|^2  + O(\delta^2)\\
 \end{split}
\]
In particular,
\begin{itemize}
 \item Since $0 \neq p \perp \ul{b}$ we see that $(\ul{\tilde{b}}^1,\ul{b})$ and $(\ul{\tilde{b}}^2,\ul{b})$ are linearly independent for small $\sqrt{\delta}$. Since $\sigma \neq \sqrt{\delta}$ we also see $(\ul{\tilde{b}}^1,\ul{\tilde{b}}^2)$ must be linearly independent for small $\delta$.
 \item Since $\sqrt{\delta} \neq \sigma$ we have \eqref{eq:DQMnspantb1tb2}: Actually we have
 \begin{equation}\label{eq:DQMnspansb}
 \begin{split}
       \span \{\ul{\tilde{b}}^1,\ul{\tilde{b}}^2\} =& \span \{\ul{b}\} + \span \{p+\delta^2 \brac{\xi-\langle \xi,\ul{b}\rangle \ul{b}}\}\\
       \overset{\eqref{eq:DQp}}{=} &\span \{\ul{b}\} + \span \{\ul{\overline{b}}+\delta^2 \xi\}
\end{split}
       \end{equation}
Indeed, \[
          \span \{\ul{\tilde{b}}^1,\ul{\tilde{b}}^2\} \subset \span \{\ul{b}\} + \span \{p+\delta^2 \brac{\xi-\langle \xi,\ul{b}\rangle \ul{b}}\}
        \]
        and
        \[
 \span \{p+\delta^2 \brac{\xi-\langle \xi,\ul{b}\rangle \ul{b}}\} \subset \span \{\ul{\tilde{b}}^1,\ul{\tilde{b}}^2\} + \span \{\ul{b}\}
\]
are obvious from the definition of $\ul{\tilde{b}^1}$ and $\ul{\tilde{b}^2}$.

On the other hand we can explicitly compute $\eta_1,\eta_2 \in \R$ so that
\[
\ul{b} = \eta_1 \ul{\tilde{b}}^1 + \eta_2 \ul{\tilde{b}}^2.
\]
We have
\[
   \underline{\tilde{b}}^1 \frac{\sigma}{\sigma-\sqrt{\delta}} \abs{
 \ul{b}+\sqrt{\delta} \brac{\underline{\overline{b}}+\delta^2 \ul{\xi}-\langle \underline{\overline{b}}+\delta^2 \ul{\xi} , \underline{b}\rangle \underline{b}} }
 = \frac{\sigma}{\sigma-\sqrt{\delta}} \ul{b}+\frac{\sigma \sqrt{\delta}}{\sigma-\sqrt{\delta}} \brac{\underline{\overline{b}}+\delta^2 \ul{\xi}-\langle \underline{\overline{b}}+\delta^2 \ul{\xi} , \underline{b}\rangle \underline{b}} \\
\]
and
\[
-  \underline{\tilde{b}}^2 \frac{\sqrt{\delta}}{\sigma-\sqrt{\delta}} \abs{
 \ul{b}+\sigma \brac{\underline{\overline{b}}+\delta^2 \ul{\xi}-\langle \underline{\overline{b}}+\delta^2 \ul{\xi} , \underline{b}\rangle \underline{b}} } = \frac{-\sqrt{\delta}}{{\sigma-\sqrt{\delta}}} \ul{b}-\frac{\sqrt{\delta} \sigma}{\sigma-\sqrt{\delta}} \brac{\underline{\overline{b}}+\delta^2 \ul{\xi}-\langle \underline{\overline{b}}+\delta^2 \ul{\xi} , \underline{b}\rangle \underline{b}}
\]
Adding the two we find (using $\ul{b} \perp p$)
\[
\begin{split}
\eta_1 =& \frac{\sigma}{\sigma-\sqrt{\delta}} \abs{
 \ul{b}+\sqrt{\delta} \brac{\underline{\overline{b}}+\delta^2 \ul{\xi}-\langle \underline{\overline{b}}+\delta^2 \ul{\xi} , \underline{b}\rangle \underline{b}} }\\
 =&\frac{\sigma}{\sigma-\sqrt{\delta}}+0 +O(\delta)
\end{split}
 \]
and
\[
\begin{split}
 \eta_2 =& -\frac{\sqrt{\delta}}{\sigma-\sqrt{\delta}} + \abs{
 \ul{b}+\sigma \brac{\underline{\overline{b}}+\delta^2 \ul{\xi}-\langle \underline{\overline{b}}+\delta^2 \ul{\xi} , \underline{b}\rangle \underline{b}} }\\
 =&-\frac{\sqrt{\delta}}{\sigma-\sqrt{\delta}} + O(\delta).
 \end{split}
\]
Since
\[
 \sigma = -\sqrt{\delta} \frac{\lambda}{1-\lambda},
\]
we have the asymptotics
\[
\begin{split}
\eta_1 =& \lambda+O(\delta)
\end{split}
 \]
and
\[
\begin{split}
 \eta_2 =& (1-\lambda) + O(\delta).
 \end{split}
\]
\end{itemize}

We already computed $\eta_1,\eta_2 \in \R$ such that
\[
\ul{b} = \eta_1 \ul{\tilde{b}}^1 + \eta_2 \ul{\tilde{b}}^2.
\]

By \eqref{eq:DQMnspansb} there must be $\mu_1,\mu_2 \in \R$ such that
\begin{equation}\label{eq:DQmn:mu1mu2eq}
\delta \brac{\ul{\overline{b}}+\delta^2 \ul{\xi}} = \mu_1 \ul{\tilde{b}}^1 + \mu_2 \ul{\tilde{b}}^2\\
\end{equation}

Thus
\[
\begin{split}
&\delta \overline{C}_\delta + \underline{C} \\
 =& \delta \sum_{\alpha=1}^n\brac{\brac{\underline{\overline{b}}_\alpha  + \delta^2 \ul{\xi}_\alpha }dx^\alpha \wedge \underline{\overline{a}}} + \sum_{\alpha=1}^n\underline{b}_\alpha dx^\alpha \wedge \underline{a}\\
=& \sum_{\alpha=1}^n\brac{\mu_1 \ul{\tilde{b}}^1 +\mu_2 \ul{\tilde{b}}^2 }_\alpha dx^\alpha \wedge \underline{\overline{a}} + \sum_{\alpha=1}^n\brac{\eta_1 \ul{\tilde{b}}^1 +\eta_2 \ul{\tilde{b}}^2}_\alpha dx^\alpha \wedge \underline{a}\\
=& \sum_{\alpha=1}^n\ul{\tilde{b}}^1_\alpha dx^\alpha \wedge \mu_1\underline{\overline{a}} +\sum_{\alpha=1}^n\ul{\tilde{b}}^2_\alpha dx^\alpha \wedge  \mu_2\underline{\overline{a}} \\
&+ \sum_{\alpha=1}^n \ul{\tilde{b}}^1_\alpha dx^\alpha \wedge \eta_1\underline{a} + \sum_{\alpha=1}^n\ul{\tilde{b}}^2_\alpha dx^\alpha \wedge \eta_2 \underline{a}\\
=& \sum_{\alpha=1}^n\ul{\tilde{b}}^1_\alpha dx^\alpha \wedge \brac{\mu_1\underline{\overline{a}}+\eta_1\underline{a}}
+\sum_{\alpha=1}^n\ul{\tilde{b}}^2_\alpha dx^\alpha \wedge \brac{\mu_2\underline{\overline{a}} +\eta_2 \underline{a}}\\
\end{split}
 \]
So we set
\[
 \ul{\tilde{a}}^1 \coloneqq  \brac{ \mu_1\underline{\overline{a}}+\eta_1\underline{a}}
\]
\[
  \ul{\tilde{a}}^2 = \brac{ \mu_2\underline{\overline{a}} +\eta_2 \underline{a}}
\]
and we have found a decomposition as in \eqref{eq:onerankdec:dec}.

We now need to discuss the asymptotics: We first observe
\begin{itemize}
\item We have $ \underline{\tilde{b}}^1- \underline{\tilde{b}}^2=(\sqrt{\delta} - \sigma)p+o(\sqrt{\delta})$
\item
Using $p \perp \underline{b}$ we also find
\[
\begin{split}
 \langle  \underline{\tilde{b}}^1, \underline{\tilde{b}}^2\rangle =& 1+0+\brac{\sqrt{\delta} \sigma +\delta+\sigma^2}|p|^2 + 0 + O(\delta^{2})\\
=&  1+\delta \brac{\frac{\sqrt{\delta} \sigma + \delta + \sigma^2}{\delta}}\abs{p}^2  + 0 + O(\delta^{2})\\
 \end{split}
\]
and thus
\[
\begin{split}
 \langle  \underline{\tilde{b}}^1, \underline{\tilde{b}}^2\rangle^2 =&
 1+2\delta \brac{\frac{\sqrt{\delta} \sigma + \delta + \sigma^2}{\delta}}\abs{p}^2  + O(\delta^{2})\\
 \end{split}
\]
It is worth noticing that for our choice of $\sigma$, $\frac{\sqrt{\delta} \sigma + \delta + \sigma^2}{\delta} \neq 0$.
\end{itemize}

We now compute $\mu_1$ and $\mu_2$. Multiplying \eqref{eq:DQmn:mu1mu2eq} with
\[
\ul{\tilde{b}}^1-\langle \ul{\tilde{b}}^1,\ul{\tilde{b}}^2\rangle \ul{\tilde{b}}^2 \perp \ul{\tilde{b}}^2
\]
we find
\[
\begin{split}
&\delta \langle{\ul{\overline{b}}+\delta^2 \ul{\xi}}, \ul{\tilde{b}}^1-\langle \ul{\tilde{b}}^1,\ul{\tilde{b}}^2\rangle \ul{\tilde{b}}^2 \rangle = \mu_1 \brac{1-\langle \ul{\tilde{b}}^1,\ul{\tilde{b}}^2\rangle^2}\\
\Leftrightarrow &\langle{\ul{\overline{b}}+\delta^2 \underline{\xi}}, \ul{\tilde{b}}^1-\langle \ul{\tilde{b}}^1,\ul{\tilde{b}}^2\rangle \ul{\tilde{b}}^2 \rangle = \mu_1 \brac{2\brac{\frac{\sqrt{\delta} \sigma + \delta + \sigma^2}{\delta}}\abs{p}^2  + O(\delta)}\\
\Leftrightarrow &
\langle\ul{\overline{b}}, \ul{\tilde{b}}^1-\ul{\tilde{b}}^2 \rangle +O(\delta)= \mu_1 \brac{2\brac{\frac{\sqrt{\delta} \sigma + \delta + \sigma^2}{\delta}}\abs{p}^2  + O(\delta)}\\
\Leftrightarrow &\langle\ul{\overline{b}}, (\sqrt{\delta}-\sigma)p\rangle + o(\sqrt{\delta})= \mu_1 \brac{2\brac{\frac{\sqrt{\delta} \sigma + \delta + \sigma^2}{\delta}}\abs{p}^2  + O(\delta)}
\end{split}
\]
Thus
\[
 \mu_1 = \frac{1}{2}\frac{\sqrt{\delta}-\sigma}{\frac{\sqrt{\delta} \sigma + \delta + \sigma^2}{\delta}} \langle\ul{\overline{b}}, \frac{p}{|p|^2}\rangle + o(\sqrt{\delta})
\]

Observe that
\[
\begin{split}
& \frac{\sqrt{\delta}-\sigma}{\frac{\sqrt{\delta} \sigma + \delta + \sigma^2}{\delta}}\\
=& \sqrt{\delta}\frac{1+\frac{\lambda}{1-\lambda}}{ \frac{\lambda}{1-\lambda}+  1+ \brac{\frac{\lambda}{1-\lambda}}^2}\\
=& \sqrt{\delta}\frac{1-\lambda}{ 1-\lambda+ \lambda^2}\\
 \end{split}
\]
Thus
\[
 \mu_1 = \frac{1}{2}\sqrt{\delta}\frac{1-\lambda}{ 1-\lambda+ \lambda^2}\langle\ul{\overline{b}}, \frac{p}{|p|^2}\rangle + o(\sqrt{\delta})
\]
Multiplying \eqref{eq:DQmn:mu1mu2eq} with
\[
\ul{\tilde{b}}^2-\langle \ul{\tilde{b}}^1,\ul{\tilde{b}}^2\rangle \ul{\tilde{b}}^1 \perp \ul{\tilde{b}}^1
\]
we get the same asymptotics only with a sign change for $\mu_2$,
\[
 \mu_2 = -\frac{1}{2}\sqrt{\delta}\frac{1-\lambda}{ 1-\lambda+ \lambda^2} \langle\ul{\overline{b}}, \frac{p}{|p|^2}\rangle + o(\sqrt{\delta})
\]

These asymptotics in particular imply
\begin{itemize}
\item $\lim_{\delta \to 0}|\tilde{\underline{b}}^\tni-\underline{b}| = 0$ for $\tni=1,2$
\item $\lim_{\delta \to 0} |\tilde{\underline{a}}^1-\lambda \underline{a}| + |\tilde{\underline{a}}^2-(1-\lambda) \underline{a}|= 0$.
\item We also observe that if $(\ul{\overline{a}}',\ul{a}')$ are linearly independent vectors of $\R^M$ then
\[
\begin{split}
&\zeta  \ul{\tilde{a}}'^1 = \vartheta  \ul{\tilde{a}}'^2\\
\Leftrightarrow &\zeta  \brac{ \mu_1\underline{\overline{a}}'+\eta_1\underline{a}'} = \vartheta  \brac{ \mu_2\underline{\overline{a}}' +\eta_2 \underline{a}'}\\
\Leftrightarrow &\mu_1 \zeta  \underline{\overline{a}}'+\eta_1 \zeta  \underline{a}' =  \mu_2\vartheta  \underline{\overline{a}}' +\eta_2 \vartheta  \underline{a}'\\
\Leftrightarrow &\brac{\eta_1 \zeta   -\eta_2 \vartheta  }\underline{a}'=  \brac{\mu_2\vartheta   -\mu_1 \zeta } \underline{\overline{a}}'\\
\Leftrightarrow &\eta_1 \zeta   -\eta_2 \vartheta  = 0 \text{ and }  \mu_2\vartheta   -\mu_1 \zeta =0\\
\Leftrightarrow &
\left ( \begin{array}{cc}
\eta_1 &-\eta_2  \\
-\mu_1 & \mu_2
\end{array} \right ) \left ( \begin{array}{c} \zeta   \\ \vartheta  \end{array} \right )= 0
\end{split}
\]

We know that by\[
0 \neq \langle\ul{\overline{b}}, p \rangle \overset{\eqref{eq:DQp}}{=} \abs{\ul{\overline{b}}}^2- \brac{\langle \underline{\overline{b}} , \underline{b}\rangle}^2 \overset{\eqref{eq:barbbulbnotld}}{\neq} 0,
\]
we have the asymptotics
\[
0 \neq  \mu_2 = -\mu_1 + o(\sqrt{\delta})
\]
and
\[
\eta_1 = \lambda+o(\sqrt{\delta}), \quad \text{and} \quad  \eta_2 = (1-\lambda) + o(\sqrt{\delta}).
 \]
Since
\[
 \det \left ( \begin{array}{cc}
\lambda &-(1-\lambda) \\
1 & 1
\end{array} \right ) =1 \neq 0.
\]
we find that  for suitably small $\delta$
\[
\begin{split}
&\zeta  \ul{\tilde{a}}'^1 = \vartheta  \ul{\tilde{a}}'^2\\
\Leftrightarrow & \zeta   =0 \text{ and } \vartheta  = 0.
\end{split}
\]
Thus $\ul{\tilde{a}}'^1,\ul{\tilde{a}}'^2$ are linearly independent in $\R^M$.
\end{itemize}

Moreover,
\[
\begin{split}
&\underline{\tilde{C}}^1 - \lambda \underline{C}\\
=& \sum_{\alpha=1}^n\ul{\tilde{b}}^1_\alpha dx^\alpha \wedge \brac{\mu_1\underline{\overline{a}}+\eta_1\underline{a}} - \lambda \sum_{\alpha=1}^n\ul{b}_\alpha dx^\alpha \wedge \underline{a}\\
=& \sum_{\alpha=1}^n\ul{\tilde{b}}^1_\alpha dx^\alpha \wedge \brac{\mu_1\underline{\overline{a}}+\lambda \underline{a}} - \lambda \sum_{\alpha=1}^n\ul{b}_\alpha dx^\alpha \wedge \underline{a} + O(\delta)\\
=& \mu_1\sum_{\alpha=1}^n\ul{\tilde{b}}^1_\alpha dx^\alpha \wedge \underline{\overline{a}} + \lambda \sum_{\alpha=1}^n\brac{\ul{\tilde{b}}^1 -\ul{b}}_\alpha dx^\alpha \wedge \underline{a} + O(\delta)\\
=& \sqrt{\delta}\frac{1}{2}\brac{\frac{1-\lambda}{ 1-\lambda+ \lambda^2} \langle\ul{\overline{b}}, \frac{p}{|p|^2}\rangle \sum_{\alpha=1}^n\ul{b}_\alpha dx^\alpha \wedge \underline{\overline{a}} + \lambda \sum_{\alpha=1}^n p_\alpha dx^\alpha \wedge \underline{a}} + o(\sqrt{\delta})\\
 \end{split}
\]
and
\[
\begin{split}
&\underline{\tilde{C}}^2 - (1-\lambda) \underline{C}\\
=& \sum_{\alpha=1}^n\ul{\tilde{b}}^2_\alpha dx^\alpha \wedge \brac{\mu_2\underline{\overline{a}}+\eta_2\underline{a}} - (1-\lambda) \sum_{\alpha=1}^n\ul{b}_\alpha dx^\alpha \wedge \underline{a}\\
=& -\sqrt{\delta}\frac{1}{2}\frac{1-\lambda}{ 1-\lambda+ \lambda^2} \langle\ul{\overline{b}}, \frac{p}{|p|^2}\rangle\sum_{\alpha=1}^n\ul{b}_\alpha dx^\alpha \wedge \underline{\overline{a}} + (1-\lambda)\sum_{\alpha=1}^n\brac{\sigma p}_\alpha dx^\alpha \wedge \underline{a} + o(\sqrt{\delta})\\
=& -\sqrt{\delta}\frac{1}{2}\frac{1-\lambda}{ 1-\lambda+ \lambda^2} \langle\ul{\overline{b}}, \frac{p}{|p|^2}\rangle\sum_{\alpha=1}^n\ul{b}_\alpha dx^\alpha \wedge \underline{\overline{a}} - \lambda \sqrt{\delta}\sum_{\alpha=1}^np_\alpha dx^\alpha \wedge \underline{a} + o(\sqrt{\delta})\\
=& -\sqrt{\delta}\frac{1}{2}\brac{\frac{1-\lambda}{ 1-\lambda+ \lambda^2} \langle\ul{\overline{b}}, \frac{p}{|p|^2}\rangle \sum_{\alpha=1}^n\ul{b}_\alpha dx^\alpha \wedge \underline{\overline{a}} + \lambda\sum_{\alpha=1}^n p_\alpha dx^\alpha \wedge \underline{a}} + o(\sqrt{\delta})\\
 \end{split}
\]
This establishes the asymptotic behavior in \eqref{eq:DQMn:C1lambdaC} and \eqref{eq:DQMn:C21mlambdaC}.

As for \eqref{eq:DQMn:kappaCkappa1Ct}, by the previous results
\[
\lim_{\delta \to 0}\abs{ \kappa \ul{C} -\tilde{\kappa}_1 \ul{\tilde{C}}^1} = \lim_{\delta \to 0} \abs{\kappa \ul{C} -\tilde{\kappa}_1 \lambda \ul{C}} \overset{\eqref{eq:DQkappalpambda}}{=}0
\]
Similarly
\[
\lim_{\delta \to 0}\abs{\kappa \ul{C}  - \brac{\ul{\tilde{C}}^1+\tilde{\kappa}_2 \ul{\tilde{C}}^2}} = \lim_{\delta \to 0}
\abs{\kappa \ul{C}  - \brac{\lambda\ul{C}^1+\tilde{\kappa}_2 (1-\lambda)\ul{C}^2}}
\overset{\eqref{eq:DQkappalpambda}}{=}0
\]

\[
 \frac{\kappa}{\lambda} = \frac{\kappa  - \lambda}{1-\lambda}
\]

\[
\lim_{\delta \to 0}\abs{\kappa \ul{C} -\tilde{\kappa}_1 \ul{\tilde{C}}^1}+\abs{\kappa \ul{C}  - \brac{\ul{\tilde{C}}^1+\tilde{\kappa}_2 \ul{\tilde{C}}^2}}=0
\]

We can conclude.
\end{proof}

The following is the main building block for \Cref{th:pleaspleaseplease}.

\begin{proposition}\label{pr:DonQuijotte22}
Assume for some given $\tnn \geq 3$ there exists a $T_\tnn$-configuration $\left (\begin{array}{c}
  \R^M \otimes \Ep^1 \R^n\\
  \R^M \otimes \Ep^{n-1} \R^n\\
 \end{array}\right )$
which is non-degenerate except possibly for the spanning condition \eqref{eq:nondegenerate}.
Let it be given by
\[
 (0,(C'^\tni)_{\tni=1}^{\tnn},(\kappa_\tni)_{\tni=1}^{\tnn}) \in \left (\begin{array}{c}
  \R^M \otimes \Ep^1 \R^n\\
   \end{array}\right ) \times \left (\begin{array}{c}
  \R^M \otimes \Ep^1 \R^n\\
   \end{array}\right ) ^\tnn\times (1,\infty)^\tnn
\]

\[
 (0,(C'^\tni)_{\tni=1}^{\tnn},(\kappa_\tni)_{\tni=1}^{\tnn}) \in \left (
  \R^M \otimes \Ep^{n-1} \R^n \right ) \times \left (
  \R^M \otimes \Ep^{n-1} \R^n \right )^\tnn\times (1,\infty)^\tnn
\]
with endpoints
\[
 Z^\tni = \left ( \begin{array}{c}
              X^\tni\\
              Y^\tni
             \end{array} \right ) = \sum_{\tnk=1}^\tni \left ( \begin{array}{c}
              C'^\tni\\
              C''^\tni
             \end{array} \right ) + (\kappa_\tni-1) \left ( \begin{array}{c}
              C'^\tni\\
              C''^\tni
             \end{array} \right ), \tni = 1,\ldots,\tnn,
\]
where \[C'^\tni = b^\tni_\alpha dx^\alpha \wedge a'^\tni \] \[C''^\tni = b^\tni_{\alpha} \wedge a''^\tni\] for $a^\tni=\left (\begin{array}{c} a'^\tni\\a''^\tni \end{array} \right )\in \left (\begin{array}{c} \R^M\\
  \R^M \otimes \Ep^{n-2} \R^n \end{array} \right )$ and $b^{\tni} \in \R^n$, $|b^\tni|=1$.

Assume it satisfies, with the identification  \eqref{eq:hatvsbarYv2v3} and \eqref{eq:hatvsbarYv2v3X}, for some $c^\tni \in \R$ and $d^\tni \in \R$
\begin{equation}\label{eq:szineqmxnvDQ}
\begin{split}
c^{\tnj} >&c^{\tni} + \langle \hat{Y}^\tni, {X}^\tnj - {X}^\tni\rangle +d^\tni  \det_{2 \times 2}({X}^{\tnj}_{(1,2),(1,2)}-{X}^{\tni}_{(1,2),(1,2)})
\end{split}
\end{equation}

Assume there exists $\tni_1\neq  \tni_2 \in\{1,\ldots,\tnn\}$ and a $\ul{\overline{a}}' \in \R^M$, $\ul{\overline{a}}'' = \ul{\overline{a}}_{\alpha \beta} \hdg \brac{dx^\alpha \wedge dx^\beta} \in \R^M \otimes \Ep^{n-2} \R^n$,  $\ul{\overline{b}} \in \R^n$, $|\ul{\overline{b}}|=1$, such that if we set
\[
\ul{p}^{\tni} \coloneqq  \underline{\overline{b}}-\langle \underline{\overline{b}} , \underline{b}^{\tni}\rangle \underline{b}^{\tni},
\]
then  we have

\begin{equation}\label{eq:DC:condYXV1}
 \langle\ul{\overline{b}}, \frac{\ul{p}^{\tni_1}}{|\ul{p}^{\tni_1}|^2}\rangle_{\R^n}
\brac{
\sum_{\alpha,\beta=1}^n\brac{\langle \underline{a}''^{\tni_1}_{\alpha \beta},\underline{\overline{a}}' \rangle_{\R^M} -
\langle \underline{\overline{a}}_{\alpha \beta}'' ,\underline{a}'^{\tni_1}\rangle_{\R^M}} \ul{p}^{\tni_1}_\beta \ul{b}^{\tni_1}_\alpha }
>   -2d^{\tni_1}    \langle\ul{\overline{b}}, \frac{\ul{p}^{\tni_1}}{|\ul{p}^{\tni_1}|^2}\rangle_{\R^n}\
\langle \brac{a'^{\tni_1}}^\perp, \overline{a}'\rangle_{\R^2}\, \langle \brac{p^{\tni_1}}^\perp, b^{\tni_1}\rangle_{\R^2}
\end{equation}
and
\begin{equation}\label{eq:DC:condYXV2}
\langle\ul{\overline{b}}, \frac{\ul{p}^{\tni_2}}{|\ul{p}^{\tni_2}|^2}\rangle_{\R^n}
\brac{
\sum_{\alpha,\beta=1}^n\brac{\langle \underline{a}''^{\tni_2}_{\alpha \beta},\underline{\overline{a}}' \rangle_{\R^M} -
\langle \underline{\overline{a}}_{\alpha \beta}'' ,\underline{a}'^{\tni_2}\rangle_{\R^M}} \ul{p}^{\tni_2}_\beta \ul{b}^{\tni_2}_\alpha }
<  -2d^{\tni_2}    \langle\ul{\overline{b}}, \frac{\ul{p}^{\tni_2}}{|\ul{p}^{\tni_2}|^2}\rangle_{\R^n}\
\langle \brac{a'^{\tni_2}}^\perp, \overline{a}'\rangle_{\R^2}\, \langle \brac{p^{\tni_2}}^\perp, b^{\tni_2}\rangle_{\R^2}
\end{equation}
Here, on the right-hand side for a vector $\ul{v} \in \R^M$ or $\ul{v} \in \R^n$ we call the restriction to the first two entries $v$. Namely,
\[
 a' = \left ( \begin{array}{c} \ul{a}'_1\\
 \ul{a}'_2\\ \end{array}
 \right ), \quad b = \left ( \begin{array}{c}\ul{b}_1\\\ul{b}_2\end{array}\right ),\quad  p = \left ( \begin{array}{c}\ul{p}_1\\\ul{p}_2\end{array}\right )
\]
And, for an $\R^2$-vector we set $v^\perp = (-v_2,v_1)^T$.
(We also observe with some relief that it is irrelevant if $C^{\tni_1} = a^{\tni_1} \otimes b^{\tni}$ or $C^{\tni_1} = (-a^{\tni_1}) \otimes (-b^{\tni})$)

Then there exists a $T_{\tnn+2}$-configuration
\[
 (0,(\tilde{C}^\tni)_{\tni=1}^{\tnn},(\tilde{\kappa}_\tni)_{\tni=1}^{\tnn})
\]
which is still non-degenerate except possibly for the spanning condition \eqref{eq:nondegenerate}, and it still satisfies \eqref{eq:szineqmxnvDQ}, \eqref{eq:DC:condYXV1} and \eqref{eq:DC:condYXV2}.
Also, if the original $T_{\tnn}$-configuration satisfies $\dim \span \{b^1,\ldots,b^\tnn\} = L < n$ then we can ensure that the new $T_{\tnn}$-configuration satisfies $\dim \span \{\ul{\tilde{b}}^1,\ldots,\ul{\tilde{b}}^{\tnn+2}\} \geq L+1$.

Lastly, regarding the wildness of the $T_{\tnn}$: If the original $T_{\tnn}$-configuration was wild for some $\{\alpha_1,\ldots,\alpha_{k}\}\subset \{1,\ldots,n\}$ then so is the new $T_{\tnn+2}$-configuration, and indeed for any $\beta  \in \{1,\ldots,n\}\setminus \{\alpha_1,\ldots,\alpha_{k}\}$ we can ensure that the new $T_{\tnn+2}$-configuration is wild for $\{\alpha_1,\ldots,\alpha_{k},\beta\}$.

\end{proposition}

\begin{proof}
W.l.o.g. $\tni_1 < \tni_2$.

If $\dim \span \{b^1,\ldots,b^\tnn\}  < n$ take $\ul{\xi} \in \span \{b^1,\ldots,b^\tnn\}^\perp \subset \R^n$, $|\ul{\xi}| =1$.

The orthogonality assumption on $\ul{\xi}$ ignores the wildness claim, but it is easy to check that a slight rotation of $\ul{\xi}$ ensures that the construction below improves the wildness.

Fix $\underline{\overline{a}} \in \left ( \begin{array}{c} \R^M\\ \R^M \otimes \Ep^{n-2} \R^n \end{array}\right )$ and $\underline{\overline{b}} \in \R^n$, $\tni_1 \neq \tni_2 \in \{1,\ldots,\tnn\}$ as in \eqref{eq:DC:condYXV1}, \eqref{eq:DC:condYXV2}.
Obviously $\overline{b} \neq 0$ and $\overline{a} \neq 0$, so without loss of generality we may assume $|\overline{a}|=1$ and $|\overline{b}|=1$. Since the inequalities \eqref{eq:DC:condYXV1} and \eqref{eq:DC:condYXV2} are strict, we can also assume that $\overline{a}', \overline{b}$ are each linearly independent with each $a'^\tni$, $b^\tni$, respectively.

Set for small $\delta >0$
\[
 \underline{\overline{C}}_\delta = \brac{\underline{\overline{b}}_\alpha  + \delta^2 \underline{\xi}_\alpha }dx^\alpha \wedge \underline{\overline{a}} \in \mathscr{R}^o.
\]

The idea is to ``replace'' $\ul{C}^{\tni_1}$ with $\ul{C}^{\tni_1}+\delta \ul{\overline{C}}_\delta$ and $C^{\tni_2}$ with $\ul{C}^{\tni_2} - \delta \ul{\overline{C}}_\delta$, for $\delta$ sufficiently small, so that the new endpoints $\tilde{Z}^\tni$ are close to the original endpoints. Cf. \Cref{fig:TNext1}.

\begin{figure}
 \includegraphics[width=\textwidth]{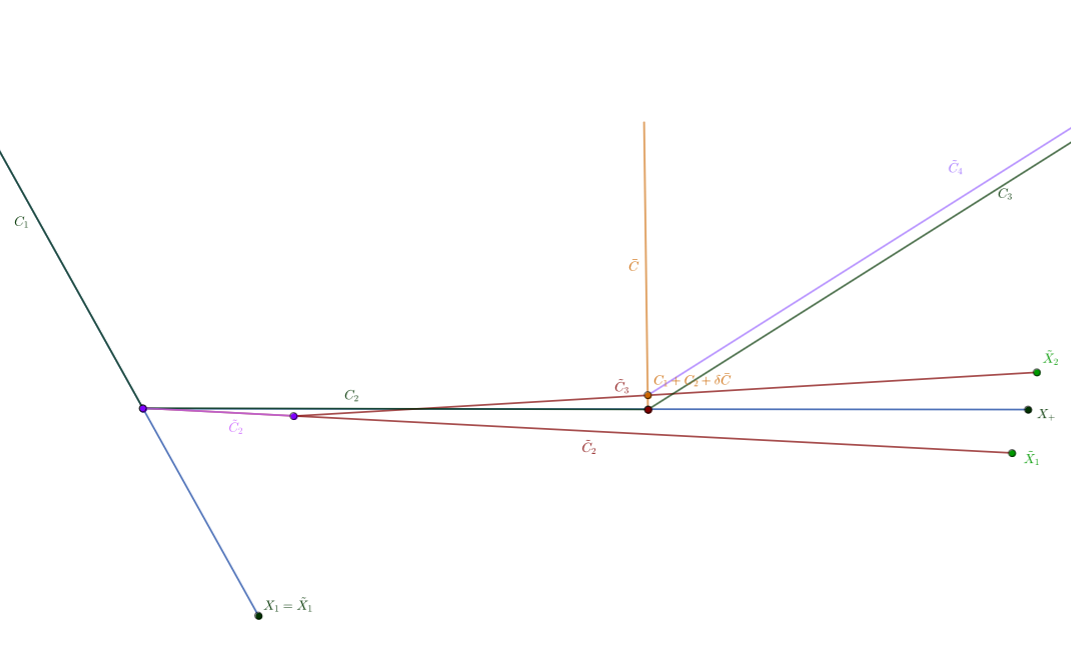}
 \caption{\label{fig:TNext1} The construction for \eqref{eq:DQ:Cnt1}: We replace $C^{2}+\delta \overline{C}_\delta$ with $\tilde{C}^2+\tilde{C}^3$ which adds one element to the $T_{\tnn+1}$-configuration but with only slightly shifted endpoints. At some other point we do the same with $-\delta \overline{C}_\delta$ to cancel the global effect.}
\end{figure}

We apply \Cref{la:onerankdecMxn} for $\lambda^{\tni_1} \coloneqq  \kappa^{\tni_1} - \sqrt{\kappa^{\tni_1}-\kappa^{\tni_1}} \in (0,1)$  and obtain
\begin{equation}\label{eq:DQ:Cnt1}
 C^{\tni_1}+\delta \overline{C}_\delta =: \tilde{C}^{\tni_1}+\tilde{C}^{\tni_1+1},
\end{equation}
and $\tilde{\kappa}_{\tni_1+1}$ and $\tilde{\kappa}_{\tni_1+1}$.

Denote
\[
\ul{p}^{\tni} \coloneqq  \underline{\overline{b}}-\langle \underline{\overline{b}} , \underline{b}^{\tni}\rangle \underline{b}^{\tni},
\]
then we have the asymptotic expansion
\begin{equation}\label{eq:DQ:ass1}
\begin{split}
&\underline{\tilde{C}}^{\tni_1} - \lambda^{\tni_1} \underline{C}^{\tni_1}\\
=& \sqrt{\delta}\brac{\lambda^{\tni_1} \sum_{\alpha=1}^n \ul{p}^{\tni_1}_\alpha dx^\alpha \wedge \underline{a}^{\tni_1}+\frac{1}{2}\frac{1-\lambda^{\tni_1}}{ 1-\lambda^{\tni_1}+ \brac{\lambda^{\tni_1}}^2} \langle\ul{\overline{b}}, \frac{\ul{p}^{\tni_1}}{|\ul{p}^{\tni_1}|^2}\rangle \sum_{\alpha=1}^n\ul{b}^{\tni_1}_\alpha dx^\alpha \wedge \underline{\overline{a}} } + o(\sqrt{\delta})\\
 \end{split}
\end{equation}
%

\begin{equation}\label{eq:DQ:ass2}
\begin{split}
&\underline{\tilde{C}}^{\tni_1+1} - (1-\lambda^{\tni_1}) \underline{C}^{\tni_1}\\
=& -\sqrt{\delta}\brac{\lambda^{\tni_1} \sum_{\alpha=1}^n \ul{p}^{\tni_1}_\alpha dx^\alpha \wedge \underline{a}^{\tni_1}+\frac{1}{2}\frac{1-\lambda^{\tni_1}}{ 1-\lambda^{\tni_1}+ \brac{\lambda^{\tni_1}}^2} \langle\ul{\overline{b}}, \frac{\ul{p}^{\tni_1}}{|\ul{p}^{\tni_1}|^2}\rangle \sum_{\alpha=1}^n\ul{b}^{\tni_1}_\alpha dx^\alpha \wedge \underline{\overline{a}} } + o(\sqrt{\delta})\\
 \end{split}
\end{equation}

Similarly, from \Cref{la:onerankdecMxn} for $\lambda^{\tni_2} \coloneqq  \kappa^{\tni_2} - \sqrt{\kappa^{\tni_2}-\kappa^{\tni_2}} \in (0,1)$ we obtain $\tilde{\kappa}_{\tni_2+1}, \tilde{\kappa}_{\tni_1+2} > 1$ and a decomposition
\begin{equation}\label{eq:DQ:Cnt2}
 \ul{C}^{\tni_2}-\delta \ul{\overline{C}}_\delta =: \ul{\tilde{C}}^{\tni_2+1}+\ul{\tilde{C}}^{\tni_2+2}.
\end{equation}
Observe that this is like the $\tni_1$-case, just taking $-\ul{\overline{a}}$ because of the different sign of the $-\delta \ul{\overline{C}}_\delta$. Then we have the asymptotic expansion
\[
\begin{split}
&\underline{\tilde{C}}^{\tni_2+1} - \lambda^{\tni_2} \underline{C}^{\tni_2}\\
=& \sqrt{\delta}\brac{\lambda^{\tni_2} \sum_{\alpha=1}^n \ul{p}^{\tni_2}_\alpha dx^\alpha \wedge \underline{a}^{\tni_2}{-}\frac{1}{2}\frac{1-\lambda^{\tni_2}}{ 1-\lambda^{\tni_2}+ \brac{\lambda^{\tni_2}}^2} \langle\ul{\overline{b}}, \frac{\ul{p}^{\tni_2}}{|\ul{p}^{\tni_2}|^2}\rangle \sum_{\alpha=1}^n\ul{b}^{\tni_2}_\alpha dx^\alpha \wedge \underline{\overline{a}} } + o(\sqrt{\delta})\\
 \end{split}
\]
and
\[
\begin{split}
&\underline{\tilde{C}}^{\tni_2+2} - (1-\lambda^{\tni_2}) \underline{C}^{\tni_2}\\
=& -\sqrt{\delta}\brac{\lambda^{\tni_2} \sum_{\alpha=1}^n \ul{p}^{\tni_2}_\alpha dx^\alpha \wedge \underline{a}^{\tni_2}{-}\frac{1}{2}\frac{1-\lambda^{\tni_2}}{ 1-\lambda^{\tni_2}+ \brac{\lambda^{\tni_2}}^2} \langle\ul{\overline{b}}, \frac{\ul{p}^{\tni_2}}{|\ul{p}^{\tni_2}|^2}\rangle \sum_{\alpha=1}^n\ul{b}^{\tni_2}_\alpha dx^\alpha \wedge \underline{\overline{a}} } + o(\sqrt{\delta})\\
 \end{split}
\]
We define the new $T_{\tnn+2}$-configuration.
\[
\begin{split}
 \ul{\tilde{C}}^\tni =& \ul{C}^\tni,\ \tilde{\kappa}^\tni = \kappa^\tni\\
 \ul{\tilde{C}}^{\tni_1}=&\text{above} ,\ \tilde{\kappa}^{\tni_1} = \text{above}\\
 \ul{\tilde{C}}^{\tni_1+1}=&\text{above} ,\ \tilde{\kappa}^{\tni_1+1} = \text{above}\\
 \ul{\tilde{C}}^{\tni+1}=& C^{\tni},\ \tilde{\kappa}^{\tni +1}= \kappa^\tni\quad \tni_1+1 \leq \tni \leq \tni_2\\
 \ul{\tilde{C}}^{\tni_2+1}=& \text{above} ,\ \tilde{\kappa}^{\tni_2+1} = \text{above}\\
 \ul{\tilde{C}}^{\tni_2+2}=& \text{above} ,\ \tilde{\kappa}^{\tni_2+2} = \text{above}\\
 \ul{\tilde{C}}^{\tni+2}=& C^{\tni} \quad \tilde{\kappa}^{\tni +2}= \kappa^\tni \quad \tni \geq \tni_2+1
\end{split}
 \]
By construction, globally the $\delta \ul{\overline{C}}_\delta$-distortion cancels and we have
\[
 \sum_{\tni=1}^{\tnn+2} \ul{\tilde{C}}^\tni =  \sum_{\tni=1}^{\tnn} C^\tni  =0.
\]
\begin{itemize}
 \item Neighboring $(\ul{\tilde{a}}'^\tni, \ul{\tilde{a}}'^{\tni+1})$ are still linearly independent in $\R^M$: the only places where we replaced them they are very close to being collinear with $\ul{a}^{\tni_1}$, $\ul{a}^{\tni_2}$, respectively -- which are linearly independent with their respective predecessor and successor. And the two neighboring vectors are linearly independent by \Cref{la:onerankdecMxn}.
 \item Similarly, the neighboring $\ul{\tilde{b}}^\tni$ are still pairwise linearly independent.
 \item By \eqref{eq:DQMnspantb1tb2}
\[
\span\{\ul{\tilde{b}}_1,\ldots,\ul{\tilde{b}}_{\tnn+2}\} = \span\{\ul{b}_1,\ldots,\ul{b}_{\tnn}\}+\span \{\ul{p}^{\tni_1}+\delta^2 \ul{\xi}\}+\span \{\ul{p}^{\tni_2}+\delta^2 \ul{\xi}\}
\]
If $\dim \span \{\ul{b}_1,\ldots,\ul{b}_{\tnn}\} < n$ we chose $\xi \perp \span \dim \{\ul{b}_1,\ldots,\ul{b}_{\tnn}\}$, for suitably small $\delta$
\[
\span\{\ul{b}_1,\ldots,\ul{b}_{\tnn}\} \cap \span \{\ul{p}^{\tni_1}+\delta^2 \ul{\xi}\} = \{0\}
\]
and thus
\[
\dim \span\{\ul{\tilde{b}}_1,\ldots,\ul{\tilde{b}}_{\tnn+2}\} \geq \dim \span\{\ul{b}_1,\ldots,\ul{b}_{\tnn}\}+1.
\]
\end{itemize}

It remains to show that the new $T_{\tnn+2}$-configuration still satisfies \eqref{eq:szineq2x4vDQ}, \eqref{eq:DC:condYXV1}, \eqref{eq:DC:condYXV2}.

By constructions, the endpoints of our new $T_{\tnn+2}$-configuration did not change too much:
\begin{equation}\label{eq:DQ:ZZ}
 \begin{split}
\ul{\tilde{Z}}^{\tni_1}-\ul{Z}^{\tni_1}=&   \tilde{\kappa}^{\tni_1} \ul{\tilde{C}}^{\tni_1} -  \kappa^{\tni} \ul{C}^{\tni_1}\\
\ul{\tilde{Z}}^{\tni_1+1}-\ul{Z}^{\tni_1}=& \ul{\tilde{C}}^{\tni_1} + \tilde{\kappa}^{\tni_1+1}\ul{\tilde{C}}^{\tni_1+1}-  \kappa^{\tni} \ul{C}^{\tni_1}\\
\ul{\tilde{Z}}^{\tni_2+1}-\ul{Z}^{\tni_2}=&\delta \ul{\overline{C}}_\delta + \tilde{\kappa}^{\tni_2+1} \ul{\tilde{C}}^{i_2+1} - \kappa^{\tni_2} \ul{C}^{\tni_2}\\
\ul{\tilde{Z}}^{\tni_2+2}-\ul{Z}^{\tni_2}=&\delta \ul{\overline{C}}_\delta +\ul{\tilde{C}}^{i_2+1}+ \tilde{\kappa}^{\tni_2+1} \ul{\tilde{C}}^{i_2+2}- \kappa_{\tni_2} \ul{C}^{\tni_2}
 \end{split}
\end{equation}
Thus by the construction, \eqref{eq:DQMn:kappaCkappa1Ct},
\[
| \ul{\tilde{Z}}^{\tni_1}-\ul{Z}^{\tni}|+| \ul{\tilde{Z}}^{\tni_1+1}-\ul{Z}^{\tni_1}|+|\ul{\tilde{Z}}^{\tni_1+1}-\ul{\tilde{Z}}^{\tni_1}| \xrightarrow{\delta \to 0} 0.
\]
\[
| \ul{\tilde{Z}}^{\tni_2+1}-\ul{Z}^{\tni_2}|+| \ul{\tilde{Z}}^{\tni_2+2}-\ul{Z}^{\tni_2}|+|\ul{\tilde{Z}}^{\tni_2+2}-\ul{\tilde{Z}}^{\tni_2+1}| \xrightarrow{\delta \to 0} 0.
\]

All other $\ul{\tilde{Z}}^\tni$ differ by at most $\delta \ul{\overline{C}_\delta}$  from $\ul{Z}^{\tni}$, $\ul{Z}^{\tni-1}$ or $\ul{Z}^{\tni-2}$.

Set
\[
\gamma_0 \coloneqq  -\max_{\tni\neq \tnj \in \{1,\ldots,\tnn\}} c^{\tni} - c^{\tnj} + \langle \hat{Y}^\tni, {X}^\tnj - {X}^\tni\rangle +d^\tni  \det_{2 \times 2}({X}^{\tnj}_{(1,2),(1,2)}-{X}^{\tni}_{(1,2),(1,2)})
\]
and observe that by \eqref{eq:szineqmxnvDQ} we have $\gamma_0 > 0$.

Fix some $\gamma_1,\gamma_2 \in (-\frac{\gamma_0}{2},\frac{\gamma_0}{2})$, then we can define our new variables
\[
\begin{split}
 \ul{\tilde{Z}}^\tni =& \ul{Z}^\tni,\ \tilde{c}^\tni\coloneqq c^\tni,\ \tilde{d}^\tni \coloneqq  d^\tni  \quad \tni \leq \tni_1-1\\
|\ul{\tilde{Z}}^{\tni_1}-\ul{Z}^{\tni_1}| \ll& 1,\ \tilde{c}^{\tni_1}\coloneqq c^{\tni_1},\ \tilde{d}^{\tni_1} \coloneqq  d^{\tni_1}\\
|\ul{\tilde{Z}}^{\tni_1+1}-\ul{Z}^{\tni_1}| \ll& 1,\ \tilde{c}^{\tni_1+1}\coloneqq c^{\tni_1}+\gamma_1,\ \tilde{d}^{\tni_1+1} \coloneqq  d^{\tni_1}\\
|\ul{\tilde{Z}}^{\tni+1}-\ul{Z}^{\tni}| \ll& 1  ,\ \tilde{c}^{\tni+1}\coloneqq c^\tni,\ \tilde{d}^{\tni +1}\coloneqq  d^\tni  \quad \tni_1 <\tni \leq \tni_2-1\\
|\ul{\tilde{Z}}^{\tni_2+1}-\ul{Z}^{\tni_2}| \ll& 1,\ \tilde{c}^{\tni_2+1}\coloneqq c^{\tni_2},\ \tilde{d}^{\tni_2+1} \coloneqq  d^{\tni_2} \\
|\ul{\tilde{Z}}^{\tni_2+2}-\ul{Z}^{\tni_2}| \ll& 1,\ \tilde{c}^{\tni_2+2}\coloneqq c^{\tni_2}+\gamma_2,\ \tilde{d}^{\tni_2+2} \coloneqq  d^{\tni_2} \\
|\ul{\tilde{Z}}^{\tni+2}-\ul{Z}^{\tni}| \ll& 1 ,\ \tilde{c}^{\tni+2}\coloneqq c^\tni,\ \tilde{d}^{\tni +2}\coloneqq  d^\tni \quad \tni \geq \tni_2+1\\
\end{split}
 \]

For $\delta$ small enough (depending on $\gamma_0$ and $\abs{\ul{\overline{a}}}$, $\abs{\ul{\overline{b}}}$, $\abs{\ul{\overline{\xi}}}$), we then have
\[
0 >  \tilde{c}^{\tni} - \tilde{c}^{\tnj} + \langle \hat{\ul{\tilde{Y}}}^\tni, \ul{\tilde{X}}^\tnj - \ul{\tilde{X}}^\tni\rangle +\tilde{d}^\tni  \det_{2 \times 2}(\tilde{X}^{\tnj}_{(1,2),(1,2)}-{X}^{\tni}_{(1,2),(1,2)})
\]
for all $\tni\neq \tnj \in \{1,\ldots,\tnn\}$ except when $\{\tni,\tnj\} = \{\tni_1,\tni_1+1\}$ or $\{\tni,\tnj\} = \{\tni_2+1,\tni_2+2\}$ . (We drop the underline for $\ul{X}$ if we only consider $2\times2$-matrices).

We also observe, that for $\delta$ suitably small we still have \eqref{eq:DC:condYXV1} and \eqref{eq:DC:condYXV2}, e.g. for $\tni_1$ and $\tni_2+1$.

That is, in order to conclude it remains to show (observe that $\tilde{c}^{i_1+1}=\tilde{c}^{i_1} +\gamma$ etc.). We now drop the subscript in $\det_{2\times2}$ but it is always applied to the top-left square submatrix.
\[
0 > -\gamma_1+ \langle \widehat{\ul{\tilde{Y}}^{\tni_1}} , \ul{\tilde{X}}^{\tni_1+1} - \ul{\tilde{X}}^{\tni_1}\rangle + d^{\tni_1}  \det_{2\times 2}(\tilde{X}^{\tni_1+1} - \tilde{X}^{\tni_1})
\]
\[
0 > \gamma_1 + \langle \widehat{\ul{\tilde{Y}}^{\tni_1+1}} , \ul{\tilde{X}}^{\tni_1} - \ul{\tilde{X}}^{\tni_1+1}\rangle + d^{\tni_1}  \det_{2\times 2}({\tilde{X}}^{\tni_1} - {\tilde{X}}^{\tni_1+1})
\]
\[
0 > -\gamma_2 +\langle \widehat{\ul{\tilde{Y}}^{\tni_2+1}} , \ul{\tilde{X}}^{\tni_2+2} - \ul{\tilde{X}}^{\tni_2+1}\rangle + d^{\tni_2}  \det_{2\times 2}({\tilde{X}}^{\tni_2+2} - {\tilde{X}}^{\tni_2+1})
\]
\[
0 >  \gamma_2 +\langle \widehat{\ul{\tilde{Y}}^{\tni_2+2}} , \ul{\tilde{X}}^{\tni_2+1} - \ul{\tilde{X}}^{\tni_2+2}\rangle + d^{\tni_2}  \det_{2\times 2}({\tilde{X}}^{\tni_2+1} -{\tilde{X}}^{\tni_2+2})
\]
Equivalently,
\[
-\langle \widehat{\ul{\tilde{Y}}^{\tni_1+1}}, \ul{\tilde{X}}^{\tni_1} - \ul{\tilde{X}}^{\tni_1+1}\rangle - d^{\tni_1}  \det_{2\times 2}(\tilde{X}^{\tni_1} - \tilde{X}^{\tni_1+1}) > \gamma_1 >  \langle \widehat{\ul{\tilde{Y}}^{\tni_1}}, \ul{\tilde{X}}^{\tni_1+1} - \ul{\tilde{X}}^{\tni_1}\rangle + d^{\tni_1}  \det_{2\times 2}(\tilde{X}^{\tni_1+1} - \tilde{X}^{\tni_1})
\]
\[
-\langle \widehat{\ul{\tilde{Y}}^{\tni_2+2}}, \ul{\tilde{X}}^{\tni_2+1} - \ul{\tilde{X}}^{\tni_2+2}\rangle - d^{\tni_2}  \det_{2\times 2}(\tilde{X}^{\tni_2+1} - \tilde{X}^{\tni_2+2})> \gamma_2 >  \langle \widehat{\ul{\tilde{Y}}^{\tni_2+1}}, \ul{\tilde{X}}^{\tni_2+2} - \ul{\tilde{X}}^{\tni_2+1}\rangle + d^{\tni_2}  \det_{2\times 2}(\tilde{X}^{\tni_2+2} - \tilde{X}^{\tni_2+1})
\]
For both inequalities, it suffices to show
\[
-\langle \widehat{\ul{\tilde{Y}}^{\tni_1+1}}, \ul{\tilde{X}}^{\tni_1} - \ul{\tilde{X}}^{\tni_1+1}\rangle - d^{\tni_1}  \det_{2\times 2}(\tilde{X}^{\tni_1} - \tilde{X}^{\tni_1+1}) >  \langle \widehat{\ul{\tilde{Y}}^{\tni_1}}, \ul{\tilde{X}}^{\tni_1+1} - \ul{\tilde{X}}^{\tni_1}\rangle + d^{\tni_1}  \det_{2\times 2}(\tilde{X}^{\tni_1+1} - \tilde{X}^{\tni_1})
\]
\[
-\langle \widehat{\ul{\tilde{Y}}^{\tni_2+2}}, \ul{\tilde{X}}^{\tni_2+1} - \ul{\tilde{X}}^{\tni_2+2}\rangle - d^{\tni_2}  \det_{2\times 2}(\tilde{X}^{\tni_2+1} - \tilde{X}^{\tni_2+2}) >  \langle \widehat{\ul{\tilde{Y}}^{\tni_2+1}}, \ul{\tilde{X}}^{\tni_2+2} - \ul{\tilde{X}}^{\tni_2+1}\rangle + d^{\tni_2}  \det_{2\times 2}(\tilde{X}^{\tni_2+2} - \tilde{X}^{\tni_2+1})
\]
Indeed, if we had shown the above inequalities there must be some $\gamma_1, \gamma_2 \in \R$ such that
\[
-\langle \widehat{\ul{\tilde{Y}}^{\tni_1+1}}, \ul{\tilde{X}}^{\tni_1} - \ul{\tilde{X}}^{\tni_1+1}\rangle - d^{\tni_1}  \det_{2\times 2}(\tilde{X}^{\tni_1} - \tilde{X}^{\tni_1+1}) >  \gamma_1 > \langle \widehat{\ul{\tilde{Y}}^{\tni_1}}, \ul{\tilde{X}}^{\tni_1+1} - \ul{\tilde{X}}^{\tni_1}\rangle + d^{\tni_1}  \det_{2\times 2}(\tilde{X}^{\tni_1+1} - \tilde{X}^{\tni_1})
\]
\[
-\langle \widehat{\ul{\tilde{Y}}^{\tni_2+2}}, \ul{\tilde{X}}^{\tni_2+1} - \ul{\tilde{X}}^{\tni_2+2}\rangle - d^{\tni_2}  \det_{2\times 2}(\tilde{X}^{\tni_2+1} - \tilde{X}^{\tni_2+2}) >  \gamma_2 > \langle \widehat{\ul{\tilde{Y}}^{\tni_2+1}}, \ul{\tilde{X}}^{\tni_2+2} -\ul{\tilde{Y}} \ul{\tilde{X}}^{\tni_2+1}\rangle + d^{\tni_2}  \det_{2\times 2}(\tilde{X}^{\tni_2+2} - \tilde{X}^{\tni_2+1})
\]
and since $|\ul{\tilde{Z}}^{\tni_1+1} -\ul{\tilde{Z}}^{\tni_1}|\ll 1$ and $|\ul{\tilde{Z}}^{\tni_2+2}-\ul{\tilde{Z}}^{\tni_2+1}| \ll 1$ for suitably small $\delta$ we may assume that $|\gamma_1| + |\gamma_2| < \frac{\gamma_0}{2}$.

That is, we need to show (the square-determinants are symmetric)
\begin{equation}\label{eq:condYXaskljdDQtni1}
\langle \widehat{\ul{\tilde{Y}}^{\tni_1}} -\widehat{\ul{\tilde{Y}}^{\tni_1+1}} , \ul{\tilde{X}}^{\tni_1} - \ul{\tilde{X}}^{\tni_1+1}\rangle >   2d^{\tni_1}  \det_{2\times 2}(\tilde{X}^{\tni_1+1} - \tilde{X}^{\tni_1})
\end{equation}
and
\begin{equation}\label{eq:condYXaskljdDQtni2}
\langle \widehat{\ul{\tilde{Y}}^{\tni_2+1}} -\widehat{\ul{\tilde{Y}}^{\tni_2+2}}, \ul{\tilde{X}}^{\tni_2+1} - \ul{\tilde{X}}^{\tni_2+2}\rangle  >  2d^{\tni_2}  \det_{2\times 2}(\tilde{X}^{\tni_2+2} - \tilde{X}^{\tni_2+1})
\end{equation}
From \eqref{eq:DQ:ZZ}
\[
 \begin{split}
\ul{\tilde{Z}}^{\tni_1}-\ul{\tilde{Z}}^{\tni_1+1}=&   \brac{\tilde{\kappa}^{\tni_1}-1} \ul{\tilde{C}}^{\tni_1} - \tilde{\kappa}^{\tni_1+1}\ul{\tilde{C}}^{\tni_1+1}\\
\overset{\eqref{eq:DQkappalpambda}}{=}&\brac{\tilde{\kappa}^{\tni_1}-1} \brac{\ul{\tilde{C}}^{\tni_1}-\lambda^{\tni_1} \ul{C}^{\tni_1}}  - \tilde{\kappa}^{\tni_1+1}\brac{\ul{\tilde{C}}^{\tni_1+1}-(1-\lambda^{i_1}) \ul{C}^{\tni_1}} \end{split}
\]
In view of \eqref{eq:DQ:ass1}, \eqref{eq:DQ:ass2}
\[
\begin{split}
&\ul{\tilde{Z}}^{\tni_1}-\ul{\tilde{Z}}^{\tni_1+1}\\
 =&\sqrt{\delta}\brac{\tilde{\kappa}^{\tni_1}+\tilde{\kappa}^{\tni_1+1}-1} \brac{
 \lambda^{\tni_1} \sum_{\beta=1}^n \ul{p}^{\tni_1}_\beta dx^\beta \wedge \underline{a}^{\tni_1}+\frac{1}{2}\frac{1-\lambda^{\tni_1}}{ 1-\lambda^{\tni_1}+ \brac{\lambda^{\tni_1}}^2} \langle\ul{\overline{b}}, \frac{\ul{p}^{\tni_1}}{|\ul{p}^{\tni_1}|^2}\rangle \sum_{\beta=1}^n\ul{b}^{\tni_1}_\beta dx^\beta \wedge \underline{\overline{a}}
 } + o(\sqrt{\delta})\\
 \end{split}
\]
By the identifications \eqref{eq:hatvsbarYv2v3}
\[
\begin{split}
&\brac{\widehat{\ul{\tilde{Y}}^{\tni_1}}-\widehat{\ul{\tilde{Y}}^{\tni_1+1}}}_{\beta}\\
 =&\sqrt{\delta}\brac{\tilde{\kappa}^{\tni_1}+\tilde{\kappa}^{\tni_1+1}-1} \Big (
 \lambda^{\tni_1} \sum_{\beta=1}^n \ul{p}^{\tni_1}_\beta \hdg \brac{dx^\alpha \wedge dx^\beta \wedge \underline{a}''^{\tni_1}}\\
&\quad  +\frac{1}{2}\frac{1-\lambda^{\tni_1}}{ 1-\lambda^{\tni_1}+ \brac{\lambda^{\tni_1}}^2} \langle\ul{\overline{b}}, \frac{\ul{p}^{\tni_1}}{|\ul{p}^{\tni_1}|^2}\rangle \sum_{\beta=1}^n\ul{b}^{\tni_1}_\beta \hdg \brac{dx^\alpha \wedge dx^\beta \wedge \underline{\overline{a}}''}\Big )
 \\
 &+ o(\sqrt{\delta})\\
 \end{split}
\]
We can write
\[
 \underline{a}''^{\tni_1}= \underbrace{\underline{a}''^{\tni_1}_{\alpha \beta}}_{\in \R^M} \hdg \brac{dx^{\alpha} \wedge dx^\beta}
\]
Then
\[
\begin{split}
\R^M\ni &\brac{\widehat{\ul{\tilde{Y}}^{\tni_1}}-\widehat{\ul{\tilde{Y}}^{\tni_1+1}}}_{\alpha}\\
 =&\sqrt{\delta}\brac{\tilde{\kappa}^{\tni_1}+\tilde{\kappa}^{\tni_1+1}-1} \big (
 \lambda^{\tni_1} \sum_{\beta=1}^n \ul{p}^{\tni_1}_\beta \underbrace{\underline{a}''^{\tni_1}_{\alpha \beta}}_{\in \R^M}+\frac{1}{2}\frac{1-\lambda^{\tni_1}}{ 1-\lambda^{\tni_1}+ \brac{\lambda^{\tni_1}}^2} \langle\ul{\overline{b}}, \frac{\ul{p}^{\tni_1}}{|\ul{p}^{\tni_1}|^2}\rangle \sum_{\beta=1}^n\ul{b}^{\tni_1}_\beta \underbrace{\underline{\overline{a}}_{\alpha \beta}''}_{\in \R^M} \big )
  + o(\sqrt{\delta})\\
 \end{split}
\]
With \eqref{eq:hatvsbarYv2v3X},
\begin{equation}\label{eq:DQMn:Xi1mXi2}
\begin{split}
\R^M \ni &\brac{\ul{\tilde{X}}^{\tni_1}-\ul{\tilde{X}}^{\tni_1+1}}_{\alpha}\\
 =&\sqrt{\delta}\brac{\tilde{\kappa}^{\tni_1}+\tilde{\kappa}^{\tni_1+1}-1} \big (
 \lambda^{\tni_1} \ul{p}^{\tni_1}_\alpha \underbrace{\underline{a}'^{\tni_1}}_{\in \R^M}+\frac{1}{2}\frac{1-\lambda^{\tni_1}}{ 1-\lambda^{\tni_1}+ \brac{\lambda^{\tni_1}}^2} \langle\ul{\overline{b}}, \frac{\ul{p}^{\tni_1}}{|\ul{p}^{\tni_1}|^2}\rangle \ul{b}^{\tni_1}_\alpha \underbrace{\underline{\overline{a}}'}_{\in \R^M}
\big ) + o(\sqrt{\delta})\\
 \end{split}
\end{equation}
Observe that $a''_{\alpha \beta}$ are necessarily antisymmetric, thus we have cancellations: $\sum_{\alpha,\beta}\ul{p}^{\tni_1}_\beta \underline{a}''^{\tni_1}_{\alpha \beta}\ul{p}^{\tni_1}_\alpha =0$, and $\sum_{\alpha,\beta} \ul{b}^{\tni_1}_\beta \underline{\overline{a}}_{\alpha \beta}'' \ul{b}^{\tni_1}_\alpha  = 0$
Thus
\[
\begin{split}
& \langle \widehat{\ul{\tilde{Y}}^{\tni_1}}-\widehat{\ul{\tilde{Y}}^{\tni_1+1}}, {\tilde{X}^{\tni_1}-\tilde{X}^{\tni_1+1}} \rangle_{\R^{M \times n}}\\
=&\delta\brac{\tilde{\kappa}^{\tni_1}+\tilde{\kappa}^{\tni_1+1}-1}^2 \cdot \\
\quad &\cdot \sum_{\alpha,\beta=1}^n\big \langle
 \lambda^{\tni_1} \ul{p}^{\tni_1}_\beta \underline{a}''^{\tni_1}_{\alpha \beta}+\frac{1}{2}\frac{1-\lambda^{\tni_1}}{ 1-\lambda^{\tni_1}+ \brac{\lambda^{\tni_1}}^2} \langle\ul{\overline{b}}, \frac{\ul{p}^{\tni_1}}{|\ul{p}^{\tni_1}|^2}\rangle \ul{b}^{\tni_1}_\beta \underline{\overline{a}}_{\alpha \beta}'' ,\\
 & \quad \quad \quad \quad \lambda^{\tni_1} \ul{p}^{\tni_1}_\alpha \underline{a}'^{\tni_1}+\frac{1}{2}\frac{1-\lambda^{\tni_1}}{ 1-\lambda^{\tni_1}+ \brac{\lambda^{\tni_1}}^2} \langle\ul{\overline{b}}, \frac{\ul{p}^{\tni_1}}{|\ul{p}^{\tni_1}|^2}\rangle \ul{b}^{\tni_1}_\alpha \underline{\overline{a}}' \Big \rangle_{\R^M}+o(\delta)\\
=&\delta\brac{\tilde{\kappa}^{\tni_1}+\tilde{\kappa}^{\tni_1+1}-1}^2
 \lambda^{\tni_1} \frac{1}{2}\frac{1-\lambda^{\tni_1}}{ 1-\lambda^{\tni_1} +\brac{\lambda^{\tni_1}}^2}\langle\ul{\overline{b}}, \frac{\ul{p}^{\tni_1}}{|\ul{p}^{\tni_1}|^2}\rangle
\brac{
\sum_{\alpha,\beta=1}^n\langle \underline{a}''^{\tni_1}_{\alpha \beta},\underline{\overline{a}}' \rangle_{\R^M} \ul{p}^{\tni_1}_\beta \ul{b}^{\tni_1}_\alpha +
\sum_{\alpha,\beta=1}^n \langle \underline{\overline{a}}_{\alpha \beta}'' ,\underline{a}'^{\tni_1}\rangle_{\R^M} \ul{p}^{\tni_1}_\alpha \ul{b}^{\tni_1}_\beta }\\
&+o(\delta)\\
=&\delta\brac{\tilde{\kappa}^{\tni_1}+\tilde{\kappa}^{\tni_1+1}-1}^2
 \lambda^{\tni_1} \frac{1}{2}\frac{1-\lambda^{\tni_1}}{ 1-\lambda^{\tni_1} +\brac{\lambda^{\tni_1}}^2}\langle\ul{\overline{b}}, \frac{\ul{p}^{\tni_1}}{|\ul{p}^{\tni_1}|^2}\rangle
\brac{
\sum_{\alpha,\beta=1}^n\brac{\langle \underline{a}''^{\tni_1}_{\alpha \beta},\underline{\overline{a}}' \rangle_{\R^M} -
\sum_{\alpha,\beta=1}^n \langle \underline{\overline{a}}_{\alpha \beta}'' ,\underline{a}'^{\tni_1}\rangle_{\R^M}} \ul{p}^{\tni_1}_\beta \ul{b}^{\tni_1}_\alpha }\\
&+o(\delta).
 \end{split}
\]

Thus
\[
\begin{split}
& \langle \widehat{\ul{\tilde{Y}}^{\tni_1}}-\widehat{\ul{\tilde{Y}}^{\tni_1+1}}, {\tilde{X}^{\tni_1}-\tilde{X}^{\tni_1+1}} \rangle_{\R^{M \times n}}\\
=&\delta\brac{\tilde{\kappa}^{\tni_1}+\tilde{\kappa}^{\tni_1+1}-1}^2 \cdot \\
\quad &\cdot \sum_{\alpha,\beta=1}^n\big \langle
 \lambda^{\tni_1} \ul{p}^{\tni_1}_\beta \underline{a}''^{\tni_1}_{\alpha \beta}+\frac{1}{2}\frac{1-\lambda^{\tni_1}}{ 1-\lambda^{\tni_1}+ \brac{\lambda^{\tni_1}}^2} \langle\ul{\overline{b}}, \frac{\ul{p}^{\tni_1}}{|\ul{p}^{\tni_1}|^2}\rangle \ul{b}^{\tni_1}_\beta \underline{\overline{a}}_{\alpha \beta}'' ,\\
 & \quad \quad \quad \quad \lambda^{\tni_1} \ul{p}^{\tni_1}_\alpha \underline{a}'^{\tni_1}+\frac{1}{2}\frac{1-\lambda^{\tni_1}}{ 1-\lambda^{\tni_1}+ \brac{\lambda^{\tni_1}}^2} \langle\ul{\overline{b}}, \frac{\ul{p}^{\tni_1}}{|\ul{p}^{\tni_1}|^2}\rangle \ul{b}^{\tni_1}_\alpha \underline{\overline{a}}' \Big \rangle_{\R^M}+o(\delta)\\
=&\delta\brac{\tilde{\kappa}^{\tni_1}+\tilde{\kappa}^{\tni_1+1}-1}^2
 \lambda^{\tni_1} \frac{1}{2}\frac{1-\lambda^{\tni_1}}{ 1-\lambda^{\tni_1} +\brac{\lambda^{\tni_1}}^2}\langle\ul{\overline{b}}, \frac{\ul{p}^{\tni_1}}{|\ul{p}^{\tni_1}|^2}\rangle
\brac{
\sum_{\alpha,\beta=1}^n\langle \underline{a}''^{\tni_1}_{\alpha \beta},\underline{\overline{a}}' \rangle_{\R^M} \ul{p}^{\tni_1}_\beta \ul{b}^{\tni_1}_\alpha +
\sum_{\alpha,\beta=1}^n \langle \underline{\overline{a}}_{\alpha \beta}'' ,\underline{a}'^{\tni_1}\rangle_{\R^M} \ul{p}^{\tni_1}_\alpha \ul{b}^{\tni_1}_\beta }\\
&+o(\delta)\\
=&\delta\brac{\tilde{\kappa}^{\tni_1}+\tilde{\kappa}^{\tni_1+1}-1}^2
 \lambda^{\tni_1} \frac{1}{2}\frac{1-\lambda^{\tni_1}}{ 1-\lambda^{\tni_1} +\brac{\lambda^{\tni_1}}^2}\langle\ul{\overline{b}}, \frac{\ul{p}^{\tni_1}}{|\ul{p}^{\tni_1}|^2}\rangle
\brac{
\sum_{\alpha,\beta=1}^n\brac{\langle \underline{a}''^{\tni_1}_{\alpha \beta},\underline{\overline{a}}' \rangle_{\R^M} -
\langle \underline{\overline{a}}_{\alpha \beta}'' ,\underline{a}'^{\tni_1}\rangle_{\R^M}} \ul{p}^{\tni_1}_\beta \ul{b}^{\tni_1}_\alpha }\\
&+o(\delta).
 \end{split}
\]
Similarly (recall that this case is just taking $-\ul{\overline{a}}$ instead of $\ul{\overline{a}}$)
\[
\begin{split}
& \langle \widehat{\ul{\tilde{Y}}^{\tni_2+1}}-\widehat{\ul{\tilde{Y}}^{\tni_2+2}}, {\tilde{X}^{\tni_2+1}-\tilde{X}^{\tni_2+2}} \rangle_{\R^{M \times n}}\\
=&{-}\delta\brac{\tilde{\kappa}^{\tni_2+1}+\tilde{\kappa}^{\tni_2+2}-1}^2
 \lambda^{\tni_2} \frac{1}{2}\frac{1-\lambda^{\tni_2}}{ 1-\lambda^{\tni_2} \brac{\lambda^{\tni_2}}^2}\langle\ul{\overline{b}}, \frac{\ul{p}^{\tni_2}}{|\ul{p}^{\tni_2}|^2}\rangle
\brac{
\sum_{\alpha,\beta=1}^n\brac{\langle \underline{a}''^{\tni_2}_{\alpha \beta},\underline{\overline{a}}' \rangle_{\R^M} -
\langle \underline{\overline{a}}_{\alpha \beta}'' ,\underline{a}'^{\tni_2}\rangle_{\R^M}} \ul{p}^{\tni_2}_\beta \ul{b}^{\tni_2}_\alpha }\\
&+o(\delta).
 \end{split}
\]

Next, we go for the $2\times 2$ determinant term:
For a vector $\ul{v} \in \R^M$ or $\ul{v} \in \R^n$ we call the restriction to the first two entries $v$. Namely,
\[
 a' = \left ( \begin{array}{c} \ul{a}'_1\\
 \ul{a}'_2\\ \end{array}
 \right ), \quad b = \left ( \begin{array}{c}\ul{b}_1\\\ul{b}_2\end{array}\right ),\quad  p = \left ( \begin{array}{c}\ul{p}_1\\\ul{p}_2\end{array}\right )
\]
Similarly, for a $\R^{M \times n}$-matrix $\ul{X}$ we call $X$ its restriction to the upper left $2\times 2$ square.
From \eqref{eq:DQMn:Xi1mXi2}
\[
\begin{split}
\R^{2 \times 2}\ni&\tilde{X}^{\tni_1}-\tilde{X}^{\tni_1+1}\\
 =&\sqrt{\delta}\brac{\tilde{\kappa}^{\tni_1}+\tilde{\kappa}^{\tni_1+1}-1} \big (
 \lambda^{\tni_1} a'^{\tni_1} \otimes p^{\tni_1} +\frac{1}{2}\frac{1-\lambda^{\tni_1}}{ 1-\lambda^{\tni_1}+ \brac{\lambda^{\tni_1}}^2} \langle\ul{\overline{b}}, \frac{\ul{p}^{\tni_1}}{|\ul{p}^{\tni_1}|^2}\rangle_{\R^n} \overline{a}'\otimes b^{\tni_1}
\big ) + o(\sqrt{\delta})\\
 \end{split}
\]
By \Cref{la:detoftworank1}  $\det_{2\times 2} (a \otimes b + c \otimes d)=-\langle a^\perp, c\rangle  \langle b^\perp ,d \rangle$, so
\[
\begin{split}
&\det_{2 \times 2}\brac{\tilde{X}^{\tni_1}-\tilde{X}^{\tni_1+1}}\\
=&-\delta\, \brac{\tilde{\kappa}^{\tni_1}+\tilde{\kappa}^{\tni_1+1}-1}^2  \lambda^{\tni_1} \frac{1}{2}\frac{1-\lambda^{\tni_1}}{ 1-\lambda^{\tni_1}+ \brac{\lambda^{\tni_1}}^2} \langle\ul{\overline{b}}, \frac{\ul{p}^{\tni_1}}{|\ul{p}^{\tni_1}|^2}\rangle_{\R^n}\
\langle \brac{a'^{\tni_1}}^\perp, \overline{a}'\rangle_{\R^2}\, \langle \brac{p^{\tni_1}}^\perp, b^{\tni_1}\rangle_{\R^2}
\\
&+ o(\delta)
\end{split}
\]
And similarly,
\[
\begin{split}
&\det_{2 \times 2}\brac{\tilde{X}^{\tni_2+1}-\tilde{X}^{\tni_2+2}}\\
=&+\delta\, \brac{\tilde{\kappa}^{\tni_2+1}+\tilde{\kappa}^{\tni_2+2}-1}^2  \lambda^{\tni_2} \frac{1}{2}\frac{1-\lambda^{\tni_2}}{ 1-\lambda^{\tni_2}+ \brac{\lambda^{\tni_2}}^2} \langle\ul{\overline{b}}, \frac{\ul{p}^{\tni_2}}{|\ul{p}^{\tni_2}|^2}\rangle_{\R^n}\
\langle \brac{a'^{\tni_2}}^\perp, \overline{a}'\rangle_{\R^2}\, \langle \brac{p^{\tni_2}}^\perp, b^{\tni_2}\rangle_{\R^2}
\\
&+ o(\delta)
\end{split}
\]
So, in order to establish \eqref{eq:condYXaskljdDQtni1} we need
\[
 \langle\ul{\overline{b}}, \frac{\ul{p}^{\tni_1}}{|\ul{p}^{\tni_1}|^2}\rangle
\brac{
\sum_{\alpha,\beta=1}^n\brac{\langle \underline{a}''^{\tni_1}_{\alpha \beta},\underline{\overline{a}}' \rangle_{\R^M} -
\langle \underline{\overline{a}}_{\alpha \beta}'' ,\underline{a}'^{\tni_1}\rangle_{\R^M}} \ul{p}^{\tni_1}_\beta \ul{b}^{\tni_1}_\alpha }
>   -2d^{\tni_1}    \langle\ul{\overline{b}}, \frac{\ul{p}^{\tni_1}}{|\ul{p}^{\tni_1}|^2}\rangle_{\R^n}\
\langle \brac{a'^{\tni_1}}^\perp, \overline{a}'\rangle_{\R^2}\, \langle \brac{p^{\tni_1}}^\perp, b^{\tni_1}\rangle_{\R^2}
\]
and in order to have \eqref{eq:condYXaskljdDQtni2} we need to assume
\[
 -\langle\ul{\overline{b}}, \frac{\ul{p}^{\tni_2}}{|\ul{p}^{\tni_2}|^2}\rangle
\brac{
\sum_{\alpha,\beta=1}^n\brac{\langle \underline{a}''^{\tni_2}_{\alpha \beta},\underline{\overline{a}}' \rangle_{\R^M} -
\langle \underline{\overline{a}}_{\alpha \beta}'' ,\underline{a}'^{\tni_2}\rangle_{\R^M}} \ul{p}^{\tni_2}_\beta \ul{b}^{\tni_2}_\alpha }
>   +2d^{\tni_2}    \langle\ul{\overline{b}}, \frac{\ul{p}^{\tni_2}}{|\ul{p}^{\tni_2}|^2}\rangle_{\R^n}\
\langle \brac{a'^{\tni_2}}^\perp, \overline{a}'\rangle_{\R^2}\, \langle \brac{p^{\tni_2}}^\perp, b^{\tni_2}\rangle_{\R^2}
\]

This is exactly \eqref{eq:DC:condYXV1} and \eqref{eq:DC:condYXV2}, so we can conclude.
\end{proof}

The conditions \eqref{eq:DC:condYXV1},  \eqref{eq:DC:condYXV2} in \Cref{pr:DonQuijotte22} are not too difficult to satisfy. We observe the following
\begin{lemma}\label{la:TnnweirdCconditioniseasy}
Assume we have a $T_{\tnn}$-configuration in $\R^{4\times 2}$ that such that for two $C^{\tni_1} = a^{\tni_1} \otimes b^{\tni_1}$ and $C^{\tni_2}=a^{\tni_2} \otimes b^{\tni_2}$ (as usual $a^\tnk = \left ( \begin{array}{c}a'^{\tnk}\\a''^{\tnk}\end{array} \right )\in \R^4$, $b^{\tnk} \in \R^2$, $|b^{\tnk}|=1$) we have
\begin{itemize}
\item $b^{\tni_1},b^{\tni_2}$ are linear independent in $\R^2$ but not orthogonal to each other.
\item $a'^{\tni_1},a'^{\tni_2}$ are linear independent in $\R^2$
\end{itemize}
then whatever $d^{\tni_1}$, $d^{\tni_2} \in \R$ is, we can find $\overline{a} \in \R^4$, $\overline{b} \in \R^2$, $|\overline{b}|=1$ such that \eqref{eq:DC:condYXV1},  \eqref{eq:DC:condYXV2} are satisfied for the lifted version of the $T_{\tnn}$ of \Cref{la:TnnLift}.
\end{lemma}
\begin{proof}
Set
\[
\overline{b} \coloneqq  \frac{\brac{b^{\tni_1}}^\perp+\brac{b^{\tni_2}}^\perp}{\abs{b^{\tni_1}+b^{\tni_2}}}.
\]
Set
\[
 p^{\tni} \coloneqq \overline{b}-\langle \overline{b} , b^{\tni}\rangle_{\R^2} b^{\tni}.
\]
Since $b^{\tni_1}$ and $b^{\tni_2}$ are linearly independent but not perpendicular in $\R^2$,
\begin{equation}\label{eq:baskldjasdv}
\begin{split}
 \langle \brac{\overline{b}^\tni}^\perp, p^{\tni}\rangle_{\R^2} =& \langle \brac{\overline{b}^\tni}^\perp, \overline{b}\rangle_{\R^2}\\
 =& \frac{1+\langle \brac{b^{\tni_1}}^\perp, \brac{b^{\tni_2}}^\perp \rangle}{\abs{b^{\tni_1}+b^{\tni_2}}}\\
  =& \frac{1}{2}\frac{\abs{ b^{\tni_1}- \brac{b^{\tni_2}}}^2}{\abs{b^{\tni_1}+b^{\tni_2}}}\\
  =& \frac{\abs{ b^{\tni_1}- b^{\tni_2}}}{2} \in (0,1)\\
 \end{split}
\end{equation}
\[
\langle \overline{b},b^\tni\rangle_{\R^2} = \pm \frac{\langle b^{\tni_1},\brac{b^{\tni_2}}^\perp \rangle}{\abs{b^{\tni_1}+b^{\tni_2}}} \in (-1,1).
\]
Indeed, the latter number is necessarily between $(-1,1)$: For this assume $b^{i_1} = \left ( \begin{array}{c}1\\0\end{array} \right )$ then the above is (in absolute value)
\[
\begin{split}
&\frac{\sqrt{1-\theta^2}}{\sqrt{2+\theta}} < 1 \\
\Leftrightarrow\quad &1-\theta^2< 2+\theta \\
\Leftrightarrow\quad &0< \theta^2 +1+\theta \\
\end{split}
\]
which is always true for $\theta \in [0,1]$.

Thus,
\[
\langle \overline{b}, p^{\tni}\rangle_{\R^2}= 1-\langle \overline{b} , b^{\tni}\rangle_{\R^2}^2 \in (0,1) \quad \tni=\tni_1,\tni_2.
\]

Considering now the extension by zeroes $\ul{\overline{b}} \in \R^n$ and $\ul{b}^\tni$ we have for
\[
\ul{p}^{\tni} \coloneqq  \underline{\overline{b}}-\langle \underline{\overline{b}} , \underline{b}^{\tni}\rangle \underline{b}^{\tni},
\]
that
\[
\langle\ul{\overline{b}}, \ul{p}^{\tni}\rangle_{\R^n} =\langle \overline{b}, p^{\tni}\rangle_{\R^2}>0 \quad \tni=\tni_1,\tni_2.
\]
Thus, in order to show \eqref{eq:DC:condYXV1} it suffices to show
\[
\brac{
\sum_{\alpha,\beta=1}^n\brac{\langle \underline{a}''^{\tni_1}_{\alpha \beta},\underline{\overline{a}}' \rangle_{\R^M} -
\langle \underline{\overline{a}}_{\alpha \beta}'' ,\underline{a}'^{\tni_1}\rangle_{\R^M}} \ul{p}^{\tni_1}_\beta \ul{b}^{\tni_1}_\alpha }
>   -2d^{\tni_1}    \langle \brac{a'^{\tni_1}}^\perp, \overline{a}'\rangle_{\R^2}\, \langle \brac{p^{\tni_1}}^\perp, b^{\tni_1}\rangle_{\R^2}
\]
and for \eqref{eq:DC:condYXV2} it suffices to show
\[
\brac{
\sum_{\alpha,\beta=1}^n\brac{\langle \underline{a}''^{\tni_1}_{\alpha \beta},\underline{\overline{a}}' \rangle_{\R^M} -
\langle \underline{\overline{a}}_{\alpha \beta}'' ,\underline{a}'^{\tni_1}\rangle_{\R^M}} \ul{p}^{\tni_1}_\beta \ul{b}^{\tni_1}_\alpha }
< - 2d^{\tni_1} \langle \brac{a'^{\tni_1}}^\perp, \overline{a}'\rangle_{\R^2}\, \langle \brac{p^{\tni_1}}^\perp, b^{\tni_1}\rangle_{\R^2}
\]
If we choose $\overline{a}'\aeq 0$, it actually suffices to establish
\[
\langle \underline{\overline{a}}_{\alpha \beta}'' ,\underline{a}'^{\tni_1}\rangle_{\R^M}\, \ul{p}^{\tni_1}_\beta \ul{b}^{\tni_1}_\alpha<   0
\]
and
\[
\langle \underline{\overline{a}}_{\alpha \beta}'' ,\underline{a}'^{\tni_2}\rangle_{\R^M} \ul{p}^{\tni_2}_\beta \ul{b}^{\tni_2}_\alpha >  0
\]
Take
\[
 \overline{a}'' \coloneqq  -\frac{a'^{\tni_1}}{\abs{a'^{\tni_1}}}+\frac{a'^{\tni_2}}{\abs{a'^{\tni_2}}} \in \R^2
\]
then
\[
\R^M \ni \ul{\overline{a}''}_{\alpha \beta} =  \left ( \begin{array}{c}
                                                   \overline{a}''\\
                                                   0_{M-2}
                                                  \end{array} \right ) \brac{\delta_{\alpha 1}\delta_{\beta 2}-\delta_{\alpha 2}\delta_{\beta 1}}
\]
Thus,
\[
\langle \underline{\overline{a}}_{\alpha \beta}'' ,\underline{a}'^{\tni_1}\rangle_{\R^M}\, \ul{p}^{\tni_1}_\beta \ul{b}^{\tni_1}_\alpha = \langle \overline{a}'',a'^{\tni_1}\rangle_{\R^2} \underbrace{\langle \brac{b^{\tni_1}}^\perp, p^{\tni_1}\rangle_{\R^2}}_{>0}.
\]
and
\[
\langle \underline{\overline{a}}_{\alpha \beta}'' ,\underline{a}'^{\tni_2}\rangle_{\R^M} \ul{p}^{\tni_2}_\beta \ul{b}^{\tni_2}_\alpha = \langle \overline{a}'',a'^{\tni_2}\rangle_{\R^2} \underbrace{\langle \brac{b^{\tni_2}}^\perp, p^{\tni_2}\rangle_{\R^2}}_{>0}.
\]

Thus, it suffices to establish
\[
\langle \overline{a}'',a'^{\tni_1}\rangle_{\R^2}<   0
\]
and
\[
\langle \overline{a}'',a'^{\tni_2}\rangle_{\R^2} >  0
\]

Since $a'^{\tni_1}$ and $a'^{\tni_2}$ are linearly independent in $\R^2$ we observe
\[
\abs{ \left \langle \frac{a'^{\tni_1}}{\abs{a'^{\tni_1}}}, \frac{a'^{\tni_2}}{\abs{a'^{\tni_2}} } \right \rangle_{\R^2}} < 1
\]
So we take
\[
 \overline{a}'' \coloneqq  -\frac{a'^{\tni_1}}{\abs{a'^{\tni_1}}}+\frac{a'^{\tni_2}}{\abs{a'^{\tni_2}}}
\]
Then
\[
\langle \overline{a}'',\frac{a'^{\tni_1}}{\abs{a'^{\tni_1}}}\rangle_{\R^2}=-1+\left \langle \frac{a'^{\tni_1}}{\abs{a'^{\tni_1}}}, \frac{a'^{\tni_2}}{\abs{a'^{\tni_2}} } \right \rangle_{\R^2}<   0
\]
and
\[
\langle \overline{a}'',\frac{a'^{\tni_2}}{\abs{a'^{\tni_2}}}\rangle_{\R^2} =1-\left \langle \frac{a'^{\tni_1}}{\abs{a'^{\tni_1}}}, \frac{a'^{\tni_2}}{\abs{a'^{\tni_2}} } \right \rangle_{\R^2} >0
\]
Thus \eqref{eq:DC:condYXV1},  \eqref{eq:DC:condYXV2} are established. We can conclude.
\end{proof}

\begin{proof}[Proof of \Cref{th:pleaspleaseplease}]
In \cite{Sz04} an example of a $T_5$-configuration in $\R^{4 \times 2}$ is given which satisfies \eqref{eq:szineq2x4vDQ} if \Cref{pr:DonQuijotte22}, and the components $C'^{\tni}$ are given by
\[
\begin{split}
C_1=&\left(
\begin{array}{cc}
 1 & 1\\
 -1 & -1\\
\end{array}
\right) = \left ( \begin{array}{c}
                   1\\
                   -1\\
                  \end{array} \right ) \otimes \left ( \begin{array}{c}
                   1\\
                   1\\
                  \end{array} \right )\\
C_2=& \left(
\begin{array}{cc}
 1 & 2\\
 -2 & -4\\
\end{array}
\right)    = \left ( \begin{array}{c}
                   1\\
                   -2\\
                  \end{array} \right ) \otimes \left ( \begin{array}{c}
                   1\\
                   2\\
                  \end{array} \right )\\
C_3=& \left(
\begin{array}{cc}
 1 & 0\\
 -3 & 0\\
\end{array}
\right) = \left ( \begin{array}{c}
                   1\\
                   -3\\
                  \end{array} \right ) \otimes \left ( \begin{array}{c}
                   1\\
                   0\\
                  \end{array} \right )\\
C_4=& \left(
\begin{array}{cc}
 -3 & -3\\
 7 & 7\\
\end{array}
\right)                 = \left ( \begin{array}{c}
                   -3\\
                   7\\
                  \end{array} \right ) \otimes \left ( \begin{array}{c}
                   1\\
                   1\\
                  \end{array} \right )\\
C_5=&\left(
\begin{array}{cc}
 0 & 0\\
 -1 & -2\\
\end{array}
\right) = \left ( \begin{array}{c}
                   0\\
                   -1\\
                  \end{array} \right ) \otimes \left ( \begin{array}{c}
                   1\\
                   2\\
                  \end{array} \right )\\
\end{split}
\]
So take e.g. $\tni_1=1$, $\tni_2=2$. Then $a'^{\tni_1} = \sqrt{2}\left ( \begin{array}{c}
                   1\\
                   -1\\
                  \end{array} \right )$ is linear independent to $a'^{\tni_2} = \sqrt{5}\left ( \begin{array}{c}
                   1\\
                   -2\\
                  \end{array} \right )$, and and $b^{\tni_1}=\frac{1}{\sqrt{2}}\left ( \begin{array}{c}
                   1\\
                   1\\
                  \end{array} \right )$  is linear independent to $b^{\tni_2}=\frac{1}{\sqrt{5}}\left ( \begin{array}{c}
                   1\\
                   2\\
                  \end{array} \right )$, but they are not orthogonal. So the assumptions of \Cref{la:TnnweirdCconditioniseasy} are satisfied, thus the lifted version of \Cref{la:TnnLift} satisfy \eqref{eq:DC:condYXV1},  \eqref{eq:DC:condYXV2}, and \eqref{eq:linearsystemass} is satisfied as well. So we can apply \Cref{pr:DonQuijotte22} until we have a non-degenerate $T_{\tnn}$-configuration that satisfies \eqref{eq:linearsystemass} . We can conclude.

\end{proof}

\section{Proof of Theorem~\ref{th:main} and Theorem~\ref{th:uniqueness}}\label{th:proofmain}

\Cref{th:existenceofFwithTNN} ensures the existence of some polyconvex $F_0$ that has a non-degenerate $T_{\tnn}$-configuration. \Cref{th:conditionCgen} shows that we can distort $F_0$ into some $F$ such that $F$ has all the previous properties, but additionally satisfies condition (C). By \Cref{th:TnnwithCgivesexample} we find a non-$C^1$-solution to the Euler--Lagrange equation. Since $D^2F(\overline{X}_{\tni})$ is positive definite and $\dist(du,\{\overline{Z}_1,\ldots,\overline{Z}_{\tnn})) \ll 1$  a.e. in $\Omega$ we find that the assumptions of \cite[Theorem 6.1]{KT03} are satisfied, and thus $u$ is a weak local minimizer. This proves \Cref{th:main}.

As for \Cref{th:uniqueness}: Taking $u_2 \coloneqq \Aff$ from \Cref{th:TnnwithCgivesexample} we see that $\Aff$ is a distributional solution of the Euler--Lagrange system \eqref{eq:ELsystem1} (since $D\Aff$ is constant) and clearly $\Aff \not\equiv u$ since $u$ is nowhere $C^1$.

\appendix
\section{Basic facts from Analysis and Linear Algebra}
\subsection{Analysis}
\begin{theorem}[Vitali covering theorem]\label{th:vitali}
Let $E \subset \R^n$ be a measurable set with finite Lebesgue measure. Assume that $\mathcal{V}$ is a regular family of closed subsets $V \subset \R^d$, which is a Vitali covering for $E$, i.e.
\begin{itemize}
\item regular family: \[
 \diam(V)^d \aleq \mathcal{L}(V) \quad \forall V \in \mathcal{V}.
\]
\item Vitali covering:
\[
 \forall x \in E,\ \forall \delta >0 \quad \exists V \in \mathcal{V}: \quad x \in V, \quad \diam(V) \in (0,\delta).
\]
\end{itemize}
Then there exists an at most countable subcollection $(V_j)_{j=1}^\infty \subset \mathcal{V}$ which is pairwise disjoint and
\[
 \mathcal{L}^d(E \setminus \bigcup_{j} U_j) = 0.
\]
\end{theorem}

\begin{lemma}\label{la:IFTbiLipschitz}
Let $f: \mathcal{M}^d \to \R^N$ be a smooth map which around a point $\overline{x}$ is bi-Lipschitz, i.e.
\[
 \limsup_{y \to \overline{x}} \frac{|f(\overline{x})-f(y)|}{|\overline{x}-y|} \geq \lambda  > 0
\]
Then there exists an $\eps > 0$ such that $f\Big |_{B(x,\eps) \cap \mathcal{M}^d}$ is a smooth diffeomorphism  onto its own image.
\end{lemma}
\begin{proof}
For any $v \in T_{\overline{x}} \mathcal{M}^d$ with $|v| =1$ we have
\[
 \lambda \leq \lim_{t \to 0} \frac{|f(\overline{x}+tv)-f(\overline{x})|}{t} = |Df(\overline{x})v|.
\]
In particular $Df(\overline{x})$ is injective, i.e. has full rank. Thus, by the implicit function theorem, we have that $f$ is locally a diffeomorphism onto its target.
\end{proof}

\subsection{Linear Algebra}
\begin{lemma}\label{la:RCsets}
Let $X,Y \subset Z$ be two linearly independent spaces such that $Z=X+Y$.

Assume $A: X \to Z$ and $B: Y \to Z$
\[
 S\coloneqq  \{z \in Z: z=x+y, \quad Ax=By\}
\]
Then
\[
 \dim S \leq \dim \ker A + \dim Y.
\]
\end{lemma}
\begin{proof}
We observe that
\[
 \tilde{A}\coloneqq  A \Big|_{(\ker A)^\perp} \to Z
\]
is injective and thus for $T \coloneqq  \im \tilde{A}$ we have
\[
 \tilde{A} : \ker A^\perp \to T
\]
is bijective.

Take a basis $(\overline{x}_1,\ldots,\overline{x}_{K})$ of $\ker A$ and

Assume $z = x+y\in S$ for some $x \in X$ and $y \in Y$. Then (by uniqueness of decomposition) $Ax= By$.
Set
\[
 \tilde{x} \coloneqq  x-\overline{x}_i\langle \overline{x}_i, x\rangle \in (\ker A)^\perp
\]
We still have
\[
\tilde{A}\tilde{x} = A\tilde{x} = By.
\]
In particular, $By \in T$, and thus we have
\[
 \tilde{x} = \tilde{A}^{-1} By.
\]
We have shown
\[
  z = (x-\tilde{x})+(I+\tilde{A}^{-1} B)y \subset
\]
That is, for $C \coloneqq  (I+\tilde{A}^{-1} B)$ we have shown
\[
 z \in \ker_X A + \im_Y C.
\]
Since $C$ is a linear operation on $Y$, we have that $\dim \im_Y C \leq \dim Y$. We can conclude.

\end{proof}

\begin{lemma}\label{la:detoftworank1}
Let $a,b,c,d \in \R^2$. Recall the notation $v^\perp = (-v_2,v_1)^T$. Then

\[
 \det (a \otimes b - c \otimes d)=\langle a^\perp, c\rangle  \langle b^\perp ,d \rangle
\]
\end{lemma}
\begin{proof}
\[
 a \otimes b - c \otimes d = \left ( \begin{array}{cc}
                                      a^1 b^1 & a^1 b^2\\
                                      a^2 b^1 & a^2 b^2
                                     \end{array} \right )  + \left (\begin{array}{cc}
                                      c^1 d^1 & c^1 d^2\\
                                      c^2 d^1 & c^2 d^2
                                     \end{array} \right) = \left ( \begin{array}{cc}
                                      a^1 b^1+c^1 d^1 & a^1 b^2+c^1 d^2\\
                                      a^2 b^1 + c^2 d^1& a^2 b^2+c^2 d^2
                                     \end{array} \right )
\]
So
\[
\begin{split}
 \det (a \otimes b - c \otimes d) =& \brac{ a^1 b^1+c^1 d^1} \brac{a^2 b^2+c^2 d^2} - \brac{a^1 b^2+c^1 d^2} \brac{a^2 b^1 + c^2 d^1}\\
 =&  a^1 c^2  \brac{b^1d^2 -  b^2  d^1}\\
 &+  a^2 c^1 \brac{d^1  b^2 - d^2  b^1 }\\
 =&  a^1 c^2  \langle b^\perp ,d \rangle-  a^2 c^1 \langle b^\perp, d\rangle \\
  =&  \langle a^\perp, c\rangle  \langle b^\perp ,d \rangle\\
\end{split}
 \]
\end{proof}

\begin{lemma}\label{la:aotimesbJ}
Let
\[
 J = \left ( \begin{array}{cc}
              0 & -1\\
              1 & 0
             \end{array}
\right )
\]
Recall the notation $v^\perp = (-v_2,v_1)^T$.
Then
\[
\brac{a \otimes b} J= -a \otimes b^\perp,
\]
and thus
\[
\langle \brac{a \otimes b} J, c\otimes d\rangle= -\langle a,c\rangle\, \langle b^\perp, d\rangle
\]

 \end{lemma}
 \begin{proof}
We have
\[
a \otimes b = \left ( \begin{array}{cc}
                       a^1b^1 & a^1b^2\\
                       a^2b^1 & a^2b^2\\
                      \end{array}
\right)
\]
and
\[
a \otimes b^\perp = \left ( \begin{array}{cc}
                       -a^1b^2 & a^1b^1\\
                       -a^2b^2 & a^2b^1\\
                      \end{array}
\right)
\]

So
\[
a \otimes b J= \left ( \begin{array}{cc}
                       a^1b^1 & a^1b^2\\
                       a^2b^1 & a^2b^2\\
                      \end{array}
\right) \left ( \begin{array}{cc}
              0 & -1\\
              1 & 0
             \end{array}
\right ) = \left ( \begin{array}{cc}
              a^1b^2 & -a^1b^1 \\
              a^2b^2 & - a^2 b^1
             \end{array}
\right )
\]

 \end{proof}

\section{Polyconvexity, rank one convexity, Legendre-Hadamard condition}\label{s:polyconvLH}

\begin{definition}\label{def:stronglegendrehad}
A map $F \in C^2(\R^{M \times n},\R)$ satisfies the Legendre-Hadamard condition if
\[
 \partial_{X_{i\alpha}}  \partial_{X_{j\beta}} F(X) \xi_i \xi_j\, \eta_\alpha \eta_\beta \geq 0
\]
for all $\xi \in \R^M$ and $\eta \in \R^n$.

$F$ satisfies the strong Legendre-Hadamard condition if for some $\lambda > 0$
\begin{equation}\label{eq:stronglegendre}
 \partial_{X_{i\alpha}}  \partial_{X_{j\beta}} F(X) \xi_i \xi_j\, \eta_\alpha \eta_\beta \geq \lambda |\xi|_{\R^M}^2 |\eta|_{\R^n}^2
\end{equation}
for all $\xi \in \R^M$ and $\eta \in \R^n$.
\end{definition}

It is important not to confuse the Legendre-Hadamard-condition with convexity which takes the form
\[
 \partial_{X_{i\alpha}}  \partial_{X_{j\beta}} F(X) X_{i\alpha} X_{j\beta} \geq 0 \quad \forall X\in \R^{M \times n}
\]

\begin{lemma}\label{la:legendrehadamardinvertible}
Assume that $F$ satisfies the strong Legendre-Hadamard condition and fix $X \in \R^{M \times n}$.
\[
 \R^{M \times n} \ni Z \mapsto D^2 F(X) Z \in  \R^{M \times n}
\]
is a bijective linear map.
\end{lemma}
\begin{proof}
$D^2 F(X)$ is a linear map from $\R^{M \times n} \to \R^{M \times n}$ defined via
\[
 \R^{M \times n} \ni Z \mapsto D^2 F(X) Z = \brac{\sum_{j=1}^M \sum_{\alpha =1}^n \partial_{X_{i\alpha}}  \partial_{X_{j\beta}} F(X) Z_{j\beta}}_{i \in \{1,\ldots,M\}, \alpha \in \{1,\ldots,n\}}.
\]
It has full rank if it is injective, i.e. if $\ker D^2 F = \{0\}$.

Let $\xi^1,\ldots,\xi^M$ be an orthonormal basis of $\R^M$ and $\eta^1,\ldots, \eta^n$ an orthonormal basis of $\R^n$.
Then
\[
 \brac{\xi^i \otimes \eta^\alpha}_{i=1,\ldots,M; \alpha=1,\ldots,n}
\]
is an orthonormal basis of $\R^{M \times n}$ -- orthonormal in the sense of the Hilbert-Schmidt scalar product.

By the strong Legendre-Hadamard condition we have
\[
\langle D^2 F(X) \xi^i \otimes \eta^\alpha, \xi^i \otimes \eta^\alpha \rangle_{\R^{M \times n}} > 0
\]
and thus the image $D^2 F(X)$ is $\R^{M \times n}$. Thus $Z \mapsto D^2 F(X)Z$ is surjective, thus as a linear map between the same space $Z \mapsto D^2 F(X)Z$ must be bijective.
\end{proof}

\begin{lemma}\label{la:rankoneconveximpliesLH}
Assume $F \in C^2(\R^{M \times n},\R)$ is rank-one convex, i.e. assume that for any $X,Y \in \R^{M \times n}$ with $\rank (X-Y) \leq 1$ we have
\[
 F(\lambda X + (1-\lambda) Y) \leq \lambda F(X) + (1-\lambda) F(Y) \quad \forall \lambda \in [0,1].
\]
Then $F$ satisfies the Legendre-Hadamard condition.
\end{lemma}
\begin{proof}
Fix $X \in \R^{M \times n}$, $\xi \in \R^m$ and $\eta \in \R^n$ and set
\[
 \phi(t) \coloneqq F(X + t \xi \otimes \eta).
\]
Then
\[
 \phi(\lambda t_0 + (1-\lambda) t_1) = F(\lambda (X + t_0 \xi \otimes \eta) + (1-\lambda) (X + t_1 \xi \otimes \eta) ),
\]
and thus $\phi$ is convex. This implies
\[
0 \leq \phi''(t) = \partial_{X_{i\alpha}}  \partial_{X_{j\beta}} F(X) \xi_i \xi_j\, \eta_\alpha \eta_\beta.
\]
\end{proof}

Recall that a function $F: \R^{M \times n} \to \R$ is called polyconvex if the map
\[
\R^{M \times n} \ni X \mapsto F(X)
\]
can be written as a convex function of the $\ell \times \ell$ subdeterminants of $X$, $\ell \in \{1,\ldots, \min\{n,M\}\}$. And $F$ is strongly polyconvex if $F(X) - \gamma |X|^2$ is polyconvex for some $\gamma > 0$.

\begin{lemma}
Assume that $F \in C^2(\R^{M \times n})$ is polyconvex. Then $F$ is rank-one convex, and thus $F$ satisfies the Legendre-Hadamard condition.

If $F$ is strongly polyconvex, then $F$ satisfies the strong Legendre-Hadamard condition.
\end{lemma}
\begin{proof}
Fix $t \in (0,1)$, $X,Y \in \R^{M \times n}$ with $\rank (X-Y) \leq 1$. Fix any $k \times k$ subspace of $\R^{M \times n}$, and denote by $X'$, $Y'$ the projections into this subspace identified with $\R^{k \times k}$.

Set
\[
\varphi(t) \coloneqq  \det_{k \times k} (tX' + (1-t) Y')
\]
Since $\rank(X'-Y') \leq 1$ the map $\varphi$ is affine linear in $t$. Moreover we know $\varphi(1) = \det_{k \times k} (X')$ and $\varphi(0) = \det(Y')$. We conclude
\[
 \det_{k \times k} (tX' + (1-t) Y') =  t\det_{k \times k} X' + (1-t) \det_{k \times k} Y'.
 \quad t \in \R.
\]
Since $F$ is polyconvex it can be written as as a convex function $k \times k$ subdeterminants of $X$
\[
 F(X) = G(X,\det_{1 \times 1} (X'), \ldots, \det_{k \times k} (X'), \ldots ).
\]
Since by the above argument all the (sub-)determinants behave linearly on $tX + (1-t)Y$, $F$ is rank-one convex by convexity of $G$.

Thus, $F$ satisfies the Legendre-Hadamard condition by \Cref{la:rankoneconveximpliesLH}.

If $F$ is moreover strongly polyconvex, then $F(X) = \eps |X|^2 + \tilde{F}$, where $\tilde{F}$ is polyconvex. Since $\tilde{F}$ thus satisfies the Legendre-Hadamard condition we find
\[
 \partial_{X_{i\alpha}}  \partial_{X_{j\beta}} F(X) \xi_i \xi_j\, \eta_\alpha \eta_\beta \geq 2\eps\, |\xi|^2_{\R^M} |\eta|^2_{\R^n}.
\]
Thus a strongly polyconvex $F$ satisfies the strong Legendre-Hadamard condition.
\end{proof}

As a corollary of the previous lemmata we have
\begin{lemma}\label{la:stronglypolyconvexinvertible}
Assume that $F \in C^2(\R^{M \times n})$ is strongly polyconvex. Then
\[
 \R^{M \times n} \ni Z \mapsto D^2 F(X) Z \in  \R^{M \times n}
\]
is a bijective linear map.
\end{lemma}

\section{Explicit formulas for \texorpdfstring{$T_{\tnn}$}{TN}-configurations}\label{s:mathematica}
Following the ideas of \cite{Sz04}, using Mathematica we computed explicit formulas for $T_{\tnn}$-configurations that satisfy the conclusions of \Cref{th:pleaspleaseplease}.

\subsection{A \texorpdfstring{$T_{14}$}{T14}-configuration in two dimensions}
The following is an example of a $T_{14}$ configuration for $M=n=2$. We set
\[
 C_i=\begin{pmatrix}
A_i\\
B_i
\end{pmatrix} = \begin{pmatrix}
a_i\otimes n_i\\
b_i \otimes n_i
\end{pmatrix} \in \R^{4\times2}, \quad
Z_i = \begin{pmatrix}
X_i\\
Y_i
\end{pmatrix} \in \R^{4\times 2}
\]
We take $a_i$'s:
\[
\left(
\begin{array}{c}
 1 \\
 1 \\
\end{array}
\right),\left(
\begin{array}{c}
 1 \\
 2 \\
\end{array}
\right),\left(
\begin{array}{c}
 0 \\
 1 \\
\end{array}
\right),\left(
\begin{array}{c}
 5 \\
 2 \\
\end{array}
\right),\left(
\begin{array}{c}
 11 \\
 9 \\
\end{array}
\right),\left(
\begin{array}{c}
 17 \\
 7 \\
\end{array}
\right),\left(
\begin{array}{c}
 47 \\
 19 \\
\end{array}
\right),\left(
\begin{array}{c}
 23 \\
 9 \\
\end{array}
\right),\left(
\begin{array}{c}
 41 \\
 18 \\
\end{array}
\right),
\]
\[\left(
\begin{array}{c}
 68 \\
 1 \\
\end{array}
\right),\left(
\begin{array}{c}
 299 \\
 171 \\
\end{array}
\right),\left(
\begin{array}{c}
 101 \\
 21 \\
\end{array}
\right),\left(
\begin{array}{c}
 -3084 \\
 -1857 \\
\end{array}
\right),\left(
\begin{array}{c}
 2401 \\
 1590 \\
\end{array}
\right)
\]
The $n_i$'s:
\[
 \left(
\begin{array}{c}
 1 \\
 0 \\
\end{array}
\right),\left(
\begin{array}{c}
 0 \\
 1 \\
\end{array}
\right),\left(
\begin{array}{c}
 2 \\
 1 \\
\end{array}
\right),\left(
\begin{array}{c}
 3 \\
 2 \\
\end{array}
\right),\left(
\begin{array}{c}
 1 \\
 2 \\
\end{array}
\right),\left(
\begin{array}{c}
 1 \\
 1 \\
\end{array}
\right),\left(
\begin{array}{c}
 1 \\
 5 \\
\end{array}
\right),\left(
\begin{array}{c}
 1 \\
 0 \\
\end{array}
\right),\left(
\begin{array}{c}
 0 \\
 1 \\
\end{array}
\right),
\]
\[\left(
\begin{array}{c}
 1 \\
 2 \\
\end{array}
\right),\left(
\begin{array}{c}
 1 \\
 13 \\
\end{array}
\right),\left(
\begin{array}{c}
 2 \\
 1 \\
\end{array}
\right),\left(
\begin{array}{c}
 1 \\
 3 \\
\end{array}
\right),\left(
\begin{array}{c}
 1 \\
 2 \\
\end{array}
\right)
\]

Which gives $A_i$'s:
\[
 \left(
\begin{array}{cc}
 1 & 0 \\
 1 & 0 \\
\end{array}
\right),\left(
\begin{array}{cc}
 0 & 1 \\
 0 & 2 \\
\end{array}
\right),\left(
\begin{array}{cc}
 0 & 0 \\
 2 & 1 \\
\end{array}
\right),\left(
\begin{array}{cc}
 15 & 10 \\
 6 & 4 \\
\end{array}
\right),\left(
\begin{array}{cc}
 11 & 22 \\
 9 & 18 \\
\end{array}
\right),\left(
\begin{array}{cc}
 17 & 17 \\
 7 & 7 \\
\end{array}
\right),\left(
\begin{array}{cc}
 47 & 235 \\
 19 & 95 \\
\end{array}
\right),
\left(
\begin{array}{cc}
 23 & 0 \\
 9 & 0 \\
\end{array}
\right),
\]

\[\left(
\begin{array}{cc}
 0 & 41 \\
 0 & 18 \\
\end{array}
\right),\left(
\begin{array}{cc}
 68 & 136 \\
 1 & 2 \\
\end{array}
\right),\left(
\begin{array}{cc}
 299 & 3887 \\
 171 & 2223 \\
\end{array}
\right),\left(
\begin{array}{cc}
 202 & 101 \\
 42 & 21 \\
\end{array}
\right),\left(
\begin{array}{cc}
 -3084 & -9252 \\
 -1857 & -5571 \\
\end{array}
\right),
\left(
\begin{array}{cc}
 2401 & 4802 \\
 1590 & 3180 \\
\end{array}
\right)
\]
We set $\kappa_i=2$ for all $i\in\{1,\ldots,14\}$.

This gives $X_i$'s:
\[
 \left(
\begin{array}{cc}
 2 & 0 \\
 2 & 0 \\
\end{array}
\right),\left(
\begin{array}{cc}
 1 & 2 \\
 1 & 4 \\
\end{array}
\right),\left(
\begin{array}{cc}
 1 & 1 \\
 5 & 4 \\
\end{array}
\right),\left(
\begin{array}{cc}
 31 & 21 \\
 15 & 11 \\
\end{array}
\right),\left(
\begin{array}{cc}
 38 & 55 \\
 27 & 43 \\
\end{array}
\right),\left(
\begin{array}{cc}
 61 & 67 \\
 32 & 39 \\
\end{array}
\right),\left(
\begin{array}{cc}
 138 & 520 \\
 63 & 222 \\
\end{array}
\right),\left(
\begin{array}{cc}
 137 & 285 \\
 62 & 127 \\
\end{array}
\right),
\]
\[
\left(
\begin{array}{cc}
 114 & 367 \\
 53 & 163 \\
\end{array}
\right),\left(
\begin{array}{cc}
 250 & 598 \\
 55 & 149 \\
\end{array}
\right),\left(
\begin{array}{cc}
 780 & 8236 \\
 396 & 4593 \\
\end{array}
\right),\left(
\begin{array}{cc}
 885 & 4551 \\
 309 & 2412 \\
\end{array}
\right),\left(
\begin{array}{cc}
 -5485 & -14054 \\
 -3447 & -8751 \\
\end{array}
\right),\left(
\begin{array}{cc}
 2401 & 4802 \\
 1590 & 3180 \\
\end{array}
\right)
\]
The matrix of the determinants of the differences $\det(X_i-X_j)$ is:
\begin{tiny}
 \[
 \left(
\begin{array}{cccccccccccccc}
 0 & -2 & -7 & 46 & 173 & 291 & -1528 & 45 & -461 & 5258 & 328370 & 732639 &
   -455509 & 3244 \\
 -2 & 0 & 4 & -56 & 65 & 85 & -2250 & -535 & -1013 & 3921 & 322401 & 727580
   & -435158 & -4800 \\
 -7 & 4 & 0 & 10 & 255 & 318 & -236 & 540 & 399 & 6255 & 354946 & 745472 &
   -487930 & 12815 \\
 46 & -56 & 10 & 0 & -184 & 58 & -1375 & -112 & -532 & 7142 & 302003 &
   718634 & -396458 & -19545 \\
 173 & 65 & 255 & -184 & 0 & -152 & 1160 & 266 & 1008 & 7268 & 357311 &
   738671 & -445404 & -6830 \\
 291 & 85 & 318 & 58 & -152 & 0 & 48 & 148 & 272 & 8577 & 300810 & 713284 &
   -377619 & -27190 \\
 -1528 & -2250 & -236 & -1375 & 1160 & 48 & 0 & -140 & -114 & -7552 & 236754
   & 644304 & -699561 & 155340 \\
 45 & -535 & 540 & -112 & 266 & 148 & -140 & 0 & -90 & 4677 & 216004 &
   655478 & -403435 & 10016 \\
 -461 & -1013 & 399 & -532 & 1008 & 272 & -114 & -90 & 0 & -2366 & 251313 &
   662875 & -564014 & 83284 \\
 5258 & 3921 & 6255 & 7142 & 7268 & 8577 & -7552 & 4677 & -2366 & 0 &
   -249238 & 432943 & -269804 & 66541 \\
 328370 & 322401 & 354946 & 302003 & 357311 & 300810 & 236754 & 216004 &
   251313 & -249238 & 0 & -549600 & -2060310 & 1809723 \\
 732639 & 727580 & 745472 & 718634 & 738671 & 713284 & 644304 & 655478 &
   662875 & 432943 & -549600 & 0 & 1227930 & 842757 \\
 -455509 & -435158 & -487930 & -396458 & -445404 & -377619 & -699561 &
   -403435 & -564014 & -269804 & -2060310 & 1227930 & 0 & -889806 \\
 3244 & -4800 & 12815 & -19545 & -6830 & -27190 & 155340 & 10016 & 83284 &
   66541 & 1809723 & 842757 & -889806 & 0 \\
\end{array}
\right)
\]
\end{tiny}
In particular all $\det(X_i - X_j)\neq 0$ for all $i\neq j$.

With these data we find:
$b_i$'s
\[
 \left(
\begin{array}{c}
 0 \\
 -2390 \\
\end{array}
\right),\left(
\begin{array}{c}
 -478 \\
 2665 \\
\end{array}
\right),\left(
\begin{array}{c}
 -550 \\
 769 \\
\end{array}
\right),\left(
\begin{array}{c}
 167 \\
 -492 \\
\end{array}
\right),\left(
\begin{array}{c}
 -1240 \\
 2977 \\
\end{array}
\right),\left(
\begin{array}{c}
 0 \\
 -293 \\
\end{array}
\right),\left(
\begin{array}{c}
 -2442 \\
 5945 \\
\end{array}
\right),\left(
\begin{array}{c}
 -1086 \\
 2706 \\
\end{array}
\right),
\]
\[
\left(
\begin{array}{c}
 -2018 \\
 4704 \\
\end{array}
\right),\left(
\begin{array}{c}
 2898 \\
 6192 \\
\end{array}
\right),\left(
\begin{array}{c}
 -38226 \\
 67612 \\
\end{array}
\right),\left(
\begin{array}{c}
 -197436 \\
 313998 \\
\end{array}
\right),\left(
\begin{array}{c}
 -165154 \\
 173736 \\
\end{array}
\right),\left(
\begin{array}{c}
 600721 \\
 -884543 \\
\end{array}
\right)
\]
This gives $B_i$'s
\[
 \left(
\begin{array}{cc}
 0 & 0 \\
 -2390 & 0 \\
\end{array}
\right),\left(
\begin{array}{cc}
 0 & -478 \\
 0 & 2665 \\
\end{array}
\right),\left(
\begin{array}{cc}
 -1100 & -550 \\
 1538 & 769 \\
\end{array}
\right),\left(
\begin{array}{cc}
 501 & 334 \\
 -1476 & -984 \\
\end{array}
\right),\left(
\begin{array}{cc}
 -1240 & -2480 \\
 2977 & 5954 \\
\end{array}
\right),
\]
\[
\left(
\begin{array}{cc}
 0 & 0 \\
 -293 & -293 \\
\end{array}
\right),\left(
\begin{array}{cc}
 -2442 & -12210 \\
 5945 & 29725 \\
\end{array}
\right),\left(
\begin{array}{cc}
 -1086 & 0 \\
 2706 & 0 \\
\end{array}
\right),\left(
\begin{array}{cc}
 0 & -2018 \\
 0 & 4704 \\
\end{array}
\right),\left(
\begin{array}{cc}
 2898 & 5796 \\
 6192 & 12384 \\
\end{array}
\right),
\]
\[
\left(
\begin{array}{cc}
 -38226 & -496938 \\
 67612 & 878956 \\
\end{array}
\right),\left(
\begin{array}{cc}
 -394872 & -197436 \\
 627996 & 313998 \\
\end{array}
\right),\left(
\begin{array}{cc}
 -165154 & -495462 \\
 173736 & 521208 \\
\end{array}
\right),\left(
\begin{array}{cc}
 600721 & 1201442 \\
 -884543 & -1769086 \\
\end{array}
\right)
\]
And $Y_i$'s
\[
 \left(
\begin{array}{cc}
 0 & 0 \\
 -4780 & 0 \\
\end{array}
\right),\left(
\begin{array}{cc}
 0 & -956 \\
 -2390 & 5330 \\
\end{array}
\right),\left(
\begin{array}{cc}
 -2200 & -1578 \\
 686 & 4203 \\
\end{array}
\right),\left(
\begin{array}{cc}
 -98 & -360 \\
 -3804 & 1466 \\
\end{array}
\right),\left(
\begin{array}{cc}
 -3079 & -5654 \\
 3626 & 14358 \\
\end{array}
\right),
\]
\[
\left(
\begin{array}{cc}
 -1839 & -3174 \\
 63 & 7818 \\
\end{array}
\right),\left(
\begin{array}{cc}
 -6723 & -27594 \\
 12246 & 67561 \\
\end{array}
\right),\left(
\begin{array}{cc}
 -6453 & -15384 \\
 11713 & 37836 \\
\end{array}
\right),\left(
\begin{array}{cc}
 -5367 & -19420 \\
 9007 & 47244 \\
\end{array}
\right),
\]
\[
\left(
\begin{array}{cc}
 429 & -5810 \\
 21391 & 67308 \\
\end{array}
\right),\left(
\begin{array}{cc}
 -78921 & -1005482 \\
 150423 & 1812836 \\
\end{array}
\right),\left(
\begin{array}{cc}
 -830439 & -903416 \\
 1338803 & 1561876 \\
\end{array}
\right),\left(
\begin{array}{cc}
 -765875 & -1696904 \\
 1058279 & 2290294 \\
\end{array}
\right),
\]
\[\left(
\begin{array}{cc}
 600721 & 1201442 \\
 -884543 & -1769086 \\
\end{array}
\right)
\]
For
\[
\begin{split}
 c&=(0, -4938, -18804, -54520, -192834, -176960, -992337, -553838,-682858, 167765, 29184771,\\
 &\phantom{=( \ } 123336787, 123336787, -10696958)
 \end{split}
\]
and
\[
 d=(-44, 128, -189, 42, -59, 29, 140, 102, 140, 199, 241, -503, 186, -337)
\]
With these data it's possible that all $13\cdot 14 = 182$ inequalities are satisfied:
\[
 c_i - c_j + d_i \det(X_i - X_j) + \langle X_i - X_j, Y_i J \rangle <0.
\]
We have exactly

\begin{tiny}
 \[
  \left(
\begin{array}{cccccccccccccc}
 0 & -14094 & -8 & -84 & -20318 & -22264 & -1591 & -55202 & -75998 &
   -1111337 & -65587591 & -167102263 & -61464611 & -4646178 \\
 -8 & 0 & -6942 & -20256 & -202 & -8618 & -21109 & -8664 & -886 & -67141 &
   -250717 & -36763141 & -144984275 & -3687990 \\
 -3838 & -10 & 0 & -62 & -291 & -41937 & -1703 & -132411 & -130913 &
   -2399922 & -111671426 & -272490671 & -369903 & -3001900 \\
 -68 & -13720 & -2848 & 0 & -9546 & -266 & -3327 & -3264 & -31950 & -483619
   & -35079069 & -102909587 & -102243013 & -4157616 \\
 -34208 & -5848 & -57545 & -6086 & 0 & -106 & -3106 & -1060 & -38580 &
   -1280328 & -60252787 & -171624418 & -67043568 & -1415309 \\
 -491 & -309 & -9119 & -242 & -126 & 0 & -2729 & -7504 & -26058 & -645499 &
   -35937776 & -110475298 & -99454131 & -3531578 \\
 -60472 & -81109 & -48804 & -69076 & -30146 & -34503 & 0 & -1597 & -5980 &
   -4718 & -161415 & -30415525 & -152191972 & -1832422 \\
 -4374 & -2198 & -23573 & -18270 & -32330 & -19876 & -1953 & 0 & -946 &
   -3408 & -9449058 & -55634422 & -130221766 & -5207692 \\
 -11446 & -30670 & -2733 & -31488 & -2816 & -11722 & -2579 & -4314 & 0 &
   -1106 & -286694 & -30537574 & -149251076 & -3155154 \\
 -103250 & -216455 & -7784 & -136018 & -40558 & -79267 & -3849 & -80239 &
   -75380 & 0 & -149790 & -316577 & -171133360 & -74451 \\
 -585154 & -767541 & -82753 & -1295659 & -413553 & -1533006 & -3154731 &
   -3530376 & -2758743 & -11458950 & 0 & -561952 & -171398130 & -147875 \\
 -13265973 & -6363564 & -20766439 & -2984648 & -10736775 & -334280 &
   -74865005 & -12298686 & -42511131 & -18539785 & -379448 & 0 & -742748 &
   -101297011 \\
 -52307366 & -45222386 & -63419413 & -26275272 & -43020936 & -16275094 &
   -168954753 & -34276306 & -107134832 & -18797070 & -1409240700 & -206702 &
   0 & -908356 \\
 -10654898 & -10843506 & -10290272 & -11879017 & -11087582 & -12283048 &
   -394405 & -9564210 & -4999540 & -8915559 & -1492927 & -334307895 &
   -44196336 & 0 \\
\end{array}
\right)
 \]
\end{tiny}
In particular all $i\neq j$ entries are negative.

\subsection{A \texorpdfstring{$T_{14}$}{T14} in three dimensions}
The following is a $T_{14}$ configuration in $\R^3\times \R^3$:
\[
 C_i=\begin{pmatrix}
A_i\\
B_i,
\end{pmatrix}
\]

\[
\begin{split}
 &\quad A_i = \left ( \begin{array}{c|c|c}
          n_1^i a^1_i & n_2^i a^1_i & n_3^i a^1_i\\
          n_1^i a^2_i & n_2^i a^2_i & n_3^i a^2_i\\
          n_1^i a^3_i & n_2^i a^3_i & n_3^i a^3_i\\ \end{array}
\right ) = a_i\otimes n^i \\
&\quad B =          \left ( \begin{array}{c|c|c}
          n_2^i q^{1 3}_i-n_3^i q^{1 2}_i & n_3^i q^{1 1}_i-n_1^i q^{1 3}_i & n_1^i q^{1 2}_i-n_2^i q^{1 1}_i\\
          n_2^i q^{2 3}_i-n_3^i q^{2 2}_i & n_3^i q^{2 1}_i-n_1^i q^{2 3}_i & n_1^i q^{2 2}_i-n_2^i q^{2 1}_i\\
          n_2^i q^{3 3}_i-n_3^i q^{3 2}_i & n_3^i q^{3 1}_i-n_1^i q^{3 3}_i & n_1^i q^{3 2}_i-n_2^i q^{3 1}_i\\
\end{array}
\right ) = \begin{pmatrix}
            n^i \times \vec{q}_1^i\\
            n^i \times \vec{q}_2^i\\
            n^i \times \vec{q}_3^i
           \end{pmatrix}.
\end{split}
\]
\[
 n^i \in \R^3,\, a_i \in \R^3, Q^i=\begin{pmatrix}
                                       \vec{q}_1^i\\
                                       \vec{q}_2^i\\
                                       \vec{q}_3^i
                                      \end{pmatrix}=
                                      \begin{pmatrix}
                                       q^{11}_i & q^{12}_i & q^{13}_i\\
                                       q^{21}_i & q^{22}_i & q^{23}_i\\
                                       q^{31}_i & q^{32}_i & q^{33}_i
                                      \end{pmatrix} \in \R^{3 \times 3}
\]
As $A_i$'s we take:
\[
 \left(
\begin{array}{ccc}
 1 & 0 & 0 \\
 1 & 0 & 0 \\
 0 & 0 & 0 \\
\end{array}
\right),\left(
\begin{array}{ccc}
 0 & 1 & 0 \\
 0 & 2 & 0 \\
 0 & 0 & 0 \\
\end{array}
\right),\left(
\begin{array}{ccc}
 0 & 0 & 0 \\
 2 & 1 & 0 \\
 0 & 0 & 0 \\
\end{array}
\right),\left(
\begin{array}{ccc}
 15 & 10 & 0 \\
 6 & 4 & 0 \\
 0 & 0 & 0 \\
\end{array}
\right),\left(
\begin{array}{ccc}
 11 & 22 & 0 \\
 9 & 18 & 0 \\
 0 & 0 & 0 \\
\end{array}
\right),\left(
\begin{array}{ccc}
 17 & 17 & 0 \\
 7 & 7 & 0 \\
 0 & 0 & 0 \\
\end{array}
\right),
\]
\[\left(
\begin{array}{ccc}
 47 & 235 & 0 \\
 19 & 95 & 0 \\
 0 & 0 & 0 \\
\end{array}
\right),\left(
\begin{array}{ccc}
 23 & 0 & 0 \\
 9 & 0 & 0 \\
 0 & 0 & 0 \\
\end{array}
\right),\left(
\begin{array}{ccc}
 0 & 41 & 0 \\
 0 & 18 & 0 \\
 0 & 0 & 0 \\
\end{array}
\right),\left(
\begin{array}{ccc}
 -106 & -212 & 0 \\
 -54 & -108 & 0 \\
 0 & 0 & 0 \\
\end{array}
\right),\left(
\begin{array}{ccc}
 -9 & -117 & 0 \\
 -4 & -52 & 0 \\
 0 & 0 & 0 \\
\end{array}
\right),
\]
\[
\left(
\begin{array}{ccc}
 -1 & -1 & -1 \\
 -1 & -1 & -1 \\
 -1 & -1 & -1 \\
\end{array}
\right),\left(
\begin{array}{ccc}
 1 & 3 & 0 \\
 5 & 15 & 0 \\
 0 & 0 & 0 \\
\end{array}
\right),\left(
\begin{array}{ccc}
 1 & 1 & 1 \\
 1 & 1 & 1 \\
 1 & 1 & 1 \\
\end{array}
\right)
\]
As $B_i$'s we take
\[
\left(
\begin{array}{ccc}
 0 & 0 & 1 \\
 0 & 43595 & 0 \\
 0 & 0 & 0 \\
\end{array}
\right),\left(
\begin{array}{ccc}
 61797 & 0 & 0 \\
 -962 & 0 & -1 \\
 0 & 0 & 0 \\
\end{array}
\right),\left(
\begin{array}{ccc}
 24101 & -48202 & 0 \\
 155 & -310 & 0 \\
 0 & 0 & 0 \\
\end{array}
\right),\left(
\begin{array}{ccc}
 112530 & -168795 & 0 \\
 -285490 & 428235 & 0 \\
 0 & 0 & 0 \\
\end{array}
\right),
\]
\[\left(
\begin{array}{ccc}
 534254 & -267127 & 0 \\
 -705914 & 352957 & 0 \\
 0 & 0 & 0 \\
\end{array}
\right),\left(
\begin{array}{ccc}
 198074 & -198074 & 0 \\
 -484546 & 484546 & 0 \\
 0 & 0 & 0 \\
\end{array}
\right),\left(
\begin{array}{ccc}
 2371885 & -474377 & 0 \\
 -5869210 & 1173842 & 0 \\
 0 & 0 & 0 \\
\end{array}
\right),
\]
\[
\left(
\begin{array}{ccc}
 0 & -217131 & 0 \\
 0 & 558346 & 0 \\
 0 & 0 & 1 \\
\end{array}
\right),\left(
\begin{array}{ccc}
 475997 & 0 & 0 \\
 -1086645 & 0 & 0 \\
 0 & 0 & 0 \\
\end{array}
\right),\left(
\begin{array}{ccc}
 -2938894 & 1469447 & 0 \\
 5747526 & -2873763 & 0 \\
 0 & 0 & 0 \\
\end{array}
\right),
\]
\[
\left(
\begin{array}{ccc}
 -1483326 & 114102 & 0 \\
 3324373 & -255721 & 0 \\
 0 & 0 & 1 \\
\end{array}
\right),\left(
\begin{array}{ccc}
 3 & -3 & 0 \\
 383268 & -383269 & 1 \\
 0 & 0 & 0 \\
\end{array}
\right),\left(
\begin{array}{ccc}
 1201149 & -400383 & -367028 \\
 -1859523 & 619841 & 688669 \\
 -6 & 2 & 2 \\
\end{array}
\right),
\]
\[
\left(
\begin{array}{ccc}
 -557570 & 190543 & 367027 \\
 836968 & -148299 & -688669 \\
 6 & -2 & -4 \\
\end{array}
\right)
\]
As $n_i$'s:
\[
 \left(
\begin{array}{c}
 1 \\
 0 \\
 0 \\
\end{array}
\right),\left(
\begin{array}{c}
 0 \\
 1 \\
 0 \\
\end{array}
\right),\left(
\begin{array}{c}
 2 \\
 1 \\
 0 \\
\end{array}
\right),\left(
\begin{array}{c}
 3 \\
 2 \\
 0 \\
\end{array}
\right),\left(
\begin{array}{c}
 1 \\
 2 \\
 0 \\
\end{array}
\right),\left(
\begin{array}{c}
 1 \\
 1 \\
 0 \\
\end{array}
\right),\left(
\begin{array}{c}
 1 \\
 5 \\
 0 \\
\end{array}
\right),\left(
\begin{array}{c}
 1 \\
 0 \\
 0 \\
\end{array}
\right),\left(
\begin{array}{c}
 0 \\
 1 \\
 0 \\
\end{array}
\right),\left(
\begin{array}{c}
 1 \\
 2 \\
 0 \\
\end{array}
\right),\left(
\begin{array}{c}
 1 \\
 13 \\
 0 \\
\end{array}
\right),
\]
\[\left(
\begin{array}{c}
 1 \\
 1 \\
 1 \\
\end{array}
\right),\left(
\begin{array}{c}
 1 \\
 3 \\
 0 \\
\end{array}
\right),\left(
\begin{array}{c}
 1 \\
 1 \\
 1 \\
\end{array}
\right)
\]
As $a_i$'s:
\[
 \left(
\begin{array}{c}
 1 \\
 1 \\
 0 \\
\end{array}
\right),\left(
\begin{array}{c}
 1 \\
 2 \\
 0 \\
\end{array}
\right),\left(
\begin{array}{c}
 0 \\
 1 \\
 0 \\
\end{array}
\right),\left(
\begin{array}{c}
 5 \\
 2 \\
 0 \\
\end{array}
\right),\left(
\begin{array}{c}
 11 \\
 9 \\
 0 \\
\end{array}
\right),\left(
\begin{array}{c}
 17 \\
 7 \\
 0 \\
\end{array}
\right),\left(
\begin{array}{c}
 47 \\
 19 \\
 0 \\
\end{array}
\right),\left(
\begin{array}{c}
 23 \\
 9 \\
 0 \\
\end{array}
\right),\left(
\begin{array}{c}
 41 \\
 18 \\
 0 \\
\end{array}
\right),\left(
\begin{array}{c}
 -106 \\
 -54 \\
 0 \\
\end{array}
\right),
\]
\[
\left(
\begin{array}{c}
 -9 \\
 -4 \\
 0 \\
\end{array}
\right),\left(
\begin{array}{c}
 -1 \\
 -1 \\
 -1 \\
\end{array}
\right),\left(
\begin{array}{c}
 1 \\
 5 \\
 0 \\
\end{array}
\right),\left(
\begin{array}{c}
 1 \\
 1 \\
 1 \\
\end{array}
\right)
\]

As $Q_i$'s:
\[
 \left(
\begin{array}{ccc}
 171 & 1 & 0 \\
 -75 & 0 & -43595 \\
 -155 & 0 & 0 \\
\end{array}
\right),\left(
\begin{array}{ccc}
 0 & 8 & 61797 \\
 1 & 58 & -962 \\
 0 & -101 & 0 \\
\end{array}
\right),\left(
\begin{array}{ccc}
 0 & 0 & 24101 \\
 0 & 0 & 155 \\
 0 & 0 & 0 \\
\end{array}
\right),\left(
\begin{array}{ccc}
 0 & 0 & 56265 \\
 0 & 0 & -142745 \\
 0 & 0 & 0 \\
\end{array}
\right),
\]
\[
\left(
\begin{array}{ccc}
 0 & 0 & 267127 \\
 0 & 0 & -352957 \\
 0 & 0 & 0 \\
\end{array}
\right),\left(
\begin{array}{ccc}
 0 & 0 & 198074 \\
 0 & 0 & -484546 \\
 0 & 0 & 0 \\
\end{array}
\right),\left(
\begin{array}{ccc}
 0 & 0 & 474377 \\
 0 & 0 & -1173842 \\
 0 & 0 & 0 \\
\end{array}
\right),\left(
\begin{array}{ccc}
 65 & 0 & 217131 \\
 91 & 0 & -558346 \\
 -204 & 1 & 0 \\
\end{array}
\right),
\]
\[\left(
\begin{array}{ccc}
 0 & -137 & 475997 \\
 0 & -212 & -1086645 \\
 0 & 1 & 0 \\
\end{array}
\right),\left(
\begin{array}{ccc}
 0 & 0 & -1469447 \\
 0 & 0 & 2873763 \\
 0 & 0 & 0 \\
\end{array}
\right),\left(
\begin{array}{ccc}
 0 & 0 & -114102 \\
 0 & 0 & 255721 \\
 0 & 1 & 0 \\
\end{array}
\right),\left(
\begin{array}{ccc}
 0 & 0 & 3 \\
 0 & 1 & 383269 \\
 0 & 0 & 0 \\
\end{array}
\right),
\]
\[
\left(
\begin{array}{ccc}
 0 & -367028 & 400383 \\
 0 & 688669 & -619841 \\
 0 & 2 & -2 \\
\end{array}
\right),\left(
\begin{array}{ccc}
 0 & 367027 & -190543 \\
 0 & -688669 & 148299 \\
 0 & -4 & 2 \\
\end{array}
\right)
\]
For
\[
\begin{split}
 c&=(0, -193352, -170909, -2127668, -8107342, -10761971, 32323016, \\
&\phantom{=(} -7686922, 5363342, -32272240, 18951257, 2303591, 2997661, 277746),
\end{split}
\]
\[
 m=(-165, 85074, -3130948, -6333, 689, 134, 51, 907, 97, 40, -17, -16, 43, -221)
\]
\[
 d=\left(
\begin{array}{ccccccccc}
 0 & 0 & 0 & 0 & 2 & -182467 & 0 & 1033781 & -25405 \\
 0 & 0 & 0 & 0 & 3 & 0 & 0 & 894424 & -26015 \\
 0 & 0 & 0 & 0 & 1 & 0 & 0 & 863315 & -29063 \\
 0 & 0 & 0 & 0 & 1 & 0 & 121 & -289167 & -27231 \\
 0 & 0 & 0 & 0 & 1 & 0 & 21110 & 1278617 & -28088 \\
 0 & 0 & 0 & 0 & 1 & 1 & 19661 & 320873 & -27865 \\
 0 & 0 & 0 & 0 & 0 & 0 & 2579165 & -7138143 & -24395 \\
 0 & 0 & 0 & 0 & 1 & 0 & 47457 & 785115 & -25976 \\
 0 & 0 & 0 & 0 & 0 & 0 & 257708 & 327493 & -25470 \\
 0 & 0 & 0 & 0 & 1 & 3 & 1191038 & -3050248 & -27779 \\
 0 & 0 & 0 & 0 & 0 & 87 & 3073 & 555789 & -27773 \\
 0 & 0 & 0 & 0 & 266745 & 611036 & 766267 & -1237038 & -24914 \\
 0 & 0 & 0 & 0 & -1556334 & -573705 & -1094427 & 2661105 & -49244 \\
 0 & 0 & 0 & 0 & -103426 & -37004 & -235151 & 639412 & -27210 \\
\end{array}
\right)
\]

we can solve $13\cdot14 = 182$ equations

\begin{equation}\label{eq:15inLaszlo}
\begin{split}
0&> c_i -c_j + \langle Y_i , X_j-X_i\rangle +\sum_{k,\ell=1}^3 d_{i,k,\ell} \brac{\det_{2 \times 2} \brac{X_j-X_i}_{k,\ell}}\\
    &\quad+ m_i \brac{\det_{3 \times 3} X_j - \det_{3 \times 3} X_i-\langle\partial_{X_{\ell \alpha}}\det_{3 \times 3} X_j, X_j-X_i\rangle},
 \end{split}
\end{equation}
for $i,j\in{1,\ldots,14}$, $i\neq j$.

\bibliographystyle{abbrv}%
\bibliography{bib}%

\end{document}